\documentclass[a4paper, 10pt]{book}

\usepackage{amsmath, amsthm, mathtools, tikz-cd}
\usepackage{amsfonts, amssymb, mathrsfs, bbm}
\usepackage{comment, verbatim}
\usepackage{hyperref} \hypersetup{linktocpage}
\usepackage{breakurl}
\usepackage[hypcap]{caption}
\usepackage{enumitem}
\usepackage{array}
\usepackage{multirow}
\usepackage{textcomp} % For the \dcup and \dcap symbols
\usepackage{fancyhdr} % Page header tweaks
\usepackage[figuresleft]{rotating} % For sideways figures
\usepackage[usenames,dvipsnames]{pstricks}
\usepackage{epsfig}
\usepackage{pst-grad} % For gradients
\usepackage{pst-plot} % For axes
\usepackage[linecolor=cyan, bordercolor=cyan, backgroundcolor=white, textsize=scriptsize, disable]{todonotes}

\newcommand{\Obj}{\mathrm{Obj}}
\newcommand{\Mor}{\mathrm{Mor}}
\newcommand{\Hom}{\mathrm{Hom}}
\newcommand{\TL}{\mathrm{TL}}
\newcommand{\id}{\mathrm{id}}
\newcommand{\tr}{\mathrm{tr}}
\newcommand{\Kar}{\mathcal{K}ar\,}
\newcommand{\End}{\mathrm{End}}
\newcommand{\spn}{\mathrm{span}}
\newcommand{\Mat}{\mathcal{M}at\,}
\newcommand{\TLJ}{\mathrm{TLJ}}
\newcommand{\Net}{\mathrm{Net}}
\mathchardef\mhyphen="2D % Define a "math hyphen"
\def \dcup{
	\setbox0 \hbox{$\,\cup\,$}
	\setbox1 \vbox {\hbox {\hss\raisebox{0.75ex}[6pt][0pt]{\kern-0.1ex\textemdash}\hss}}
	\rlap{\box1}{\box0}
}
\def \dcap{
	\setbox0 \hbox{$\,\cap\,$}
	\setbox1 \vbox {\hbox {\hss\raisebox{-0.7ex}[0pt][0pt]{\kern-0.1ex\textemdash}\hss}}
	\rlap{\box1}{\box0}
}
\def\sigmatwo[#1][#2]{
	\setbox0 \hbox{$\displaystyle\sum_{#1}^{#2}$}
	\setbox1 \hbox to \wd0 {\hss$\ \, \scriptstyle 2$\hss}
	\rlap{\box1}\box0
}
\newcommand{\sumtwo}[2]{\sigmatwo[#1][#2]}
\DeclarePairedDelimiter\abs{\lvert}{\rvert}
\DeclareMathOperator{\plus}{+}
\newcommand{\nplus}{\plus_n}
\def\qjpre[#1,#2,#3][#4,#5,#6]{\left\{\begin{matrix}#1&#2&#3\\#4&#5&#6\end{matrix}\right\}}
\newcommand{\qj}[2]{\qjpre[#1][#2]}

\theoremstyle{definition}	\newtheorem{defn}{Definition}[section]

\theoremstyle{plain}		\newtheorem{prop}[defn]{Proposition}
\theoremstyle{plain}		\newtheorem{lemma}[defn]{Lemma}
\theoremstyle{plain}		\newtheorem{theorem}[defn]{Theorem}
\theoremstyle{plain}		\newtheorem{corollary}[defn]{Corollary}
\theoremstyle{remark}		\newtheorem{remark}[defn]{Remark}
\theoremstyle{remark}		\newtheorem*{remark*}{Remark}

\renewcommand{\baselinestretch}{1.2}

\title{The Temperley-Lieb categories and skein modules}

\begin{document}
\begin{titlepage}
\begin{center}
\vspace*{\fill} \Huge
                        The Temperley-Lieb categories and skein modules
\\ \vfill
\Large
                          Joshua Chen
\\
                          May 2014
\\
\vfill\vfill \normalsize
         A thesis submitted for the degree of Bachelor of Science (Hons) \\
         of the Australian National University
\end{center}
\end{titlepage}
\pagestyle{empty}
\frontmatter
\newpage
\newpage
\vspace*{\fill}
\hfill\Large\emph{
Ars longa, vita brevis.
}
\vfill\vfill\vfill

\newpage
\normalsize
\chapter*{Acknowledgements}
\addcontentsline{toc}{chapter}{Acknowledgements}
With deepest acknowledgement and heartfelt thanks to
\begin{itemize}
\item[---] Scott, for being such a kind and patient supervisor and for teaching me everything in this thesis. It's been a blast learning from you.
\item[---] Vigleik, Dennis and David, who I probably exasperated by being such a slow learner in their classes. Thank you for challenging me to aim higher and work harder!
\item[---] Joan, without whom the proof of the spine theorem would be in even worse shape than it currently is.
\item[---] Pete, Benedict, Chris, Yitao, Hannah, Hannah, Angus, Chat, Seb, Ben, Mark and Michael. Good company, stimulating conversation, loads of maths and laughs. What more could a friend ask for?
\item[---] Mum, Dan, Sam and Ben, my loving and supportive mum and brothers! Thank you for your love and prayers.
\item[---] My God and saviour Jesus, for blessing and sustaining me every step of the long way. \emph{Soli Deo gloria}, to him be the glory.
\end{itemize}

\tableofcontents

\mainmatter
\pagestyle{fancy}
\setcounter{secnumdepth}{-1}
\chapter{Introduction}
In this thesis we develop the theory of diagrammatic Temperley-Lieb categories in order to construct examples of spherical fusion categories, which we then use to define Turaev-Viro skein modules for $n$-holed disks.

The basic Temperley-Lieb category contains as endomorphism spaces the Temperley-Lieb algebras, which have many surprising links to statistical mechanics, knot theory and representation theory.
From this basic definition one constructs the Temperley-Lieb-Jones categories, which have very nice structure: at roots of unity $q$ they are equivalent as braided spherical tensor categories to the semisimplified category $\mathcal{R}ep\,U_q(\mathfrak{sl}_2)$ of representations of the quantum algebra $U_q(\mathfrak{sl}_2)$.
(We do not address this any further in this thesis, the interested reader may see for example \cite{ST2008} or Chapter XII of \cite{Turaev1994}).
Further, the Temperley-Lieb-Jones categories are examples of spherical fusion categories, that is, semisimple spherical linear categories with additional nice properties, and one of the reasons these categories are interesting is that they allow us to construct skein modules.

A skein module is a module associated to a $2$-manifold, and is the first step towards constructing a $(2+1)$-dimensional \emph{topological quantum field theory} (TQFT).
Briefly speaking, a $(n+1)$-dimensional TQFT is a functor from $(n+1)$-$\mathbf{Cob}$ to $\mathbf{FdVect}$, assigning to every $n$-manifold a finite-dimensional vector space and to every $(n+1)$-cobordism between $n$-manifolds a linear transformation between the corresponding vector spaces, in a manner that respects composition and the rigid symmetric monoidal structure of the categories.
These were first used to construct topological invariants by Witten in a seminal paper in 1989 \cite{Witten1989}, and were axiomatized by Atiyah \cite{Atiyah1988} around the same time.
(For a general introduction to topological quantum field theory from the algebraic point of view see for instance \cite{Atiyah1988} or the first section of \cite{Baez1995}.)

Following Witten's work, one of the next TQFTs to be discovered was the Turaev-Viro TQFT \cite{TV1992}, which takes as input a spherical fusion category in order to construct vector spaces (free modules) for $2$-surfaces and linear maps for $3$-cobordisms.
In this thesis we deal only with the $2$-dimensional aspect of the theory and use Temperley-Lieb-Jones at roots of unity to construct skein module invariants associated to a specific class of surfaces, namely $n$-holed disks.

The outline of this thesis is as follows.
Chapter 1 begins with some preliminary definitions and results.
In Chapter 2 we define generic Temperley-Lieb, introduce the all-important Jones-Wenzl idempotents and use them to construct the generic Temperley-Lieb-Jones categories $\TLJ$.
In Chapter 3 we take a necessary detour into some Temperley-Lieb skein theory, proving the results we will need in order to show that generic $\TLJ$ is semisimple, which we do in Chapter 4.
In Chapter 5 we consider what happens for $\TLJ$ evaluated at a root of unity, and show that after taking the quotient by the negligible ideal we obtain a semisimple category with finitely many simple objects, which is also spherical fusion.
Finally, in Chapter 6 we present an alternative approach to constructing the Turaev-Viro skein modules for $n$-holed disks.

\setcounter{secnumdepth}{1}
\chapter{Preliminaries}
In this chapter we introduce some basic notions and notation that will be used throughout the rest of this thesis.

For a category $\mathcal{C}$ let $\Obj(\mathcal{C})$ denote the set of objects of $\mathcal{C}$ and $\Mor(\mathcal{C})$ the set of all morphisms in $\mathcal{C}$.
We write $\mathbbm{1}_a$ for the identity morphism on $a$.
All our categories are small, that is, $\Obj(\mathcal{C})$ and $\Mor(\mathcal{C})$ are sets.

\begin{defn}
A \textbf{monoidal category} $(\mathcal{C}, \otimes, e)$ is a category $\mathcal{C}$ together with a \emph{tensor product} bifunctor $\otimes \colon \mathcal{C} \times \mathcal{C} \rightarrow \mathcal{C}$ and a distinguished object $e \in \Obj(\mathcal{C})$ satisfying the following axioms:
\begin{enumerate}
\item (Identity).
There are natural isomorphisms $\lambda$ and $\rho$ with components
\[ \lambda_a \colon e \otimes a \cong a \]
and
\[ \rho_a \colon a \otimes e \cong a \]
for each $a \in \Obj(\mathcal{C})$.

\item (Associativity).
There is a natural isomorphism $\alpha$ with components
\[ \alpha_{a,b,c} \colon (a \otimes b)\otimes c \cong a\otimes(b\otimes c) \]
for all $a,b,c \in \Obj(\mathcal{C})$.

\item (Coherence).
The triangle diagram
\[
\begin{tikzcd}[column sep=small]
(a\otimes e)\otimes b \arrow[swap]{rd}{\rho_a\otimes\mathbbm{1}_b} \arrow{rr}{\alpha_{a,e,b}} & & a\otimes(e\otimes b) \arrow{ld}{\mathbbm{1}_a\otimes\lambda_b} \\
& a\otimes b &
\end{tikzcd}
\]
and the pentagon diagram
\[
\begin{tikzcd}[column sep=small]
& ((a\otimes b)\otimes c)\otimes d \arrow[swap]{dl}{\alpha_{a,b,c}\otimes\mathbbm{1}_d} \arrow{dr}{\alpha_{a\otimes b,c,d}} \\
(a\otimes(b\otimes c))\otimes d \arrow{d}{\alpha_{a,b\otimes c,d}} & & (a\otimes b)\otimes(c\otimes d) \arrow{d}{\alpha_{a,b,c\otimes d}} \\
a\otimes((b\otimes c)\otimes d) \arrow[swap]{rr}{\mathbbm{1}_a\otimes\alpha_{b,c,d}} & & a\otimes(b\otimes(c\otimes d))
\end{tikzcd}
\]
are commutative for all $a,b,c,d \in \Obj(\mathcal{C})$.
This implies that the order in which we parenthesize a tensor product of $a_1, \dotsc, a_n$ (with arbitrary insertions of the tensor identity $e$) does not matter: any two such parenthesized tensor products $x$ and $y$ are isomorphic via a sequence of morphisms $\lambda, \rho, \alpha$ and their inverses, and furthermore any two such sequences beginning at $x$ and ending at $y$ in fact yield the same isomorphism $x\cong y$.
(See Chapter VII of \cite{MacLane1998} or Section 1.9 of \cite{Etingof2009} for more information.)
\end{enumerate}

Note that $\otimes$ being a bifunctor means that in particular the ``exchange relation'' $(f\otimes g)\circ (h\otimes k) = (f\circ h)\otimes(g\circ k)$ holds for all morphisms $f,g,h,k \in \Mor(\mathcal{C})$.

A monoidal category is called \textbf{strict} if the natural isomorphisms $\lambda$, $\rho$ and $\alpha$ are in fact identities.
\end{defn}

We will construct skein modules using monoidal categories that have some additional structure.
Here we introduce the first of these.

\begin{defn}
A \textbf{linear category} is a category enriched over the category of vector spaces.
Explicitly, a $\mathbbm{F}$-linear category is a category whose hom-sets are vector spaces over some field $\mathbbm{F}$, in which composition of morphisms is bilinear with respect to addition in the hom-space.
A \textbf{linear monoidal category} is category that is linear, monoidal, and whose tensor product is bilinear with respect to addition.
\end{defn}

We then have the following easy fact, stated without proof:

\begin{prop}
The endomorphism spaces of $\mathbbm{F}$-linear categories are in fact $\mathbbm{F}$-algebras, with multiplication given by composition of morphisms.
\label{prop:endo_alg}
\todo{Also every hom-space $\Hom(x,y)$ is a $\Hom(y,y),\Hom(x,x)$ bimodule in the obvious way, but I'm not sure this needs to be mentioned.}
\end{prop}

\noindent For this reason $\mathbbm{F}$-linear categories are often known as \textit{$\mathbbm{F}$-algebroids} in the literature.

Until stated otherwise, throughout this paper we will take our ground field to be $\mathbbm{F} = \mathbbm{C}(q)$, the fraction field of the ring of complex polynomials in a formal parameter $q$.

\begin{defn}
\label{defn:quantumint}
The \textbf{$n$-th quantum integer} $[n]$ is the element of $\mathbbm{F}$ given by
\begin{align*}
[n] & \coloneqq \frac{q^n - q^{-n}}{q-q^{-1}} \\
& = q^{-n+1} + q^{-n+3} + \dotsb + q^{n-3} + q^{n-1} \quad \text{when} \ n \neq 0.
\end{align*} 
\end{defn}

An important relation we will make use of is the following, whose proof follows easily from the definitions.

\begin{prop}[Recursion relation for ${[}n{]}$]
\label{prop:quantumint_relation}
For all integers $n>0$,
\[ [n+1] = [2][n] - [n-1]. \]
\end{prop}

Next we define the Temperley-Lieb diagrams.

\begin{defn}
Let $m,n$ be non-negative integers, $I$ the unit interval $[0,1]$ and consider the unit square $I \times I$ with $m$ and $n$ points distinguished on the bottom and top edges $I\times \{0\}$ and $I\times \{1\}$ respectively.
For $m+n$ even, a \textbf{simple Temperley-Lieb (TL) diagram} from $m$ to $n$ points consists of smooth arcs with endpoints on a top or bottom edge connecting pairs of distinguished points, together with finitely many (possibly zero) loops drawn in the unit square.
All arcs and loops are mutually disjoint, and planar isotopic diagrams are to be considered equal.
If $m+n$ is odd, there are no simple TL diagrams from $m$ to $n$ points.
\end{defn}

Arcs connecting two points on the top edge are known as \emph{cups}, those connecting points on the bottom edge \emph{caps}, and arcs connecting a point on the top edge with a point on the bottom edge are called \emph{through-strings}.
For convenience we call points connected by cups and caps, \emph{capped points} and \emph{cupped points} respectively.

There is a composition rule for simple TL diagrams: if $f \colon m \rightarrow n$ and $g \colon n \rightarrow k$ are diagrams from $m$ to $n$ and $n$ to $k$ points respectively, their composition $g\circ f \colon m \rightarrow k$ is the diagram obtained by stacking $g$ on top of $f$, joining the ends of the corresponding arcs and rescaling the diagram into the unit square (Fig. \ref{fig:TL_composition}).

Note also that $m$ and $n$ are allowed to be zero, in which case the simple diagrams consist of zero or more disjoint loops in the square.

% ---------------- Figure: Composition of TL diagrams ----------------
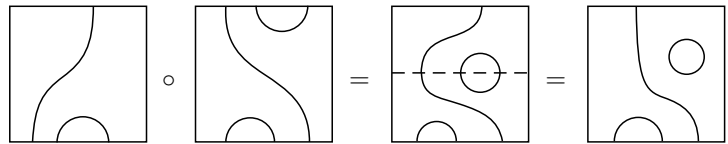
\begin{figure}[htb]
\[
\begin{gathered}
\psscalebox{1.0 1.0} % Change this value to rescale the drawing.
{
\begin{pspicture}(0,-0.9090909)(1.8181823,0.9090909)
\rput{-180.0}(1.8181823,0.0){\psframe[linecolor=black, linewidth=0.02, dimen=outer](1.8181821,0.9090907)(0.0,-0.9090911)}
\rput{0.0}(0.0,0.0){\psarc[linecolor=black, linewidth=0.02, dimen=outer](0.9713553,-0.90566057){0.33962265}{0.0}{180.0}}
\psbezier[linecolor=black, linewidth=0.02](0.30909118,-0.9000002)(0.33676416,-0.40377378)(0.4424245,-0.23619232)(0.7090912,-0.03619232)(0.97575784,0.16380768)(1.1090912,0.36380768)(1.1090912,0.8999998)
\end{pspicture}
}
\end{gathered}
\ \circ \ 
\begin{gathered}
\psscalebox{1.0 1.0} % Change this value to rescale the drawing.
{
\begin{pspicture}(0,-0.9090912)(1.8181823,0.9090912)
\psframe[linecolor=black, linewidth=0.02, dimen=outer](1.8181821,0.90909094)(0.0,-0.90909094)
\rput{-180.0}(2.293654,1.8113208){\psarc[linecolor=black, linewidth=0.02, dimen=outer](1.146827,0.9056604){0.33962265}{0.0}{180.0}}
\psbezier[linecolor=black, linewidth=0.02](0.40909117,0.9)(0.38141823,0.40377358)(0.60909116,0.2)(0.9090912,0.0)(1.2090912,-0.2)(1.5090911,-0.4)(1.5090911,-0.9)
\psarc[linecolor=black, linewidth=0.02, dimen=outer](0.72909117,-0.9){0.32}{0.0}{180.0}
\end{pspicture}
}
\end{gathered}
\ = \  
\begin{gathered}
\psscalebox{1.0 1.0} % Change this value to rescale the drawing.
{
\begin{pspicture}(0,-0.9090912)(1.8181823,0.9090912)
\psframe[linecolor=black, linewidth=0.02, dimen=outer](1.8181821,0.90909094)(0.0,-0.90909094)
\psline[linecolor=black, linewidth=0.02, linestyle=dashed, dash=0.17638889cm 0.10583334cm](0.012316993,0.016129032)(1.7865106,0.016129032)
\psarc[linecolor=black, linewidth=0.02, dimen=outer](1.1736073,0.016129032){0.2580645}{0.0}{180.0}
\psarc[linecolor=black, linewidth=0.02, dimen=outer](0.59941375,-0.88709676){0.2580645}{0.0}{180.0}
\psbezier[linecolor=black, linewidth=0.02](0.39941376,0.016129032)(0.40909117,-0.3)(0.60909116,-0.3)(0.9090912,-0.4)(1.2090912,-0.5)(1.4090912,-0.6)(1.4639299,-0.88709676)
\psbezier[linecolor=black, linewidth=0.02](0.39480546,0.014285714)(0.40909117,0.3)(0.5090912,0.4)(0.8090912,0.5)(1.1090912,0.6)(1.1805197,0.7)(1.1948055,0.9)
\rput{-180.0}(2.3453965,0.034076247){\psarc[linecolor=black, linewidth=0.02, dimen=outer](1.1726983,0.017038124){0.2580645}{0.0}{180.0}}
\end{pspicture}
}
\end{gathered}
\ = \ 
\begin{gathered}
\psscalebox{1.0 1.0} % Change this value to rescale the drawing.
{
\begin{pspicture}(0,-0.9090912)(1.8181823,0.9090912)
\psframe[linecolor=black, linewidth=0.02, dimen=outer](1.8181821,0.90909094)(0.0,-0.90909094)
\psarc[linecolor=black, linewidth=0.02, dimen=outer](0.6733864,-0.8898365){0.31559876}{0.0}{180.0}
\psbezier[linecolor=black, linewidth=0.02](0.64623266,0.891989)(0.6559101,-0.10875542)(0.75967073,-0.22370917)(1.0089858,-0.31159115)(1.2583009,-0.39947313)(1.4529268,-0.5287671)(1.4639299,-0.88709676)
\pscircle[linecolor=black, linewidth=0.02, dimen=outer](1.3090912,0.22876713){0.24109589}
\end{pspicture}
}
\end{gathered}
\]
\caption{Composition of TL diagrams. \label{fig:TL_composition}}
\end{figure}
% ---------------- END FIG ----------------

\chapter{The Temperley-Lieb categories}
\section{Generic Temperley-Lieb}
As indicated in the previous section, we will until otherwise stated work over $\mathbbm{F} = \mathbbm{C}(q)$ where $q$ is a formal parameter.

\begin{defn}	
The \textbf{generic Temperley-Lieb category} $\TL$ has as objects non-negative integers $n \in \mathbbm{N}$, with morphisms defined as follows.
$\Hom(m,n)$ is defined to be the $\mathbbm{F}$-linear span of all simple TL diagrams from $m$ to $n$ points, modulo the \emph{$d$-equivalence relation} --- a diagram with a loop is equal to $d$ times the same diagram without the loop, where $d = [2] = q+q^{-1}$ is the \emph{loop variable} (Fig. \ref{fig:TL_dequiv}).
Composition of morphisms is given by the composition of simple TL diagrams extended linearly, and $\mathbbm{1}_n$ is the diagram containing exactly $n$ through-strings.
\label{defn:TL_category}
\end{defn}

% ---------------- Figure: d-equivalence of TL diagrams ------------------
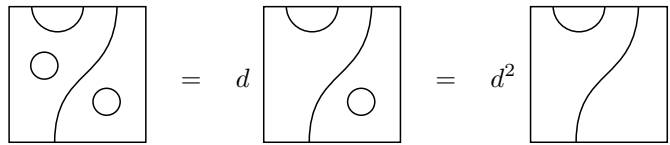
\begin{figure}[htb]
\[
\begin{gathered}
\psscalebox{1.0 1.0} % Change this value to rescale the drawing.
{
\begin{pspicture}(0,-0.9090912)(1.8181823,0.9090912)
\psframe[linecolor=black, linewidth=0.02, dimen=outer](1.8181821,0.90909094)(0.0,-0.90909094)
\rput{-180.0}(1.2936541,1.8113208){\psarc[linecolor=black, linewidth=0.02, dimen=outer](0.64682704,0.9056604){0.33962265}{0.0}{180.0}}
\pscircle[linecolor=black, linewidth=0.02, dimen=outer](0.47324213,0.116981134){0.18867925}
\pscircle[linecolor=black, linewidth=0.02, dimen=outer](1.2921101,-0.36226416){0.18867925}
\psbezier[linecolor=black, linewidth=0.02](1.4298459,0.8962264)(1.4090912,0.4)(1.2090912,0.2)(1.0090911,0.0)(0.8090912,-0.2)(0.60909116,-0.4)(0.60909116,-0.9)
\end{pspicture}
}
\end{gathered}
\quad = \quad d \ 
\begin{gathered}
\psscalebox{1.0 1.0} % Change this value to rescale the drawing.
{
\begin{pspicture}(0,-0.9090912)(1.8181823,0.9090912)
\psframe[linecolor=black, linewidth=0.02, dimen=outer](1.8181821,0.90909094)(0.0,-0.90909094)
\rput{-180.0}(1.2936541,1.8113208){\psarc[linecolor=black, linewidth=0.02, dimen=outer](0.64682704,0.9056604){0.33962265}{0.0}{180.0}}
\pscircle[linecolor=black, linewidth=0.02, dimen=outer](1.2921101,-0.36226416){0.18867925}
\psbezier[linecolor=black, linewidth=0.02](1.4298459,0.8962264)(1.4090912,0.4)(1.2090912,0.2)(1.0090911,0.0)(0.8090912,-0.2)(0.60909116,-0.4)(0.60909116,-0.9)
\end{pspicture}
}
\end{gathered}
\quad = \quad d^2 \ 
\begin{gathered}
\psscalebox{1.0 1.0} % Change this value to rescale the drawing.
{
\begin{pspicture}(0,-0.9090912)(1.8181823,0.9090912)
\psframe[linecolor=black, linewidth=0.02, dimen=outer](1.8181821,0.90909094)(0.0,-0.90909094)
\rput{-180.0}(1.2936541,1.8113208){\psarc[linecolor=black, linewidth=0.02, dimen=outer](0.64682704,0.9056604){0.33962265}{0.0}{180.0}}
\psbezier[linecolor=black, linewidth=0.02](1.4298459,0.8962264)(1.4090912,0.4)(1.2090912,0.2)(1.0090911,0.0)(0.8090912,-0.2)(0.60909116,-0.4)(0.60909116,-0.9)
\end{pspicture}
}
\end{gathered}
\]
\caption{$d$-equivalence of TL diagrams. \label{fig:TL_dequiv}}
\end{figure}
% ---------------- END FIG ----------------

We call the morphisms in $\TL$ \emph{formal diagrams} to distinguish them from the simple TL diagrams, though in the interests of brevity we will often refer to them as \emph{diagrams}.
Whenever we really mean simple TL diagrams this will be stated explicitly.

\begin{prop}
$\TL$ is a strict linear monoidal category: it has tensor product given by $a \otimes b = a+b$ for objects $a$, $b$.
For simple diagrams $f \colon m \rightarrow n$, $g \colon m^\prime \rightarrow n^\prime$, the tensor product $f \otimes g \colon m \otimes m^\prime \rightarrow n \otimes n^\prime$ is the simple diagram formed by juxtaposition, placing $f$ to the left of $g$ (Fig. \ref{fig:tensor_diag}) and rescaling into the unit square.
Extending bilinearly gives the tensor product of formal diagrams.
\end{prop}

\begin{proof}
This is a simple exercise in verifying the axioms:
since $\otimes$ on objects is addition of integers, the tensor product is strictly associative with tensor identity $0$, and the components of $\lambda$ and $\rho$ are equalities.
Since $\lambda, \rho, \alpha$ are all equalities the triangle and pentagon diagrams are trivially commutative.
Finally, by construction the hom-sets of $\TL$ are $\mathbbm{F}$-vector spaces, and tensor product and composition distribute over addition.
\end{proof}

% --------------- FIG: Tensor product of TL diagrams ---------------
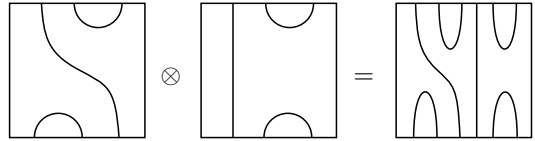
\begin{figure}[hbt]
\[
\begin{gathered}
\psscalebox{1.0 1.0} % Change this value to rescale the drawing.
{
\begin{pspicture}(0,-0.9)(1.8,0.9)
\psframe[linecolor=black, linewidth=0.02, dimen=outer](1.8,0.9)(0.0,-0.9)
\psarc[linecolor=black, linewidth=0.02, dimen=outer](0.6509804,-0.88235295){0.3137255}{0.0}{180.0}
\psarc[linecolor=black, linewidth=0.02, dimen=outer](1.1745098,0.88235295){0.3137255}{180.0}{360.0}
\psbezier[linecolor=black, linewidth=0.02](0.42941177,0.88235295)(0.45882353,0.46470588)(0.5,0.23921569)(0.9,0.019607844)(1.3,-0.2)(1.4,-0.2)(1.4490196,-0.88235295)
\end{pspicture}
}
\end{gathered}
\ \otimes \ 
\begin{gathered}
\psscalebox{1.0 1.0} % Change this value to rescale the drawing.
{
\begin{pspicture}(0,-0.9)(1.8,0.9)
\psframe[linecolor=black, linewidth=0.02, dimen=outer](1.8,0.9)(0.0,-0.9)
\psarc[linecolor=black, linewidth=0.02, dimen=outer](1.1509804,-0.88235295){0.3137255}{0.0}{180.0}
\psarc[linecolor=black, linewidth=0.02, dimen=outer](1.1745098,0.88235295){0.3137255}{180.0}{360.0}
\psline[linecolor=black, linewidth=0.02](0.42941177,0.88235295)(0.42941177,-0.88235295)
\end{pspicture}
}
\end{gathered}
\ = \ 
\begin{gathered}
\psscalebox{1.0 1.0} % Change this value to rescale the drawing.
{
\begin{pspicture}(0,-0.9)(1.8,0.9)
\psframe[linecolor=black, linewidth=0.02, dimen=outer](1.8,0.9)(0.0,-0.9)
\psbezier[linecolor=black, linewidth=0.02](0.2653906,0.88235295)(0.28222606,0.46470588)(0.3057957,0.23921569)(0.53475785,0.019607844)(0.76372,-0.2)(0.8209605,-0.2)(0.8490196,-0.88235295)
\psline[linecolor=black, linewidth=0.02](1.0694118,0.88001525)(1.0694118,-0.8846906)
\psbezier[linecolor=black, linewidth=0.02](1.2844156,0.8805195)(1.2844156,0.08051948)(1.5909091,0.08571429)(1.5909091,0.8857143)
\psbezier[linecolor=black, linewidth=0.02](1.6012987,-0.8805195)(1.6012987,-0.08051948)(1.2844156,-0.09090909)(1.2844156,-0.8909091)
\psbezier[linecolor=black, linewidth=0.02](0.57272726,0.8805195)(0.57272726,0.08051948)(0.874026,0.08571429)(0.874026,0.8857143)
\psbezier[linecolor=black, linewidth=0.02](0.5467532,-0.8857143)(0.5467532,-0.08571429)(0.24025974,-0.08051948)(0.24025974,-0.8805195)
\end{pspicture}
}
\end{gathered}
\]
\caption{Tensor product of TL diagrams. \label{fig:tensor_diag}}
\end{figure}
% --------------- END FIG ---------------

\begin{remark}
\label{rem:loopless-diagrams}
The $d$-equivalence relation means that all morphisms $f\in\Hom(m,n)$ can be written canonically as a linear combination of simple TL diagrams, each having no closed loops.
Thus the set of all simple TL diagrams from $m$ to $n$ points without loops forms a basis for $\Hom(m,n)$, and it is not hard to show that there are exactly the Catalan number $c_k = \frac{1}{k+1}\binom{2k}{k}$ of these diagrams, where $k = \frac{m+n}{2}$. (See for instance Proposition 1.12 of \cite{Wang2010}.)
Hence in particular $\Hom(m,n)$ is finite-dimensional.
\end{remark}

\begin{remark}
\label{rem:identifying-Hom(0,0)-with-F}
Since the only diagram in $\Hom(0,0)$ that does not contain any closed loops is the empty diagram $\mathbbm{1}_0$, this means that every morphism $f\in\Hom(0,0)$ is a $\mathbbm{F}$-multiple of $\mathbbm{1}_0$.
Thus we may canonically identify $\Hom(0,0)$ with the field $\mathbbm{F}$.
\end{remark}

We define the following operations on morphisms in $\TL$:

\begin{defn}
There is an \textbf{anti-involution} $\overline{\phantom{f}}$ of formal diagrams which takes $f \colon n \rightarrow m$ to the diagram $\overline{f} \colon m \rightarrow n$ obtained by complex-conjugating the coefficients of $f$, sending $q \mapsto q^{-1}$ and reflecting every simple diagram in $f$ in the horiontal line $I\times \frac{1}{2}$ through the middle of the diagram (Fig. \ref{fig:antiinv-dual-TL_diags}).
\end{defn}

\begin{defn}
\label{defn:TL-dual}
The \textbf{dual} of a simple TL diagram $f \colon m \rightarrow n$ is the diagram $f^\ast \colon n \rightarrow m$ obtained by rotating $f$ around its center by $180^\circ$ (Fig. \ref{fig:antiinv-dual-TL_diags}); this is extended to all formal diagrams by complex-conjugating coefficients and sending $q \mapsto q^{-1}$.
\end{defn}

% ------- FIG: Anti-involution and dual of TL diagrams --------
\begin{figure}[ht]
\begin{align*}
\overline{
\left((1+i)\,
\begin{gathered}
\psscalebox{1.0 1.0} % Change this value to rescale the drawing.
{
\begin{pspicture}(0,-0.60821617)(1.2164323,0.60821617)
\psframe[linecolor=black, linewidth=0.02, dimen=outer](1.2164325,0.6082162)(0.0,-0.6082162)
\psarc[linecolor=black, linewidth=0.02, dimen=outer](0.46016893,-0.5993919){0.21225}{0.0}{180.0}
\psbezier[linecolor=black, linewidth=0.02](0.98766893,-0.5936419)(0.9813847,-0.30199093)(0.78240997,-0.12934202)(0.61266893,-0.012391891)(0.44292787,0.10455824)(0.24810052,0.26305276)(0.24391891,0.6063581)
\psarc[linecolor=black, linewidth=0.02, dimen=outer](0.7601689,0.5943581){0.21225}{180.0}{360.0}
\end{pspicture}
}
\end{gathered} \ 
\right)
}
\ &= \ 
(1-i)\,
\begin{gathered}
\psscalebox{1.0 1.0} % Change this value to rescale the drawing.
{
\begin{pspicture}(0,-0.60821617)(1.2164323,0.60821617)
\psframe[linecolor=black, linewidth=0.02, dimen=outer](1.2164325,0.6082162)(0.0,-0.6082162)
\psbezier[linecolor=black, linewidth=0.02](0.23766892,-0.5936419)(0.23138472,-0.30199093)(0.44490996,-0.060592018)(0.61266893,0.018858109)(0.7804279,0.098308235)(0.98560053,0.2534142)(0.9814189,0.59671956)
\psarc[linecolor=black, linewidth=0.02, dimen=outer](0.7695213,-0.59868705){0.20722891}{0.0}{180.0}
\psarc[linecolor=black, linewidth=0.02, dimen=outer](0.44662976,0.60131294){0.21686748}{180.0}{360.0}
\end{pspicture}
}
\end{gathered} \\\\
\begin{gathered}
\psscalebox{1.0 1.0} % Change this value to rescale the drawing.
{
\begin{pspicture}(0,-0.60821617)(1.2164323,0.60821617)
\psframe[linecolor=black, linewidth=0.02, dimen=outer](1.2164325,0.6082162)(0.0,-0.6082162)
\psarc[linecolor=black, linewidth=0.02, dimen=outer](0.43217194,-0.59868705){0.20722891}{0.0}{180.0}
\psarc[linecolor=black, linewidth=0.02, dimen=outer](0.44662976,0.60131294){0.21686748}{180.0}{360.0}
\psline[linecolor=black, linewidth=0.02](0.938196,0.5916744)(0.938196,-0.59868705)
\end{pspicture}
}
\end{gathered}^{\displaystyle\,\ast}
&= \ 
\begin{gathered}
\psscalebox{1.0 1.0} % Change this value to rescale the drawing.
{
\begin{pspicture}(0,-0.6082162)(1.2164322,0.6082162)
\psframe[linecolor=black, linewidth=0.02, dimen=outer](1.2164323,0.6082163)(0.0,-0.6082161)
\psarc[linecolor=black, linewidth=0.02, dimen=outer](0.75661623,-0.598687){0.20722891}{0.0}{180.0}
\psarc[linecolor=black, linewidth=0.02, dimen=outer](0.75774074,0.601313){0.21686748}{180.0}{360.0}
\psline[linecolor=black, linewidth=0.02](0.25375146,0.59167445)(0.25375146,-0.598687)
\end{pspicture}
}
\end{gathered}
\end{align*}
\caption{Anti-involution and dual of TL diagrams.}\label{fig:antiinv-dual-TL_diags}
\end{figure}
% --------------- END FIG ---------------

\begin{defn}
\label{defn:TL-trace}
Let $f \in \Hom(n,n)$ be a simple diagram on $n$ points.
The \textbf{trace} of $f$ is the diagram $\tr(f)$ obtained by joining corresponding points on the top and bottom edges by $n$ disjoint arcs drawn around the outside of $f$ (Fig. \ref{fig:TL_trace}).
Extending linearly gives the trace of any endomorphism on $n$ points.
\end{defn}

\begin{remark}
\label{rem:trace-map}
By Remark \ref{rem:identifying-Hom(0,0)-with-F}, we consider the trace as a map $\tr\colon\Hom(n,n)\rightarrow\mathbbm{F}$.
\end{remark}

% ---------------- Figure: Trace of TL diagrams ----------------
\begin{figure}[htb]
\[
\tr\left(
\begin{gathered}
\psscalebox{1.0 1.0} % Change this value to rescale the drawing.
{
\begin{pspicture}(0,-0.6075906)(1.2151812,0.6075906)
\psframe[linecolor=black, linewidth=0.02, dimen=outer](1.2151812,0.6075906)(0.0,-0.60759044)
\psarc[linecolor=black, linewidth=0.02, dimen=outer](0.42679742,-0.5933381){0.1952381}{0.0}{180.0}
\psarc[linecolor=black, linewidth=0.02, dimen=outer](0.8077498,0.5971381){0.1952381}{180.0}{360.0}
\psbezier[linecolor=black, linewidth=0.02](0.22679743,0.60190004)(0.23235384,0.11520752)(0.5226748,0.039045878)(0.62679744,0.0066619436)(0.73092,-0.02572199)(1.0030128,-0.0562186)(1.0125117,-0.6028619)
\end{pspicture}
}
\end{gathered}
\right) \ = \ 
\begin{gathered}
\psscalebox{1.0 1.0} % Change this value to rescale the drawing.
{
\begin{pspicture}(0,-1.5816131)(2.1817126,1.5816131)
\psframe[linecolor=black, linewidth=0.02, dimen=outer](1.2151812,0.61079055)(0.0,-0.6043905)
\psarc[linecolor=black, linewidth=0.02, dimen=outer](1.1962291,0.6038711){0.19354838}{0.0}{180.0}
\psarc[linecolor=black, linewidth=0.02, dimen=outer](1.1962291,0.6038711){0.58064514}{0.0}{180.0}
\psarc[linecolor=black, linewidth=0.02, dimen=outer](1.1962291,0.6038711){0.9677419}{0.0}{180.0}
\rput{-180.0}(2.407942,-1.2077417){\psarc[linecolor=black, linewidth=0.02, dimen=outer](1.203971,-0.60387087){0.19354838}{0.0}{180.0}}
\rput{-180.0}(2.407942,-1.2077417){\psarc[linecolor=black, linewidth=0.02, dimen=outer](1.203971,-0.60387087){0.58064514}{0.0}{180.0}}
\rput{-180.0}(2.407942,-1.2077417){\psarc[linecolor=black, linewidth=0.02, dimen=outer](1.203971,-0.60387087){0.9677419}{0.0}{180.0}}
\psline[linecolor=black, linewidth=0.02](1.3887022,0.6051)(1.3934641,-0.6044238)
\psline[linecolor=black, linewidth=0.02](1.7744164,0.6057349)(1.7839403,-0.60855085)
\psline[linecolor=black, linewidth=0.02](2.1601307,0.6095444)(2.1696546,-0.5999794)
\psarc[linecolor=black, linewidth=0.02, dimen=outer](0.42679742,-0.59013814){0.1952381}{0.0}{180.0}
\psarc[linecolor=black, linewidth=0.02, dimen=outer](0.8077498,0.60033804){0.1952381}{180.0}{360.0}
\psbezier[linecolor=black, linewidth=0.02](0.22679743,0.6051)(0.23235384,0.11840744)(0.5226748,0.042245798)(0.62679744,0.009861864)(0.73092,-0.022522068)(1.0030128,-0.05301868)(1.0125117,-0.59966195)
\end{pspicture}
}
\end{gathered}
\ \ = \ \ d
\]
\caption{Trace of TL diagrams. \label{fig:TL_trace}}
\end{figure}
% ---------------- END FIG ----------------

\section{Hom-spaces and ideals in Temperley-Lieb}
Recall from Proposition \ref{prop:endo_alg} that the endomorphism spaces of $\mathbbm{F}$-linear categories are in fact $\mathbbm{F}$-algebras, with multiplication in the algebra given by composition in the category.
Because of this we will often abuse notation and write $gf = g\circ f$ for the composition of arbitrary morphisms $f,g \in \Mor(\TL)$.

\begin{defn}
The \textbf{$n$-th Temperley-Lieb algebra} $\TL_n$ is the endomorphism space $\Hom(n,n)$ in $\TL$ consisting of all formal diagrams on $n$ points.
\end{defn}

$\TL_n$ is finitely generated as an algebra by the $n$ simple diagrams $\id_n, U_1, \dotsc, U_{n-1}$, where $\id_n=\mathbbm{1}_n$ is the identity diagram with $n$ through-strings, and $U_i$ is the simple diagram with a cup joining the $i$-th and $(i+1)$-th points on the top edge, a cap joining the $i$-th and $(i+1)$-th points on the bottom edge and through-strings connecting the remaining points (Fig. \ref{fig:TL_gen}).
% ---------------- Figure: Generators of TL_n ------------------
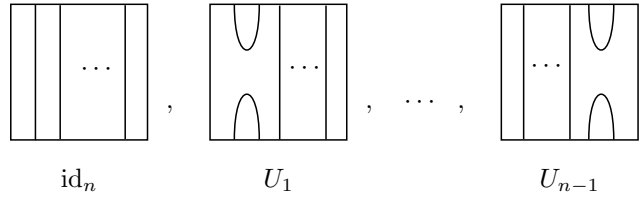
\begin{figure}[ht]
\[
\begin{gathered}
\psscalebox{1.0 1.0} % Change this value to rescale the drawing.
{
\begin{pspicture}(0,-0.9090912)(1.8181823,0.9090912)
\psframe[linecolor=black, linewidth=0.02, dimen=outer](1.8181821,0.90909094)(0.0,-0.90909094)
\psline[linecolor=black, linewidth=0.02](0.6596288,0.88709676)(0.6596288,-0.88709676)
\psline[linecolor=black, linewidth=0.02](1.5147157,0.88709676)(1.5147157,-0.88709676)
\rput[bl](0.92976946,-0.03664185){$\cdots$}
\psline[linecolor=black, linewidth=0.02](0.3342017,0.8895668)(0.3342017,-0.9035367)
\end{pspicture}
} \\
\id_n
\end{gathered}\ \ ,\quad
\begin{gathered}
\psscalebox{1.0 1.0} % Change this value to rescale the drawing.
{
\begin{pspicture}(0,-0.9090912)(1.8181823,0.9090912)
\psframe[linecolor=black, linewidth=0.02, dimen=outer](1.8181821,0.90909094)(0.0,-0.90909094)
\psline[linecolor=black, linewidth=0.02](1.5352285,0.88196856)(1.5352285,-0.89222497)
\rput[bl](1.0477182,-0.031513646){$\cdots$}
\psline[linecolor=black, linewidth=0.02](0.92629546,0.88196856)(0.92629546,-0.89222497)
\psbezier[linecolor=black, linewidth=0.02](0.329604,0.8948718)(0.329604,0.0948718)(0.65780914,0.08974359)(0.65780914,0.88974357)
\psbezier[linecolor=black, linewidth=0.02](0.65780914,-0.88974357)(0.65780914,-0.08974359)(0.329604,-0.0948718)(0.329604,-0.8948718)
\end{pspicture}
} \\
U_1
\end{gathered}\ \ ,\quad \cdots \ \ ,\quad
\begin{gathered}
\psscalebox{1.0 1.0} % Change this value to rescale the drawing.
{
\begin{pspicture}(0,-0.9090912)(1.8181823,0.9090912)
\rput{-180.0}(1.8181823,0.0){\psframe[linecolor=black, linewidth=0.02, dimen=outer](1.8181821,0.90909094)(0.0,-0.90909094)}
\psline[linecolor=black, linewidth=0.02](0.3034667,-0.88709676)(0.3034667,0.88709676)
\psline[linecolor=black, linewidth=0.02](0.9123997,-0.88709676)(0.9123997,0.88709676)
\psbezier[linecolor=black, linewidth=0.02](1.4885783,-0.8948718)(1.4885783,-0.0948718)(1.1603732,-0.08974359)(1.1603732,-0.88974357)
\psbezier[linecolor=black, linewidth=0.02](1.1603732,0.88974357)(1.1603732,0.08974359)(1.4885783,0.0948718)(1.4885783,0.8948718)
\rput[bl](0.39627066,-0.0025641026){$\cdots$}
\end{pspicture}
} \\
U_{n-1}
\end{gathered}
\]
\caption{Generators of $\TL_n$. \label{fig:TL_gen}}
\end{figure}
% ---------------- END FIG ----------------

If we can express any simple diagram without loops as a product of these generators then we can add in any number of loops by multiplying by appropriate powers of $d$, and take linear combinations to obtain any formal diagram on $n$ points.
So to show that these diagrams generate $\TL_n$ it suffices to show how to write any simple diagram without loops as a product of the generators.
We can do this by ``wriggling the strings'' of the diagram to obtain an isotopic diagram for which the decomposition is obvious.
This is illustrated with a particular example beneath (Fig. \ref{fig:generator_eg}), for further details see Section 3 of \cite{Kauffman1990}.

% ---------------- Figure: A simple diagram as a product of generators ------------------
\begin{figure}[ht]
\[
\begin{gathered}
\psscalebox{1.0 1.0} % Change this value to rescale the drawing.
{
\begin{pspicture}(0,-0.9)(1.8,0.9)
\psframe[linecolor=black, linewidth=0.02, dimen=outer](1.8,0.9)(0.0,-0.9)
\psarc[linecolor=black, linewidth=0.02, dimen=outer](0.4904762,-0.8825397){0.24126984}{0.0}{180.0}
\psarc[linecolor=black, linewidth=0.02, dimen=outer](1.0873016,-0.8888889){0.24126984}{0.0}{180.0}
\psarc[linecolor=black, linewidth=0.02, dimen=outer](1.0619048,0.8888889){0.4952381}{180.0}{360.0}
\psarc[linecolor=black, linewidth=0.02, dimen=outer](1.0555556,0.8952381){0.24126984}{180.0}{360.0}
\psbezier[linecolor=black, linewidth=0.02](0.30634922,0.8888889)(0.31530154,0.2455707)(0.66272616,0.07004297)(0.88412696,-0.025396826)(1.1055279,-0.12083662)(1.4945313,-0.12153985)(1.5253968,-0.8825397)
\end{pspicture}
}
\end{gathered}
\quad = \quad
\begin{gathered}
\psscalebox{1.0 1.0} % Change this value to rescale the drawing.
{
\begin{pspicture}(0,-2.00625)(3.6609313,2.00625)
\psline[linecolor=black, linewidth=0.02, linestyle=dashed, dash=0.17638889cm 0.10583334cm](0.0028959515,-0.40091714)(3.6251183,-0.40091714)
\psline[linecolor=black, linewidth=0.02, linestyle=dashed, dash=0.17638889cm 0.10583334cm](0.0011764945,0.39628014)(3.6322875,0.39628014)
\psline[linecolor=black, linewidth=0.02, linestyle=dashed, dash=0.17638889cm 0.10583334cm](0.0,1.1934775)(3.6666667,1.1934775)
\psline[linecolor=black, linewidth=0.02, linestyle=dashed, dash=0.17638889cm 0.10583334cm](0.012971367,-1.1981142)(3.6263046,-1.1981142)
\psarc[linecolor=black, linewidth=0.02, dimen=outer](0.9121569,-1.2077206){0.28235295}{180.0}{360.0}
\rput{-180.0}(1.8164706,-3.9919116){\psarc[linecolor=black, linewidth=0.02, dimen=outer](0.9082353,-1.9959558){0.28235295}{180.0}{360.0}}
\psline[linecolor=black, linewidth=0.02](1.7592157,-1.2077206)(1.7592157,-1.9920343)
\psline[linecolor=black, linewidth=0.02](2.9090197,-1.2116953)(2.9090197,-1.9960091)
\psline[linecolor=black, linewidth=0.02](2.3239217,-1.1998775)(2.3239217,-1.9841912)
\psline[linecolor=black, linewidth=0.02](0.6337255,-0.3959559)(0.6337255,-1.1802696)
\psline[linecolor=black, linewidth=0.02](1.1984314,-0.4194853)(1.1984314,-1.203799)
\psarc[linecolor=black, linewidth=0.02, dimen=outer](2.0454903,-0.41164216){0.28235295}{180.0}{360.0}
\rput{-180.0}(4.0831375,-2.399755){\psarc[linecolor=black, linewidth=0.02, dimen=outer](2.0415688,-1.1998775){0.28235295}{180.0}{360.0}}
\psline[linecolor=black, linewidth=0.02](2.905098,-0.42811275)(2.905098,-1.2124264)
\psarc[linecolor=black, linewidth=0.02, dimen=outer](2.6219609,0.37659314){0.28235295}{180.0}{360.0}
\rput{-180.0}(5.2360783,-0.8232843){\psarc[linecolor=black, linewidth=0.02, dimen=outer](2.6180391,-0.41164216){0.28235295}{180.0}{360.0}}
\psarc[linecolor=black, linewidth=0.02, dimen=outer](1.4847059,1.1844363){0.28235295}{180.0}{360.0}
\rput{-180.0}(2.9615686,0.79240197){\psarc[linecolor=black, linewidth=0.02, dimen=outer](1.4807843,0.39620098){0.28235295}{180.0}{360.0}}
\psline[linecolor=black, linewidth=0.02](1.1945099,0.40012255)(1.1945099,-0.3841912)
\psline[linecolor=black, linewidth=0.02](1.7570869,0.40583682)(1.7570869,-0.3784769)
\psline[linecolor=black, linewidth=0.02](2.337367,1.188694)(2.337367,0.40438026)
\psline[linecolor=black, linewidth=0.02](2.899944,1.1829797)(2.899944,0.39866596)
\psline[linecolor=black, linewidth=0.02](0.63444334,0.3916945)(0.63444334,-0.39261922)
\psline[linecolor=black, linewidth=0.02](0.63730145,1.1684974)(0.63730145,0.3841837)
\psarc[linecolor=black, linewidth=0.02, dimen=outer](2.0613093,1.9956692){0.28235295}{180.0}{360.0}
\rput{-180.0}(4.1147757,2.4148679){\psarc[linecolor=black, linewidth=0.02, dimen=outer](2.0573878,1.2074339){0.28235295}{180.0}{360.0}}
\psline[linecolor=black, linewidth=0.02](1.2046794,1.989953)(1.2046794,1.2056394)
\psline[linecolor=black, linewidth=0.02](0.6378332,1.9883046)(0.6378332,1.2039909)
\psline[linecolor=black, linewidth=0.02](2.9017081,1.9888364)(2.9017081,1.2045227)
\psframe[linecolor=black, linewidth=0.02, dimen=outer](3.6609313,2.00625)(0.010931397,-2.00625)
\end{pspicture}
}
\end{gathered}
\quad = \quad U_3 U_2 U_4 U_3 U_1
\]
\caption{A simple diagram as a product of generators. \label{fig:generator_eg}}
\end{figure}
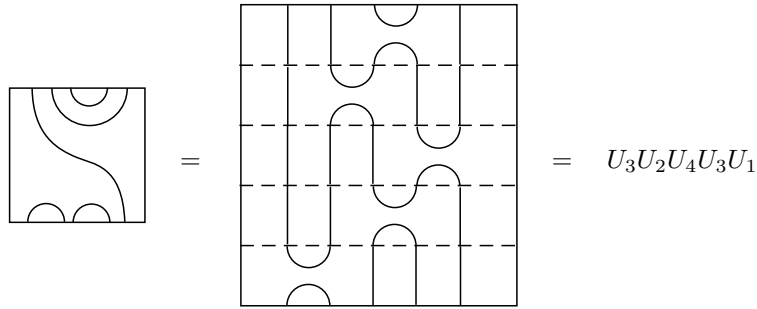
% ---------------- END FIG ----------------

\begin{remark}
In general there is more than one way to express a simple diagram as a product of generators; see the next remark about relations in $\TL_n$.
\end{remark}
\begin{remark}
$\TL_n$ is often presented as an abstract $\mathbbm{F}$-algebra in terms of the generators $\id_n, U_1, \dotsc, U_{n-1}$ and relations
\begin{gather*}
U_i^2 = d U_i \\
U_i U_j = U_j U_i, \quad \abs{i-j}>1 \\
U_i U_{i+1} U_i = U_i \ \text{and}\ U_{i+1} U_i U_{i+1} = U_{i+1}, \quad 1 \leq i < n
\end{gather*}
Our diagrammatic definition is due to Kauffman, and in this setting it is easy to see that the above relations hold.
(Proving that only these relations hold is slightly less trivial, see Theorem 4.3 of \cite{Kauffman1990}.)
\end{remark}

% Do I need to talk about the dimension of TL_n here yet?
% Or can I wait until we actually need it?
% WELL, THIS SPACE IS RESERVED FOR TALKING ABOUT TL_n AS A VECTOR SPACE IF WE NEED
%
%Observe that by $d$-equivalence all formal diagrams can be written as a $\mathbbm{F}$-linear combination of simple diagrams without loops.
%Also, the set of all non-isotopic simple diagrams without loops are linearly independent, thus these form a basis for $\TL_n$.
%We have the following result:
%<Proposition on dimension of $\TL_n$>

\begin{defn}
An \textbf{ideal} $\mathcal{I}$ in a category $\mathcal{C}$ is a collection of morphisms that is closed under composition by arbitrary morphisms in $\mathcal{C}$ (whenever such a composition is defined).
If $\mathcal{C}$ is a linear monoidal category, $\mathcal{I}$ is known as a \textbf{tensor ideal} if it is further closed under tensor product with arbitrary morphisms, and for all pairs of objects $x,y \in \Obj(\mathcal{C})$ the subset $\mathcal{I} \cap \Hom(x,y) \subset \mathcal{I}$ is a linear subspace of $\Hom(x,y)$.
\end{defn}

The following result, though simple, has a very useful corollary that we will make use of in many proofs to come.

\begin{lemma}
Suppose $f \colon m \rightarrow n,\ g \colon n \rightarrow k$ are simple TL diagrams with $b$ and $c$ through-strings respectively.
Then the composite diagram $gf$ has at most $\min(b,c)$ through-strings.
\end{lemma}
\begin{proof}
Since $f$ has $b$ through-strings it must have $m-b$ points on its bottom edge not connected by through-strings --- that is, it has that many capped points.
Similarly since $g$ has $c$ through-strings it must have $k-c$ points on its top edge not connected by through-strings (that is, cupped points).
It is clear that capped and cupped points remain capped and cupped after composition of diagrams, so composition never decreases the number of capped or cupped points.
Thus $gf \colon m \rightarrow k$ has at least $m-b$ capped points and $k-c$ cupped points, hence at most $\min(b,c)$ points connected by through-strings.
\end{proof}

We say that a formal diagram $f$ has $c$ through-strings if $c$ is the greatest number of through-strings possessed by any simple diagram in $f$.
By the above lemma and bilinearity of composition we get the following result.

\begin{corollary}
The set of all formal diagrams $f \in \Mor(\TL)$ with at most $c$ through-strings forms an ideal $\mathcal{I}_c$ in $\TL$.
\label{cor:thru-string_ideal}
\end{corollary}

Note that this is not however a tensor ideal, since we can always tensor with an identity diagram $\id_n$ to add $n$ more through-strings.

\section{Jones-Wenzl idempotents}
One of our main goals will be to show that certain categories built from generic Temperley-Lieb are semisimple (definitions to come), which will enable us to construct finite-dimensional vector spaces associated to surfaces.
To this end we introduce an important class of endomorphisms in $\TL$, first discovered by Jones \cite{Jones1983} in the context of subfactors of von Neumann algebras.
The inductive definition we use here is due to Wenzl \cite{Wenzl1987}.

First, some notation:

\begin{defn}
Let $n \geq 2$ and $1 \leq i < n$.
Define the \textbf{cup} $\cup_{i,n}$ to be the simple TL diagram from $n-2$ to $n$ points having $i-1$ through-strings connecting corresponding points on the top and bottom edges, followed by a cup connecting the $i$-th and $i+1$-th points on the top edge, followed by through-strings connecting the remaining points (reading from left to right).
That is,
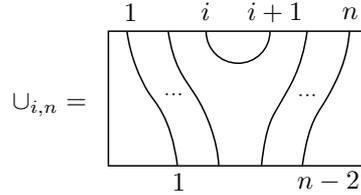
\begin{equation*}
\cup_{i,n} = \ 
\begin{gathered}
\psscalebox{1.0 1.0} % Change this value to rescale the drawing.
{
\begin{pspicture}(0,-1.2553478)(3.440982,1.2553478)
\psarc[linecolor=black, linewidth=0.02, dimen=outer](1.7150847,0.86036295){0.42}{180.0}{360.0}
\psframe[linecolor=black, linewidth=0.02, dimen=outer](3.4409823,0.87521887)(0.0,-0.92990935)
\rput[bl](1.2284564,1.0153478){$i$}
\rput[bl](1.8367347,0.9872189){$i+1$}
\rput[bl](0.7412167,-0.032760445){$\scriptstyle\cdots$}
\rput[bl](2.5,-0.043236636){$\scriptstyle\cdots$}
\psbezier[linecolor=black, linewidth=0.02](0.7969216,0.8601359)(0.867366,0.4054418)(0.9869811,0.27199215)(1.1456395,0.029366683)(1.304298,-0.21325879)(1.4510474,-0.6081435)(1.4608603,-0.92196125)
\psbezier[linecolor=black, linewidth=0.02](0.25506315,0.8563849)(0.30247793,0.5101009)(0.43107215,0.22967732)(0.5264917,0.07067056)(0.6219113,-0.088336185)(0.82833993,-0.24622673)(0.9191154,-0.92296153)
\psbezier[linecolor=black, linewidth=0.02](2.6243727,0.86797905)(2.5531678,0.35262477)(2.3830504,0.11435683)(2.277706,-0.07602095)(2.1723616,-0.26639873)(2.0853245,-0.34156963)(2.0216794,-0.9238471)
\psbezier[linecolor=black, linewidth=0.02](3.1812353,0.86797905)(3.1514633,0.48997474)(3.0251765,0.23296292)(2.8985686,0.003979052)(2.7719607,-0.2250048)(2.6484172,-0.30199003)(2.5616276,-0.9124131)
\rput[bl](0.21652947,0.99346924){$1$}
\rput[bl](3.0871177,1.0248418){$n$}
\rput[bl](0.84870785,-1.2553478){$1$}
\rput[bl](2.493056,-1.2505643){$n-2$}
\end{pspicture}
}
\end{gathered}
\end{equation*}
and similarly define the \textbf{cap} $\cap_{i,n}\colon n \rightarrow n-2$ by
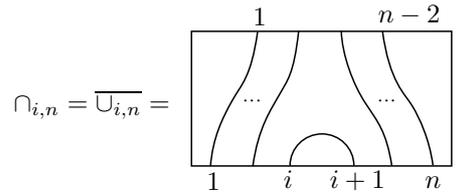
\begin{equation*}
\cap_{i,n} = \overline{\cup_{i,n}} = \ 
\begin{gathered}
\psscalebox{1.0 1.0} % Change this value to rescale the drawing.
{
\begin{pspicture}(0,-1.2150948)(3.440982,1.2150948)
\rput{-180.0}(3.4517949,-1.7705653){\psarc[linecolor=black, linewidth=0.02, dimen=outer](1.7258974,-0.88528264){0.42}{180.0}{360.0}}
\rput{-180.0}(3.440982,0.00485111){\psframe[linecolor=black, linewidth=0.02, dimen=outer](3.4409823,0.90498966)(0.0,-0.90013856)}
\rput[bl](1.2217914,-1.186966){$i$}
\rput[bl](1.8300695,-1.2150948){$i+1$}
\rput[bl](0.68,-0.08301637){$\scriptstyle\cdots$}
\rput[bl](2.45,-0.08815923){$\scriptstyle\cdots$}
\psbezier[linecolor=black, linewidth=0.02](2.6440604,-0.8850556)(2.573616,-0.43036148)(2.454001,-0.29691184)(2.2953424,-0.05428636)(2.136684,0.1883391)(1.9899348,0.5832239)(1.9801219,0.8970416)
\psbezier[linecolor=black, linewidth=0.02](3.185919,-0.8813045)(3.1385043,-0.5350206)(3.0099099,-0.25459698)(2.9144905,-0.09559024)(2.8190708,0.06341651)(2.6126423,0.22130705)(2.5218666,0.8980419)
\psbezier[linecolor=black, linewidth=0.02](0.8166095,-0.89289874)(0.8878144,-0.37754443)(1.0579318,-0.1392765)(1.1632762,0.05110127)(1.2686206,0.24147905)(1.3556577,0.31664994)(1.4193026,0.8989274)
\psbezier[linecolor=black, linewidth=0.02](0.25974676,-0.89289874)(0.28951883,-0.5148944)(0.41580552,-0.2578826)(0.5424134,-0.028898729)(0.66902137,0.20008513)(0.79256487,0.27707037)(0.8793546,0.88749343)
\rput[bl](0.20986441,-1.2088445){$1$}
\rput[bl](3.0804527,-1.177472){$n$}
\rput[bl](0.819926,0.98031133){$1$}
\rput[bl](2.4642742,0.9850948){$n-2$}
\end{pspicture}
}
\end{gathered}\,.
\end{equation*}
\end{defn}

We will in general write $\cap_i$ and $\cup_i$ for $\cap_{i,n}$ and $\cup_{i,n}$ since $n$ will always be clear from the context.
Observe that $\cup_i \cap_i = U_i$ and $\cap_i \cup_i = d\cdot\id$.

\begin{theorem}
\label{thm:JWP}
For $n = 0, 1, 2, \dotsc$ define $p_n \in \TL_n$ inductively by
% -------------- DIAGRAM: Inductive definition of JWPs --------------
\begin{align}
p_0 & = \begin{gathered}
\psscalebox{1.0 1.0} % Change this value to rescale the drawing.
{
\begin{pspicture}(0,-0.46)(0.92,0.46)
\psframe[linecolor=black, linewidth=0.02, dimen=outer](0.92,0.46)(0.0,-0.46)
\end{pspicture}
}
\end{gathered} \quad \text{(the empty diagram)}, \nonumber \\
p_1 & = \begin{gathered}
\psscalebox{1.0 1.0} % Change this value to rescale the drawing.
{
\begin{pspicture}(0,-0.46)(0.92,0.46)
\psframe[linecolor=black, linewidth=0.02, dimen=outer](0.92,0.46)(0.0,-0.46)
\psline[linecolor=black, linewidth=0.02](0.47142857,0.45285714)(0.47142857,-0.44)
\end{pspicture}
}
\end{gathered} = \id_1 \ , \nonumber \\
p_{n+1} & = \ 
\begin{gathered}
\psscalebox{1.0 1.0} % Change this value to rescale the drawing.
{
\begin{pspicture}(0,-1.1136682)(1.8117054,1.1136682)
\psframe[linecolor=black, linewidth=0.02, dimen=outer](1.5,0.40326387)(0.0,-0.39673612)
\psline[linecolor=black, linewidth=0.02](0.3622188,0.39571947)(0.3622188,1.1136682)
\psline[linecolor=black, linewidth=0.02](0.36956573,-1.1136683)(0.36956573,-0.3957196)
\rput[bl](0.6,-0.11525464){$p_n$}
\psline[linecolor=black, linewidth=0.02](1.1608701,-1.1136683)(1.1608701,-0.3957196)
\psline[linecolor=black, linewidth=0.02](1.1535231,0.39571947)(1.1535231,1.1136682)
\rput[bl](0.5739131,0.76948464){$\cdots$}
\rput[bl](0.5739131,-0.7967361){$\cdots$}
\psline[linecolor=black, linewidth=0.02](1.8017051,1.098234)(1.8017051,-1.0978444)
\end{pspicture}
}
\end{gathered}
\quad - \ \frac{[n]}{[n+1]} \ \ 
\begin{gathered}
\psscalebox{1.0 1.0} % Change this value to rescale the drawing.
{
\begin{pspicture}(0,-1.7498428)(1.8986871,1.7498428)
\psframe[linecolor=black, linewidth=0.02, dimen=outer](1.5,1.1776017)(0.0,0.3776017)
\psline[linecolor=black, linewidth=0.02](0.26174784,1.1598009)(0.26174784,1.7481275)
\psline[linecolor=black, linewidth=0.02](0.26956573,-0.41279984)(0.26956573,0.37861827)
\rput[bl](0.6,0.6590832){$p_n$}
\psline[linecolor=black, linewidth=0.02](0.9608701,-0.41279984)(0.9608701,0.37861827)
\psline[linecolor=black, linewidth=0.02](1.2514824,1.1615162)(1.2514824,1.7498428)
\rput[bl](0.5739131,1.4611104){$\cdots$}
\rput[bl](0.44314387,-0.055051334){$\cdots$}
\psarc[linecolor=black, linewidth=0.02, dimen=outer](1.5724934,0.390379){0.30204082}{180.0}{360.0}
\psframe[linecolor=black, linewidth=0.02, dimen=outer](1.5122449,-0.40199012)(0.012244951,-1.2019901)
\rput[bl](0.61224496,-0.9205086){$p_n$}
\rput{-180.0}(3.1694765,-0.8274052){\psarc[linecolor=black, linewidth=0.02, dimen=outer](1.5847383,-0.4137026){0.30204082}{180.0}{360.0}}
\psline[linecolor=black, linewidth=0.02](0.26582947,-1.7355052)(0.26582947,-1.1870624)
\psline[linecolor=black, linewidth=0.02](1.255564,-1.7339063)(1.255564,-1.1854633)
\rput[bl](0.5779947,-1.454622){$\cdots$}
\psline[linecolor=black, linewidth=0.02](1.8749068,0.3797085)(1.8749068,1.7440187)
\psline[linecolor=black, linewidth=0.02](1.8886871,-0.39500389)(1.8886871,-1.7498426)
\end{pspicture}
}
\end{gathered}
\label{eq:JWPrelation}
\end{align}
% -------------- END DIAG --------------

We call $p_n$ the \textbf{\emph{Jones-Wenzl idempotents}}.
They are the unique endomorphisms $p_n \colon n \rightarrow n$ in $\TL$ satisfying the following properties:
\begin{enumerate}[label=\roman{*}., ref=\roman{*}]
\item $p_n$ is nonzero and can be written as $p_n = \id_n + m$ where $m = \sum c_j m_j$ is a $\mathbbm{F}$-linear combination of non-identity diagrams $m_j\colon n\rightarrow n$,\label{eq:JWP1}
\item $\cap_i p_n = p_n \cup_i = 0$ for $i = 1, \dotsc, n-1$, and \label{eq:JWP2}
\item $p_n^2 = p_n$. \label{eq:JWP3}
\end{enumerate}
\end{theorem}

\begin{proof}
It is trivial to check that properties \eqref{eq:JWP1} through \eqref{eq:JWP3} hold for $p_0$ and $p_1$, and that $p_1$ satisfies the relation \eqref{eq:JWPrelation}.
Let $n \geq 1$, assume the required properties hold for $p_n$, $p_{n-1}$, and consider $p_{n+1}$.

\begin{enumerate}
\item Since the coefficient of $\id_n$ in $p_n$ is $1$, the first term in \eqref{eq:JWPrelation} contributes an identity diagram with coefficient $1$.
For the second term note that
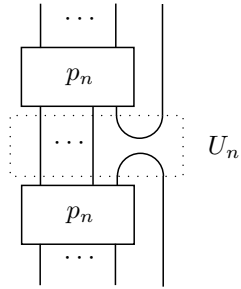
\begin{center}
\psscalebox{1.0 1.0} % Change this value to rescale the drawing.
{

}
\end{center}
is a product involving $U_n$ which has $n-1$ through-strings, hence by Corollary \ref{cor:thru-string_ideal} is a formal diagram that cannot contain the identity on $n+1$ points.
Thus the coefficient of $\id_{n+1}$ in $p_{n+1}$ is exactly $1$.

\item For $i \neq n$,
\begin{equation*}
\cap_i p_{n+1} = \ 
\begin{gathered}
\psscalebox{1.0 1.0} % Change this value to rescale the drawing.
{
%
}
\end{gathered}
\label{eq:1str_trace}
\end{equation}

Consider the \emph{one-strand (right) trace} $\tr_1(p_n)$ of $p_n$, then we get the following relation:
\begin{align*}
\tr_1(p_n) &= \ 
\begin{gathered}
\psscalebox{1.0 1.0} % Change this value to rescale the drawing.
{
%
}
\end{gathered} \\
&= d\cdot p_{n-1} - \frac{[n-1]}{[n]}\cdotp_{n-1} \\
&= \left(d - \frac{[n-1]}{[n]}\right)p_{n-1}
\end{align*}
and by substitution into \eqref{eq:1str_trace} we get that
\begin{align*}
\cap_n p_{n+1} &= \ 
\begin{gathered}
\psscalebox{1.0 1.0} % Change this value to rescale the drawing.
{
%
}
\end{gathered}
\end{align*}
as
\[
d - \frac{[n-1]}{[n]} = \frac{d[n]-[n-1]}{[n]} = \frac{[n+1]}{[n]}
\]
by Proposition \ref{prop:quantumint_relation}.

Now we can ``absorb'' $p_{n-1}$ into $p_n$: by property \eqref{eq:JWP1} we have
\begin{equation*}
\begin{gathered}
\psscalebox{1.0 1.0} % Change this value to rescale the drawing.
{
%
}
\end{gathered}
\ \right)
\quad = \ p_n
\end{equation*}
since the summation term is zero.
This is because each $m_j$ is a non-identity diagram and hence must contain a cap on its bottom edge, so by the induction hypothesis $m_j p_n = 0$.
Thus we finally have that
\begin{align*}
\cap_n p_{n+1} &= \ 
\begin{gathered}
\psscalebox{1.0 1.0} % Change this value to rescale the drawing.
{
%
}
\end{gathered} \\
&= 0.
\end{align*}

Lastly observe that by the inductive formula \eqref{eq:JWPrelation} and invariance of the quantum integers under $q \mapsto q^{-1}$ we have that $\overline{p_{n+1}} = p_{n+1}$, hence
\[ p_{n+1} \cup_i = \overline{p_{n+1}}\cdot\overline{\cap_i} = \overline{\cap_i p_{n+1}} = \overline{0} = 0. \]

\item Write $\mu_n$ for $\frac{[n]}{[n+1]}$.
Using the one-strand trace relation and the idempotent property of $p_n$ we calculate that
\begin{align*}
p_{n+1}^2 &= \ 
\begin{gathered}
\psscalebox{1.0 1.0} % Change this value to rescale the drawing.
{
\begin{pspicture}(0,-1.1136682)(1.8117054,1.1136682)
\psframe[linecolor=black, linewidth=0.02, dimen=outer](1.5,0.40326387)(0.0,-0.39673612)
\psline[linecolor=black, linewidth=0.02](0.3622188,0.39571947)(0.3622188,1.1136682)
\psline[linecolor=black, linewidth=0.02](0.36956573,-1.1136683)(0.36956573,-0.3957196)
\rput[bl](0.6,-0.11525464){$p_n$}
\psline[linecolor=black, linewidth=0.02](1.1608701,-1.1136683)(1.1608701,-0.3957196)
\psline[linecolor=black, linewidth=0.02](1.1535231,0.39571947)(1.1535231,1.1136682)
\rput[bl](0.5739131,0.76948464){$\cdots$}
\rput[bl](0.5739131,-0.7967361){$\cdots$}
\psline[linecolor=black, linewidth=0.02](1.8017051,1.098234)(1.8017051,-1.0978444)
\end{pspicture}
}
\end{gathered}
\quad - \ 2\mu_n \ \ 
\begin{gathered}
\psscalebox{1.0 1.0} % Change this value to rescale the drawing.
{
\begin{pspicture}(0,-1.7498428)(1.8986871,1.7498428)
\psframe[linecolor=black, linewidth=0.02, dimen=outer](1.5,1.1776017)(0.0,0.3776017)
\psline[linecolor=black, linewidth=0.02](0.26174784,1.1598009)(0.26174784,1.7481275)
\psline[linecolor=black, linewidth=0.02](0.26956573,-0.41279984)(0.26956573,0.37861827)
\rput[bl](0.6,0.6590832){$p_n$}
\psline[linecolor=black, linewidth=0.02](0.9608701,-0.41279984)(0.9608701,0.37861827)
\psline[linecolor=black, linewidth=0.02](1.2514824,1.1615162)(1.2514824,1.7498428)
\rput[bl](0.5739131,1.4611104){$\cdots$}
\rput[bl](0.44314387,-0.055051334){$\cdots$}
\psarc[linecolor=black, linewidth=0.02, dimen=outer](1.5724934,0.390379){0.30204082}{180.0}{360.0}
\psframe[linecolor=black, linewidth=0.02, dimen=outer](1.5122449,-0.40199012)(0.012244951,-1.2019901)
\rput[bl](0.61224496,-0.9205086){$p_n$}
\rput{-180.0}(3.1694765,-0.8274052){\psarc[linecolor=black, linewidth=0.02, dimen=outer](1.5847383,-0.4137026){0.30204082}{180.0}{360.0}}
\psline[linecolor=black, linewidth=0.02](0.26582947,-1.7355052)(0.26582947,-1.1870624)
\psline[linecolor=black, linewidth=0.02](1.255564,-1.7339063)(1.255564,-1.1854633)
\rput[bl](0.5779947,-1.454622){$\cdots$}
\psline[linecolor=black, linewidth=0.02](1.8749068,0.3797085)(1.8749068,1.7440187)
\psline[linecolor=black, linewidth=0.02](1.8886871,-0.39500389)(1.8886871,-1.7498426)
\end{pspicture}
}
\end{gathered}
\quad + \ \mu_n^2(d - \mu_{n-1}) \ \ 
\begin{gathered}
\psscalebox{1.0 1.0} % Change this value to rescale the drawing.
{
\begin{pspicture}(0,-2.177474)(1.9248853,2.177474)
\psframe[linecolor=black, linewidth=0.02, dimen=outer](1.5042553,-0.8291536)(0.0042552566,-1.6291536)
\psline[linecolor=black, linewidth=0.02](0.26600304,-0.8469544)(0.26600304,-0.25862777)
\rput[bl](0.60425526,-1.3476721){$p_n$}
\psline[linecolor=black, linewidth=0.02](0.26108667,-2.1774743)(0.26108667,-1.6290313)
\psline[linecolor=black, linewidth=0.02](1.2508211,-2.1758752)(1.2508211,-1.6274323)
\rput[bl](0.5627256,-1.92817){$\cdots$}
\rput[bl](0.40159568,-0.5839439){$\cdots$}
\psframe[linecolor=black, linewidth=0.02, dimen=outer](1.1675773,0.25972635)(0.0035772908,-0.26793325)
\psline[linecolor=black, linewidth=0.02](0.9628125,-0.848183)(0.9628125,-0.25985634)
\psline[linecolor=black, linewidth=0.02](0.25583354,0.26209715)(0.25583354,0.8504238)
\rput[bl](0.39142618,0.5251076){$\cdots$}
\psline[linecolor=black, linewidth=0.02](0.9587036,0.25596854)(0.9587036,0.84429514)
\rput[bl](0.23777379,-0.121066965){$p_{n-1}$}
\psframe[linecolor=black, linewidth=0.02, dimen=outer](1.4999999,1.6346762)(0.0,0.8346762)
\rput[bl](0.56595737,1.1161577){$p_n$}
\psline[linecolor=black, linewidth=0.02](0.25683135,1.6225258)(0.25683135,2.1709688)
\psline[linecolor=black, linewidth=0.02](1.2465658,1.6241248)(1.2465658,2.1725676)
\rput[bl](0.5584703,1.8718301){$\cdots$}
\psarc[linecolor=black, linewidth=0.02, dimen=outer](1.5829704,-0.83535516){0.3319149}{0.0}{180.0}
\psarc[linecolor=black, linewidth=0.02, dimen=outer](1.5999917,0.8497512){0.30638298}{180.0}{360.0}
\psline[linecolor=black, linewidth=0.02](1.9063747,0.83272994)(1.8978641,2.1774108)
\psline[linecolor=black, linewidth=0.02](1.9148853,-0.8438658)(1.9148853,-2.1630147)
\end{pspicture}
}
\end{gathered} \\
&= \ 
\begin{gathered}
\psscalebox{1.0 1.0} % Change this value to rescale the drawing.
{
\begin{pspicture}(0,-1.1136682)(1.8117054,1.1136682)
\psframe[linecolor=black, linewidth=0.02, dimen=outer](1.5,0.40326387)(0.0,-0.39673612)
\psline[linecolor=black, linewidth=0.02](0.3622188,0.39571947)(0.3622188,1.1136682)
\psline[linecolor=black, linewidth=0.02](0.36956573,-1.1136683)(0.36956573,-0.3957196)
\rput[bl](0.6,-0.11525464){$p_n$}
\psline[linecolor=black, linewidth=0.02](1.1608701,-1.1136683)(1.1608701,-0.3957196)
\psline[linecolor=black, linewidth=0.02](1.1535231,0.39571947)(1.1535231,1.1136682)
\rput[bl](0.5739131,0.76948464){$\cdots$}
\rput[bl](0.5739131,-0.7967361){$\cdots$}
\psline[linecolor=black, linewidth=0.02](1.8017051,1.098234)(1.8017051,-1.0978444)
\end{pspicture}
}
\end{gathered}
\quad - \ (2\mu_n - \mu_n^2(d-\mu_{n-1})) \ \ 
\begin{gathered}
\psscalebox{1.0 1.0} % Change this value to rescale the drawing.
{
\begin{pspicture}(0,-1.7498428)(1.8986871,1.7498428)
\psframe[linecolor=black, linewidth=0.02, dimen=outer](1.5,1.1776017)(0.0,0.3776017)
\psline[linecolor=black, linewidth=0.02](0.26174784,1.1598009)(0.26174784,1.7481275)
\psline[linecolor=black, linewidth=0.02](0.26956573,-0.41279984)(0.26956573,0.37861827)
\rput[bl](0.6,0.6590832){$p_n$}
\psline[linecolor=black, linewidth=0.02](0.9608701,-0.41279984)(0.9608701,0.37861827)
\psline[linecolor=black, linewidth=0.02](1.2514824,1.1615162)(1.2514824,1.7498428)
\rput[bl](0.5739131,1.4611104){$\cdots$}
\rput[bl](0.44314387,-0.055051334){$\cdots$}
\psarc[linecolor=black, linewidth=0.02, dimen=outer](1.5724934,0.390379){0.30204082}{180.0}{360.0}
\psframe[linecolor=black, linewidth=0.02, dimen=outer](1.5122449,-0.40199012)(0.012244951,-1.2019901)
\rput[bl](0.61224496,-0.9205086){$p_n$}
\rput{-180.0}(3.1694765,-0.8274052){\psarc[linecolor=black, linewidth=0.02, dimen=outer](1.5847383,-0.4137026){0.30204082}{180.0}{360.0}}
\psline[linecolor=black, linewidth=0.02](0.26582947,-1.7355052)(0.26582947,-1.1870624)
\psline[linecolor=black, linewidth=0.02](1.255564,-1.7339063)(1.255564,-1.1854633)
\rput[bl](0.5779947,-1.454622){$\cdots$}
\psline[linecolor=black, linewidth=0.02](1.8749068,0.3797085)(1.8749068,1.7440187)
\psline[linecolor=black, linewidth=0.02](1.8886871,-0.39500389)(1.8886871,-1.7498426)
\end{pspicture}
}
\end{gathered}
\end{align*}
by the absorption relation, and it is a straightforward calculation that $2\mu_n - \mu_n^2(d-\mu_{n-1}) = \mu_n$.
\end{enumerate}

Finally for uniqueness, let $p_n = \id+m$ and $p_n^\prime = \id + m^\prime$ be endomorphisms on $n$ points satisfying properties \eqref{eq:JWP1} through \eqref{eq:JWP3}.
Then
\[ p_n = \id\,p_n = (\id+m^\prime)p_n = p_n^\prime p_n = p_n^\prime(\id+m) = p_n^\prime\,\id = p_n^\prime. \qedhere \]
\end{proof}

\begin{remark}
The absorption property proved in \eqref{eq:JWP2} holds more generally: if $m<n$ then
\[ (\id \otimes p_m \otimes \id) \circ p_n = p_n = p_n \circ (\id \otimes p_m \otimes \id) \]
whenever the composition is defined.
The proof is essentially the same as that given before.
\end{remark}

\begin{remark}
In property \eqref{eq:JWP1} the requirement that the coefficient of $\id$ is $1$ was included purely as a matter of convenience in simplifying the proof;
it can in fact be shown that this is a necessary consequence of the other properties.
For assume $p_n$ is nonzero and satisfies properties \eqref{eq:JWP2} and \eqref{eq:JWP3}.
Then we may write $p_n = c\cdot\id +m$ for some $c \in\mathbbm{F}$, and
\[ p_n = p_n^2 = p_n(c\cdot\id+m)=c\,p_n+p_n\,m = c\,p_n, \]
so $c=1$.
\end{remark}

\begin{remark}
\label{rem:JWP_lat_refl_dual}
We have seen that $p_n = \overline{p_n}$, in particular the Jones-Wenzl idempotents are invariant under reflection in the horizontal.
In fact they are also invariant under the lateral reflection $p_n \mapsto p_n^r$, where $p_n^r$ is the diagram obtained by reflecting all simple diagrams in $p_n$ in the vertical line $\frac{1}{2} \times I$ through the middle of the square.
This is easily seen as properties \eqref{eq:JWP1} through \eqref{eq:JWP3} are invariant under lateral reflection, and thus satisfied by $p_n^r$, which by uniqueness is equal to $p_n$.
In particular this means that the laterally reflected version of the relation \eqref{eq:JWPrelation} holds.
Furthermore, since the dual $\displaystyle ^\ast$ is the composition of the anti-involution and lateral reflection, we have that $p_n^\ast = p_n$, i.e. the Jones-Wenzl idempotents are also invariant under taking duals.
\end{remark}

In the proof of Theorem \ref{thm:JWP} we saw that the Jones-Wenzl idempotents satisfy the relation
\[ \tr_1(p_n) = \frac{[n+1]}{[n]}p_{n-1}. \]
From this we also have the following fact.

\begin{prop}[Trace of $p_n$]
\label{prop:JWP_trace}
\[ \tr(p_n) = [n+1] \]
\end{prop}
\begin{proof}
This is true for the empty diagram $p_0$ since $\tr(p_0) = d^0 = 1 = [1]$.
Suppose the result holds for $p_{n-1}$ where $n-1\geq 0$.
Then
\begin{align*}
\tr(p_n) &= \tr(\tr_1(p_n)) \\
		 &= \tr\left(\frac{[n+1]}{[n]}p_{n-1}\right) \\
		 &= \frac{[n+1]}{[n]}\tr(p_{n-1}) \\
		 &= \frac{[n+1]}{[n]}[n] \\
		 &= [n+1]. \qedhere
\end{align*}
\end{proof}

\section{Generic Temperley-Lieb-Jones}
We now wish to promote the Jones-Wenzl idempotents to be objects in a new category constructed from $\TL$.
To do this we first need the following notion.

\begin{defn}The \textbf{Karoubi envelope} or \emph{idempotent completion} of a category $\mathcal{C}$ is the category $\Kar\mathcal{C}$ constructed from $\mathcal{C}$ as follows:
\begin{itemize}
\item Objects in $\Kar\mathcal{C}$ are idempotent morphisms $p\colon x \rightarrow x$ in $\Mor(\mathcal{C})$.
\item Let $p\colon x\rightarrow x$ and $q\colon y\rightarrow y$ be objects in $\Kar\mathcal{C}$; a morphism $f\colon p \rightarrow q \in \Mor(\Kar\mathcal{C})$ is a morphism $f \colon x\rightarrow y \in \Mor(\mathcal{C})$, such that $fp = f = qf$ when considered as morphisms in $\mathcal{C}$.
Composition of morphisms $g$ and $f$ in $\Kar\mathcal{C}$ is the same as composition in $\mathcal{C}$, except we now consider $gf$ as a morphism between objects in $\Kar(\mathcal{C})$ instead of the underlying objects in $\mathcal{C}$.
\end{itemize}
\end{defn}

The identity morphism $\mathbbm{1}_p$ on $p\colon x\rightarrow x$ in $\Kar\mathcal{C}$ is then $p$ itself.
The original category $\mathcal{C}$ injects fully and faithfully into $\Kar\mathcal{C}$ via the functor that sends $x$ to the identity morphism $\mathbbm{1}_x$ and $f\colon x\rightarrow y$ to $f\colon\mathbbm{1}_x\rightarrow\mathbbm{1}_y$.

\begin{remark}
The defining property of the Karoubi envelope is that all idempotent morphisms \emph{split} in $\Kar\mathcal{C}$: that is, for every idempotent $f\colon p\rightarrow p$ there is an object $q$ and morphisms $g\colon q\rightarrow p, h\colon p\rightarrow q$ such that $g\circ h = f$ and $h\circ g = \mathbbm{1}_q$.
However the approach we take in our applications to skein modules will not rely on this particular property.
\end{remark}

Let us now take the idempotent completion $\Kar\TL$ of Temperley-Lieb.
This has as objects all idempotent TL diagrams on $n$ points, with morphisms also being TL diagrams $f\colon p\rightarrow q$ that are invariant under pre-composition with their domain and post-composition with their codomain (that is, $fp = f$ and $qf = f$).

Recall that $\TL$ is a strict linear monoidal category with tensor product $\otimes$; this lifts to a tensor product in $\Kar\TL$ as follows.
Observe that the tensor product of idempotents in $\TL$ is also an idempotent, hence an object in $\Kar\TL$.
Furthermore $\otimes$ of diagrams is strictly associative, the empty diagram $p_0$ is the tensor identity, and the coherence conditions hold trivially.
It is also easy to see that if $f,g\colon p\rightarrow q$ are morphisms from $p$ to $q$ in $\Kar\TL$ then so is $cf + g$ for all $c \in \mathbbm{F}$, and $\otimes$ is already bilinear.
Thus $\Kar\TL$ is also a strict $\mathbbm{F}$-linear monoidal category.

The following lemma gives an easy way to determine the hom-sets in the Karoubi envelope of any category.

\begin{lemma}
Let $\mathcal{C}$ be a category.
For all idempotents $p\colon x\rightarrow x$ and $q\colon y \rightarrow y$ in $\Kar\mathcal{C}$ there is a surjection
\[ \phi \colon \Hom_\mathcal{C}(x,y) \rightarrow \Hom_{\Kar\mathcal{C}}(p,q) \]
given by
\[ \phi(f) = qfp. \]
\label{lemma:Kar_hom}
\end{lemma}

\begin{proof}
If $f \in \Hom_{\mathcal{C}}(x,y)$ then $\phi(f) = qfp \in \Hom_{\mathcal{C}}(x,y)$.
Furthermore $\phi(f)$ satisfies
\[ \phi(f)p = qfpp = qfp = \phi(f) \]
and
\[ q\phi(f) = qqfp = qfp = \phi(f). \]
Hence $\phi(f)$ is a morphism from $p$ to $q$ in $\Kar\mathcal{C}$ and $\phi$ defines a map from $\Hom_{\mathcal{C}}(x,y)$ to $\Hom_{\Kar\mathcal{C}}(p,q)$.

For surjectivity observe that if $f\in\Hom_{\Kar\mathcal{C}}(p,q)$ then $f$ is a morphism in $\Hom_{\mathcal{C}}(x,y)$ satisfying $fp = f = qf$, and thus $f = qfp$.
\end{proof}

Next we prove an important property of the Jones-Wenzl idempotents.
We need the following definitions.

\begin{defn}
An object $x$ in a $\mathbbm{F}$-linear category $\mathcal{C}$ is called \textbf{simple} if $\Hom(x,x) = \spn_\mathbbm{F}\{\mathbbm{1}_x\}$.
Let $S \subset \Obj(\mathcal{C})$ be a collection of simple objects, following \cite{Muger2001} we further say that $S$ is a collection of \textbf{disjoint simple objects} if $\Hom(x,y) = \{0\}$ for all $x\neq y$ in $S$.
\end{defn}

\begin{defn}
An integer triple $(a,b,c)$ is called \textbf{admissible} if $a+b+c$ is even and $a+b\geq c$, $b+c\geq a$, and $a+c\geq b$.
\end{defn}

% -------------- THEOREM: JWPs have 0 or 1-d hom-spaces between them --------------
\begin{theorem}
\label{thm:JWP_hom_dim}
Let $p_a, p_b, p_c$ be Jones-Wenzl idempotents in $\Kar\TL$.
The hom-space $\Hom(p_a\otimes p_b, p_c)$ is $1$-dimensional if and only if $(a,b,c)$ is admissible, and zero otherwise.
\end{theorem}

\begin{proof}
Let us first consider the diagrams $g = p_c\circ f \circ (p_a\otimes p_b)$ where $f$ is a simple diagram from $a+b$ to $c$ points (see Fig. \ref{fig:Kar_mor}).
By Remark \ref{rem:loopless-diagrams} we may assume $f$ has no closed loops.
We will prove that if $(a,b,c)$ is not admissible the composite $g$ is zero for all simple diagrams $f$, while for admissible triples it is zero for all but one $f$.
The result will then follow from Lemma \ref{lemma:Kar_hom} and bilinearity of composition.

% ------------- Fig: A morphism from p_a \otimes p_b to p_c in Kar TL -------------
\begin{figure}
\[
g \ \ = \ \ \begin{gathered}
\psscalebox{1.0 1.0} % Change this value to rescale the drawing.
{
\begin{pspicture}(0,-2.092213)(2.9044442,2.092213)
\psline[linecolor=black, linewidth=0.02](1.9740884,2.087108)(1.9740884,1.7735944)
\psline[linecolor=black, linewidth=0.02](0.9524668,2.092213)(0.9524668,1.7570778)
\psline[linecolor=black, linewidth=0.02](2.7470696,-1.7951238)(2.7470696,-2.0877686)
\psline[linecolor=black, linewidth=0.02](1.7080712,-1.8060784)(1.7080712,-2.080069)
\psline[linecolor=black, linewidth=0.02](2.7426252,-0.96575564)(2.7426252,-1.2913415)
\psline[linecolor=black, linewidth=0.02](1.7036269,-0.96607846)(1.7036269,-1.2978468)
\psline[linecolor=black, linewidth=0.02](1.974644,1.26933)(1.974644,0.9358165)
\psline[linecolor=black, linewidth=0.02](0.95302236,1.2655463)(0.95302236,0.93041116)
\psframe[linecolor=black, linewidth=0.02, dimen=outer](2.0912466,1.7790859)(0.8131978,1.2620127)
\psframe[linecolor=black, linewidth=0.02, dimen=outer](2.8564265,-1.2884343)(1.5783777,-1.8055074)
\psframe[linecolor=black, linewidth=0.02, linestyle=dashed, dash=0.17638889cm 0.10583334cm, dimen=outer](2.9044445,0.9486544)(0.0,-0.9713456)
\rput[bl](1.2573835,1.905018){$\cdots$}
\rput[bl](1.2755654,1.0386544){$\cdots$}
\rput[bl](2.0225928,-1.196627){$\cdots$}
\rput[bl](2.0473258,-2.0391955){$\cdots$}
\rput[bl](1.2922221,1.3886544){$p_c$}
\rput[bl](1.3572221,-0.19316378){$f$}
\rput[bl](2.060861,-1.6731349){$p_b$}
\psline[linecolor=black, linewidth=0.02](1.2181808,-1.7995683)(1.2181808,-2.092213)
\psline[linecolor=black, linewidth=0.02](0.1791824,-1.8105229)(0.1791824,-2.0845134)
\psline[linecolor=black, linewidth=0.02](1.2137364,-0.97020006)(1.2137364,-1.295786)
\psline[linecolor=black, linewidth=0.02](0.17473795,-0.9705229)(0.17473795,-1.3022913)
\psframe[linecolor=black, linewidth=0.02, dimen=outer](1.3275377,-1.2928787)(0.04948889,-1.8099519)
\rput[bl](0.48481497,-1.1921825){$\cdots$}
\rput[bl](0.509548,-2.0347512){$\cdots$}
\rput[bl](0.5319723,-1.6775794){$p_a$}
\end{pspicture}
}
\end{gathered}
\]
\caption{A morphism $g\colon p_a\otimes p_b \rightarrow p_c$ in $\Kar\TL$.}\label{fig:Kar_mor}
\end{figure}
% ------------- END FIG -------------

If $a+b+c$ is odd then there are no such $f$ and hence no morphisms from $p_a\otimes p_b$ to $p_c$ apart from the zero morphism.
So $\Hom(p_a\otimes p_b,p_c) = \{0\}$ in this case.

Let $a+b+c$ be even.
We have the following cases.

Case 1: $a+b<c$.
In this case any $f\colon a+b\rightarrow c$ must have at most $a+b$ through-strings and hence has a cup on its top edge.
Then $g =(p_c\circ f) \circ (p_a\otimes p_b) = 0 \circ (p_a\otimes p_b)$ is always zero and $\Hom(p_a \otimes p_b, p_c) = \{0\}$.

Case 2: $a+b \geq c$.
Now $f$ must have at most $c$ through-strings, but if it has any fewer then it contains a cup on its top edge and $g = 0$ as before.
Thus we only need consider $f$ with exactly $c$ through-strings; in this case the bottom edge of $f$ has $a+b-c$ points connected by caps.
However, if $f$ caps off either of $p_a$ or $p_b$ then again $g=0$, so to find nonzero morphisms $g$ we may restrict our attention further to those simple diagrams whose caps connect a strand of $p_a$ to a strand of $p_b$.
There is only one possibility for such a diagram $f$: it consists of $k = \frac{a+b-c}{2}$ successively nested caps connecting the $k$ rightmost strands of $p_a$ with the $k$ leftmost strands of $p_b$, and through-strings connecting the remaining $2c$ points (Fig. \ref{fig:nonzero-yielding_f}).

% ------------- Fig: Only f yielding nonzero g -------------
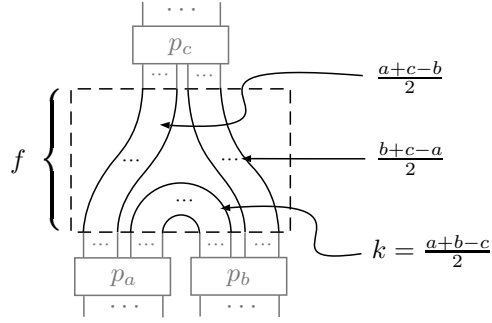
\begin{figure}[htb]
\begin{center}
\psscalebox{1.0 1.0} % Change this value to rescale the drawing.
{
\begin{pspicture}(0,-2.0922132)(5.6452446,2.0922132)
\definecolor{colour1}{rgb}{0.46666667,0.46666667,0.46666667}
\psline[linecolor=colour1, linewidth=0.02](2.7622545,2.087108)(2.7622545,1.7735944)
\psline[linecolor=colour1, linewidth=0.02](1.7406329,2.0922132)(1.7406329,1.7742207)
\psline[linecolor=colour1, linewidth=0.02](3.535236,-1.7951237)(3.535236,-2.0877686)
\psline[linecolor=colour1, linewidth=0.02](2.4962373,-1.8060783)(2.4962373,-2.080069)
\psline[linecolor=colour1, linewidth=0.02](3.5307913,-0.9657556)(3.5307913,-1.2913415)
\psline[linecolor=colour1, linewidth=0.02](2.491793,-0.96607834)(2.491793,-1.2978468)
\psline[linecolor=colour1, linewidth=0.02](2.76281,1.2693301)(2.76281,0.93581665)
\psline[linecolor=colour1, linewidth=0.02](1.7411884,1.2655463)(1.7411884,0.9304112)
\psframe[linecolor=colour1, linewidth=0.02, dimen=outer](2.8794127,1.779086)(1.6013638,1.2620128)
\psframe[linecolor=colour1, linewidth=0.02, dimen=outer](3.6445925,-1.2884341)(2.3665438,-1.8055073)
\psframe[linecolor=black, linewidth=0.02, linestyle=dashed, dash=0.17638889cm 0.10583334cm, dimen=outer](3.6926105,0.9486545)(0.78816605,-0.97134554)
\rput[bl](2.0455496,1.9050181){\textcolor{colour1}{$\cdots$}}
\rput[bl](2.835492,-2.0391955){\textcolor{colour1}{$\cdots$}}
\rput[bl](2.0803883,1.3886545){\textcolor{colour1}{$p_c$}}
\rput[bl](0.0,-0.17505312){$f$}
\rput[bl](2.8490272,-1.6731348){\textcolor{colour1}{$p_b$}}
\psline[linecolor=colour1, linewidth=0.02](2.006347,-1.7995682)(2.006347,-2.092213)
\psline[linecolor=colour1, linewidth=0.02](0.96734846,-1.8105228)(0.96734846,-2.0845134)
\psline[linecolor=colour1, linewidth=0.02](2.0019023,-0.9702)(2.0019023,-1.2957859)
\psline[linecolor=colour1, linewidth=0.02](0.962904,-0.9705228)(0.962904,-1.3022912)
\psframe[linecolor=colour1, linewidth=0.02, dimen=outer](2.1157038,-1.2928786)(0.83765495,-1.8099518)
\rput[bl](1.2977141,-2.034751){\textcolor{colour1}{$\cdots$}}
\rput[bl](1.3201383,-1.6775793){\textcolor{colour1}{$p_a$}}
\psarc[linecolor=black, linewidth=0.02, dimen=outer](2.2446406,-0.976194){0.24242425}{0.0}{180.0}
\psarc[linecolor=black, linewidth=0.02, dimen=outer](2.2426653,-0.9692804){0.6597082}{0.0}{180.0}
\psbezier[linecolor=black, linewidth=0.02](1.7404878,0.9363537)(1.7182246,0.4052874)(1.5090024,0.16808036)(1.3684949,-0.0024117369)(1.2279874,-0.17290384)(0.9832228,-0.53215486)(0.96024096,-0.9599426)
\psbezier[linecolor=black, linewidth=0.02](1.4108529,-0.96067125)(1.4170504,-0.52274495)(1.730664,-0.22030145)(1.807487,-0.11639207)(1.8843099,-0.0124826925)(2.190435,0.3177961)(2.1911886,0.943724)
\psbezier[linecolor=black, linewidth=0.02](2.7648828,0.9515636)(2.8015378,0.31432924)(3.083288,0.012312092)(3.1817718,-0.09465865)(3.2802556,-0.20162939)(3.5068116,-0.45356354)(3.531105,-0.95688087)
\psbezier[linecolor=black, linewidth=0.02](2.3364627,0.9308269)(2.3702757,0.24885371)(2.6224895,-0.0047052153)(2.7094824,-0.13102704)(2.7964756,-0.25734887)(3.087661,-0.5426759)(3.0871477,-0.9625065)
\rput[bl](2.164121,-0.6061507){$\scriptstyle \cdots$}
\rput[bl](2.7602112,-0.061990313){$\scriptstyle \cdots$}
\rput[bl](1.4633691,-0.076025404){$\scriptstyle \cdots$}
\rput{90.0}(-0.2655724,-1.6068038){\rput[bl](0.67061573,-0.9361881){$\overbrace{\qquad\qquad\quad\ }$}}
\psbezier[linecolor=black, linewidth=0.02, arrowsize=0.05291666666666667cm 2.0,arrowlength=1.4,arrowinset=0.0]{<-}(2.7617035,-0.64720565)(5.0897036,-0.47120565)(3.1377034,-1.4632057)(4.6177034,-1.2472056)
\rput[bl](4.756122,-1.4256243){$k=\frac{a+b-c}{2}$}
\psbezier[linecolor=black, linewidth=0.02, arrowsize=0.05291666666666667cm 2.0,arrowlength=1.4,arrowinset=0.0]{<-}(1.9654243,0.46274784)(4.443249,0.31391063)(1.4707633,1.3743757)(4.635192,1.1139107)
\psbezier[linecolor=black, linewidth=0.02, arrowsize=0.05291666666666667cm 2.0,arrowlength=1.4,arrowinset=0.0]{<-}(3.0356874,0.0126832295)(3.8285322,0.014276368)(4.0905166,0.007903818)(4.6687403,0.011090092)
\rput[bl](4.7947187,0.8796389){$\frac{a+c-b}{2}$}
\rput[bl](4.805245,-0.19755408){$\frac{b+c-a}{2}$}
\psline[linecolor=colour1, linewidth=0.02](2.3350368,0.9465721)(2.3350368,1.2665721)
\psline[linecolor=colour1, linewidth=0.02](2.18837,0.9465721)(2.18837,1.2665721)
\psline[linecolor=colour1, linewidth=0.02](3.0863469,-0.9702)(3.0863469,-1.2957859)
\psline[linecolor=colour1, linewidth=0.02](2.8996801,-0.9702)(2.8996801,-1.2957859)
\psline[linecolor=colour1, linewidth=0.02](1.5841247,-0.9702)(1.5841247,-1.2957859)
\psline[linecolor=colour1, linewidth=0.02](1.4107913,-0.9657556)(1.4107913,-1.2913415)
\rput[bl](2.5774543,-1.179484){\textcolor{colour1}{$\scriptstyle \cdots$}}
\rput[bl](3.1952322,-1.179484){\textcolor{colour1}{$\scriptstyle \cdots$}}
\rput[bl](1.6618987,-1.188373){\textcolor{colour1}{$\scriptstyle \cdots$}}
\rput[bl](1.0618988,-1.188373){\textcolor{colour1}{$\scriptstyle \cdots$}}
\rput[bl](1.8396766,1.0649604){\textcolor{colour1}{$\scriptstyle \cdots$}}
\rput[bl](2.4263432,1.060516){\textcolor{colour1}{$\scriptstyle \cdots$}}
\end{pspicture}
}
\end{center}
\caption{Nonzero morphism $g\colon p_a\otimes p_b \rightarrow p_c$.}\label{fig:nonzero-yielding_f}
\end{figure}
% ------------- END FIG -------------

Now $f$ exists precisely when
\begin{align*}
& a \geq \frac{a+b-c}{2} \quad \text{and} \quad b \geq \frac{a+b-c}{2} \\
\iff \ & a+c \geq b \quad \text{and} \quad b+c \geq a,
\end{align*}
and in this case writing $p_a = \id_a + m_a, p_b = \id_b + m_b, p_c = \id_c+m_c$ as in property \eqref{eq:JWP1} of Theorem \ref{thm:JWP} and expanding, we have that
\begin{align}
g &= p_c \circ f \circ (p_a \otimes p_b) \notag \\
  &= f + f \circ (m_a \otimes \id_b) + f \circ (\id_a \otimes m_b) + f \circ (m_a\otimes m_b) + m_c \circ f \circ (p_a \otimes p_b).
\label{eq:g}
\end{align}

We claim that the coefficient of $f$ in the above expression is exactly $1$, so in particular $g$ is nonzero.
To see this, let us partition the points on the top and bottom edges of $g$ into three groups: the $a$ leftmost points on the bottom edge, the remaining rightmost $b$ points on the bottom edge, and the $c$ points on the top edge.
Observe that every arc in $f$ connects points from distinct groups $a, b, c$.
On the other hand each of the other terms in \eqref{eq:g} contains a $m_a$, $m_b$ or $m_c$ term, each of which is a sum of non-identity diagrams on $a$, $b$ or $c$ points which contain cups and caps.
Hence these terms all have arcs which connect two points in some group $a,b,c$, and so do not contain $f$ as a summand, thus proving our claim.
\end{proof}
% -------------- END THM --------------

\begin{corollary}
\label{cor:JWP-disjointsimpleobjs}
The Jones-Wenzl idempotents are simple, and form a collection of disjoint simple objects in $\Kar\TL$.
\end{corollary}
\begin{proof}
$(0,a,a)$ is admissible hence $\Hom(p_a, p_a) = \Hom(p_0 \otimes p_a , p_a)$ is one-dimensional, spanned by $g = p_a \circ \id_a \circ p_a = p_a = \mathbbm{1}_{p_a}$.
On the other hand if $b \neq c$ then $(0,b,c)$ is not admissible, hence $\Hom(p_b, p_c) = \Hom(p_0 \otimes p_b , p_c) = \{0\}$.
\end{proof}

Finally, let us extend our category one last time.

\begin{defn}
Let $\mathcal{C}$ be a $\mathrm{Ab}$-enriched category, i.e. one whose hom-sets are additive abelian groups, and where composition distributes over addition.
The \textbf{additive completion} or \emph{matrix category} $\Mat\mathcal{C}$ of $\mathcal{C}$ is the category formed by taking all formal finite direct sums $\bigoplus_i x_i$ of objects $x_i$ in $\mathcal{C}$.
A morphism $f\colon \bigoplus_{i=1}^n x_i \rightarrow \bigoplus_{j=1}^m y_j$ is a $m\times n$ matrix with columns indexed by $x_i$ and rows by $y_j$, where the $(j,i)$-th entry is a morphism $f_{j,i}\colon x_i \rightarrow y_j$.
Composition of morphisms is given by matrix multiplication, and the identity morphism on $\bigoplus_i x_i$ is the matrix direct sum
\begin{equation*}
\mathbbm{1}_{\bigoplus_i x_i} = \bigoplus_i \begin{bmatrix} \mathbbm{1}_{x_i} \end{bmatrix}.
\end{equation*}
\end{defn}

Recall, or refer to Chapter VIII of \cite{MacLane1998}, that an additive category is a $\mathrm{Ab}$-enriched category which further has a zero object and finite biproducts (equivalently, products or coproducts) for all objects.
The additive completion of a category $\mathcal{C}$ is then an additive category containing $\mathcal{C}$ as a subcategory, where the biproduct of objects in $\Mat\mathcal{C}$ is given by their direct sum.

If $\mathcal{C}$ is a $\mathbbm{F}$-linear category then so is $\Mat\mathcal{C}$: its hom-spaces $\Hom(\bigoplus_i x_i, \bigoplus_j y_j)$ are $\mathbbm{F}$-vector spaces isomorphic to the direct sum $\bigoplus_{i,j}\Hom(x_i,y_j)$.
If in addition $\mathcal{C}$ is linear monoidal, we can further lift this structure to $\Mat\mathcal{C}$ by defining $\otimes$ in $\Mat\mathcal{C}$ to be the tensor product in $\mathcal{C}$ on single (unary direct sums of) objects and extending bilinearly over $\oplus$, while we define $\otimes$ of morphisms to be the usual Kronecker product of matrices with composition of the entries in place of multiplication.

We are now ready to introduce the main ingredient in the construction of our skein modules.

\begin{defn}
Take the full subcategory of $\Mat(\Kar\TL)$ having as objects the closure of the set of Jones-Wenzl idempotents under the direct sum and tensor product.
This is a $\mathbbm{F}$-linear monoidal category which we call the \textbf{generic Temperley-Lieb-Jones category} $\TLJ$.
\end{defn}
\todo{For the reasons of 1. avoiding another possibly lengthy proof of semisimplicity of the whole $\Mat\Kar\TL$, and 2. the fact that we seem to only need direct sums of tensor products of $p_n$, I have reverted back to my original definition of $\TLJ$.}

Observe that $\Kar\TL$ embeds fully and faithfully into $\TLJ$ (with objects and morphisms corresponding to unary direct sums and $1\times 1$ matrices in the obvious way), and that the Jones-Wenzl idempotents still form a collection of disjoint simple objects in $\TLJ$.

We wish to show that $\TLJ$ is \emph{semisimple}; that is, every object in $\TLJ$ is isomorphic to a direct sum of simple objects, namely the Jones-Wenzl idempotents.
In order to do so we need to develop a bit more theory.

\chapter{Temperley-Lieb skein theory}
\section{Trivalent graphs and TL diagrams}
In the previous section we saw that $\Hom(p_a\otimes p_b,p_c)$ is a one-dimensional $\mathbbm{F}$-vector space precisely when $(a,b,c)$ is admissible, in which case $\Hom(p_a\otimes p_b,p_c)$ is spanned by the formal diagram
\begin{equation*}
g_{a,b,c} \ \ = \ \
\begin{gathered}
\psscalebox{0.8} % Change this value to rescale the drawing.
{
\begin{pspicture}(0,-1.7630463)(2.8069377,1.7630463)
\definecolor{colour1}{rgb}{0.003921569,0.003921569,0.003921569}
\psline[linecolor=colour1, linewidth=0.02](1.9329329,1.7579412)(1.9329329,1.4444277)
\psline[linecolor=colour1, linewidth=0.02](0.91131127,1.7630464)(0.91131127,1.445054)
\psline[linecolor=colour1, linewidth=0.02](2.6975808,-1.4659572)(2.6975808,-1.7586019)
\psline[linecolor=colour1, linewidth=0.02](1.6585824,-1.4769118)(1.6585824,-1.7509023)
\psframe[linecolor=colour1, linewidth=0.02, dimen=outer](2.050091,1.4499193)(0.7720423,0.9328461)
\psframe[linecolor=colour1, linewidth=0.02, dimen=outer](2.8069377,-0.95926756)(1.528889,-1.4763408)
\rput[bl](1.2510667,1.0594878){\textcolor{colour1}{$p_c$}}
\rput[bl](2.0113723,-1.3439683){\textcolor{colour1}{$p_b$}}
\psline[linecolor=colour1, linewidth=0.02](1.168692,-1.4704015)(1.168692,-1.7630464)
\psline[linecolor=colour1, linewidth=0.02](0.12969357,-1.4813561)(0.12969357,-1.7553468)
\psframe[linecolor=colour1, linewidth=0.02, dimen=outer](1.2780489,-0.963712)(0.0,-1.4807851)
\rput[bl](0.48248345,-1.3484126){\textcolor{colour1}{$p_a$}}
\psarc[linecolor=black, linewidth=0.02, dimen=outer](1.4069856,-0.9803607){0.24242425}{0.0}{180.0}
\psarc[linecolor=black, linewidth=0.02, dimen=outer](1.4050103,-0.97344714){0.6597082}{0.0}{180.0}
\psbezier[linecolor=black, linewidth=0.02](0.902833,0.93218696)(0.8805697,0.40112072)(0.67134756,0.16391365)(0.53084004,-0.0065784496)(0.3903325,-0.17707056)(0.14556788,-0.5363216)(0.12258605,-0.9641093)
\psbezier[linecolor=black, linewidth=0.02](0.573198,-0.96483797)(0.5793955,-0.5269117)(0.8930092,-0.22446816)(0.96983206,-0.12055878)(1.0466549,-0.016649405)(1.35278,0.31362942)(1.3535337,0.93955725)
\psbezier[linecolor=black, linewidth=0.02](1.927228,0.9473969)(1.9638828,0.3101625)(2.2456331,0.008145379)(2.344117,-0.098825365)(2.4426007,-0.20579611)(2.6691568,-0.45773026)(2.6934502,-0.9610476)
\psbezier[linecolor=black, linewidth=0.02](1.4988079,0.9266602)(1.5326208,0.244687)(1.7848345,-0.008871928)(1.8718276,-0.13519377)(1.9588207,-0.2615156)(2.2500062,-0.54684263)(2.249493,-0.96667314)
\rput[bl](0.31904852,0.10296096){$i$}
\rput[bl](2.4523818,0.027960962){$j$}
\rput[bl](1.3440485,-0.25537238){$k$}
\rput[bl](1.2107152,1.5446277){$\cdots$}
\rput[bl](0.42738187,-1.6803724){$\cdots$}
\rput[bl](1.9940486,-1.672039){$\cdots$}
\rput[bl](0.6523819,-0.07203904){$\scriptstyle\cdots$}
\rput[bl](1.3023819,-0.5803724){$\scriptstyle\cdots$}
\rput[bl](1.9273819,-0.08037237){$\scriptstyle\cdots$}
\end{pspicture}
}
\end{gathered}
\ \ = \ \
\begin{gathered}
\psscalebox{0.8} % Change this value to rescale the drawing.
{
\begin{pspicture}(0,-1.6027273)(3.5945454,1.6027273)
\definecolor{colour0}{rgb}{0.003921569,0.003921569,0.003921569}
\psline[linecolor=black, linewidth=0.02](1.7917753,0.5658509)(1.7917753,1.3465613)
\psframe[linecolor=colour0, linewidth=0.02, dimen=outer](2.2844281,0.57011575)(1.3098263,0.2823323)
\psline[linecolor=black, linewidth=0.02](0.98861605,-0.9593419)(0.25413743,-1.4093337)
\rput{120.0}(0.884416,-2.2798777){\psframe[linecolor=colour0, linewidth=0.02, dimen=outer](1.5876529,-0.7407383)(0.6130511,-1.0285217)}
\psline[linecolor=black, linewidth=0.02](2.6056383,-0.9495448)(3.3525999,-1.4000793)
\rput{-120.0}(4.4989376,0.8282032){\psframe[linecolor=colour0, linewidth=0.02, dimen=outer](2.9758513,-0.7407381)(2.0012496,-1.0285215)}
\psbezier[linecolor=black, linewidth=0.02](1.6057084,0.28550643)(1.6057084,-0.30868474)(1.5100133,-0.3971813)(1.1092901,-0.63738626)
\psbezier[linecolor=black, linewidth=0.02](1.9944695,0.28550643)(1.9944695,-0.3592542)(2.1439931,-0.4098237)(2.472945,-0.6247439)
\psbezier[linecolor=black, linewidth=0.02](1.3006803,-0.9913725)(1.7073843,-0.7574887)(1.8808317,-0.7574887)(2.2755737,-0.98505133)
\rput[bl](2.2231793,-0.34031078){$j$}
\rput[bl](1.7483072,-1.1739248){$k$}
\rput[bl](1.288032,-0.27906227){$i$}
\rput[bl](1.7363636,1.4427273){$c$}
\rput[bl](0.0,-1.5754546){$a$}
\rput[bl](3.4545455,-1.6027273){$b$}
\end{pspicture}
}
\end{gathered}
\end{equation*}
where $i = \frac{a+c-b}{2}$, $j=\frac{b+c-a}{2}$, $k=\frac{a+b-c}{2}$, and in the right hand simplified diagram:
\begin{itemize}
\item an edge marked by $n$ represents a collection of $n$ parallel strands, and
\item a rectangle
$\begin{gathered}
\psscalebox{1.0 1.0} % Change this value to rescale the drawing.
{
\begin{pspicture}(0,-0.395102)(0.48,0.395102)
\definecolor{colour0}{rgb}{0.99607843,0.99607843,0.99607843}
\psline[linecolor=black, linewidth=0.02](0.24571429,0.39510202)(0.24571429,-0.39510205)
\psframe[linecolor=black, linewidth=0.02, fillstyle=solid,fillcolor=colour0, dimen=outer](0.48,0.04938774)(0.0,-0.09061226)
\rput[bl](0.31396824,-0.3320635){$\scriptstyle n$}
\end{pspicture}
}
\end{gathered}$
represents a Jones-Wenzl idempotent $p_n$ on $n$ strands.
\end{itemize}
We also introduce the use of edge-labelled uni-trivalent graphs to represent a certain class of TL diagrams built by connecting Jones-Wenzl idempotents.

\begin{defn}[Trivalent representations of TL diagrams]
A \emph{uni-trivalent} graph is one whose vertices have degree $1$ or $3$.
Let $\Gamma$ be a planar uni-trivalent graph with edges labelled by natural numbers, such that at every trivalent vertex $v$ the labels $(a,b,c)$ of the edges incident to $v$ form an admissible triple.
We implicitly assume that $\Gamma$ is drawn in the unit square $I\times I$, and require that its univalent vertices are either on the top or bottom edge.
Then $\Gamma$ represents a formal TL diagram in the following manner: an edge labelled by $n$ indicates the presence of a Jones-Wenzl idempotent $p_n$, and a trivalent vertex with incident edges $a,b,c$ represents the basis element $g_{a,b,c}$ of $\Hom(p_a\otimes p_b,p_c)$.
\end{defn}

Thus for instance
\begin{equation*}
g_{a,b,c} \ \ = \ \ 
\begin{gathered}
\psscalebox{1.0 1.0} % Change this value to rescale the drawing.
{
\begin{pspicture}(0,-0.5875841)(1.4244444,0.5875841)
\psline[linecolor=black, linewidth=0.02](0.69777775,0.5875841)(0.69777775,-0.19146354)
\psline[linecolor=black, linewidth=0.02](0.69821805,-0.18495561)(1.3728931,-0.5744794)
\psline[linecolor=black, linewidth=0.02](0.701782,-0.18940005)(0.02710693,-0.5789239)
\rput[bl](0.82222223,0.39647296){$c$}
\rput[bl](1.2844445,-0.37241593){$b$}
\rput[bl](0.0,-0.37686038){$a$}
\end{pspicture}
}
\end{gathered}
\end{equation*}
is the representation of $g_{a,b,c}$ as a trivalent graph.
We will use uni-trivalent graphs and formal diagrams interchangeably, as required.

\section{Theta nets}
\begin{defn}
The TL diagram represented by the graph
\begin{equation*}
\theta(a,b,c) \ \ = \ \ 
\begin{gathered}
\psscalebox{1.0 1.0} % Change this value to rescale the drawing.
{
\begin{pspicture}(0,-0.97727275)(1.46,0.97727275)
\psellipse[linecolor=black, linewidth=0.02, dimen=outer](0.73,-0.15727273)(0.73,0.82)
\psline[linecolor=black, linewidth=0.02](0.003902439,-0.18498892)(1.4514071,-0.18498892)
\rput[bl](0.65636367,0.8172727){$a$}
\rput[bl](0.6490909,-0.095454544){$b$}
\rput[bl](0.64363635,-0.87545455){$c$}
\end{pspicture}
}
\end{gathered}
\end{equation*}
is called a \textbf{theta net}.
\end{defn}

Expanding fully we see that the theta nets are given by
\begin{align*}
\theta(a,b,c)
\ \ &= \ \ 
\begin{gathered}
\psscalebox{1.0 1.0} % Change this value to rescale the drawing.
{
\begin{pspicture}(0,-1.3447769)(2.7054546,1.3447769)
\rput{-60.0}(0.56472343,2.2468257){\psframe[linecolor=black, linewidth=0.02, dimen=outer](2.3133795,0.92801374)(2.14296,0.3406822)}
\rput{60.0}(0.51468676,-2.2757144){\psframe[linecolor=black, linewidth=0.02, dimen=outer](2.3133795,-0.3984596)(2.14296,-0.98579115)}
\psframe[linecolor=black, linewidth=0.02, dimen=outer](1.9109626,0.25207853)(1.7405431,-0.33525303)
\rput{-60.0}(0.8633737,0.111156434){\psframe[linecolor=black, linewidth=0.02, dimen=outer](0.61316085,-0.39845955)(0.44274142,-0.98579115)}
\rput{60.0}(0.813337,-0.14004502){\psframe[linecolor=black, linewidth=0.02, dimen=outer](0.6131608,0.9280138)(0.44274136,0.34068227)}
\psframe[linecolor=black, linewidth=0.02, dimen=outer](1.0155778,0.25207883)(0.8451584,-0.3352527)
\psline[linecolor=black, linewidth=0.02](1.0119902,-0.034495503)(1.7463896,-0.034495503)
\psbezier[linecolor=black, linewidth=0.02](0.5868573,0.5215616)(0.5592852,0.32676482)(0.6630008,0.094940886)(0.8486228,0.098334566)
\psbezier[linecolor=black, linewidth=0.02](1.8953633,0.09691807)(2.1213768,0.14192994)(2.1880813,0.2929162)(2.1790507,0.5258838)
\psbezier[linecolor=black, linewidth=0.02](0.85698265,-0.16903448)(0.73314244,-0.18121794)(0.5923353,-0.27259383)(0.611889,-0.55890495)
\psbezier[linecolor=black, linewidth=0.02](1.897931,-0.1507593)(2.120811,-0.18584405)(2.169669,-0.35178626)(2.169669,-0.58327186)
\psbezier[linecolor=black, linewidth=0.02](0.37724432,0.6411651)(0.16867128,0.30002844)(0.20126082,-0.38833663)(0.39679804,-0.68683124)
\psbezier[linecolor=black, linewidth=0.02](2.3912377,0.6419849)(2.580257,0.34958205)(2.5672612,-0.357878)(2.3717241,-0.68073946)
\psbezier[linecolor=black, linewidth=0.02](2.1957407,0.70208234)(1.8698452,1.1955122)(0.89215904,1.2016039)(0.540192,0.7142658)
\psbezier[linecolor=black, linewidth=0.02](0.57278156,-0.7477485)(0.9377844,-1.1558942)(1.8372557,-1.1498024)(2.1957407,-0.7538402)
\rput[bl](1.5643063,0.03700517){$\scriptstyle b$}
\rput[bl](0.0,-0.12841533){$\scriptstyle l$}
\rput[bl](2.6454546,-0.13750625){$\scriptstyle l$}
\rput[bl](0.40223655,0.09016543){$\scriptstyle m$}
\rput[bl](2.1623335,0.09195706){$\scriptstyle m$}
\rput[bl](0.43259212,-0.3716182){$\scriptstyle n$}
\rput[bl](2.1872327,-0.34033903){$\scriptstyle n$}
\psframe[linecolor=black, linewidth=0.02, fillstyle=solid, dimen=outer](1.4709626,1.344777)(1.3005432,0.7574454)
\psframe[linecolor=black, linewidth=0.02, fillstyle=solid, dimen=outer](1.484296,-0.7574453)(1.3138765,-1.3447769)
\psframe[linecolor=black, linewidth=0.02, fillstyle=solid, dimen=outer](1.4709626,0.25311026)(1.3005432,-0.33422127)
\rput[bl](1.1054827,0.03700517){$\scriptstyle b$}
\rput[bl](1.7165507,1.1256105){$\scriptstyle a$}
\rput[bl](1.7215507,-1.198032){$\scriptstyle c$}
\rput[bl](0.9459625,1.1256105){$\scriptstyle a$}
\rput[bl](0.9509625,-1.198032){$\scriptstyle c$}
\end{pspicture}
}
\end{gathered} \\
&= \ \ 
\begin{gathered}
\psscalebox{1.0 1.0} % Change this value to rescale the drawing.
{
\begin{pspicture}(0,-1.0528004)(2.9010253,1.0528004)
\psframe[linecolor=black, linewidth=0.02, dimen=outer](0.6833824,-0.008439892)(0.0,-0.25818348)
\psframe[linecolor=black, linewidth=0.02, dimen=outer](2.9010255,-0.008439704)(2.217643,-0.25818336)
\psframe[linecolor=black, linewidth=0.02, dimen=outer](1.7922039,-0.00843977)(1.1088215,-0.2581834)
\rput[bl](0.23950326,-0.22118896){$\scriptstyle p_a$}
\rput[bl](1.3374624,-0.2252706){$\scriptstyle p_b$}
\rput[bl](2.4354217,-0.22118896){$\scriptstyle p_c$}
\psbezier[linecolor=black, linewidth=0.02](0.45824167,-0.2481834)(0.47203815,-0.5297881)(1.3537407,-0.53526753)(1.3510988,-0.2481834)
\psbezier[linecolor=black, linewidth=0.02](1.6010988,-0.25532627)(1.6696898,-0.54456306)(2.42468,-0.52098185)(2.4532514,-0.2481834)
\psbezier[linecolor=black, linewidth=0.02](0.18681309,-0.2481834)(0.32368952,-1.2034414)(2.6241908,-1.207805)(2.7225273,-0.2481834)
\psbezier[linecolor=black, linewidth=0.02](0.4455215,-0.013251896)(0.46202588,0.31551522)(1.3217387,0.31551522)(1.36059,-0.013251896)
\psbezier[linecolor=black, linewidth=0.02](1.5852475,-0.018731348)(1.6236037,0.32099468)(2.4126449,0.32099468)(2.451001,-0.013251896)
\psbezier[linecolor=black, linewidth=0.02](0.17702836,-0.013251896)(0.20442562,1.0662001)(2.6756585,1.0662001)(2.719494,-0.013251896)
\rput[bl](0.7826888,0.29110676){$\scriptstyle m$}
\rput[bl](1.9591311,0.28092653){$\scriptstyle n$}
\rput[bl](1.443529,0.8728004){$\scriptstyle l$}
\end{pspicture}
}
\end{gathered}
\end{align*}
by the idempotent property, where the $l,m,n$ are uniquely determined by the $a,b,c$ and vice versa:
given an admissible triple $(a,b,c)$ we have the by-now familiar formulae
\begin{equation}
m = \frac{a+b-c}{2}, \quad n = \frac{b+c-a}{2}, \quad l = \frac{a+c-b}{2},
\label{eq:theta_lblconv1}
\end{equation}
and given any triple $(l,m,n)$ we have
\begin{equation}
a = m+l, \quad b = m+n, \quad c=n+l.
\label{eq:theta_lblconv2}
\end{equation}
Hence we may equivalently describe a theta net $\theta(a,b,c)$ by specifying the labels $m,n,l$.

\begin{defn}
Define
\[ \Net(m,n,l) \coloneqq \theta(a,b,c) \]
where $m,n,l$ and $a,b,c$ are related by equations \eqref{eq:theta_lblconv1} and \eqref{eq:theta_lblconv2}.
\end{defn}

We will need to know the value of the theta nets $\theta(a,b,c)=\Net(m,n,l)$.
Kauffman and Lins devote an entire chapter to finding an explicit formula for this trace; we simply state their result here and refer the reader to Chapter 6 of \cite{KL1994} for a proof.

\begin{theorem}[Kauffman and Lins, 1994]
\label{thm:theta_net_trace}
The value of the theta net $\theta(a,b,c) = \Net(m,n,l)$ is given by
\[ \Net(m,n,l) = \frac{[m]![n]![l]![m+n+l+1]!}{[m+n]![n+l]![m+l]!} \]
where the quantum factorial $[n]!$ is the product of the quantum integers ranging from $[n]$ down to $[1]$.
In particular,
\[ \Net(m,n,0) = [m+n+1]. \]
\end{theorem}

From this theorem we observe the following fact.

\begin{corollary}
\label{cor:theta_trace_invariance}
The value of a theta net $\Net(m,n,l)$ is invariant under any permutation of its arguments $m,n,l$.
Hence by the relations in equation \eqref{eq:theta_lblconv1}, $\theta(a,b,c)$ is also invariant under permutations of $a,b,c$.
\end{corollary}

\section{Some diagrammatic identities}
In this section we prove the main results we need to show that $\TLJ$ is semisimple, beginning with a lemma due to Kauffman and Lins.

% -------------- LEMMA: Roundabout lemma --------------
\begin{lemma}[cf. Lemma 7 of \cite{KL1994}]
\label{lemma:roundabout_lemma}
Let $(a,c,d)$ and $(b,c,d)$ be admissible triples.
If $a \neq b$ then
\begin{equation*}
\begin{gathered}
\psscalebox{1.0 1.0} % Change this value to rescale the drawing.
{
\begin{pspicture}(0,-0.97975385)(1.1033772,0.97975385)
\psellipse[linecolor=black, linewidth=0.02, dimen=outer](0.5526531,-0.019766627)(0.32,0.4494703)
\psline[linecolor=black, linewidth=0.02](0.5508031,0.4216411)(0.5508031,0.9797538)
\psline[linecolor=black, linewidth=0.02](0.55251956,-0.979754)(0.55251956,-0.4624576)
\rput[bl](0.63673466,0.6766606){$\scriptstyle a$}
\rput[bl](0.64081633,-0.8457884){$\scriptstyle b$}
\rput[bl](0.0,-0.06619656){$\scriptstyle c$}
\rput[bl](0.9733772,-0.07241117){$\scriptstyle d$}
\end{pspicture}
}
\end{gathered}
\ \ = \ \ 0
\end{equation*}
while if $a=b$ then
\begin{equation*}
\begin{gathered}
\psscalebox{1.0 1.0} % Change this value to rescale the drawing.
{
\begin{pspicture}(0,-0.97975385)(1.1033772,0.97975385)
\psellipse[linecolor=black, linewidth=0.02, dimen=outer](0.5526531,-0.019766627)(0.32,0.4494703)
\psline[linecolor=black, linewidth=0.02](0.5508031,0.4216411)(0.5508031,0.9797538)
\psline[linecolor=black, linewidth=0.02](0.55251956,-0.979754)(0.55251956,-0.4624576)
\rput[bl](0.63673466,0.6766606){$\scriptstyle a$}
\rput[bl](0.64081633,-0.8457884){$\scriptstyle b$}
\rput[bl](0.0,-0.06619656){$\scriptstyle c$}
\rput[bl](0.9733772,-0.07241117){$\scriptstyle d$}
\end{pspicture}
}
\end{gathered}
\ \ = \ \ 
\frac{\theta(a,c,d)}{[a+1]}
\ 
\begin{gathered}
\psscalebox{1.0 1.0} % Change this value to rescale the drawing.
{
\begin{pspicture}(0,-0.68108106)(0.64,0.68108106)
\definecolor{colour0}{rgb}{0.99607843,0.99607843,0.99607843}
\psline[linecolor=black, linewidth=0.02](0.3226912,0.6810811)(0.3226912,-0.68108106)
\psframe[linecolor=black, linewidth=0.02, fillstyle=solid,fillcolor=colour0, dimen=outer](0.64,0.1291777)(0.0,-0.090822294)
\rput[bl](0.41787678,-0.45121837){$\scriptstyle a$}
\end{pspicture}
}
\end{gathered} \ .
\end{equation*}
\end{lemma}

\begin{proof}
If $a>b$ then
\begin{equation*}
\begin{gathered}
\psscalebox{1.0 1.0} % Change this value to rescale the drawing.
{
\begin{pspicture}(0,-0.97975385)(1.1033772,0.97975385)
\psellipse[linecolor=black, linewidth=0.02, dimen=outer](0.5526531,-0.019766627)(0.32,0.4494703)
\psline[linecolor=black, linewidth=0.02](0.5508031,0.4216411)(0.5508031,0.9797538)
\psline[linecolor=black, linewidth=0.02](0.55251956,-0.979754)(0.55251956,-0.4624576)
\rput[bl](0.63673466,0.6766606){$\scriptstyle a$}
\rput[bl](0.64081633,-0.8457884){$\scriptstyle b$}
\rput[bl](0.0,-0.06619656){$\scriptstyle c$}
\rput[bl](0.9733772,-0.07241117){$\scriptstyle d$}
\end{pspicture}
}
\end{gathered}
\ \ = \ \ 
\begin{gathered}
\psscalebox{1.0 1.0} % Change this value to rescale the drawing.
{
\begin{pspicture}(0,-1.4369283)(3.0876243,1.4369283)
\psframe[linecolor=black, linewidth=0.02, fillstyle=solid, dimen=outer](2.108718,0.08284173)(1.5687178,-0.09715827)
\psframe[linecolor=black, linewidth=0.02, fillstyle=solid, dimen=outer](0.890647,0.08284181)(0.350647,-0.097158186)
\psframe[linecolor=black, linewidth=0.02, fillstyle=solid, dimen=outer](1.4996821,-0.90219116)(0.9596821,-1.0821911)
\psbezier[linecolor=black, linewidth=0.02](1.128218,-0.80993986)(1.1192547,-0.6399987)(1.0400562,-0.61092955)(0.83306533,-0.5040128)(0.6260745,-0.39709604)(0.5181001,-0.33742884)(0.49492928,-0.0888106)
\psbezier[linecolor=black, linewidth=0.02](1.3397918,-0.80575305)(1.3539429,-0.6220738)(1.4691813,-0.606467)(1.6607364,-0.5057667)(1.8522915,-0.4050663)(1.9434257,-0.31995028)(1.9644328,-0.088440575)
\psbezier[linecolor=black, linewidth=0.02](1.1104332,0.78709847)(1.1029717,0.67203534)(0.90975744,0.6200584)(0.75697887,0.539589)(0.60420024,0.4591196)(0.49101967,0.3504878)(0.49340066,0.06971692)
\psbezier[linecolor=black, linewidth=0.02](1.3390481,0.79922175)(1.3515114,0.66274256)(1.4857805,0.64079094)(1.6895777,0.5390806)(1.8933748,0.43737024)(1.9634086,0.2890694)(1.9683832,0.06599119)
\psline[linecolor=black, linewidth=0.02](1.2295853,-1.073038)(1.2295853,-1.4193066)
\rput[bl](1.325878,-1.4369283){$\scriptstyle b$}
\psframe[linecolor=black, linewidth=0.02, linestyle=dotted, dotsep=0.10583334cm, dimen=outer](2.4643989,0.8046488)(0.0,-0.8151213)
\rput{-90.0}(1.804853,3.3697062){\rput[bl](2.5872796,0.7824266){$\overbrace{\qquad\qquad\ \,}$}}
\rput[bl](2.8976245,-0.1899872){$f$}
\psbezier[linecolor=black, linewidth=0.02](1.7282535,0.07807877)(1.6699412,0.3448717)(1.4065144,0.37952805)(1.2297028,0.37952805)(1.0528913,0.37952805)(0.7822786,0.33653137)(0.73115206,0.07807877)
\psbezier[linecolor=black, linewidth=0.02](0.7398844,-0.08424007)(0.79886144,-0.32124835)(1.0906807,-0.33389643)(1.238435,-0.33351544)(1.3861895,-0.33313444)(1.6534241,-0.30816713)(1.7195945,-0.08424007)
\psline[linecolor=black, linewidth=0.02](1.2295853,1.0428984)(1.2295853,1.4369283)
\rput[bl](1.3065232,1.259846){$\scriptstyle a$}
\psline[linecolor=black, linewidth=0.02, linestyle=dashed, dash=0.17638889cm 0.10583334cm](1.1066666,0.786871)(1.1066666,0.9113155)
\psline[linecolor=black, linewidth=0.02, linestyle=dashed, dash=0.17638889cm 0.10583334cm](1.1244444,-0.7997956)(1.1288887,-0.9153512)
\psline[linecolor=black, linewidth=0.02, linestyle=dashed, dash=0.17638889cm 0.10583334cm](1.3422221,-0.7997956)(1.3422221,-0.91090673)
\psframe[linecolor=black, linewidth=0.02, fillstyle=solid, dimen=outer](1.4996823,1.0549715)(0.95968235,0.8749714)
\psline[linecolor=black, linewidth=0.02, linestyle=dashed, dash=0.17638889cm 0.10583334cm](1.3377776,0.7913155)(1.3377776,0.8890933)
\end{pspicture}
}
\end{gathered}
\end{equation*}
and since $f$ is a diagram from $b<a$ to $a$ points it must have a cup on its top edge, so composition with $p_a$ gives $0$.
A similar argument applies to the case $a<b$.

On the other hand if $a=b$ then
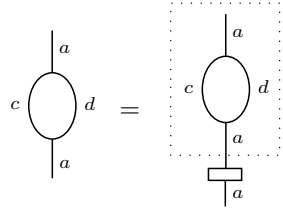
\begin{equation*}
\begin{gathered}
\psscalebox{1.0 1.0} % Change this value to rescale the drawing.
{
\begin{pspicture}(0,-0.97975385)(1.1033772,0.97975385)
\psellipse[linecolor=black, linewidth=0.02, dimen=outer](0.5526531,-0.019766627)(0.32,0.4494703)
\psline[linecolor=black, linewidth=0.02](0.5508031,0.4216411)(0.5508031,0.9797538)
\psline[linecolor=black, linewidth=0.02](0.55251956,-0.979754)(0.55251956,-0.4624576)
\rput[bl](0.63673466,0.6766606){$\scriptstyle a$}
\rput[bl](0.64081633,-0.8457884){$\scriptstyle a$}
\rput[bl](0.0,-0.06619656){$\scriptstyle c$}
\rput[bl](0.9733772,-0.07241117){$\scriptstyle d$}
\end{pspicture}
}
\end{gathered}
\ \ = \ \ 
\begin{gathered}
\psscalebox{1.0 1.0} % Change this value to rescale the drawing.
{
\begin{pspicture}(0,-1.3528796)(1.4523079,1.3528796)
\psellipse[linecolor=black, linewidth=0.02, dimen=outer](0.7353847,0.19893283)(0.32,0.4494703)
\psline[linecolor=black, linewidth=0.02](0.73353475,0.64034057)(0.73353475,1.1984533)
\psline[linecolor=black, linewidth=0.02](0.7352512,-0.85252815)(0.7352512,-0.24375814)
\rput[bl](0.8194663,0.89536005){$\scriptstyle a$}
\rput[bl](0.18273161,0.1525029){$\scriptstyle c$}
\rput[bl](1.1561089,0.14628828){$\scriptstyle d$}
\psframe[linecolor=black, linewidth=0.02, dimen=outer](0.95244145,-0.8393876)(0.49609226,-1.0171654)
\psline[linecolor=black, linewidth=0.02](0.7278383,-1.0028796)(0.7278383,-1.3528796)
\rput[bl](0.81600153,-1.246145){$\scriptstyle a$}
\psframe[linecolor=black, linewidth=0.02, linestyle=dotted, dotsep=0.10583334cm, dimen=outer](1.4523077,1.3528796)(0.0,-0.6840434)
\rput[bl](0.8163894,-0.5107938){$\scriptstyle a$}
\end{pspicture}
}
\end{gathered}
\end{equation*}
by the idempotent property of $p_a$.
Writing the dotted portion as a linear combination of simple diagrams in $\TL_n$ and observing that every diagram that is not multiple of the identity $\id_a$ is a product of generators $U_i$ and hence contains caps, we have that
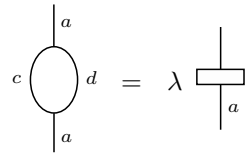
\begin{equation*}
\begin{gathered}
\psscalebox{1.0 1.0} % Change this value to rescale the drawing.
{
\begin{pspicture}(0,-0.97975385)(1.1033772,0.97975385)
\psellipse[linecolor=black, linewidth=0.02, dimen=outer](0.5526531,-0.019766627)(0.32,0.4494703)
\psline[linecolor=black, linewidth=0.02](0.5508031,0.4216411)(0.5508031,0.9797538)
\psline[linecolor=black, linewidth=0.02](0.55251956,-0.979754)(0.55251956,-0.4624576)
\rput[bl](0.63673466,0.6766606){$\scriptstyle a$}
\rput[bl](0.64081633,-0.8457884){$\scriptstyle a$}
\rput[bl](0.0,-0.06619656){$\scriptstyle c$}
\rput[bl](0.9733772,-0.07241117){$\scriptstyle d$}
\end{pspicture}
}
\end{gathered}
\ \ = \ \ \lambda \ 
\begin{gathered}
\psscalebox{1.0 1.0} % Change this value to rescale the drawing.
{
\begin{pspicture}(0,-0.68108106)(0.64,0.68108106)
\definecolor{colour0}{rgb}{0.99607843,0.99607843,0.99607843}
\psline[linecolor=black, linewidth=0.02](0.3226912,0.6810811)(0.3226912,-0.68108106)
\psframe[linecolor=black, linewidth=0.02, fillstyle=solid,fillcolor=colour0, dimen=outer](0.64,0.1291777)(0.0,-0.090822294)
\rput[bl](0.41787678,-0.45121837){$\scriptstyle a$}
\end{pspicture}
}
\end{gathered}
\end{equation*}
since only the identity terms survive the extra $p_a$.
Taking the trace we have that
\begin{equation*}
\begin{gathered}
\psscalebox{1.0 1.0} % Change this value to rescale the drawing.
{
\begin{pspicture}(0,-0.7046765)(1.732531,0.7046765)
\psellipse[linecolor=black, linewidth=0.02, dimen=outer](0.5526531,-0.011323825)(0.32,0.4494703)
\rput[bl](1.602531,-0.058240876){$\scriptstyle a$}
\rput[bl](0.0,-0.057753757){$\scriptstyle c$}
\rput[bl](0.9733772,-0.063968375){$\scriptstyle d$}
\psbezier[linecolor=black, linewidth=0.02](0.5553197,0.42344353)(0.7137322,0.5871186)(0.8423447,0.6938192)(1.0406531,0.6581102)(1.2389615,0.6224012)(1.4987438,0.45361143)(1.4948475,-0.0032787002)(1.4909512,-0.4601688)(1.2506683,-0.64551365)(1.0459864,-0.6698898)(0.84130454,-0.694266)(0.6693735,-0.60831684)(0.5344502,-0.44542605)
\end{pspicture}
}
\end{gathered}
\ \ = \ \ 
\lambda \ 
\begin{gathered}
\psscalebox{1.0 1.0} % Change this value to rescale the drawing.
{
\begin{pspicture}(0,-0.55151504)(1.2229252,0.55151504)
\definecolor{colour0}{rgb}{0.99607843,0.99607843,0.99607843}
\rput[bl](1.0929252,-0.051988628){$\scriptstyle a$}
\psellipse[linecolor=black, linewidth=0.02, dimen=outer](0.6551515,0.0)(0.36666667,0.55151516)
\psframe[linecolor=black, linewidth=0.02, fillstyle=solid,fillcolor=colour0, dimen=outer](0.64,0.12363632)(0.0,-0.09636368)
\end{pspicture}
}
\end{gathered}
\ \ = \ \ 
\lambda [a+1],
\end{equation*}
so that
\begin{equation*}
\lambda \ = \ 
\frac{
\begin{gathered}
\psscalebox{1.0 1.0} % Change this value to rescale the drawing.
{
\begin{pspicture}(0,-0.48309985)(0.72,0.48309985)
\psellipse[linecolor=black, linewidth=0.02, dimen=outer](0.36,-0.08309984)(0.36,0.4)
\psline[linecolor=black, linewidth=0.02](0.016216217,-0.09663297)(0.7081081,-0.09663297)
\rput[bl](0.31617466,0.37309983){$\scriptstyle a$}
\rput[bl](0.3119981,-0.038932778){$\scriptstyle d$}
\rput[bl](0.30863115,-0.4172404){$\scriptstyle c$}
\end{pspicture}
}
\end{gathered}
}{[a+1]}
\end{equation*}
as required.
\end{proof}
% ----------- END Roundabout lemma -----------

We also have the following two ``triangle-shrinking'' lemmas.

\allowdisplaybreaks

\begin{lemma}
\label{lemma:shrink_triangle_k+1}
Let $(a,b-1,k)$, $(b-1,1,b)$ and $(k,1,k+1)$ be admissible triples.
Then
\begin{equation*}
\begin{gathered}
\psscalebox{1.0 1.0} % Change this value to rescale the drawing.
{
\begin{pspicture}(0,-0.91888887)(1.9158479,0.91888887)
\rput[bl](0.7532164,0.447193){$\scriptstyle b-1$}
\rput[bl](0.49637428,-0.14134502){$\scriptstyle k$}
\rput[bl](1.3557894,-0.1485965){$\scriptstyle 1$}
\rput[bl](1.0872514,-0.91888887){$\scriptstyle k+1$}
\rput[bl](0.0,0.7524561){$\scriptstyle a$}
\rput[bl](1.805848,0.73888886){$\scriptstyle b$}
\psline[linecolor=black, linewidth=0.02](1.7453245,0.7198784)(1.4368912,0.40336388)
\psline[linecolor=black, linewidth=0.02](0.17507282,0.7069489)(0.49158728,0.3985156)
\psline[linecolor=black, linewidth=0.02](0.9793618,-0.38187093)(0.98063165,-0.85933125)
\pspolygon[linecolor=black, linewidth=0.02](0.4963635,0.3950733)(0.48902255,0.40238094)(1.4360014,0.40238094)(0.98086435,-0.37953943)
\end{pspicture}
}
\end{gathered}
\ \ = \ \ 
\begin{gathered}
\psscalebox{1.0 1.0} % Change this value to rescale the drawing.
{
\begin{pspicture}(0,-0.70747477)(1.6522115,0.70747477)
\psline[linecolor=black, linewidth=0.02](0.80185485,-0.64820975)(0.80185485,0.10923705)
\psline[linecolor=black, linewidth=0.02](0.80217934,0.10710939)(0.14621116,0.4858328)
\psline[linecolor=black, linewidth=0.02](0.80578566,0.10285407)(1.4617538,0.4815775)
\rput[bl](1.5422115,0.52747476){$\scriptstyle b$}
\rput[bl](0.0,0.541042){$\scriptstyle a$}
\rput[bl](0.9100797,-0.70747477){$\scriptstyle k+1$}
\end{pspicture}
}
\end{gathered} \ .
\end{equation*}
\end{lemma}

\begin{proof}
Expanding, we have
\begin{align*}
\begin{gathered}
\psscalebox{1.0 1.0} % Change this value to rescale the drawing.
{
\begin{pspicture}(0,-0.91888887)(1.9158479,0.91888887)
\rput[bl](0.7532164,0.447193){$\scriptstyle b-1$}
\rput[bl](0.49637428,-0.14134502){$\scriptstyle k$}
\rput[bl](1.3557894,-0.1485965){$\scriptstyle 1$}
\rput[bl](1.0872514,-0.91888887){$\scriptstyle k+1$}
\rput[bl](0.0,0.7524561){$\scriptstyle a$}
\rput[bl](1.805848,0.73888886){$\scriptstyle b$}
\psline[linecolor=black, linewidth=0.02](1.7453245,0.7198784)(1.4368912,0.40336388)
\psline[linecolor=black, linewidth=0.02](0.17507282,0.7069489)(0.49158728,0.3985156)
\psline[linecolor=black, linewidth=0.02](0.9793618,-0.38187093)(0.98063165,-0.85933125)
\pspolygon[linecolor=black, linewidth=0.02](0.4963635,0.3950733)(0.48902255,0.40238094)(1.4360014,0.40238094)(0.98086435,-0.37953943)
\end{pspicture}
}
\end{gathered}
\quad &= \quad
\begin{gathered}
\psscalebox{1.0 1.0} % Change this value to rescale the drawing.
{
\begin{pspicture}(0,-2.06187)(4.2250166,2.06187)
\definecolor{colour0}{rgb}{1.0,0.19607843,0.19607843}
\definecolor{colour1}{rgb}{0.39215687,0.78431374,1.0}
\rput{80.0}(1.2902669,-0.4546695){\psframe[linecolor=black, linewidth=0.02, fillstyle=solid,fillcolor=colour0, dimen=outer](0.99333316,0.8369599)(0.83878773,0.24605079)}
\rput{-40.0}(-0.7058619,0.71552056){\psframe[linecolor=black, linewidth=0.02, dimen=outer](0.7072801,1.6228846)(0.5527346,1.0319756)}
\psframe[linecolor=black, linewidth=0.02, fillstyle=solid,fillcolor=colour1, dimen=outer](1.6048522,1.3486652)(1.4503068,0.7577562)
\rput{-90.0}(3.343482,0.8925301){\psframe[linecolor=black, linewidth=0.02, dimen=outer](2.1952786,-0.9300214)(2.0407333,-1.5209305)}
\rput{-40.0}(0.7972607,0.9027535){\psframe[linecolor=black, linewidth=0.02, fillstyle=solid,fillcolor=colour0, dimen=outer](1.7160505,-0.3483966)(1.5615051,-0.93930566)}
\rput{40.0}(1.6913136,-1.9919859){\psframe[linecolor=black, linewidth=0.02, dimen=outer](3.6593976,1.6228846)(3.5048523,1.0319755)}
\psframe[linecolor=black, linewidth=0.02, fillstyle=solid,fillcolor=colour1, dimen=outer](2.7618256,1.3486652)(2.60728,0.7577561)
\psline[linecolor=black, linewidth=0.02](1.5974526,1.0671782)(2.6253595,1.0671782)
\psline[linecolor=black, linewidth=0.02](0.0064277425,1.858619)(0.5870088,1.3714536)
\psline[linecolor=black, linewidth=0.02](3.6380067,1.3714536)(4.218588,1.858619)
\psline[linecolor=black, linewidth=0.02](2.1252897,-1.3039755)(2.1252897,-2.0618703)
\psbezier[linecolor=black, linewidth=0.02](3.460534,1.3584573)(3.2473397,1.1895089)(3.0357325,1.0710566)(2.7544825,1.0710566)
\psbezier[linecolor=black, linewidth=0.02](3.585534,1.2209574)(3.2099388,0.934251)(2.210534,-0.7040426)(2.2230341,-1.1540426)
\psbezier[linecolor=black, linewidth=0.02](0.7674785,1.3802828)(0.99604994,1.2279018)(1.243669,1.1644098)(1.459542,1.170759)
\psbezier[linecolor=black, linewidth=0.02](0.62779593,1.2152034)(0.8182722,0.986632)(0.8373198,0.8913939)(0.81192297,0.63742566)
\psbezier[linecolor=black, linewidth=0.02](1.0087483,0.58663195)(1.0468435,0.7644098)(1.2373198,0.96123517)(1.459542,0.95488596)
\psbezier[linecolor=black, linewidth=0.02](1.7008119,-0.6768601)(2.0055737,-0.8990823)(2.0119228,-0.9752728)(2.0182722,-1.1467013)
\psbezier[linecolor=black, linewidth=0.02](0.88811344,0.47869548)(0.7611293,-0.11178072)(1.3897008,-0.46098706)(1.5865262,-0.5943204)
\rput[bl](0.11985946,1.8882192){$\scriptstyle a$}
\rput[bl](4.0309706,1.88187){$\scriptstyle b$}
\rput[bl](2.2586184,-1.9538586){$\scriptstyle k+1$}
\rput[bl](0.8627166,-0.42289183){$\scriptstyle k$}
\rput[bl](1.0563674,1.3453622){$\scriptstyle m$}
\rput[bl](1.2182722,0.6564733){$\scriptstyle n$}
\rput[bl](0.6119229,0.786632){$\scriptstyle l$}
\rput[bl](1.8976372,1.1263145){$\scriptstyle b-1$}
\rput[bl](2.9357324,-0.14670135){$\scriptstyle 1$}
\end{pspicture}
}
\end{gathered} \\
&= \quad \quad
\begin{gathered}
\psscalebox{1.0 1.0} % Change this value to rescale the drawing.
{
\begin{pspicture}(0,-1.6909611)(3.5050159,1.6909611)
\rput[bl](1.9168005,-1.5894141){$\scriptstyle k+1$}
\rput[bl](0.1198597,1.5318557){$\scriptstyle a$}
\psline[linecolor=black, linewidth=0.02](0.0064279926,1.5022554)(0.5870091,1.0150901)
\rput{-40.0}(-0.4767958,0.6321475){\psframe[linecolor=black, linewidth=0.02, dimen=outer](0.70728034,1.2665211)(0.5527349,0.67561203)}
\rput[bl](3.3109708,1.510961){$\scriptstyle b$}
\psline[linecolor=black, linewidth=0.02](2.918007,1.0005447)(3.498588,1.48771)
\rput{40.0}(1.28445,-1.6159552){\psframe[linecolor=black, linewidth=0.02, dimen=outer](2.939398,1.2519757)(2.7848525,0.6610666)}
\psline[linecolor=black, linewidth=0.02](1.7597915,-0.9330663)(1.7597915,-1.690961)
\rput{-90.0}(2.6070745,0.89794105){\psframe[linecolor=black, linewidth=0.02, dimen=outer](1.8297805,-0.5591122)(1.675235,-1.1500213)}
\psbezier[linecolor=black, linewidth=0.02](0.7544748,1.0238981)(1.3830463,0.5731045)(2.170348,0.5794537)(2.7100303,1.0238981)
\psbezier[linecolor=black, linewidth=0.02](1.6433637,-0.7792765)(1.560824,0.07786638)(1.0020939,0.50326324)(0.6211415,0.85881877)
\psbezier[linecolor=black, linewidth=0.02](1.7830462,-0.7856257)(1.8084431,-0.13800663)(2.2274907,0.4016759)(2.817967,0.9032632)
\psbezier[linecolor=black, linewidth=0.02](2.9005065,0.8016759)(2.4433637,0.33183464)(2.0370145,0.033421937)(1.9227288,-0.7856257)
\rput[bl](1.1263243,-0.06242996){$\scriptstyle l$}
\rput[bl](1.6501338,0.7915383){$\scriptstyle m$}
\rput[bl](1.8818799,0.17883988){$\scriptstyle n$}
\rput[bl](2.3548958,-0.12274742){$\scriptstyle 1$}
\end{pspicture}
}
\end{gathered}
\quad = \quad
\begin{gathered}
\psscalebox{1.0 1.0} % Change this value to rescale the drawing.
{
\begin{pspicture}(0,-0.70747477)(1.6522115,0.70747477)
\psline[linecolor=black, linewidth=0.02](0.80185485,-0.64820975)(0.80185485,0.10923705)
\psline[linecolor=black, linewidth=0.02](0.80217934,0.10710939)(0.14621116,0.4858328)
\psline[linecolor=black, linewidth=0.02](0.80578566,0.10285407)(1.4617538,0.4815775)
\rput[bl](1.5422115,0.52747476){$\scriptstyle b$}
\rput[bl](0.0,0.541042){$\scriptstyle a$}
\rput[bl](0.9100797,-0.70747477){$\scriptstyle k+1$}
\end{pspicture}
}
\end{gathered}
\end{align*}
where we use the idempotent and absorption rules to absorb the blue and red idempotents into $p_b$ and $p_{k+1}$ respectively.

% REMOVED TO BETTER SUIT WHEN WE SAY "Lemmas ... also hold under the condition of q-admissibility"
%Since the above relation gives a nonzero morphism between $p_a\otimes p_b$ and $p_{k+1}$, we have that $(a,b,k+1)$ is admissible.
%Alternatively we see directly that if $2\mid a+b-1+k$, then since
%\[ a+b+k+1 = (a+b-1+k)+2 \]
%we have that
%\[ 2\mid a+b+k+1, \]
%and the inequalities $a+b\geq k+1$, $b+k+1\geq a$ and $a+k+1\geq b$ follow from the admissibility of $(a,b-1,k)$.
\end{proof}

\begin{lemma}
\label{lemma:shrink_triangle_k-1}
Let $(a,b-1,k)$, $(b-1,1,b)$ and $(k,1,k-1)$ be admissible triples.
If $a+k > b-1$ then $n = \frac{k+b-1-a}{2} < k$, and
\begin{equation*}
\begin{gathered}
\psscalebox{1.0 1.0} % Change this value to rescale the drawing.
{
\begin{pspicture}(0,-0.91888887)(1.9158479,0.91888887)
\rput[bl](0.7532164,0.447193){$\scriptstyle b-1$}
\rput[bl](0.49637428,-0.14134502){$\scriptstyle k$}
\rput[bl](1.3557894,-0.1485965){$\scriptstyle 1$}
\rput[bl](1.0872514,-0.91888887){$\scriptstyle k-1$}
\rput[bl](0.0,0.7524561){$\scriptstyle a$}
\rput[bl](1.805848,0.73888886){$\scriptstyle b$}
\psline[linecolor=black, linewidth=0.02](1.7453245,0.7198784)(1.4368912,0.40336388)
\psline[linecolor=black, linewidth=0.02](0.17507282,0.7069489)(0.49158728,0.3985156)
\psline[linecolor=black, linewidth=0.02](0.9793618,-0.38187093)(0.98063165,-0.85933125)
\pspolygon[linecolor=black, linewidth=0.02](0.4963635,0.3950733)(0.48902255,0.40238094)(1.4360014,0.40238094)(0.98086435,-0.37953943)
\end{pspicture}
}
\end{gathered}
\ \ = \ \ 
(-1)^n\frac{[k-n]}{[k]} \ 
\begin{gathered}
\psscalebox{1.0 1.0} % Change this value to rescale the drawing.
{
\begin{pspicture}(0,-0.70747477)(1.6522115,0.70747477)
\psline[linecolor=black, linewidth=0.02](0.80185485,-0.64820975)(0.80185485,0.10923705)
\psline[linecolor=black, linewidth=0.02](0.80217934,0.10710939)(0.14621116,0.4858328)
\psline[linecolor=black, linewidth=0.02](0.80578566,0.10285407)(1.4617538,0.4815775)
\rput[bl](1.5422115,0.52747476){$\scriptstyle b$}
\rput[bl](0.0,0.541042){$\scriptstyle a$}
\rput[bl](0.9100797,-0.70747477){$\scriptstyle k-1$}
\end{pspicture}
}
\end{gathered}
\end{equation*}
\end{lemma}

\begin{proof}
Expanding as in the previous proof we have that
\begin{align}
\begin{gathered}
\psscalebox{1.0 1.0} % Change this value to rescale the drawing.
{
\begin{pspicture}(0,-0.91888887)(1.9158479,0.91888887)
\rput[bl](0.7532164,0.447193){$\scriptstyle b-1$}
\rput[bl](0.49637428,-0.14134502){$\scriptstyle k$}
\rput[bl](1.3557894,-0.1485965){$\scriptstyle 1$}
\rput[bl](1.0872514,-0.91888887){$\scriptstyle k-1$}
\rput[bl](0.0,0.7524561){$\scriptstyle a$}
\rput[bl](1.805848,0.73888886){$\scriptstyle b$}
\psline[linecolor=black, linewidth=0.02](1.7453245,0.7198784)(1.4368912,0.40336388)
\psline[linecolor=black, linewidth=0.02](0.17507282,0.7069489)(0.49158728,0.3985156)
\psline[linecolor=black, linewidth=0.02](0.9793618,-0.38187093)(0.98063165,-0.85933125)
\pspolygon[linecolor=black, linewidth=0.02](0.4963635,0.3950733)(0.48902255,0.40238094)(1.4360014,0.40238094)(0.98086435,-0.37953943)
\end{pspicture}
}
\end{gathered}
\ \ &= \ \ 
\begin{gathered}
\psscalebox{1.0 1.0} % Change this value to rescale the drawing.
{
\begin{pspicture}(0,-2.06187)(4.2250166,2.06187)
\definecolor{colour0}{rgb}{0.99607843,0.99607843,0.99607843}
\rput{80.0}(1.2902669,-0.4546695){\psframe[linecolor=black, linewidth=0.02, fillstyle=solid,fillcolor=colour0, dimen=outer](0.99333316,0.8369599)(0.83878773,0.24605079)}
\rput{-40.0}(-0.7058619,0.71552056){\psframe[linecolor=black, linewidth=0.02, dimen=outer](0.7072801,1.6228846)(0.5527346,1.0319756)}
\psframe[linecolor=black, linewidth=0.02, fillstyle=solid,fillcolor=colour0, dimen=outer](1.6048522,1.3486652)(1.4503068,0.7577562)
\rput{-90.0}(3.343482,0.8925301){\psframe[linecolor=black, linewidth=0.02, dimen=outer](2.1952786,-0.9300214)(2.0407333,-1.5209305)}
\rput{-40.0}(0.7972607,0.9027535){\psframe[linecolor=black, linewidth=0.02, fillstyle=solid,fillcolor=colour0, dimen=outer](1.7160505,-0.3483966)(1.5615051,-0.93930566)}
\rput{40.0}(1.6913136,-1.9919859){\psframe[linecolor=black, linewidth=0.02, dimen=outer](3.6593976,1.6228846)(3.5048523,1.0319755)}
\psframe[linecolor=black, linewidth=0.02, fillstyle=solid,fillcolor=colour0, dimen=outer](2.7618256,1.3486652)(2.60728,0.7577561)
\psline[linecolor=black, linewidth=0.02](1.5974526,1.0671782)(2.6253595,1.0671782)
\psline[linecolor=black, linewidth=0.02](0.0064277425,1.858619)(0.5870088,1.3714536)
\psline[linecolor=black, linewidth=0.02](3.6380067,1.3714536)(4.218588,1.858619)
\psline[linecolor=black, linewidth=0.02](2.1252897,-1.3039755)(2.1252897,-2.0618703)
\psbezier[linecolor=black, linewidth=0.02](3.460534,1.3584573)(3.2473397,1.1895089)(3.0357325,1.0710566)(2.7544825,1.0710566)
\psbezier[linecolor=black, linewidth=0.02](3.585534,1.2209574)(3.1590724,0.77574134)(2.4840305,-1.2040427)(1.7563674,-0.62626487)
\psbezier[linecolor=black, linewidth=0.02](0.7674785,1.3802828)(0.99604994,1.2279018)(1.243669,1.1644098)(1.459542,1.170759)
\psbezier[linecolor=black, linewidth=0.02](0.62779593,1.2152034)(0.8182722,0.986632)(0.8373198,0.8913939)(0.81192297,0.63742566)
\psbezier[linecolor=black, linewidth=0.02](1.0087483,0.58663195)(1.0468435,0.7644098)(1.2373198,0.96123517)(1.459542,0.95488596)
\psbezier[linecolor=black, linewidth=0.02](1.6194476,-0.76377994)(1.7935022,-0.8796855)(2.1191845,-1.0175724)(2.1293833,-1.163368)
\psbezier[linecolor=black, linewidth=0.02](0.88811344,0.47869548)(0.7611293,-0.11178072)(1.3897008,-0.46098706)(1.5865262,-0.5943204)
\rput[bl](0.11985946,1.8882192){$\scriptstyle a$}
\rput[bl](4.0309706,1.88187){$\scriptstyle b$}
\rput[bl](2.2586184,-1.9538586){$\scriptstyle k-1$}
\rput[bl](0.8627166,-0.42289183){$\scriptstyle k$}
\rput[bl](1.0563674,1.3453622){$\scriptstyle m$}
\rput[bl](1.2182722,0.6564733){$\scriptstyle n$}
\rput[bl](0.6119229,0.786632){$\scriptstyle l$}
\rput[bl](1.8976372,1.1263145){$\scriptstyle b-1$}
\rput[bl](2.9857326,-0.11336802){$\scriptstyle 1$}
\end{pspicture}
}
\end{gathered} \notag \\
&= \ \ 
\begin{gathered}
\psscalebox{1.0 1.0} % Change this value to rescale the drawing.
{
\begin{pspicture}(0,-1.6726408)(3.7017455,1.6726408)
\definecolor{colour1}{rgb}{0.39215687,0.78431374,1.0}
\rput[bl](2.025125,-1.5710937){$\scriptstyle k-1$}
\rput{28.139988}(0.84926987,-1.2733543){\psframe[linecolor=black, linewidth=0.02, dimen=outer](3.0422556,1.3530749)(2.88771,0.7621658)}
\psline[linecolor=black, linewidth=0.02](3.0287194,1.0892192)(3.6970294,1.4466631)
\rput[bl](3.4428663,1.4926407){$\scriptstyle b$}
\rput{-26.48331}(-0.40324286,0.4407651){\psframe[linecolor=black, linewidth=0.02, dimen=outer](0.81219393,1.372652)(0.6576485,0.781743)}
\psline[linecolor=black, linewidth=0.02](0.004459554,1.4479251)(0.6828241,1.1099517)
\rput[bl](0.13067278,1.514772){$\scriptstyle a$}
\rput{-90.0}(2.6970787,1.0245858){\psframe[linecolor=black, linewidth=0.02, dimen=outer](1.938105,-0.5407918)(1.7835596,-1.1317009)}
\psline[linecolor=black, linewidth=0.02](1.8681159,-0.9147459)(1.8681159,-1.6726407)
\psbezier[linecolor=black, linewidth=0.02](1.7451308,0.045601282)(1.5910535,0.4338445)(1.158952,0.72742414)(0.7412115,0.93970305)
\psbezier[linecolor=black, linewidth=0.02](1.9707917,0.044561252)(2.0545208,0.35275638)(2.4466321,0.6798117)(2.9471257,0.9561966)
\rput[bl](1.1977117,0.3293289){$\scriptstyle l$}
\rput[bl](1.7210084,0.97066015){$\scriptstyle m$}
\rput[bl](2.0062861,0.4915549){$\scriptstyle n$}
\rput[bl](2.600318,0.11590136){$\scriptstyle 1$}
\rput{-90.0}(1.8905213,1.8377006){\psframe[linecolor=black, linewidth=0.02, fillstyle=solid,fillcolor=colour1, dimen=outer](1.9413837,0.26904425)(1.7868382,-0.32186484)}
\psbezier[linecolor=black, linewidth=0.02](1.7839482,-0.099443085)(1.85334,-0.39849052)(1.8645277,-0.6660767)(1.8673198,-0.7665695)
\psbezier[linecolor=black, linewidth=0.02](1.9624922,-0.10174194)(2.0394258,-0.37543893)(2.3027554,-0.25310382)(2.3992982,0.0147036705)(2.4958408,0.28251117)(2.699168,0.67192096)(2.9859521,0.8646716)
\psbezier[linecolor=black, linewidth=0.02](0.84705365,1.164223)(1.7113394,0.7285087)(2.2150989,0.8270049)(2.8365273,1.1484334)
\end{pspicture}
}
\end{gathered}
\label{eq:expanded_k-1triangle}
\end{align}
where $n = \frac{k+b-1-a}{2}$ by equation \eqref{eq:theta_lblconv1}.

Assume that $a+k > b-1$, then
% REMOVED to better accommodate the q-admissible case
%$a+k \neq b$ (otherwise $a+k+b-1=2b-1$ is not even), hence we have that $a+k \geq b+1$.
%It is now easy to show that this together with admissibility of $(a,b-1,k)$ implies admissibility of $(a,b,k-1)$.
%Furthermore we have that
\[ n = \frac{k+b-1-a}{2} < \frac{k+b-1-(b-1)+k}{2} = k. \]

If $n=0$ then $(-1)^0\frac{[k-0]}{[k]}=1$, and in \eqref{eq:expanded_k-1triangle} we can absorb the central blue idempotent into $p_a$, obtaining
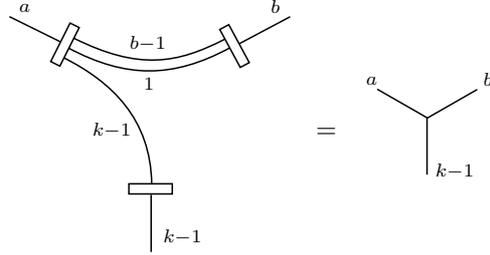
\begin{equation*}
\begin{gathered}
\psscalebox{1.0 1.0} % Change this value to rescale the drawing.
{
\begin{pspicture}(0,-1.6726407)(3.7017453,1.6726407)
\rput[bl](2.0251248,-1.5710938){$\scriptstyle k-1$}
\rput{28.139988}(0.8492698,-1.2733542){\psframe[linecolor=black, linewidth=0.02, dimen=outer](3.0422554,1.3530748)(2.8877099,0.7621657)}
\psline[linecolor=black, linewidth=0.02](3.0287192,1.0892191)(3.697029,1.4466631)
\rput[bl](3.442866,1.4926407){$\scriptstyle b$}
\rput{-26.48331}(-0.40324286,0.440765){\psframe[linecolor=black, linewidth=0.02, dimen=outer](0.81219375,1.3726519)(0.65764827,0.7817429)}
\psline[linecolor=black, linewidth=0.02](0.004459324,1.447925)(0.68282384,1.1099516)
\rput[bl](0.13067256,1.5147719){$\scriptstyle a$}
\rput{-90.0}(2.6970785,1.0245856){\psframe[linecolor=black, linewidth=0.02, dimen=outer](1.9381047,-0.5407919)(1.7835593,-1.131701)}
\psline[linecolor=black, linewidth=0.02](1.8681157,-0.914746)(1.8681157,-1.6726408)
\psbezier[linecolor=black, linewidth=0.02](0.7837706,1.0295037)(1.5379224,0.6278397)(2.152533,0.62642294)(2.900914,1.0323852)
\rput[bl](1.0944377,-0.15846881){$\scriptstyle k-1$}
\rput[bl](1.5829129,0.98970765){$\scriptstyle b-1$}
\rput[bl](1.7717464,0.48256794){$\scriptstyle 1$}
\psbezier[linecolor=black, linewidth=0.02](0.84705347,1.1642228)(1.7113391,0.7285086)(2.2150986,0.82700485)(2.836527,1.1484334)
\psbezier[linecolor=black, linewidth=0.02](0.7210259,0.9013283)(1.3556029,0.57825965)(1.8604662,0.088608734)(1.8754563,-0.7644945)
\end{pspicture}
}
\end{gathered}
\ \ = \ \ 
\begin{gathered}
\psscalebox{1.0 1.0} % Change this value to rescale the drawing.
{
\begin{pspicture}(0,-0.70747477)(1.6522115,0.70747477)
\psline[linecolor=black, linewidth=0.02](0.80185485,-0.64820975)(0.80185485,0.10923705)
\psline[linecolor=black, linewidth=0.02](0.80217934,0.10710939)(0.14621116,0.4858328)
\psline[linecolor=black, linewidth=0.02](0.80578566,0.10285407)(1.4617538,0.4815775)
\rput[bl](1.5422115,0.52747476){$\scriptstyle b$}
\rput[bl](0.0,0.541042){$\scriptstyle a$}
\rput[bl](0.9100797,-0.70747477){$\scriptstyle k-1$}
\end{pspicture}
}
\end{gathered}
\end{equation*}
as required.

In the case that $n>0$, we have the following identity for $0\leq j <n$,
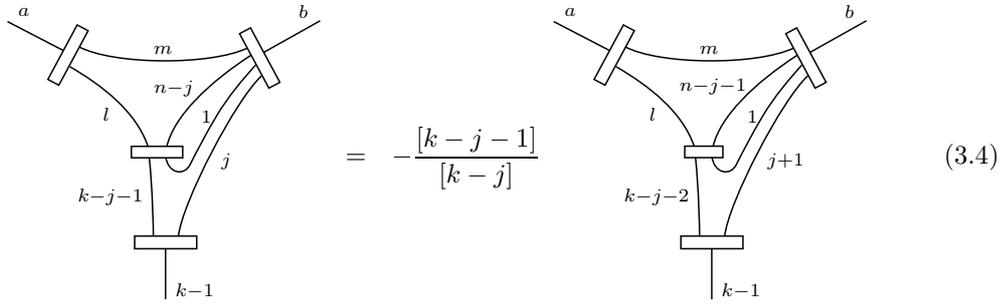
\begin{equation}
\begin{gathered}
\psscalebox{1.0 1.0} % Change this value to rescale the drawing.
{
\begin{pspicture}(0,-1.9574962)(4.1221237,1.9574962)
\rput[bl](2.2114208,-1.9461484){$\scriptstyle k-1$}
\psline[linecolor=black, linewidth=0.02](3.372877,1.3144104)(4.117202,1.7352834)
\rput[bl](3.8341298,1.7774962){$\scriptstyle b$}
\psline[linecolor=black, linewidth=0.02](0.0046319785,1.7248456)(0.7432374,1.3388219)
\rput[bl](0.14520095,1.8035548){$\scriptstyle a$}
\psline[linecolor=black, linewidth=0.02](2.0835373,-1.2873342)(2.0835373,-1.9574963)
\psbezier[linecolor=black, linewidth=0.02](1.8492413,0.07891031)(1.6920731,0.53377557)(1.2513038,0.87773305)(0.8251835,1.1264387)
\psbezier[linecolor=black, linewidth=0.02](2.1006393,0.06839888)(2.193492,0.4719026)(2.6579587,0.9198521)(3.212988,1.2817082)
\rput[bl](1.2621783,0.40775272){$\scriptstyle l$}
\rput[bl](1.9251201,1.293324){$\scriptstyle m$}
\rput[bl](1.9306372,0.72892666){$\scriptstyle n-j$}
\rput[bl](2.5509207,0.3863372){$\scriptstyle 1$}
\psbezier[linecolor=black, linewidth=0.02](1.8695295,-0.07565938)(1.9140967,-0.53022987)(1.9212822,-0.9567305)(1.9230754,-1.1094857)
\psbezier[linecolor=black, linewidth=0.02](0.9430644,1.3908001)(1.4013078,1.1473329)(2.6754093,1.1415683)(3.1588247,1.3722088)
\psbezier[linecolor=black, linewidth=0.02](3.2696881,1.1496068)(3.0983627,0.9569654)(2.907973,0.731985)(2.7574677,0.4777493)(2.6069624,0.22351359)(2.4899094,-0.0042754)(2.4072669,-0.1701818)(2.324624,-0.3360882)(2.1165056,-0.28604463)(2.0865088,-0.07168935)
\psbezier[linecolor=black, linewidth=0.02](3.3457344,1.003912)(3.008873,0.6900055)(2.304885,-0.6624101)(2.2525442,-1.1229863)
\rput{28.139988}(0.9781438,-1.4202011){\psframe[linecolor=black, linewidth=0.02, fillstyle=solid, dimen=outer](3.414081,1.660726)(3.2306788,0.8218763)}
\rput{-90.0}(1.977531,1.9657713){\psframe[linecolor=black, linewidth=0.02, fillstyle=solid, dimen=outer](2.057713,0.3420036)(1.8855892,-0.35376328)}
\rput{-28.1}(-0.5085771,0.5281356){\psframe[linecolor=black, linewidth=0.02, fillstyle=solid, dimen=outer](0.8926069,1.69961)(0.70920473,0.86076033)}
\psframe[linecolor=black, linewidth=0.02, fillstyle=solid, dimen=outer](2.5038838,-1.1095093)(1.6673398,-1.2979016)
\rput[bl](2.8163521,-0.23583712){$\scriptstyle j$}
\rput[bl](0.9283375,-0.7071668){$\scriptstyle k-j-1$}
\end{pspicture}
}
\end{gathered}
\ \ = \ \ -\frac{[k-j-1]}{[k-j]} \ 
\begin{gathered}
\psscalebox{1.0 1.0} % Change this value to rescale the drawing.
{
\begin{pspicture}(0,-1.9574962)(4.1221237,1.9574962)
\rput[bl](2.2114208,-1.9461484){$\scriptstyle k-1$}
\psline[linecolor=black, linewidth=0.02](3.372877,1.3144104)(4.117202,1.7352834)
\rput[bl](3.8341298,1.7774962){$\scriptstyle b$}
\psline[linecolor=black, linewidth=0.02](0.0046319785,1.7248456)(0.7432374,1.3388219)
\rput[bl](0.14520095,1.8035548){$\scriptstyle a$}
\psline[linecolor=black, linewidth=0.02](2.0835373,-1.2873342)(2.0835373,-1.9574963)
\psbezier[linecolor=black, linewidth=0.02](1.8492413,0.07891031)(1.6920731,0.53377557)(1.2513038,0.87773305)(0.8251835,1.1264387)
\psbezier[linecolor=black, linewidth=0.02](2.1006393,0.06839888)(2.193492,0.4719026)(2.6579587,0.9198521)(3.212988,1.2817082)
\rput[bl](1.2621783,0.40775272){$\scriptstyle l$}
\rput[bl](1.9251201,1.293324){$\scriptstyle m$}
\rput[bl](1.6572195,0.7593064){$\scriptstyle n-j-1$}
\rput[bl](2.5509207,0.3863372){$\scriptstyle 1$}
\psbezier[linecolor=black, linewidth=0.02](1.8695295,-0.07565938)(1.9140967,-0.53022987)(1.9212822,-0.9567305)(1.9230754,-1.1094857)
\psbezier[linecolor=black, linewidth=0.02](0.9430644,1.3908001)(1.4013078,1.1473329)(2.6754093,1.1415683)(3.1588247,1.3722088)
\psbezier[linecolor=black, linewidth=0.02](3.2696881,1.1496068)(3.0983627,0.9569654)(2.907973,0.731985)(2.7574677,0.4777493)(2.6069624,0.22351359)(2.4899094,-0.0042754)(2.4072669,-0.1701818)(2.324624,-0.3360882)(2.1165056,-0.28604463)(2.0865088,-0.07168935)
\psbezier[linecolor=black, linewidth=0.02](3.3457344,1.003912)(3.008873,0.6900055)(2.304885,-0.6624101)(2.2525442,-1.1229863)
\rput{28.139988}(0.9781438,-1.4202011){\psframe[linecolor=black, linewidth=0.02, fillstyle=solid, dimen=outer](3.414081,1.660726)(3.2306788,0.8218763)}
\rput{-28.1}(-0.5085771,0.5281356){\psframe[linecolor=black, linewidth=0.02, fillstyle=solid, dimen=outer](0.8926069,1.69961)(0.70920473,0.86076033)}
\psframe[linecolor=black, linewidth=0.02, fillstyle=solid, dimen=outer](2.5038838,-1.1095093)(1.6673398,-1.2979016)
\rput[bl](2.8163521,-0.23583712){$\scriptstyle j+1$}
\rput[bl](0.9283375,-0.7071668){$\scriptstyle k-j-2$}
\psframe[linecolor=black, linewidth=0.02, dimen=outer](2.2399642,0.08231499)(1.7184453,-0.08983691)
\end{pspicture}
}
\end{gathered}
\label{eq:triangle_recursive_relation}
\end{equation}
by expanding the central idempotent according to the inductive relation \eqref{eq:JWPrelation} and applying the absorption rule.

Applying \eqref{eq:triangle_recursive_relation} recursively to \eqref{eq:expanded_k-1triangle} we get that
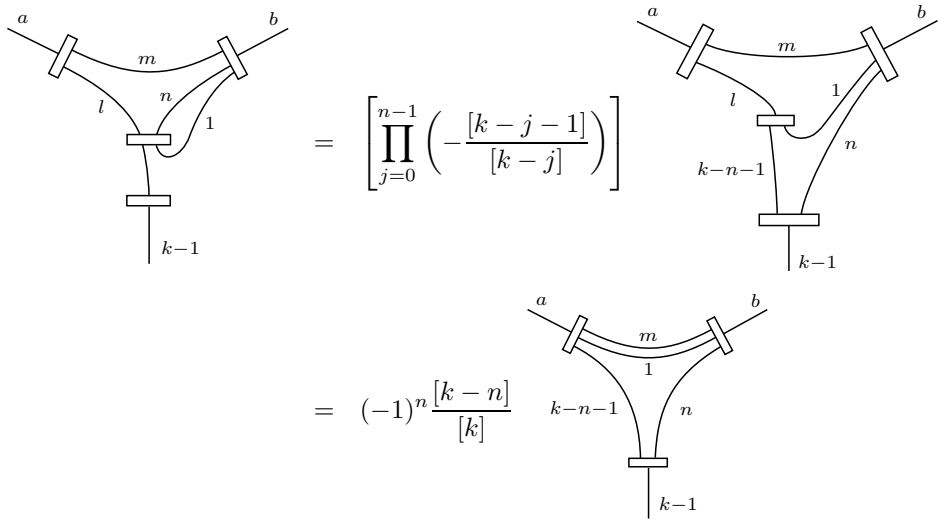
\begin{align*}
\begin{gathered}
\psscalebox{1.0 1.0} % Change this value to rescale the drawing.
{
\begin{pspicture}(0,-1.6726408)(3.7017455,1.6726408)
\rput[bl](2.025125,-1.5710937){$\scriptstyle k-1$}
\rput{28.139988}(0.84926987,-1.2733543){\psframe[linecolor=black, linewidth=0.02, dimen=outer](3.0422556,1.3530749)(2.88771,0.7621658)}
\psline[linecolor=black, linewidth=0.02](3.0287194,1.0892192)(3.6970294,1.4466631)
\rput[bl](3.4428663,1.4926407){$\scriptstyle b$}
\rput{-26.48331}(-0.40324286,0.4407651){\psframe[linecolor=black, linewidth=0.02, dimen=outer](0.81219393,1.372652)(0.6576485,0.781743)}
\psline[linecolor=black, linewidth=0.02](0.004459554,1.4479251)(0.6828241,1.1099517)
\rput[bl](0.13067278,1.514772){$\scriptstyle a$}
\rput{-90.0}(2.6970787,1.0245858){\psframe[linecolor=black, linewidth=0.02, dimen=outer](1.938105,-0.5407918)(1.7835596,-1.1317009)}
\psline[linecolor=black, linewidth=0.02](1.8681159,-0.9147459)(1.8681159,-1.6726407)
\psbezier[linecolor=black, linewidth=0.02](1.7451308,0.045601282)(1.5910535,0.4338445)(1.158952,0.72742414)(0.7412115,0.93970305)
\psbezier[linecolor=black, linewidth=0.02](1.9707917,0.044561252)(2.0545208,0.35275638)(2.4466321,0.6798117)(2.9471257,0.9561966)
\rput[bl](1.1977117,0.3293289){$\scriptstyle l$}
\rput[bl](1.7210084,0.97066015){$\scriptstyle m$}
\rput[bl](2.0062861,0.4915549){$\scriptstyle n$}
\rput[bl](2.600318,0.11590136){$\scriptstyle 1$}
\rput{-90.0}(1.8905213,1.8377006){\psframe[linecolor=black, linewidth=0.02, fillstyle=solid, dimen=outer](1.9413837,0.26904425)(1.7868382,-0.32186484)}
\psbezier[linecolor=black, linewidth=0.02](1.7839482,-0.099443085)(1.85334,-0.39849052)(1.8645277,-0.6660767)(1.8673198,-0.7665695)
\psbezier[linecolor=black, linewidth=0.02](1.9624922,-0.10174194)(2.0394258,-0.37543893)(2.3027554,-0.25310382)(2.3992982,0.0147036705)(2.4958408,0.28251117)(2.699168,0.67192096)(2.9859521,0.8646716)
\psbezier[linecolor=black, linewidth=0.02](0.84705365,1.164223)(1.7113394,0.7285087)(2.2150989,0.8270049)(2.8365273,1.1484334)
\end{pspicture}
}
\end{gathered}
\ \ &= \ \ 
\left[\prod_{j=0}^{n-1} \left( -\frac{[k-j-1]}{[k-j]} \right)\right]
\begin{gathered}
\psscalebox{1.0 1.0} % Change this value to rescale the drawing.
{
\begin{pspicture}(0,-1.7688127)(3.961637,1.7688127)
\rput[bl](2.1253252,-1.7586113){$\scriptstyle k-1$}
\psline[linecolor=black, linewidth=0.02](3.2415948,1.172515)(3.9569616,1.5508648)
\rput[bl](3.6849027,1.5888126){$\scriptstyle b$}
\psline[linecolor=black, linewidth=0.02](0.004391805,1.5414816)(0.7142617,1.19446)
\rput[bl](0.13949192,1.6122383){$\scriptstyle a$}
\psline[linecolor=black, linewidth=0.02](2.002417,-1.1663609)(2.002417,-1.7688128)
\psbezier[linecolor=black, linewidth=0.02](1.8255233,0.2960094)(1.8162656,0.5004071)(1.2368451,0.81937075)(0.7759335,0.9995398)
\rput[bl](1.226128,0.42303577){$\scriptstyle l$}
\rput[bl](1.8501631,1.1535591){$\scriptstyle m$}
\rput[bl](2.554134,0.55795705){$\scriptstyle 1$}
\psbezier[linecolor=black, linewidth=0.02](1.7454768,0.17894456)(1.818211,-0.3445286)(1.8464745,-0.8574244)(1.8481979,-1.0064814)
\psbezier[linecolor=black, linewidth=0.02](0.9063144,1.2411866)(1.3467298,1.0223182)(2.5712621,1.0171361)(3.0358703,1.2244737)
\psbezier[linecolor=black, linewidth=0.02](3.1424205,1.0243623)(2.9777606,0.8511846)(2.8631225,0.7008754)(2.6501281,0.4203864)(2.4371338,0.13989739)(2.4444234,0.13228737)(2.3436382,0.051065173)(2.2428532,-0.03015702)(1.9956591,-0.054481093)(1.9411998,0.16218907)
\psbezier[linecolor=black, linewidth=0.02](3.2155082,0.8933879)(2.8917525,0.6111972)(2.2151532,-0.6045763)(2.1648488,-1.018618)
\rput{28.139988}(0.8994202,-1.3751106){\psframe[linecolor=black, linewidth=0.02, fillstyle=solid, dimen=outer](3.2811959,1.4838402)(3.104929,0.7297443)}
\rput{-28.1}(-0.44705138,0.49711287){\psframe[linecolor=black, linewidth=0.02, fillstyle=solid, dimen=outer](0.85782,1.5187956)(0.68155307,0.7646997)}
\psframe[linecolor=black, linewidth=0.02, fillstyle=solid, dimen=outer](2.40641,-1.0065026)(1.6024119,-1.1758605)
\rput[bl](2.7323508,-0.15717602){$\scriptstyle n$}
\rput[bl](0.7940999,-0.52351564){$\scriptstyle k-n-1$}
\psframe[linecolor=black, linewidth=0.02, fillstyle=solid, dimen=outer](2.0801423,0.30428144)(1.578913,0.14952302)
\end{pspicture}
}
\end{gathered} \\
&= \ \ 
(-1)^n\frac{[k-n]}{[k]}\ 
\begin{gathered}
\psscalebox{1.0 1.0} % Change this value to rescale the drawing.
{
\begin{pspicture}(0,-1.4955456)(3.1620328,1.4955456)
\rput[bl](1.7297653,-1.4053617){$\scriptstyle k-1$}
\rput{28.139988}(0.73754656,-1.0844111){\psframe[linecolor=black, linewidth=0.02, dimen=outer](2.5981512,1.191597)(2.4662063,0.6668097)}
\psline[linecolor=black, linewidth=0.02](2.5865943,0.95726633)(3.157171,1.2747128)
\rput[bl](2.9401767,1.3155457){$\scriptstyle b$}
\rput{-26.48331}(-0.35619467,0.3794868){\psframe[linecolor=black, linewidth=0.02, dimen=outer](0.69421285,1.2089834)(0.5622681,0.68419623)}
\psline[linecolor=black, linewidth=0.02](0.0046012304,1.2758336)(0.583762,0.97567886)
\rput[bl](0.11235709,1.3352004){$\scriptstyle a$}
\rput{-90.0}(2.3422413,0.8367561){\psframe[linecolor=black, linewidth=0.02, dimen=outer](1.6554711,-0.49034894)(1.5235263,-1.0151362)}
\psline[linecolor=black, linewidth=0.02](1.5957172,-0.82245815)(1.5957172,-1.4955456)
\psbezier[linecolor=black, linewidth=0.02](0.6699463,0.904233)(1.3723478,0.54751456)(1.8385415,0.54625636)(2.4774795,0.9067921)
\rput[bl](0.29615715,-0.13617307){$\scriptstyle k-n-1$}
\rput[bl](1.4692956,0.864012){$\scriptstyle m$}
\rput[bl](1.5280751,0.41362044){$\scriptstyle 1$}
\psbezier[linecolor=black, linewidth=0.02](0.7239747,1.0238773)(1.4618676,0.6369188)(1.8919574,0.7243934)(2.4225085,1.0098547)
\psbezier[linecolor=black, linewidth=0.02](0.6163774,0.79040027)(1.158154,0.50348246)(1.4769913,0.058866635)(1.4897891,-0.6987756)
\psbezier[linecolor=black, linewidth=0.02](2.542409,0.7826398)(1.8692383,0.37776175)(1.7131406,-0.07589677)(1.6838723,-0.68565285)
\rput[bl](2.0061238,-0.118086874){$\scriptstyle n$}
\end{pspicture}
}
\end{gathered}
\end{align*}
where the product telescopes to give the required coefficient.
\end{proof}

\begin{remark}
It should be noted that by Remark \ref{rem:JWP_lat_refl_dual} all our diagrammatic identities hold under planar Euclidean transformations (specifically, reflections and rotations); this follows simply by applying the transformation to all diagrams in the proofs.
\end{remark}

Let us use the notation
\begin{equation*}
\sumtwo{k=r}{s}f(k) = f(r) + f(r+2) + \dotsb + f(s-2) + f(s)
\end{equation*}
to indicate that a summation is to be taken in steps of two over the range of the indexing variable in the case that $s-r$ is even.
Now we come to the main theorem of this chapter.

% --------------- THM: Tensor product sum of JWPs ---------------
\begin{theorem}[Tensor product identity]
\label{thm:JWP-tensorprod-identity}
For $a,b \geq 0$,
\begin{equation}
\label{eq:JWP-tensorprod-identity}
p_a \otimes p_b = 
\sumtwo{k=\abs{a-b}}{a+b} \left( \frac{[k+1]}{\theta(a,b,k)}
\begin{gathered}
\psscalebox{1.0 1.0} % Change this value to rescale the drawing.
{
\begin{pspicture}(0,-0.792027)(0.94241494,0.792027)
\psline[linecolor=black, linewidth=0.02](0.45697075,0.20267501)(0.45697075,-0.21881929)
\psline[linecolor=black, linewidth=0.02](0.4562587,0.20204714)(0.8494084,0.5487268)
\psline[linecolor=black, linewidth=0.02](0.067692176,0.5459411)(0.46084186,0.19926146)
\psline[linecolor=black, linewidth=0.02](0.06372307,-0.56065685)(0.45687276,-0.21397717)
\psline[linecolor=black, linewidth=0.02](0.45066956,-0.20976295)(0.84381926,-0.5564426)
\rput[bl](0.010811806,0.6225269){$\scriptstyle a$}
\rput[bl](0.52679443,-0.08153543){$\scriptstyle k$}
\rput[bl](0.832415,0.61202705){$\scriptstyle b$}
\rput[bl](0.0,-0.78152716){$\scriptstyle a$}
\rput[bl](0.8216032,-0.792027){$\scriptstyle b$}
\end{pspicture}
}
\end{gathered}
\ \right)
\end{equation}
\end{theorem}

\begin{proof}
First observe that $a+b-\abs{a-b} = 2\min(a,b)$ is even, and that for every $k$ in the range of the two-step sum, $k = \abs{a-b}+2i$ for some $0 \leq i \leq \min(a,b)$.
Then for all such $k$, the sum $a+b+k$ is even, and
\begin{align*}
a+b &= \abs{a-b}+2\min(a,b) \geq k, \\
b+k &\geq b+\abs{a-b} \geq \max(a,b) \geq a, \shortintertext{and}
a+k &\geq a+\abs{a-b} \geq \max(a,b) \geq b,
\end{align*}
so that $(a,b,k)$ is admissible over the range of the sum.

If $a=0$ then
\begin{align*}
\sumtwo{k=\abs{a-b}}{a+b} \left( \frac{[k+1]}{\theta(a,b,k)}
\begin{gathered}
\psscalebox{1.0 1.0} % Change this value to rescale the drawing.
{

}
\end{gathered} \\
&= \ \ p_b \\
&= \ \ p_a \otimes p_b
\end{align*}
where we simplify the coefficient using Theorem \ref{thm:theta_net_trace} and Corollary \ref{cor:theta_trace_invariance}.
The same proof applies \textit{mutatis mutandis} in the case $b=0$.

Next let $a\geq 1$; we first prove this identity for $b=1$.
Consider
\begin{equation*}
\sumtwo{k=a-1}{a+1}\frac{[k+1]}{\theta(a,1,k)}
\begin{gathered}
\psscalebox{1.0 1.0} % Change this value to rescale the drawing.
{
%
}
\end{gathered}\quad.
\end{equation*}
By Theorem \ref{thm:theta_net_trace} we have that
\begin{equation}
\theta(a,1,a-1) = \Net(1,0,a-1) = [a+1]
\label{eq:thetatrace_a-1}
\end{equation}
and
\begin{equation}
\theta(a,1,a+1) = \Net(0,1,a) = [a+2],
\label{eq:thetatrace_a+1}
\end{equation}
thus
\begin{align*}
\sumtwo{k=a-1}{a+1}\frac{[k+1]}{\theta(a,1,k)}
\begin{gathered}
\psscalebox{1.0 1.0} % Change this value to rescale the drawing.
{
%
}
\end{gathered} \\
&= p_a \otimes p_1.
\end{align*}

Suppose now that $2 \leq b\leq a$, and that the result holds for $p_a\otimes p_c$ for all $1\leq c \leq b-1$.
Then
\begin{align*}
p_a \otimes p_b
\ \ &= \ \ 
\begin{gathered}
\psscalebox{1.0 1.0} % Change this value to rescale the drawing.
{
%
}
\end{gathered} \ 
\right)
\end{align*}
by the induction hypothesis, where we write $\lambda(i,j,k)=\frac{[k+1]}{\theta(i,j,k)}$.

Now
\begin{align*}
\begin{gathered}
\psscalebox{1.0 1.0} % Change this value to rescale the drawing.
{
%
}
\end{gathered} \ 
\right)
\label{eq:JWP-tensorprod-identity_sumtwo_k}
\end{align}
by the triangle-shrinking lemmas \ref{lemma:shrink_triangle_k+1} and \ref{lemma:shrink_triangle_k-1}, and where $n_k = \frac{k+b-1-a}{2}$.

Reindexing the sum by $k^\prime = k-1$ and simplifying, we get that
\begin{align*}
p_a \otimes p_b
\ \ &= \ \ 
\sumtwo{k^\prime = a-b}{a+b-2} \lambda(a,b-1,k^\prime+1) \left(
\frac{[k^\prime+1-n_{k^\prime+1}]^2}{[k^\prime+1][k^\prime+2]}
\begin{gathered}
\psscalebox{1.0 1.0} % Change this value to rescale the drawing.
{
%
}
\end{gathered} \ 
\right) \\
\intertext{which upon expansion is}
&= \ \ 
\begin{multlined}[t]
\lambda(a,b-1,a-b+1) \frac{[a-b+1-n_{a-b+1}]^2}{[a-b+1][a-b+2]}
\begin{gathered}
\psscalebox{1.0 1.0} % Change this value to rescale the drawing.
{
%
}
\end{gathered} \ .
\end{multlined}
\end{align*}

It is now a tedious algebraic calculation to verify that the coefficients in this expression agree with those in \eqref{eq:JWP-tensorprod-identity}.
The coefficients for the first and last terms in the above expansion are straightforward to compute using Theorem \ref{thm:theta_net_trace}, these we leave to the reader.
We show here that for $k^\prime \in \{a-b+2, a-b+4, \dotsc, a+b-2\}$,
\begin{equation}
\lambda(a,b-1,k^\prime-1) + \lambda(a,b-1,k^\prime+1)\frac{[k^\prime+1-n_{k^\prime+1}]^2}{[k^\prime+1][k^\prime+2]} = \lambda(a,b,k^\prime).
\label{eq:JWP_tensorprod_coeffs}
\end{equation}

Write $k^\prime = a-b+2+2i$ for some $0 \leq i \leq b-2$.
Then one calculates that
\begin{align*}
\lambda(a,b,k^\prime) &= \frac{[a-b+3+2i]}{\theta(a,b,a-b+2+2i)} \\
                     &= \frac{[a]![b]![a-b+3+2i]!}{[b-1-i]![i+1]![a-b+1+i]![a+2+i]!} \ , \\
\intertext{and similarly}
\lambda(a,b-1,k^\prime-1) &= \frac{[a]![b-1]![a-b+2+2i]!}{[b-1-i]![i]![a-b+1+i]![a+1+i]!} \ , \\
\lambda(a,b-1,k^\prime+1) &= \frac{[a]![b-1]![a-b+4+2i]!}{[b-2-i]![i+1]![a-b+2+i]![a+2+i]!} \ ,
\end{align*}
so that the left-hand side of \eqref{eq:JWP_tensorprod_coeffs} is equal to
\begin{equation}
\lambda(a,b,k^\prime) \left( \frac{[i+1][a+2+i]+[b-1-i][a-b+2+i]}{[b][a-b+3+2i]} \right).
\label{eq:JWP_tensorprod_coeffs_LHS}
\end{equation}
But since
\begin{align*}
  & [i+1][a+2+i]+[b-1-i][a-b+2+i] \\
%\begin{aligned}
= \ & \frac{(q^{i+1}-q^{-i-1})(q^{a+2+i}-q^{-a-2-i})+(q^{b-1-i}-q^{-b+1+i})(q^{a-b+2+i}-q^{-a+b-2-i})}{(q-q^{-1})^2} \\
= \ & \frac{q^{a+2i+3}-q^{a-2b+2i+3}-q^{-a+2b-2i-3}+q^{-a-2i-3}}{(q-q^{-1})^2} \\
= \ & \frac{(q^b-q^{-b})(q^{a-b+2i+3}-q^{-a+b-2i-3})}{(q-q^{-1})^2} \\
= \ & [b][a-b+2i+3] \ ,
%\end{aligned}
\end{align*}
we have that \eqref{eq:JWP_tensorprod_coeffs_LHS} is equal to $\lambda(a,b,k^\prime)$, as required.
Hence by induction the identity holds for all $b \leq a$.

Finally, if $b > a$ observe that
\begin{equation*}
p_b \otimes p_a = 
\sumtwo{k=b-a}{a+b} \left( \frac{[k+1]}{\theta(b,a,k)}
\begin{gathered}
\psscalebox{1.0 1.0} % Change this value to rescale the drawing.
{
\begin{pspicture}(0,-0.792027)(0.94241494,0.792027)
\psline[linecolor=black, linewidth=0.02](0.45697075,0.20267501)(0.45697075,-0.21881929)
\psline[linecolor=black, linewidth=0.02](0.4562587,0.20204714)(0.8494084,0.5487268)
\psline[linecolor=black, linewidth=0.02](0.067692176,0.5459411)(0.46084186,0.19926146)
\psline[linecolor=black, linewidth=0.02](0.06372307,-0.56065685)(0.45687276,-0.21397717)
\psline[linecolor=black, linewidth=0.02](0.45066956,-0.20976295)(0.84381926,-0.5564426)
\rput[bl](0.010811806,0.6225269){$\scriptstyle b$}
\rput[bl](0.52679443,-0.08153543){$\scriptstyle k$}
\rput[bl](0.832415,0.61202705){$\scriptstyle a$}
\rput[bl](0.0,-0.78152716){$\scriptstyle b$}
\rput[bl](0.8216032,-0.792027){$\scriptstyle a$}
\end{pspicture}
}
\end{gathered}
\ \right) ,
\end{equation*}
then reflecting all diagrams about the vertical axis gives
\begin{equation*}
p_a \otimes p_b = 
\sumtwo{k=b-a}{a+b} \left( \frac{[k+1]}{\theta(b,a,k)}
\begin{gathered}
\psscalebox{1.0 1.0} % Change this value to rescale the drawing.
{
\begin{pspicture}(0,-0.792027)(0.94241494,0.792027)
\psline[linecolor=black, linewidth=0.02](0.45697075,0.20267501)(0.45697075,-0.21881929)
\psline[linecolor=black, linewidth=0.02](0.4562587,0.20204714)(0.8494084,0.5487268)
\psline[linecolor=black, linewidth=0.02](0.067692176,0.5459411)(0.46084186,0.19926146)
\psline[linecolor=black, linewidth=0.02](0.06372307,-0.56065685)(0.45687276,-0.21397717)
\psline[linecolor=black, linewidth=0.02](0.45066956,-0.20976295)(0.84381926,-0.5564426)
\rput[bl](0.010811806,0.6225269){$\scriptstyle a$}
\rput[bl](0.52679443,-0.08153543){$\scriptstyle k$}
\rput[bl](0.832415,0.61202705){$\scriptstyle b$}
\rput[bl](0.0,-0.78152716){$\scriptstyle a$}
\rput[bl](0.8216032,-0.792027){$\scriptstyle b$}
\end{pspicture}
}
\end{gathered}
\ \right)
\ \ = \ \ 
\sumtwo{k=b-a}{a+b} \left( \frac{[k+1]}{\theta(a,b,k)}
\begin{gathered}
\psscalebox{1.0 1.0} % Change this value to rescale the drawing.
{
\begin{pspicture}(0,-0.792027)(0.94241494,0.792027)
\psline[linecolor=black, linewidth=0.02](0.45697075,0.20267501)(0.45697075,-0.21881929)
\psline[linecolor=black, linewidth=0.02](0.4562587,0.20204714)(0.8494084,0.5487268)
\psline[linecolor=black, linewidth=0.02](0.067692176,0.5459411)(0.46084186,0.19926146)
\psline[linecolor=black, linewidth=0.02](0.06372307,-0.56065685)(0.45687276,-0.21397717)
\psline[linecolor=black, linewidth=0.02](0.45066956,-0.20976295)(0.84381926,-0.5564426)
\rput[bl](0.010811806,0.6225269){$\scriptstyle a$}
\rput[bl](0.52679443,-0.08153543){$\scriptstyle k$}
\rput[bl](0.832415,0.61202705){$\scriptstyle b$}
\rput[bl](0.0,-0.78152716){$\scriptstyle a$}
\rput[bl](0.8216032,-0.792027){$\scriptstyle b$}
\end{pspicture}
}
\end{gathered}
\ \right)
\end{equation*}
by Corollary \ref{cor:theta_trace_invariance}.
\end{proof}

\chapter{Semisimplicity of generic Temperley-Lieb-Jones}
\begin{defn}
A $\mathbbm{F}$-linear category $\mathcal{C}$ with biproduct $\oplus$ is \textbf{semisimple} if there is a collection of disjoint simple objects $S \subset \Obj(\mathcal{C})$ such that every object $x$ of $\mathcal{C}$ is isomorphic to a finite biproduct $\bigoplus_{i=1}^n x_i$ of objects $x_i \in S$.
\end{defn}

We remark that if $\mathcal{C}$ is also abelian, this definition is equivalent to the usual notion of semisimplicity in abelian categories.

Recall (cf. Corollary \ref{cor:JWP-disjointsimpleobjs}) that in $\TLJ$ the Jones-Wenzl idempotents form a collection of disjoint simple objects and the direct sum $\oplus$ is a biproduct.

We claim that generic Temperley-Lieb-Jones is semisimple.
In fact most of the work to show this has already been done in the previous chapter; all we still need is the following important lemma, from which semisimplicity will follow as a consequence.

\begin{lemma}
\label{lemma:JWP-tensor-isomorphism}
Let $p_a, p_b \in \Obj(\TLJ)$ be Jones-Wenzl idempotents.
Then
\begin{equation*}
p_a \otimes p_b \cong p_{\abs{a-b}} \oplus p_{\abs{a-b}+2} \oplus \dotsb \oplus p_{a+b},
\end{equation*}
where the direct sum runs from $p_{\abs{a-b}}$ to $p_{a+b}$ in steps of two.
\end{lemma}

\begin{proof}
Write $P = p_{\abs{a-b}} \oplus p_{\abs{a-b}+2} \oplus \dotsb \oplus p_{a+b}$.
Note that we showed in the proof of Theorem \ref{thm:JWP-tensorprod-identity} that $(a,b,k)$ is admissible for $k \in \{\abs{a-b}, \abs{a-b}+2, \dotsc, a+b\}$; hence let $\varphi \colon p_a \otimes p_b \rightarrow P$ be the morphism with single column indexed by $p_a \otimes p_b$ and rows indexed by $p_{\abs{a-b}}, p_{\abs{a-b}+2}, \dotsc, p_{a+b}$, whose entry in the row indexed by $p_k$ is given by
\begin{equation*}
g_{a,b,k} \ \ = \ \ 
\begin{gathered}
\psscalebox{1.0 1.0} % Change this value to rescale the drawing.
{
\begin{pspicture}(0,-0.5435213)(1.1242857,0.5435213)
\psline[linecolor=black, linewidth=0.02](0.5709957,0.29500684)(0.5709957,-0.15044771)
\psline[linecolor=black, linewidth=0.02](0.56473863,-0.14758837)(0.9296337,-0.4030906)
\psline[linecolor=black, linewidth=0.02](0.5772527,-0.14282647)(0.2123577,-0.3983287)
\rput[bl](0.47593985,0.3535213){$\scriptstyle k$}
\rput[bl](0.0,-0.5339975){$\scriptstyle a$}
\rput[bl](1.0142857,-0.5435213){$\scriptstyle b$}
\end{pspicture}
}
\end{gathered}\ .
\end{equation*}
That is,
\begin{equation*}
\varphi = \begin{bmatrix}
\begin{gathered}
\psscalebox{1.0 1.0} % Change this value to rescale the drawing.
{
\begin{pspicture}(0,-0.5435213)(1.1242857,0.5435213)
\psline[linecolor=black, linewidth=0.02](0.5709957,0.29500684)(0.5709957,-0.15044771)
\psline[linecolor=black, linewidth=0.02](0.56473863,-0.14758837)(0.9296337,-0.4030906)
\psline[linecolor=black, linewidth=0.02](0.5772527,-0.14282647)(0.2123577,-0.3983287)
\rput[bl](0.27593985,0.2935213){$\scriptstyle \abs{a-b}$}
\rput[bl](0.0,-0.5339975){$\scriptstyle a$}
\rput[bl](1.0142857,-0.5435213){$\scriptstyle b$}
\end{pspicture}
}
\end{gathered} \rule{0pt}{5ex} \ \\
\begin{gathered}
\psscalebox{1.0 1.0} % Change this value to rescale the drawing.
{
\begin{pspicture}(0,-0.5575564)(1.1242857,0.5575564)
\psline[linecolor=black, linewidth=0.02](0.5709957,0.28097174)(0.5709957,-0.1644828)
\psline[linecolor=black, linewidth=0.02](0.56473863,-0.16162346)(0.9296337,-0.4171257)
\psline[linecolor=black, linewidth=0.02](0.5772527,-0.15686156)(0.2123577,-0.4123638)
\rput[bl](0.09348371,0.3075564){$\scriptstyle|a-b|+2$}
\rput[bl](0.0,-0.5480326){$\scriptstyle a$}
\rput[bl](1.0142857,-0.5575564){$\scriptstyle b$}
\end{pspicture}
}
\end{gathered} \rule[-5ex]{0pt}{5ex} \rule{0pt}{6.5ex} \ \\
\vdots \\
\begin{gathered}
\psscalebox{1.0 1.0} % Change this value to rescale the drawing.
{
\begin{pspicture}(0,-0.5676441)(1.1242857,0.5676441)
\psline[linecolor=black, linewidth=0.02](0.5709957,0.27088404)(0.5709957,-0.17457052)
\psline[linecolor=black, linewidth=0.02](0.56473863,-0.17171118)(0.9296337,-0.4272134)
\psline[linecolor=black, linewidth=0.02](0.5772527,-0.16694927)(0.2123577,-0.4224515)
\rput[bl](0.33208022,0.3676441){$\scriptstyle a+b$}
\rput[bl](0.0,-0.5581203){$\scriptstyle a$}
\rput[bl](1.0142857,-0.5676441){$\scriptstyle b$}
\end{pspicture}
}
\end{gathered} \ \rule{0pt}{6.5ex}
\end{bmatrix}
\end{equation*}

Also let $\psi \colon P \rightarrow p_a \otimes p_b$ be the following morphism row-indexed by $p_a \otimes p_b$ and column-indexed by $p_{\abs{a-b}}, p_{\abs{a-b}+2}, \dotsb, p_{a+b}$:
\begin{equation*}
\psi = \begin{bmatrix}
\displaystyle \frac{\left[\abs{a-b}+1\right]}{\theta(a,b,\abs{a-b})}
\begin{gathered}
\psscalebox{1.0 1.0} % Change this value to rescale the drawing.
{
\begin{pspicture}(0,-0.4926521)(1.2255411,0.4926521)
\psline[linecolor=black, linewidth=0.02](0.5276145,-0.43496695)(0.5276145,0.010487583)
\psline[linecolor=black, linewidth=0.02](0.5213575,0.007628242)(0.88625246,0.2631305)
\psline[linecolor=black, linewidth=0.02](0.53387153,0.0028663373)(0.16897652,0.25836858)
\rput[bl](0.62554115,-0.4926521){$\scriptstyle \abs{a-b}$}
\rput[bl](0.0,0.32832763){$\scriptstyle a$}
\rput[bl](0.9618136,0.31265208){$\scriptstyle b$}
\end{pspicture}
}
\end{gathered} \quad ,
& \cdots \ \ ,
& \displaystyle \frac{\left[a+b+1\right]}{\theta(a,b,a+b)}
\begin{gathered}
\psscalebox{1.0 1.0} % Change this value to rescale the drawing.
{
\begin{pspicture}(0,-0.4926521)(1.2255411,0.4926521)
\psline[linecolor=black, linewidth=0.02](0.5276145,-0.43496695)(0.5276145,0.010487583)
\psline[linecolor=black, linewidth=0.02](0.5213575,0.007628242)(0.88625246,0.2631305)
\psline[linecolor=black, linewidth=0.02](0.53387153,0.0028663373)(0.16897652,0.25836858)
\rput[bl](0.62554115,-0.4926521){$\scriptstyle a+b$}
\rput[bl](0.0,0.32832763){$\scriptstyle a$}
\rput[bl](0.9618136,0.31265208){$\scriptstyle b$}
\end{pspicture}
}
\end{gathered}
\end{bmatrix},
\end{equation*}
where in similar fashion to $\varphi$ the entry in the column indexed by $p_k$ is given by
\begin{equation*}
\frac{\left[k+1\right]}{\theta(a,b,k)}
\begin{gathered}
\psscalebox{1.0 1.0} % Change this value to rescale the drawing.
{
\begin{pspicture}(0,-0.4926521)(1.2255411,0.4926521)
\psline[linecolor=black, linewidth=0.02](0.5276145,-0.43496695)(0.5276145,0.010487583)
\psline[linecolor=black, linewidth=0.02](0.5213575,0.007628242)(0.88625246,0.2631305)
\psline[linecolor=black, linewidth=0.02](0.53387153,0.0028663373)(0.16897652,0.25836858)
\rput[bl](0.62554115,-0.4926521){$\scriptstyle k$}
\rput[bl](0.0,0.32832763){$\scriptstyle a$}
\rput[bl](0.9618136,0.31265208){$\scriptstyle b$}
\end{pspicture}
}
\end{gathered} \ .
\end{equation*}
We claim that $\varphi$ and $\psi$ are inverse isomorphisms.

To see this first consider the composition $\varphi\psi \colon P \rightarrow P$.
Let us write $(\varphi\psi)_{p_i,p_j}$ to mean the entry of $\varphi\psi$ row-indexed by $p_i$ and column-indexed by $p_j$.
Then by Lemma \ref{lemma:roundabout_lemma}, observe that $\varphi\psi$ has off-diagonal entries
\begin{equation*}
(\varphi\psi)_{p_i,p_j} \ \ = \ \ \frac{[j+1]}{\theta(a,b,j)}
\begin{gathered}
\psscalebox{1.0 1.0} % Change this value to rescale the drawing.
{
\begin{pspicture}(0,-0.97975385)(1.1033772,0.97975385)
\psellipse[linecolor=black, linewidth=0.02, dimen=outer](0.5526531,-0.019766627)(0.32,0.4494703)
\psline[linecolor=black, linewidth=0.02](0.5508031,0.4216411)(0.5508031,0.9797538)
\psline[linecolor=black, linewidth=0.02](0.55251956,-0.979754)(0.55251956,-0.4624576)
\rput[bl](0.63673466,0.6766606){$\scriptstyle i$}
\rput[bl](0.64081633,-0.8457884){$\scriptstyle j$}
\rput[bl](0.0,-0.06619656){$\scriptstyle a$}
\rput[bl](0.9733772,-0.07241117){$\scriptstyle b$}
\end{pspicture}
}
\end{gathered} \ \ = \ \ 0 \quad \text{ for } p_i \neq p_j
\end{equation*}
and diagonal entries
\begin{align*}
(\varphi\psi)_{p_i,p_i} \ \ &= \ \ 
\frac{[i+1]}{\theta(a,b,i)}
\begin{gathered}
\psscalebox{1.0 1.0} % Change this value to rescale the drawing.
{
\begin{pspicture}(0,-0.97975385)(1.1033772,0.97975385)
\psellipse[linecolor=black, linewidth=0.02, dimen=outer](0.5526531,-0.019766627)(0.32,0.4494703)
\psline[linecolor=black, linewidth=0.02](0.5508031,0.4216411)(0.5508031,0.9797538)
\psline[linecolor=black, linewidth=0.02](0.55251956,-0.979754)(0.55251956,-0.4624576)
\rput[bl](0.63673466,0.6766606){$\scriptstyle i$}
\rput[bl](0.64081633,-0.8457884){$\scriptstyle i$}
\rput[bl](0.0,-0.06619656){$\scriptstyle a$}
\rput[bl](0.9733772,-0.07241117){$\scriptstyle b$}
\end{pspicture}
}
\end{gathered} \\
&= \ \ 
\frac{[i+1]}{\theta(a,b,i)} \cdot \frac{\theta(a,b,i)}{[i+1]} p_i \\
&= \ \ p_i,
\end{align*}
thus
$\varphi\psi = [p_{\abs{a-b}}] \oplus [p_{\abs{a-b}+2}] \oplus \dotsb \oplus [p_{a+b}]$
is the identity morphism on $P$.

Consider now $\psi\varphi \colon p_a \otimes p_b \rightarrow p_a \otimes p_b$, multiplying out the matrices we have that
\begin{align*}
\psi\varphi &= \sumtwo{k=\abs{a-b}}{a+b}
\frac{[k+1]}{\theta(a,b,k)}
\begin{gathered}
\psscalebox{1.0 1.0} % Change this value to rescale the drawing.
{
\begin{pspicture}(0,-0.792027)(0.94241494,0.792027)
\psline[linecolor=black, linewidth=0.02](0.45697075,0.20267501)(0.45697075,-0.21881929)
\psline[linecolor=black, linewidth=0.02](0.4562587,0.20204714)(0.8494084,0.5487268)
\psline[linecolor=black, linewidth=0.02](0.067692176,0.5459411)(0.46084186,0.19926146)
\psline[linecolor=black, linewidth=0.02](0.06372307,-0.56065685)(0.45687276,-0.21397717)
\psline[linecolor=black, linewidth=0.02](0.45066956,-0.20976295)(0.84381926,-0.5564426)
\rput[bl](0.010811806,0.6225269){$\scriptstyle a$}
\rput[bl](0.52679443,-0.08153543){$\scriptstyle k$}
\rput[bl](0.832415,0.61202705){$\scriptstyle b$}
\rput[bl](0.0,-0.78152716){$\scriptstyle a$}
\rput[bl](0.8216032,-0.792027){$\scriptstyle b$}
\end{pspicture}
}
\end{gathered} \\
&= p_a \otimes p_b
\end{align*}
by Theorem \ref{thm:JWP-tensorprod-identity}, hence $\psi\varphi$ is the identity on $p_a \otimes p_b$.
\end{proof}
\bigskip

\begin{theorem}[Semisimplicity of $\TLJ$]
Generic Temperley-Lieb-Jones is semisimple: every object $P \in \TLJ$ is isomorphic to a direct sum of Jones-Wenzl idempotents.
\end{theorem}
\begin{proof}
By the distributive property of $\otimes$ over $\oplus$ every object $P \in \TLJ$ can be written as a direct sum of tensor products $\bigotimes_{i=1}^n p_{a_i}$, so it suffices to show that every non-unary tensor product is isomorphic to a direct sum of Jones-Wenzl idempotents.

This is an easy induction on the number of tensor factors; the base case $n=2$ is Lemma \ref{lemma:JWP-tensor-isomorphism}.
Assume that for some $n \geq 2$ all tensor products with $n$ factors are isomorphic to a direct sum.
Then for a product with $n+1$ factors, using the isomorphism on the first $n$ factors and distributing the last factor over the resulting direct sum we obtain a direct sum of two-factor tensors.
Applying Lemma \ref{lemma:JWP-tensor-isomorphism} to each tensor product once more gives the required direct sum of Jones-Wenzl idempotents.
\end{proof}

\chapter{Specializing at roots of unity}
We pause for a moment to provide some motivation for this chapter and give an indication of where we are headed.

In order to construct a skein module associated to a surface $\Sigma$ we take a suitable diagrammatic category $\mathcal{C}$ and ``draw diagrams'' labelled with simple objects from $\mathcal{C}$ onto $\Sigma$.
However we would like our modules to be finite-dimensional, and in this case generic Temperley-Lieb-Jones does not suffice; if we apply our construction to $\TLJ$ we obtain an infinite-dimensional module, for the reason that there are infinitely many simple objects $p_n, n\in\mathbbm{N}$ in the category.
In this chapter we show that $\TLJ$, at a root of unity $q$ and modulo negligible elements, is a semisimple category with \emph{finitely} many simple objects (and is moreover a \emph{spherical fusion category}) which we can use to construct finite skein modules.

\section{Strict pivotal and spherical categories}
In order to develop the results we will need, we must first formalize a few properties of $\TL$.

\begin{defn}
\label{defn:strictpivotal}
Let $(\mathcal{C}, \otimes, e)$ be a strict monoidal category.
$\mathcal{C}$ is further called a \textbf{strict pivotal category} if it satisfies the following axioms.
\begin{enumerate}
\item \label{defn:strictpivotal-1} Firstly, $\mathcal{C}$ is equipped with a contravariant functor $^\ast \colon \mathcal{C}^\mathrm{op} \rightarrow \mathcal{C}$ known as the \emph{dual}, such that
\begin{enumerate}[label=\roman{*}., ref=\roman{*}]
\item $e^\ast = e$,
\end{enumerate}
and for all $a,b\in\Obj(\mathcal{C})$,
\begin{enumerate}[label=\roman{*}., ref=\roman{*}, resume]
\item $(a \otimes b)^\ast = b^\ast \otimes a^\ast$, and
\item $(a^\ast)^\ast = a$.
\end{enumerate}

\item \label{defn:strictpivotal-2} Secondly, for all objects $a \in \Obj(\mathcal{C})$ there is a \emph{unit} or \emph{co-evaluation} morphism $\eta_a \colon e \rightarrow a \otimes a^\ast$ satisfying the following three properties.
\begin{enumerate}[label=\roman{*}., ref=\roman{*}]
\item For all $a,b\in\Obj(\mathcal{C})$ and $f \colon a\rightarrow b$, the following diagram commutes:
\begin{equation*}
\begin{tikzcd}[row sep=3.5em, column sep=large]
e \arrow{r}{\eta_a} \arrow{d}[swap]{\eta_b} & a \otimes a^\ast \arrow{d}{f \otimes \mathbbm{1}} \\
b \otimes b^\ast \arrow{r}{\mathbbm{1}\otimes f^\ast} & b \otimes a^\ast
\end{tikzcd}
\end{equation*}
\item The following composition is the identity on $a^\ast$:
\begin{equation}
a^\ast = e \otimes a^\ast \xrightarrow{\eta_{a^\ast}\otimes\mathbbm{1}} a^\ast\otimes a^{\ast\ast} \otimes a^\ast = a^\ast\otimes(a\otimes a^\ast)^\ast \xrightarrow{\mathbbm{1}\otimes\eta^\ast_a} a^\ast \otimes e^\ast = a^\ast \otimes e = a^\ast
\label{eq:pivotal-id-comp}
\end{equation}
\item The following composition is $\eta_{a\otimes b}$:
\begin{equation}
e \xrightarrow{\eta_a} a \otimes a^\ast = a \otimes e \otimes a^\ast \xrightarrow{\mathbbm{1}\otimes\eta_b\otimes\mathbbm{1}} a \otimes b \otimes b^\ast \otimes a^\ast = a\otimes b\otimes(a\otimes b)^\ast
\label{eq:pivotal-eta_ab-comp}
\end{equation}
\end{enumerate}
\end{enumerate}
\end{defn}

Generic Temperley-Lieb is a strict pivotal category: the dual $a^\ast$ of an object $a \in \mathbbm{N}$ is $a$ itself, that is $a^\ast=a$, and the dual of a morphism $f$ is exactly the dual $f^\ast$ of formal diagrams as given in Definition \ref{defn:TL-dual}.
The unit morphism $\eta_a \colon 0 \rightarrow a \otimes a^\ast = 2a$ is the $a$-fold nested cup from $0$ to $2a$ points, hence $\eta_a^\ast \colon 2a \rightarrow 0$ is the $a$-fold nested cap.
One checks that the axioms for a strict pivotal category are satisfied; note especially that the compositions \eqref{eq:pivotal-id-comp} and \eqref{eq:pivotal-eta_ab-comp} are given by the diagrams in Figure \ref{fig:pivotal-composition}.

% ------------- FIG: TL is pivotal -------------
\begin{figure}[htb]
\begin{align*}
\begin{gathered}
\psscalebox{1.0 1.0} % Change this value to rescale the drawing.
{
\begin{pspicture}(0,-1.2666935)(2.5208032,1.2666935)
\psarc[linecolor=black, linewidth=0.02, dimen=outer](0.8333334,0.0){0.5688889}{180.0}{360.0}
\psarc[linecolor=black, linewidth=0.02, dimen=outer](0.83195204,0.0){0.24864864}{180.0}{360.0}
\psline[linecolor=black, linewidth=0.02](1.9155556,-1.2533065)(1.9155556,0.008915677)
\psline[linecolor=black, linewidth=0.02](2.2355556,-1.2533065)(2.2355556,0.008915677)
\psarc[linecolor=black, linewidth=0.02, dimen=outer](1.6577778,0.0){0.25777778}{0.0}{180.0}
\psarc[linecolor=black, linewidth=0.02, dimen=outer](1.6577778,0.004471232){0.5777778}{0.0}{180.0}
\psline[linecolor=black, linewidth=0.02](0.2622223,0.0)(0.2622223,1.262249)
\psline[linecolor=black, linewidth=0.02](0.5822223,0.0)(0.5822223,1.262249)
\psframe[linecolor=black, linewidth=0.02, linestyle=dotted, dotsep=0.10583334cm, dimen=outer](2.5208032,1.2666935)(0.0,-1.2666935)
\psline[linecolor=black, linewidth=0.02, linestyle=dotted, dotsep=0.10583334cm](0.042705383,-0.004417657)(2.4428985,-0.004417657)
\rput[bl](1.9536508,-0.66219544){$\scriptstyle\cdots$}
\rput[bl](1.5479366,0.34351885){$\scriptstyle\cdots$}
\rput[bl](0.7079366,-0.44505256){$\scriptstyle\cdots$}
\rput[bl](0.30222228,0.57209027){$\scriptstyle\cdots$}
\rput[bl](0.7,-0.89076686){$\scriptstyle\eta_{a}$}
\rput[bl](1.5536509,0.6920903){$\scriptstyle\eta_a^\ast$}
\rput[bl](0.68965083,0.694376){$\scriptstyle\mathbbm{1}_a$}
\rput[bl](1.5227938,-0.8919097){$\scriptstyle\mathbbm{1}_a$}
\end{pspicture}
}
\end{gathered}
\ &= \ 
\begin{gathered}
\psscalebox{1.0 1.0} % Change this value to rescale the drawing.
{
\begin{pspicture}(0,-0.56170213)(1.2,0.56170213)
\psline[linecolor=black, linewidth=0.02](0.38127682,-0.5590544)(0.38127682,0.558487)
\psline[linecolor=black, linewidth=0.02](0.7523407,-0.5590544)(0.7523407,0.558487)
\rput[bl](0.44680876,-0.05256306){$\scriptstyle\cdots$}
\psframe[linecolor=black, linewidth=0.02, dimen=outer](1.1234045,0.56170213)(0.0,-0.56170213)
\end{pspicture}
}
\end{gathered}
\ = \ \mathbbm{1}_a \\\\
\begin{gathered}
\psscalebox{1.0 1.0} % Change this value to rescale the drawing.
{
\begin{pspicture}(0,-1.1740487)(2.6696632,1.1740487)
\psarc[linecolor=black, linewidth=0.02, dimen=outer](1.3310114,0.8469023){0.5}{180.0}{360.0}
\psarc[linecolor=black, linewidth=0.02, dimen=outer](1.3310114,0.8469023){0.24}{180.0}{360.0}
\psarc[linecolor=black, linewidth=0.02, dimen=outer](1.3310113,0.08311196){1.06}{180.0}{360.0}
\psarc[linecolor=black, linewidth=0.02, dimen=outer](1.3310113,0.083111875){0.7}{180.0}{360.0}
\psline[linecolor=black, linewidth=0.02](2.0305,0.074282445)(2.0247028,0.84790236)
\psline[linecolor=black, linewidth=0.02](2.3910112,0.08273939)(2.3910112,0.85000974)
\psline[linecolor=black, linewidth=0.02](0.27285278,0.0719295)(0.2670557,0.8545382)
\psline[linecolor=black, linewidth=0.02](0.6286584,0.08038645)(0.6286584,0.85215116)
\psframe[linecolor=black, linewidth=0.02, linestyle=dotted, dotsep=0.10583334cm, dimen=outer](2.669663,0.85291785)(0.0,-1.1740484)
\psline[linecolor=black, linewidth=0.02, linestyle=dotted, dotsep=0.10583334cm](0.008988881,0.070895374)(2.633708,0.07538975)
\rput[bl](0.33939844,0.9740487){$\scriptstyle \mathbbm{1}_a$}
\rput[bl](2.0813339,0.9740487){$\scriptstyle \mathbbm{1}_a$}
\rput[bl](1.2103662,0.9940487){$\scriptstyle \eta_b$}
\rput[bl](1.2297211,-0.47756422){$\scriptstyle \eta_a$}
\rput[bl](0.32004362,0.40630674){$\scriptstyle \cdots$}
\rput[bl](2.0877855,0.39985514){$\scriptstyle \cdots$}
\rput[bl](1.2160804,0.4551547){$\scriptstyle \cdots$}
\rput[bl](1.2297211,-0.8222647){$\scriptstyle \cdots$}
\end{pspicture}
}
\end{gathered}
\ &= \ 
\begin{gathered}
\psscalebox{1.0 1.0} % Change this value to rescale the drawing.
{
\begin{pspicture}(0,-0.9151724)(1.2275866,0.50)
\rput[bl](0.5,-0.3805994){$\scriptstyle \cdots$}
\psframe[linecolor=black, linewidth=0.02, dimen=outer](1.2275865,0.3124138)(0.0,-0.9151724)
\psbezier[linecolor=black, linewidth=0.02](0.43540254,0.30306515)(0.44536424,0.051417634)(0.5195954,-0.19320084)(0.62559414,-0.19103447)(0.73159283,-0.1888681)(0.7873442,0.0348194)(0.80000025,0.3010728)
\psbezier[linecolor=black, linewidth=0.02](0.1931037,0.29172415)(0.19881018,-0.25324485)(0.4157857,-0.56789273)(0.62023014,-0.57233715)(0.8246746,-0.5767816)(1.0354079,-0.26176843)(1.0576248,0.30061305)
\rput[bl](0.18699533,0.39517242){$\displaystyle\overbrace{\qquad\ }^{2(a+b)}$}
\end{pspicture}
}
\end{gathered}
\ = \ \eta_{a\otimes b}
\end{align*}
\caption{$\TL$ is pivotal.}\label{fig:pivotal-composition}
\end{figure}
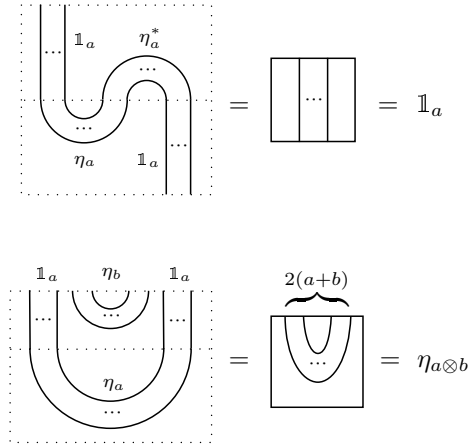
% ------------- END FIG -------------

\begin{remark}
We only define strict pivotal categories since it will turn out that $\TL$ and the rest of our categories are all strictly pivotal.
The non-strict definition is more involved: we replace the equalities in \eqref{defn:strictpivotal-1} in Definition \ref{defn:strictpivotal} with natural isomorphisms between the appropriate functors.
As a result the compositions in \eqref{defn:strictpivotal-2} have to include the identity and associative isomorphisms $\lambda, \rho, \alpha$, and we have to impose coherence conditions in the form of six different diagrams which are required to be commutative.
Fortunately for our purposes we may ignore these  details and simply refer the interested reader to Sections 1 and 2 of \cite{Barrett1993} for more information.
\end{remark}

The dual $\eta^\ast$ of the unit morphism is called the \emph{evaluation} or \emph{co-unit} morphism.
Any strict pivotal category has \emph{trace maps} on endomorphisms $f$, obtained by tensoring $f$ on the left or right with the identity morphism and then composing with the appropriate unit and evaluation morphisms.

\begin{defn}
Let $\mathcal{C}$ be a strict pivotal category.
The \textbf{right trace} of a morphism $f \colon a \rightarrow a$ in $\mathcal{C}$ is given by the following composition,
\begin{equation*}
e \xrightarrow{\eta_a} a\otimes a^\ast \xrightarrow{f\otimes\mathbbm{1}} a\otimes a^\ast \xrightarrow{\eta^\ast_a} e.
\end{equation*}
Similarly we define the \textbf{left trace} of $f$ to be
\begin{equation*}
e \xrightarrow{\eta_{a^\ast}} a^\ast\otimes a \xrightarrow{\mathbbm{1}\otimes f} a^\ast\otimes a \xrightarrow{\eta^\ast_{a^\ast}} e.
\end{equation*}
\end{defn}

It is clear that in $\TL$ the categorical trace is exactly the trace $\tr$ obtained by tracing off an endomorphism $f$ on the right or the left as in Definition \ref{defn:TL-trace}.
Recall from Remark \ref{rem:trace-map} that $\tr$ is a map from the endomorphism spaces $\Hom(a,a)$ to $\mathbbm{F}$, given by $d^m$ where $m$ is the number of loops in $\tr(f)$.
Since this is the same whichever side we trace off, the left trace and the right trace are the same for all $f$, hence $\TL$ is an example of a \emph{spherical category}.

\begin{defn}[Barrett and Westbury, 1993]
A pivotal category $\mathcal{C}$ is \textbf{spherical} if for every endomorphism $f$ in $\mathcal{C}$, the left and right trace of $f$ are equal.
\end{defn}

This justifies us simply speaking of ``the'' trace in a spherical category.

Let $f \colon a \rightarrow b, g \colon b\rightarrow a$ be TL diagrams, and consider $\tr(gf)$.
Observe that the trace of a TL diagram $D$ is invariant under isotopies of $D$, thus we may pull the diagram $g$ around the closed loop without changing the trace (Fig. \ref{fig:trace-isotopy-inv}).
This proves the following statement:

\begin{lemma}
\label{lemma:trace-cyclic-inv}
For all formal diagrams $f \colon a\rightarrow b, g \colon b \rightarrow a$ in $\TL$,
\[ \tr(gf)=\tr(fg). \]
\end{lemma}

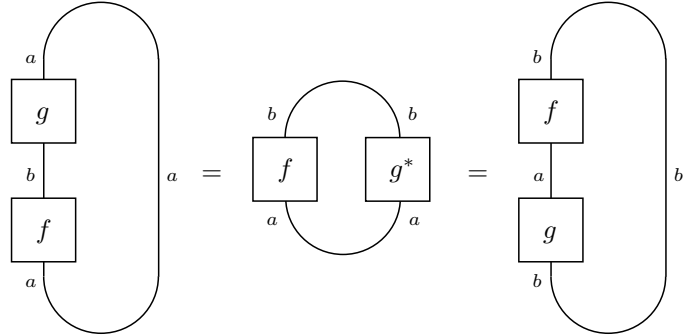
\begin{figure}
\begin{equation*}
\begin{gathered}
\psscalebox{1.0 1.0} % Change this value to rescale the drawing.
{
\begin{pspicture}(0,-2.2042856)(2.1785386,2.2042856)
\psline[linecolor=black, linewidth=0.02](0.43363753,1.4411204)(0.43363753,-1.437311)
\psframe[linecolor=black, linewidth=0.02, fillstyle=solid, dimen=outer](0.8630307,1.1820116)(0.0,0.3189813)
\rput[bl](0.19793253,-0.16928965){$\scriptstyle b$}
\psarc[linecolor=black, linewidth=0.02, dimen=outer](1.1870056,1.4399999){0.7542857}{0.0}{180.0}
\psarc[linecolor=black, linewidth=0.02, dimen=outer](1.1870055,-1.44){0.7542857}{180.0}{360.0}
\psline[linecolor=black, linewidth=0.02](1.9448462,-1.4451542)(1.9448462,1.4489635)
\rput[bl](0.18793254,1.3875731){$\scriptstyle a$}
\rput[bl](0.18793254,-1.5810543){$\scriptstyle a$}
\rput[bl](0.33090946,0.5789813){$g$}
\psframe[linecolor=black, linewidth=0.02, fillstyle=solid, dimen=outer](0.8630305,-0.3931399)(0.0,-1.2561703)
\rput[bl](0.33090925,-0.9961702){$f$}
\rput[bl](2.0485387,-0.17000265){$\scriptstyle a$}
\end{pspicture}
}
\end{gathered}
\ \ = \ \ 
\begin{gathered}
\psscalebox{1.0 1.0} % Change this value to rescale the drawing.
{
\begin{pspicture}(0,-1.159109)(2.5,1.159109)
\rput[bl](0.19793177,0.63588685){$\scriptstyle b$}
\psarc[linecolor=black, linewidth=0.02, dimen=outer](1.1950048,-0.39482358){0.7542857}{180.0}{360.0}
\rput[bl](0.18793176,-0.7638778){$\scriptstyle a$}
\psframe[linecolor=black, linewidth=0.02, fillstyle=solid, dimen=outer](0.8630297,0.4120366)(0.0,-0.4509937)
\rput[bl](0.3309085,-0.1909937){$f$}
\rput[bl](2.0485377,0.63517386){$\scriptstyle b$}
\psframe[linecolor=black, linewidth=0.02, fillstyle=solid, dimen=outer](2.345383,0.41071755)(1.4823525,-0.45231277)
\rput[bl](1.8054185,-0.19231276){$g^\ast$}
\psarc[linecolor=black, linewidth=0.02, dimen=outer](1.1870049,0.39482346){0.7542857}{0.0}{180.0}
\rput[bl](2.0599318,-0.7638778){$\scriptstyle a$}
\end{pspicture}
}
\end{gathered}
\ \ = \ \ 
\begin{gathered}
\psscalebox{1.0 1.0} % Change this value to rescale the drawing.
{
\begin{pspicture}(0,-2.2042856)(2.158538,2.2042856)
\psline[linecolor=black, linewidth=0.02](0.4336371,1.4411204)(0.4336371,-1.4373109)
\psframe[linecolor=black, linewidth=0.02, fillstyle=solid, dimen=outer](0.8630301,-0.39313984)(0.0,-1.2561702)
\psframe[linecolor=black, linewidth=0.02, fillstyle=solid, dimen=outer](0.86303025,1.1820117)(0.0,0.31898138)
\rput[bl](0.19793212,-0.16928956){$\scriptstyle a$}
\psarc[linecolor=black, linewidth=0.02, dimen=outer](1.1870053,1.4399999){0.7542857}{0.0}{180.0}
\psarc[linecolor=black, linewidth=0.02, dimen=outer](1.1870052,-1.44){0.7542857}{180.0}{360.0}
\psline[linecolor=black, linewidth=0.02](1.9448458,-1.4451541)(1.9448458,1.4489635)
\rput[bl](0.18793212,1.3875732){$\scriptstyle b$}
\rput[bl](0.18793212,-1.5810542){$\scriptstyle b$}
\rput[bl](0.33090904,0.5789814){$f$}
\rput[bl](0.33090886,-0.9961701){$g$}
\rput[bl](2.0485382,-0.17000256){$\scriptstyle b$}
\end{pspicture}
}
\end{gathered}
\end{equation*}
\caption{Invariance of the trace.} \label{fig:trace-isotopy-inv}
\end{figure}

In fact Lemma \ref{lemma:trace-cyclic-inv} is true in general for all spherical categories, see Lemma 1.5.1 of \cite{Turaev1994}.
(This justifies the name ``trace''; furthermore this construction is indeed the usual linear trace in the category $\mathbf{Vect}$ of vector spaces.)

\section{Negligible morphisms}
\begin{defn}
Let $\mathcal{C}$ be a spherical category with trace $\tr$.
A morphism $f \colon a \rightarrow b$ in $\mathcal{C}$ is \textbf{negligible} if for all $g \colon b \rightarrow a$, the trace $\tr(gf)$ is zero.
\end{defn}

We state and prove the following lemma in the setting of $\TL$, though the reader familiar with the diagrammatic calculus for pivotal categories will recognize it as a result for such categories in general.

\begin{lemma}
The set $\mathcal{N}$ of all negligible morphisms in $\TL$ forms a tensor ideal, called the \textbf{\upshape{negligible ideal}}.
\end{lemma}

\begin{proof}
Let $f \colon a \rightarrow b$ be negligible.
For all $g \colon b \rightarrow c$ and all $h \colon c \rightarrow a$,
\[ \tr(h(gf)) = \tr((hg)f) = 0 \]
and
\[ \tr(g(fh)) = \tr((gf)h) = \tr(h(gf)) = \tr((hg)f) = 0 \]
by Lemma \ref{lemma:trace-cyclic-inv}, hence $\mathcal{N}$ is closed under arbitrary composition.

Now suppose $g \colon c \rightarrow d$, we need to show that for all $h \colon b \otimes d \rightarrow a \otimes c$ the composite $h(f\otimes g)$ has zero trace.
By isotoping, we have that
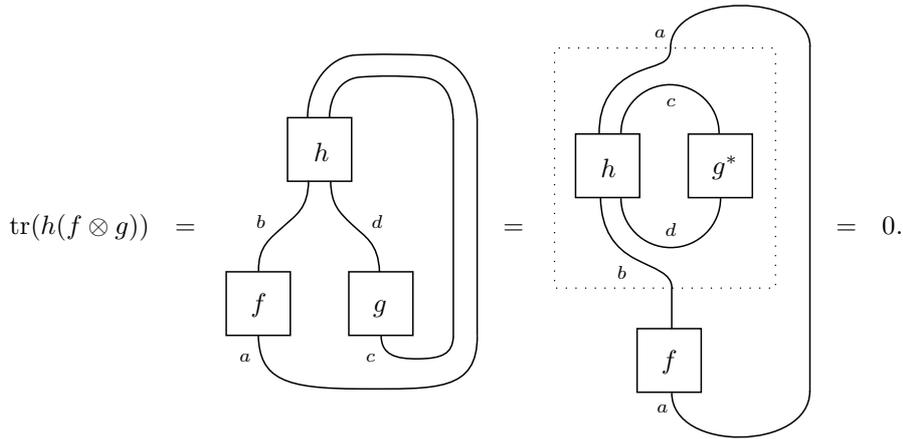
\begin{align*}
\tr(h(f\otimes g))
\ \ = \ \ 
\begin{gathered}
\psscalebox{1.0 1.0} % Change this value to rescale the drawing.
{
\begin{pspicture}(0,-2.2322364)(3.3245127,2.2322364)
\psline[linecolor=black, linewidth=0.02](3.3106408,1.3626614)(3.3106408,-1.5157701)
\rput[bl](0.40278366,-0.08394113){$\scriptstyle b$}
\psline[linecolor=black, linewidth=0.02](2.997227,-1.5363345)(2.997227,1.3577831)
\rput[bl](1.918048,-0.08394113){$\scriptstyle d$}
\rput[bl](0.1835844,-1.863995){$\scriptstyle a$}
\rput[bl](1.8458356,-1.863995){$\scriptstyle c$}
\psframe[linecolor=black, linewidth=0.02, fillstyle=solid, dimen=outer](1.6692802,1.3848789)(0.8062499,0.52184856)
\rput[bl](1.162159,0.7943486){$h$}
\psframe[linecolor=black, linewidth=0.02, fillstyle=solid, dimen=outer](0.8630302,-0.65894663)(0.0,-1.521977)
\rput[bl](0.33090895,-1.261977){$f$}
\psframe[linecolor=black, linewidth=0.02, fillstyle=solid, dimen=outer](2.4755304,-0.6589468)(1.6125001,-1.5219771)
\rput[bl](1.9434092,-1.2619771){$g$}
\psbezier[linecolor=black, linewidth=0.02](1.0927196,0.53929573)(1.0884637,0.16926613)(0.92541516,0.08305868)(0.75938624,-0.07022807)(0.5933573,-0.22351481)(0.42772293,-0.3302087)(0.43557674,-0.66070426)
\psbezier[linecolor=black, linewidth=0.02](2.026053,-0.67022806)(2.0130897,-0.28607363)(1.8326584,-0.20381409)(1.6641482,-0.04165664)(1.4956379,0.12050081)(1.3915994,0.2161523)(1.3828944,0.52856636)
\psbezier[linecolor=black, linewidth=0.02](2.9946907,1.345076)(2.9903378,1.6357408)(2.8914366,1.8913376)(2.5763755,1.9114252)(2.2613144,1.9315127)(2.0460274,1.9326644)(1.7746662,1.9131346)(1.503305,1.8936048)(1.3714857,1.6451814)(1.3677797,1.3787267)
\psbezier[linecolor=black, linewidth=0.02](3.3094444,1.3574142)(3.2875898,1.8323585)(3.0213103,2.1904504)(2.6637337,2.2029274)(2.306157,2.2154043)(2.0704403,2.222227)(1.7164658,2.2058194)(1.3624913,2.1894119)(1.0807092,1.8141103)(1.080176,1.3769264)
\psbezier[linecolor=black, linewidth=0.02](2.9970052,-1.5285614)(2.9863214,-1.7581348)(2.8468769,-1.8186903)(2.539783,-1.8230058)(2.2326891,-1.8273213)(2.0505106,-1.7704839)(2.0470054,-1.5135614)
\psbezier[linecolor=black, linewidth=0.02](3.3114498,-1.5180058)(3.3149335,-2.2577653)(2.7708926,-2.2230783)(1.7753386,-2.2213392)(0.7797846,-2.2196)(0.44709948,-2.0966394)(0.43200532,-1.5085614)
\end{pspicture}
}
\end{gathered}
\ \ = \ \ 
\begin{gathered}
\psscalebox{1.0 1.0} % Change this value to rescale the drawing.
{
\begin{pspicture}(0,-2.8715785)(3.3804357,2.8715785)
\psframe[linecolor=black, linewidth=0.02, fillstyle=solid, dimen=outer](1.1395129,1.1686046)(0.27648255,0.3055743)
\psline[linecolor=black, linewidth=0.02](3.3688185,2.328446)(3.3688185,-2.2629266)
\rput[bl](1.3235022,2.4362683){$\scriptstyle a$}
\rput[bl](1.3549408,-2.5389578){$\scriptstyle a$}
\rput[bl](0.61739165,0.58557427){$h$}
\rput[bl](1.4779621,1.5211536){$\scriptstyle c$}
\psframe[linecolor=black, linewidth=0.02, fillstyle=solid, dimen=outer](2.621866,1.1672856)(1.7588357,0.30425522)
\rput[bl](2.0819016,0.56425524){$g^\ast$}
\rput[bl](1.4658266,-0.19554514){$\scriptstyle d$}
\psarc[linecolor=black, linewidth=0.02, dimen=outer](1.5286841,1.1670805){0.645}{0.0}{180.0}
\psarc[linecolor=black, linewidth=0.02, dimen=outer](1.5386841,0.30708054){0.655}{180.0}{360.0}
\psframe[linecolor=black, linewidth=0.02, fillstyle=solid, dimen=outer](1.9489248,-1.4128451)(1.0858945,-2.2758756)
\rput[bl](1.4168036,-2.0158756){$f$}
\psbezier[linecolor=black, linewidth=0.02](0.61486053,0.31266877)(0.6241922,-0.11970544)(0.80676067,-0.35887372)(1.0706121,-0.4977826)(1.3344635,-0.63669145)(1.5521861,-0.6872584)(1.552977,-0.8843447)
\psframe[linecolor=black, linewidth=0.02, linestyle=dotted, dotsep=0.10583334cm, dimen=outer](2.9263155,2.3088474)(0.0,-0.89476156)
\psbezier[linecolor=black, linewidth=0.02](0.59732044,1.1548078)(0.5926308,1.5859135)(0.7860757,1.8283223)(1.0658919,1.963639)(1.345708,2.0989556)(1.5379009,2.0621264)(1.5355023,2.3111715)
\psline[linecolor=black, linewidth=0.02](1.5512992,-0.8803043)(1.5512992,-1.4240441)
\psbezier[linecolor=black, linewidth=0.02](1.5513312,-2.2732136)(1.5513312,-3.0732136)(3.3678017,-3.0496843)(3.3678017,-2.249684)
\psbezier[linecolor=black, linewidth=0.02](1.5352225,2.2984908)(1.5352225,3.0496018)(3.3704362,3.0420806)(3.3704362,2.3190036)
\rput[bl](0.83562905,-0.7598164){$\scriptstyle b$}
\end{pspicture}
}
\end{gathered}
\ \ = \ \ 0.
\end{align*}
A similar argument shows that $\mathcal{N}$ is closed under left tensor multiplication.

Finally since the trace is $\mathbbm{F}$-linear, if $f,g\in\Hom(a,b)$ are negligible then so are $f+g$ and $cf$ for all $c\in\mathbbm{F}$, hence the intersection of $\mathcal{N}$ with $\Hom(a,b)$ is a linear subspace of $\Hom(a,b)$.
\end{proof}

\section{Evaluating the quantum parameter}
Recall from Definition \ref{defn:quantumint} the formula for the quantum integers $[n]$.
By fixing the value of the parameter $q$ at a root of unity, we find that for certain $n$ the quantum integers $[n]$ become zero.
More specifically, $[n]=0$ precisely when $q \neq \pm 1$ is a $2n$-th root of unity, since
\[ [n] = \frac{q^n-q^{-n}}{q-q^{-1}} = 0 \iff q^n-q^{-n} = 0 \iff q^{2n}-1=0. \]

Recall that generic Temperley-Lieb $\TL$ is a $\mathbbm{F}=\mathbbm{C}(q)$-linear category; if we fix a value for $q$ we obtain a $\mathbbm{C}$-linear category.

\begin{defn}
Let $n>1$ and let $q = e^{\pi i/n}$ be a $2n$-th root of unity.
We obtain a $\mathbbm{C}$-linear category $\TL(q=e^{\pi i/n})$ with the same objects $x\in\mathbbm{N}$ as $\TL$, and whose morphisms are now $\mathbbm{C}$-linear combinations of simple TL diagrams from $x$ to $y$ points.
We say that $\TL(q)$ is the \textbf{Temperley-Lieb category evaluated at $q = e^{\pi i/n}$}.
\end{defn}

Evaluating $\TL$ at a root of unity has two effects --- the first is that all but finitely many of the Jones-Wenzl idempotents cease to exist in $\TL(q)$, and the second is that some nonzero morphisms in the category become negligible.

\begin{lemma}
Let $q = e^{\pi i/n}$ for some $n>1$, then in $\TL(q)$ only the finitely many Jones-Wenzl idempotents $p_i$ for $0 \leq i \leq n-1$ are defined.
Furthermore $p_{n-1}$ is negligible, and is the only Jones-Wenzl idempotent to be so.
\end{lemma}

\begin{proof}
Observe from the inductive definition \eqref{eq:JWPrelation} in Theorem \ref{thm:JWP} that since $[n]=0$, $p_n$ is no longer defined, hence the $p_i$ are undefined for $i \geq n$.

We claim that among the remaining Jones-Wenzl idempotents, $p_a$ is negligible if and only if $[a+1]=0$.
Let $g \in \Hom(a,a)$ and write $g = c\cdot\id_a + \sum c_j m_j$, where as usual each $m_j$ is a product of non-identity generators $U_i$ of $\TL_a$.
Then by property \eqref{eq:JWP2} of Theorem \ref{thm:JWP} we have that $g p_a = c p_a$, so that $\tr(g p_a) = c\,\tr(p_a) = c [a+1]$, and this trace is zero for all $g$ if and only if $[a+1]=0$.
Suppose $[m]=0$, then $1=q^{2m}=e^{2m\pi i/n}$ implies that $m$ is a multiple of $n$, thus $n$ is the smallest integer for which $[n]=0$ and $p_{n-1}$ is negligible.
\end{proof}

\begin{lemma}
\label{lemma:negligible_JWP_hom}
Let $q=e^{\pi i/n}$ and let $(a,b,c)$ be an admissible triple.
The morphism
\begin{equation*}
g_{a,b,c} \ \ = \ \ 
\begin{gathered}
\psscalebox{1.0 1.0} % Change this value to rescale the drawing.
{
\begin{pspicture}(0,-0.5875841)(1.4244444,0.5875841)
\psline[linecolor=black, linewidth=0.02](0.69777775,0.5875841)(0.69777775,-0.19146354)
\psline[linecolor=black, linewidth=0.02](0.69821805,-0.18495561)(1.3728931,-0.5744794)
\psline[linecolor=black, linewidth=0.02](0.701782,-0.18940005)(0.02710693,-0.5789239)
\rput[bl](0.82222223,0.39647296){$c$}
\rput[bl](1.2844445,-0.37241593){$b$}
\rput[bl](0.0,-0.37686038){$a$}
\end{pspicture}
}
\end{gathered}
\end{equation*}
is negligible in $\TL(q)$ if and only if $a+b+c \geq 2n-2$. 
\end{lemma}

\begin{proof}
By isotopy we have that
\begin{equation*}
\tr(h\circ g_{a,b,c})
\ \ = \ \ 
\begin{gathered}
\psscalebox{1.0 1.0} % Change this value to rescale the drawing.
{
\begin{pspicture}(0,-2.090995)(2.6783168,2.090995)
\psframe[linecolor=black, linewidth=0.02, linestyle=dashed, dash=0.17638889cm 0.10583334cm, dimen=outer](1.7333739,1.2964113)(0.023784796,0.16764419)
\psbezier[linecolor=black, linewidth=0.02](0.9798082,-0.40629798)(0.9987938,-0.64278907)(1.1741021,-0.6717027)(1.3265833,-0.78234136)(1.4790645,-0.89298)(1.5725181,-0.9928408)(1.5931416,-1.1751869)
\psbezier[linecolor=black, linewidth=0.02](0.77536374,-0.41074243)(0.76489466,-0.62446195)(0.54568267,-0.6997438)(0.42544505,-0.7935067)(0.30520743,-0.8872695)(0.20159031,-0.9622936)(0.17091931,-1.1707424)
\psbezier[linecolor=black, linewidth=0.02](0.40203044,-1.1796314)(0.44190526,-0.99360824)(0.61580753,-0.8811547)(0.8869085,-0.88109475)(1.1580094,-0.88103473)(1.300902,-0.9901955)(1.357586,-1.1796314)
\psframe[linecolor=black, linewidth=0.02, dimen=outer](1.7571586,-1.1721679)(1.1944444,-1.3174092)
\psframe[linecolor=black, linewidth=0.02, dimen=outer](0.56271416,-1.1721679)(0.0,-1.317409)
\psframe[linecolor=black, linewidth=0.02, dimen=outer](1.1599364,-0.2721679)(0.5972222,-0.4174091)
\psline[linecolor=black, linewidth=0.02](0.8785794,-0.28263786)(0.8785794,0.1832445)
\rput[bl](0.96502984,-0.1085202){$\scriptstyle c$}
\psbezier[linecolor=black, linewidth=0.02](1.4775223,-1.3046492)(1.514305,-1.8136556)(2.0910385,-1.8757918)(2.1733487,-1.1759825)(2.2556589,-0.4761732)(2.2350821,0.53304386)(2.1744835,1.2224195)(2.1138847,1.9117951)(1.2688313,1.9148926)(1.1889974,1.2792685)
\psbezier[linecolor=black, linewidth=0.02](0.27289167,-1.3059806)(0.20978361,-2.0058644)(2.4154932,-2.4811916)(2.5418854,-1.4204025)(2.6682773,-0.35961327)(2.6456294,0.31505919)(2.538701,1.3477246)(2.4317725,2.38039)(0.53121966,2.1593423)(0.46812978,1.2916385)
\rput[bl](1.4274703,1.4039203){$\scriptstyle b$}
\rput[bl](0.24002986,1.4160632){$\scriptstyle a$}
\rput[bl](0.06871323,-1.5794463){$\scriptstyle a$}
\rput[bl](1.7152218,-1.5687321){$\scriptstyle b$}
\rput[bl](0.7853395,0.5588482){$h$}
\rput[bl](0.284137,-0.7463178){$\scriptstyle i$}
\rput[bl](1.4327084,-0.7406035){$\scriptstyle j$}
\rput[bl](0.82747036,-1.1617941){$\scriptstyle k$}
\end{pspicture}
}
\end{gathered}
\ \ = \ \ 
\begin{gathered}
\psscalebox{1.0 1.0} % Change this value to rescale the drawing.
{
\begin{pspicture}(0,-2.050021)(2.7453618,2.050021)
\psframe[linecolor=black, linewidth=0.02, linestyle=dashed, dash=0.17638889cm 0.10583334cm, dimen=outer](2.7453616,0.4875312)(1.0357726,-0.64123595)
\psframe[linecolor=black, linewidth=0.02, dimen=outer](2.1719239,1.0767542)(1.6092097,0.931513)
\psline[linecolor=black, linewidth=0.02](1.8905673,0.4790315)(1.8905673,0.94491386)
\rput[bl](2.005589,0.65974253){$\scriptstyle c$}
\rput[bl](1.7973272,-0.2500319){$h^\ast$}
\psframe[linecolor=black, linewidth=0.02, dimen=outer](2.5929768,-1.0845568)(2.0302625,-1.229798)
\psline[linecolor=black, linewidth=0.02](2.3116198,-1.0950267)(2.3116198,-0.6291444)
\rput[bl](2.4050877,-0.93854064){$\scriptstyle a$}
\psframe[linecolor=black, linewidth=0.02, dimen=outer](1.7192926,-1.081048)(1.1565783,-1.2262892)
\psline[linecolor=black, linewidth=0.02](1.4379355,-1.0915179)(1.4379355,-0.6256356)
\rput[bl](1.2085966,-0.97354066){$\scriptstyle b$}
\psbezier[linecolor=black, linewidth=0.02](1.5227778,-1.2164823)(1.5182087,-1.4340025)(1.7251318,-1.4551191)(1.8761928,-1.4584529)(2.0272539,-1.4617867)(2.2199512,-1.423246)(2.2197833,-1.2149441)
\psbezier[linecolor=black, linewidth=0.02](1.3122514,-1.2171552)(1.1973983,-1.8051026)(0.8488454,-1.5477121)(0.7459071,-1.2564245)(0.64296883,-0.9651371)(0.54818755,-0.6680119)(0.63380396,0.6731645)(0.71942043,2.0143409)(1.7589253,1.7657866)(1.7813181,1.0631251)
\psbezier[linecolor=black, linewidth=0.02](2.4406762,-1.2257168)(2.467865,-2.2752328)(0.5626456,-2.2431269)(0.3404438,-1.4295363)(0.118241966,-0.61594576)(0.19946928,0.38059953)(0.300302,1.2467321)(0.4011347,2.1128647)(1.982653,2.426484)(2.0007536,1.0733142)
\rput[bl](1.7885088,-1.7156574){$\scriptstyle k$}
\rput[bl](0.4270802,-0.14446682){$\scriptstyle j$}
\rput[bl](0.0,-0.108514436){$\scriptstyle i$}
\end{pspicture}
}
\end{gathered}\ ,
\end{equation*}
so for all $h \colon c \rightarrow a \otimes b$ the diagram $h^\ast$ is a morphism from $b \otimes a$ to $c$.
We proved (cf. Theorem \ref{thm:JWP_hom_dim}) that the only such nonzero diagram is $g_{b,a,c}$, so the trace is certainly zero unless $h^\ast = g_{b,a,c}$, in which case we get that
\begin{equation*}
\tr(h\circ g_{a,b,c})
\ \ = \ \ 
\begin{gathered}
\psscalebox{1.0 1.0} % Change this value to rescale the drawing.
{
\begin{pspicture}(0,-1.3079076)(2.0261943,1.3079076)
\psline[linecolor=black, linewidth=0.02](0.53444445,0.43877974)(0.53444445,-0.5012202)
\psbezier[linecolor=black, linewidth=0.02](0.9748543,0.7444804)(1.1075333,0.8376038)(1.2694311,0.86202896)(1.3978032,0.7398827)(1.5261751,0.61773646)(1.6008801,0.36902323)(1.6017241,-0.06374898)(1.6025681,-0.4965212)(1.5377197,-0.6960537)(1.4086999,-0.826914)(1.27968,-0.9577743)(1.0999609,-0.91099745)(0.97025657,-0.8224196)
\psbezier[linecolor=black, linewidth=0.02](0.08908046,0.75387555)(0.079440236,0.9826691)(0.2184073,1.0818934)(0.37582377,1.1561744)(0.5332402,1.2304554)(0.8900458,1.2977089)(1.3201579,1.1857044)(1.75027,1.0736998)(2.0161824,0.8176167)(2.0129306,-0.035755336)(2.0096788,-0.8891274)(1.6572316,-1.1078804)(1.2946014,-1.2028313)(0.9319714,-1.2977822)(0.55230695,-1.263929)(0.375977,-1.2007605)(0.19964707,-1.137592)(0.09133517,-1.0501508)(0.08463602,-0.8271206)
\psline[linecolor=black, linewidth=0.02](0.5403175,0.4200496)(0.081614666,0.7565849)
\psline[linecolor=black, linewidth=0.02](0.52989054,0.41775075)(0.98859334,0.75428605)
\psline[linecolor=black, linewidth=0.02](0.5380186,-0.49398983)(0.07931582,-0.8305251)
\psline[linecolor=black, linewidth=0.02](0.5275917,-0.49169096)(0.98629445,-0.82822627)
\rput[bl](0.6311111,-0.13622025){$\scriptstyle c$}
\rput[bl](0.0,-0.71844244){$\scriptstyle a$}
\rput[bl](0.96666664,-0.73622024){$\scriptstyle b$}
\rput[bl](0.97555554,0.49489087){$\scriptstyle b$}
\rput[bl](0.008888889,0.51266867){$\scriptstyle a$}
\end{pspicture}
}
\end{gathered}
\ \ = \ \ 
\theta(a,b,c),
\end{equation*}
which by Theorem \ref{thm:theta_net_trace} is equal to
\[ \frac{[\frac{a+b-c}{2}]![\frac{b+c-a}{2}]![\frac{a+c-b}{2}]![\frac{a+b+c+2}{2}]!}{[a]![b]![c]!}. \]
But this is zero if and only if the numerator contains a factor of $[n]$, which occurs precisely when
\[ \frac{a+b+c+2}{2} \geq n \iff a+b+c \geq 2n-2. \qedhere \]
\end{proof}

The preceding two lemmas motivate the following definition, which we will shortly use.

\begin{defn}
At a root of unity $q=e^{\pi i/n}$, a triple $(a,b,c)$ of natural numbers is called \textbf{$q$-admissible} if $(a,b,c)$ is admissible, $0 \leq a,b,c \leq n-2$, and $a+b+c < 2n-2$.
\end{defn}

\section{Temperley-Lieb-Jones at roots of unity}
Throughout this section we fix the value of $q = e^{\pi i/n}$ at a $2n$-th root of unity for some $n>1$.

Now let us take the quotient of $\TL(q)$ by its negligible ideal $\mathcal{N}$.
This is the category $\TL(q)/\mathcal{N}$ having the same objects as $\TL(q)$, with each $\Hom_{\TL(q)/\mathcal{N}}(a,b)$ being the hom-space $\Hom_{\TL(q)}(a,b)$ quotiented out by its subspace $\mathcal{N}\cap\Hom_{\TL(q)}(a,b)$.
Thus $\TL(q)/\mathcal{N}$ is the category $\TL(q)$ where all negligible morphisms are identified with the zero morphism $0$.
In particular $p_{n-1}$ and $g_{a,b,c}$, where $(a,b,c)$ is not a $q$-admissible triple, are zero in $\TL(q)/\mathcal{N}$.

Note that the quotient category inherits the spherical $\mathbbm{C}$-linear monoidal structure of $\TL(q)$.
As in the unevaluated case we can take the Karoubi envelope and then the additive completion to obtain $\Mat\Kar(\TL(q)/\mathcal{N})$.

\begin{remark*}
Our quotient category $\TL/\mathcal{N}$ is a specific instance of Theorem 2.9 in \cite{Barrett1993}.
\end{remark*}

\begin{defn}
The \textbf{Temperley-Lieb-Jones category} $\TLJ(q)$ \textbf{evaluated at $q=e^{\pi i/n}$} is the full subcategory of $\Mat\Kar(\TL(q)/\mathcal{N})$ having objects $\Obj(\TLJ(q))$ being the closure of the nonzero Jones-Wenzl idempotents $p_i,\ 0\leq i\leq n-2$ under direct sum and tensor product.
\end{defn}

By Theorem \ref{thm:JWP_hom_dim} and Lemma \ref{lemma:negligible_JWP_hom}, if $p_a,p_b,p_c$ are Jones-Wenzl idempotents in $\TLJ(q)$ at a root of unity then $\Hom(p_a \otimes p_b, p_c)$ is $1$-dimensional if and only if $(a,b,c)$ is a $q$-admissible triple.
This is because for $(a,b,c)$ not $q$-admissible, $g_{a,b,c}$ is negligible in $\TL(q)$ and hence zero after taking the quotient.
From this it then follows in exactly the same manner as for the unevaluated case that

\begin{theorem}
The Jones-Wenzl idempotents $p_i,\ 0 \leq i \leq n-2$ are simple and form a collection of disjoint simple objects in $\TLJ(q)$.
\end{theorem}

As before we wish to show that $\TLJ(q)$ is semisimple.
However now the isomorphism between a tensor product and a direct sum of simple objects is more complicated.
The following notation will simplify the presentation slightly: at a $2n$-th root of unity we write
\[ a\nplus b = 
\begin{cases}
a+b & \text{if $a+b < n-1$} \\
2n - (a+b) -4 & \text{if $a+b \geq n-1$}
\end{cases}
\]
\begin{remark}
\label{rem:tensorprod-q-adm}
It is not hard to verify that for all $0 \leq a,b \leq n-2$ we have
\[ a \nplus b \leq n-2, \]
and $(a \nplus b) - \abs{a-b} \geq 0$ is even.
Now let $k \in \{\abs{a-b}, \abs{a-b}+2, \dotsc, a \nplus b \}$, then the triple $(a,b,k)$ is $q$-admissible, which we see as follows.

In the case that $a+b < n-1$ we have $a \nplus b = a+b$, and we showed in the proof of Theorem \ref{thm:JWP-tensorprod-identity} that $(a,b,k)$ is admissible in this case.
Furthermore since $k \leq a \nplus b \leq n-2$ and $a+b+k \leq a+b+n-2 \leq 2(n-2) = 2n-4$, it is also $q$-admissible.

If $a+b \geq n-1$ then by the proof mentioned earlier we have that $a+b+k$ is even, $b+k \geq a$, and $a+k \geq b$.
Furthermore
\[ a+b \geq n-1 > a \nplus b \geq k, \]
and
\[ a+b+k \leq a+b+(a \nplus b) = 2n-4, \]
so $(a,b,k)$ is also $q$-admissible in this case.
\end{remark}

Observe that at a root of unity, Lemmas \ref{lemma:roundabout_lemma}, \ref{lemma:shrink_triangle_k+1} and \ref{lemma:shrink_triangle_k-1} and their proofs still hold under condition of $q$-admissibility.
We then have the following theorem.

\begin{theorem}[Truncated tensor product identity]
\label{thm:truncated-tensorprod-identity}
At a root of unity $q=e^{\pi i/n}$,
\begin{equation}
\label{eq:truncated-tensorprod-identity}
p_a \otimes p_b = 
\sumtwo{k=\abs{a-b}}{a \nplus b} \left( \frac{[k+1]}{\theta(a,b,k)}
\begin{gathered}
\psscalebox{1.0 1.0} % Change this value to rescale the drawing.
{
\begin{pspicture}(0,-0.792027)(0.94241494,0.792027)
\psline[linecolor=black, linewidth=0.02](0.45697075,0.20267501)(0.45697075,-0.21881929)
\psline[linecolor=black, linewidth=0.02](0.4562587,0.20204714)(0.8494084,0.5487268)
\psline[linecolor=black, linewidth=0.02](0.067692176,0.5459411)(0.46084186,0.19926146)
\psline[linecolor=black, linewidth=0.02](0.06372307,-0.56065685)(0.45687276,-0.21397717)
\psline[linecolor=black, linewidth=0.02](0.45066956,-0.20976295)(0.84381926,-0.5564426)
\rput[bl](0.010811806,0.6225269){$\scriptstyle a$}
\rput[bl](0.52679443,-0.08153543){$\scriptstyle k$}
\rput[bl](0.832415,0.61202705){$\scriptstyle b$}
\rput[bl](0.0,-0.78152716){$\scriptstyle a$}
\rput[bl](0.8216032,-0.792027){$\scriptstyle b$}
\end{pspicture}
}
\end{gathered}
\ \right)
\end{equation}
for $0 \leq a,b \leq n-2$.
\end{theorem}

\begin{proof}
The structure of this proof is essentially the same as that for the non-truncated tensor product identity (Theorem \ref{thm:JWP-tensorprod-identity}), we give a sketch by following along the lines of the proof of the non-truncated case, indicating what changes need to be made as we go.
Note that we have already showed in Remark \ref{rem:tensorprod-q-adm} that $(a,b,k)$ is $q$-admissible over the range of the sum.

Firstly, the case where $a=0$ or $b=0$ is the same as for the non-truncated version.

Next we prove \eqref{eq:truncated-tensorprod-identity} for $a \geq 1$ and $b=1$.
If $a <n-2$, we have that $a+b<n-1$ and $a\nplus b=a+b$, and the truncated identity is the same as the non-truncated identity, which we have already proved.
(One checks that the proof in the non-truncated case does not involve negligible morphisms and division by zero quantum integers, and so is still valid when considered at a root of unity.)
If $a=n-2$, we have $a+b=n-1$, and hence $a\nplus b=n-3=a-1$. Then
\begin{align*}
\sumtwo{k=a-1}{a-1} \left( \frac{[k+1]}{\theta(a,b,k)}
\begin{gathered}
\psscalebox{1.0 1.0} % Change this value to rescale the drawing.
{
\begin{pspicture}(0,-0.792027)(0.94241494,0.792027)
\psline[linecolor=black, linewidth=0.02](0.45697075,0.20267501)(0.45697075,-0.21881929)
\psline[linecolor=black, linewidth=0.02](0.4562587,0.20204714)(0.8494084,0.5487268)
\psline[linecolor=black, linewidth=0.02](0.067692176,0.5459411)(0.46084186,0.19926146)
\psline[linecolor=black, linewidth=0.02](0.06372307,-0.56065685)(0.45687276,-0.21397717)
\psline[linecolor=black, linewidth=0.02](0.45066956,-0.20976295)(0.84381926,-0.5564426)
\rput[bl](0.010811806,0.6225269){$\scriptstyle a$}
\rput[bl](0.52679443,-0.08153543){$\scriptstyle k$}
\rput[bl](0.832415,0.61202705){$\scriptstyle b$}
\rput[bl](0.0,-0.78152716){$\scriptstyle a$}
\rput[bl](0.8216032,-0.792027){$\scriptstyle b$}
\end{pspicture}
}
\end{gathered}
\ \right)
\ \ &= \ \ 
\frac{[a]}{\theta(a,1,a-1)}
\begin{gathered}
\psscalebox{1.0 1.0} % Change this value to rescale the drawing.
{
\begin{pspicture}(0,-0.792027)(0.94241494,0.792027)
\psline[linecolor=black, linewidth=0.02](0.45697075,0.20267501)(0.45697075,-0.21881929)
\psline[linecolor=black, linewidth=0.02](0.4562587,0.20204714)(0.8494084,0.5487268)
\psline[linecolor=black, linewidth=0.02](0.067692176,0.5459411)(0.46084186,0.19926146)
\psline[linecolor=black, linewidth=0.02](0.06372307,-0.56065685)(0.45687276,-0.21397717)
\psline[linecolor=black, linewidth=0.02](0.45066956,-0.20976295)(0.84381926,-0.5564426)
\rput[bl](0.010811806,0.6225269){$\scriptstyle a$}
\rput[bl](0.52679443,-0.08153543){$\scriptstyle a-1$}
\rput[bl](0.832415,0.61202705){$\scriptstyle 1$}
\rput[bl](0.0,-0.78152716){$\scriptstyle a$}
\rput[bl](0.8216032,-0.792027){$\scriptstyle 1$}
\end{pspicture}
}
\end{gathered} \\
&= \ \ 
\frac{[n-2]}{[n-1]} \ 
\begin{gathered}
\psscalebox{1.0 1.0} % Change this value to rescale the drawing.
{
\begin{pspicture}(-0.31,-1.1710682)(1.415972,1.1710682)
\psframe[linecolor=black, linewidth=0.02, dimen=outer](0.8157322,0.62710935)(0.0,0.40234742)
\psframe[linecolor=black, linewidth=0.02, dimen=outer](0.81650305,-0.40278077)(0.0,-0.6275427)
\psline[linecolor=black, linewidth=0.02](0.27610415,-0.40790898)(0.27610415,0.41504583)
\psline[linecolor=black, linewidth=0.02](0.3925316,0.6144799)(0.3925316,1.1605116)
\psline[linecolor=black, linewidth=0.02](0.38320315,-1.1710682)(0.38320315,-0.6250364)
\psarc[linecolor=black, linewidth=0.02, dimen=outer](0.8560519,-0.41484427){0.31111112}{0.0}{180.0}
\psarc[linecolor=black, linewidth=0.02, dimen=outer](0.8653804,0.41822043){0.31111112}{180.0}{360.0}
\psline[linecolor=black, linewidth=0.02](1.1791992,0.40281427)(1.1791992,1.1710682)
\psline[linecolor=black, linewidth=0.02](1.163615,-0.41465753)(1.163615,-1.155834)
\rput[bl](-0.31,0.92){$\scriptstyle n-2$}
\rput[bl](-0.31,-1.15){$\scriptstyle n-2$}
\rput[bl](1.315972,0.9419145){$\scriptstyle 1$}
\rput[bl](1.267394,-1.116547){$\scriptstyle 1$}
\end{pspicture}
}
\end{gathered} \\
&= \ \ 
p_{n-1} \ \ + \ \ 
\frac{[n-2]}{[n-1]} \ 
\begin{gathered}
\psscalebox{1.0 1.0} % Change this value to rescale the drawing.
{
\begin{pspicture}(-0.31,-1.1710682)(1.415972,1.1710682)
\psframe[linecolor=black, linewidth=0.02, dimen=outer](0.8157322,0.62710935)(0.0,0.40234742)
\psframe[linecolor=black, linewidth=0.02, dimen=outer](0.81650305,-0.40278077)(0.0,-0.6275427)
\psline[linecolor=black, linewidth=0.02](0.27610415,-0.40790898)(0.27610415,0.41504583)
\psline[linecolor=black, linewidth=0.02](0.3925316,0.6144799)(0.3925316,1.1605116)
\psline[linecolor=black, linewidth=0.02](0.38320315,-1.1710682)(0.38320315,-0.6250364)
\psarc[linecolor=black, linewidth=0.02, dimen=outer](0.8560519,-0.41484427){0.31111112}{0.0}{180.0}
\psarc[linecolor=black, linewidth=0.02, dimen=outer](0.8653804,0.41822043){0.31111112}{180.0}{360.0}
\psline[linecolor=black, linewidth=0.02](1.1791992,0.40281427)(1.1791992,1.1710682)
\psline[linecolor=black, linewidth=0.02](1.163615,-0.41465753)(1.163615,-1.155834)
\rput[bl](-0.31,0.92){$\scriptstyle n-2$}
\rput[bl](-0.31,-1.15){$\scriptstyle n-2$}
\rput[bl](1.315972,0.9419145){$\scriptstyle 1$}
\rput[bl](1.267394,-1.116547){$\scriptstyle 1$}
\end{pspicture}
}
\end{gathered} \quad \text{since $p_{n-1}=0$ in $\TLJ(q)$} \\
&= \ \ 
\begin{gathered}
\psscalebox{1.0 1.0} % Change this value to rescale the drawing.
{
\begin{pspicture}(-0.3,-0.92082924)(1.1251371,0.92082924)
\psline[linecolor=black, linewidth=0.02](0.34142664,0.9207958)(0.34758046,-0.9192042)
\psframe[linecolor=black, linewidth=0.02, fillstyle=solid, dimen=outer](0.67419314,0.11688164)(0.0,-0.107880265)
\psline[linecolor=black, linewidth=0.02](0.90139127,0.9192043)(0.9075451,-0.9207957)
\rput[bl](-0.3,-0.605){$\scriptstyle n-2$}
\rput[bl](1.0251371,-0.605){$\scriptstyle 1$}
\end{pspicture}
}
\end{gathered} \\
&= \ \ p_a \otimes p_1.
\end{align*}

We have proved the truncated identity on $p_a \otimes p_1$ in the ``correct'' way; let us now introduce a convention that will greatly simplify the rest of this proof.
We have observed that the truncated identity for $p_a \otimes p_1$ coincides with the non-truncated identity when $a < n-2$.
Let us extend this by convention to include the case $a=n-2$, that is to say we will write
\begin{equation*}
p_{n-2} \otimes p_1 = \sumtwo{k=n-3}{n-1}\frac{[k+1]}{\theta(a,1,k)}
\begin{gathered}
\psscalebox{1.0 1.0} % Change this value to rescale the drawing.
{
\begin{pspicture}(-0.15,-0.792027)(0.94241494,0.792027)
\psline[linecolor=black, linewidth=0.02](0.45697075,0.20267501)(0.45697075,-0.21881929)
\psline[linecolor=black, linewidth=0.02](0.4562587,0.20204714)(0.8494084,0.5487268)
\psline[linecolor=black, linewidth=0.02](0.067692176,0.5459411)(0.46084186,0.19926146)
\psline[linecolor=black, linewidth=0.02](0.06372307,-0.56065685)(0.45687276,-0.21397717)
\psline[linecolor=black, linewidth=0.02](0.45066956,-0.20976295)(0.84381926,-0.5564426)
\rput[bl](-0.15,0.59){$\scriptstyle n-2$}
\rput[bl](0.52679443,-0.08153543){$\scriptstyle k$}
\rput[bl](0.832415,0.61202705){$\scriptstyle 1$}
\rput[bl](-0.15,-0.8){$\scriptstyle n-2$}
\rput[bl](0.8216032,-0.792027){$\scriptstyle 1$}
\end{pspicture}
}
\end{gathered}
\ \ = \ \ 
\frac{[n-2]}{\theta(n-2,1,n-3)}
\begin{gathered}
\psscalebox{1.0 1.0} % Change this value to rescale the drawing.
{
\begin{pspicture}(-0.15,-0.792027)(0.94241494,0.792027)
\psline[linecolor=black, linewidth=0.02](0.45697075,0.20267501)(0.45697075,-0.21881929)
\psline[linecolor=black, linewidth=0.02](0.4562587,0.20204714)(0.8494084,0.5487268)
\psline[linecolor=black, linewidth=0.02](0.067692176,0.5459411)(0.46084186,0.19926146)
\psline[linecolor=black, linewidth=0.02](0.06372307,-0.56065685)(0.45687276,-0.21397717)
\psline[linecolor=black, linewidth=0.02](0.45066956,-0.20976295)(0.84381926,-0.5564426)
\rput[bl](-0.15,0.59){$\scriptstyle n-2$}
\rput[bl](0.52679443,-0.08153543){$\scriptstyle n-3$}
\rput[bl](0.832415,0.61202705){$\scriptstyle 1$}
\rput[bl](-0.15,-0.8){$\scriptstyle n-2$}
\rput[bl](0.8216032,-0.792027){$\scriptstyle 1$}
\end{pspicture}
}
\end{gathered}
\quad + \ \ 
\frac{[n]}{\theta(n-2,1,n-1)}
\begin{gathered}
\psscalebox{1.0 1.0} % Change this value to rescale the drawing.
{
\begin{pspicture}(-0.15,-0.792027)(0.94241494,0.792027)
\psline[linecolor=black, linewidth=0.02](0.45697075,0.20267501)(0.45697075,-0.21881929)
\psline[linecolor=black, linewidth=0.02](0.4562587,0.20204714)(0.8494084,0.5487268)
\psline[linecolor=black, linewidth=0.02](0.067692176,0.5459411)(0.46084186,0.19926146)
\psline[linecolor=black, linewidth=0.02](0.06372307,-0.56065685)(0.45687276,-0.21397717)
\psline[linecolor=black, linewidth=0.02](0.45066956,-0.20976295)(0.84381926,-0.5564426)
\rput[bl](-0.15,0.59){$\scriptstyle n-2$}
\rput[bl](0.52679443,-0.08153543){$\scriptstyle n-1$}
\rput[bl](0.832415,0.61202705){$\scriptstyle 1$}
\rput[bl](-0.15,-0.8){$\scriptstyle n-2$}
\rput[bl](0.8216032,-0.792027){$\scriptstyle 1$}
\end{pspicture}
}
\end{gathered}
\end{equation*}
and consider the second term to be zero since $p_{n-1}=0$.
(Of course, the coefficient $\frac{[n]}{\theta(n-2,1,n-1)}=\frac{[n]}{[n]}$ is, properly speaking, not well-defined, but we will agree that it is $1$ since the numerator and denominator are equal.)
We use this convention to avoid having to break into cases depending on whether $a<n-2$ or $a=n-2$.

Next let $2 \leq b \leq a$, and assume for induction that \eqref{eq:truncated-tensorprod-identity} holds for $p_a \otimes p_c$ for all $1 \leq c \leq b-1$.
Then we have the following cases for $b$.

\noindent Case 1: $a+b<n-1$. Then $a\nplus b = a+b$, so the upper limit of summation is the same as in the non-truncated case, and one checks that the proof for the unevaluated case also holds here, with no modifications necessary.

\noindent Case 2. $a+b = n-1$, so $a \nplus b = n-3 = a+b-2$.
Then using the convention described earlier and following the calculation of the non-truncated proof we get that
\begin{multline*}
p_a \otimes p_b
\ \ = \ \ 
\lambda(a,b,a-b)
\begin{gathered}
\psscalebox{1.0 1.0} % Change this value to rescale the drawing.
{
\begin{pspicture}(0,-0.792027)(0.94241494,0.792027)
\psline[linecolor=black, linewidth=0.02](0.45697075,0.20267501)(0.45697075,-0.21881929)
\psline[linecolor=black, linewidth=0.02](0.4562587,0.20204714)(0.8494084,0.5487268)
\psline[linecolor=black, linewidth=0.02](0.067692176,0.5459411)(0.46084186,0.19926146)
\psline[linecolor=black, linewidth=0.02](0.06372307,-0.56065685)(0.45687276,-0.21397717)
\psline[linecolor=black, linewidth=0.02](0.45066956,-0.20976295)(0.84381926,-0.5564426)
\rput[bl](0.010811806,0.6225269){$\scriptstyle a$}
\rput[bl](0.832415,0.61202705){$\scriptstyle b$}
\rput[bl](0.52679443,-0.08153543){$\scriptstyle a-b$}
\rput[bl](0.0,-0.78152716){$\scriptstyle a$}
\rput[bl](0.8216032,-0.792027){$\scriptstyle b$}
\end{pspicture}
}
\end{gathered} \\
+
\sumtwo{k^\prime = a-b+2}{a+b-2} \lambda(a,b,k^\prime)
\begin{gathered}
\psscalebox{1.0 1.0} % Change this value to rescale the drawing.
{
\begin{pspicture}(0,-0.792027)(0.94241494,0.792027)
\psline[linecolor=black, linewidth=0.02](0.45697075,0.20267501)(0.45697075,-0.21881929)
\psline[linecolor=black, linewidth=0.02](0.4562587,0.20204714)(0.8494084,0.5487268)
\psline[linecolor=black, linewidth=0.02](0.067692176,0.5459411)(0.46084186,0.19926146)
\psline[linecolor=black, linewidth=0.02](0.06372307,-0.56065685)(0.45687276,-0.21397717)
\psline[linecolor=black, linewidth=0.02](0.45066956,-0.20976295)(0.84381926,-0.5564426)
\rput[bl](0.010811806,0.6225269){$\scriptstyle a$}
\rput[bl](0.832415,0.61202705){$\scriptstyle b$}
\rput[bl](0.52679443,-0.08153543){$\scriptstyle k^\prime$}
\rput[bl](0.0,-0.78152716){$\scriptstyle a$}
\rput[bl](0.8216032,-0.792027){$\scriptstyle b$}
\end{pspicture}
}
\end{gathered}
\ \ + \ \ 
\lambda(a,b-1,a+b-1)
\begin{gathered}
\psscalebox{1.0 1.0} % Change this value to rescale the drawing.
{
\begin{pspicture}(0,-0.792027)(0.94241494,0.792027)
\psline[linecolor=black, linewidth=0.02](0.45697075,0.20267501)(0.45697075,-0.21881929)
\psline[linecolor=black, linewidth=0.02](0.4562587,0.20204714)(0.8494084,0.5487268)
\psline[linecolor=black, linewidth=0.02](0.067692176,0.5459411)(0.46084186,0.19926146)
\psline[linecolor=black, linewidth=0.02](0.06372307,-0.56065685)(0.45687276,-0.21397717)
\psline[linecolor=black, linewidth=0.02](0.45066956,-0.20976295)(0.84381926,-0.5564426)
\rput[bl](0.010811806,0.6225269){$\scriptstyle a$}
\rput[bl](0.832415,0.61202705){$\scriptstyle b$}
\rput[bl](0.52679443,-0.08153543){$\scriptstyle a+b$}
\rput[bl](0.0,-0.78152716){$\scriptstyle a$}
\rput[bl](0.8216032,-0.792027){$\scriptstyle b$}
\end{pspicture}
}
\end{gathered} \ ,
\end{multline*}
where as before we write $\lambda(i,j,k) = \frac{[k+1]}{\theta(i,j,k)}$.
The last term in the above expansion comes from our convention of writing $p_{n-2} \otimes p_1$ and is negligible (it involves $p_{a+b}=p_{n-1}$) and hence zero.
Hence we have that
\begin{align*}
p_a \otimes p_b
\ \ &= \ \ 
\sumtwo{k = a-b}{a\nplus b} \lambda(a,b,k)
\begin{gathered}
\psscalebox{1.0 1.0} % Change this value to rescale the drawing.
{
\begin{pspicture}(0,-0.792027)(0.94241494,0.792027)
\psline[linecolor=black, linewidth=0.02](0.45697075,0.20267501)(0.45697075,-0.21881929)
\psline[linecolor=black, linewidth=0.02](0.4562587,0.20204714)(0.8494084,0.5487268)
\psline[linecolor=black, linewidth=0.02](0.067692176,0.5459411)(0.46084186,0.19926146)
\psline[linecolor=black, linewidth=0.02](0.06372307,-0.56065685)(0.45687276,-0.21397717)
\psline[linecolor=black, linewidth=0.02](0.45066956,-0.20976295)(0.84381926,-0.5564426)
\rput[bl](0.010811806,0.6225269){$\scriptstyle a$}
\rput[bl](0.832415,0.61202705){$\scriptstyle b$}
\rput[bl](0.52679443,-0.08153543){$\scriptstyle k$}
\rput[bl](0.0,-0.78152716){$\scriptstyle a$}
\rput[bl](0.8216032,-0.792027){$\scriptstyle b$}
\end{pspicture}
}
\end{gathered} \ .
\end{align*}

\begin{remark*}
In the calculation above we had to use the triangle-shrinking Lemma \ref{lemma:shrink_triangle_k+1} on a diagram involving the negligible Jones-Wenzl idempotent $p_{n-1}$, however this is not a problem since the lemma simply rewrites one negligible morphism as another.
(That is, zero morphisms stay zero and nonzero morphisms stay nonzero under the lemma.)
\end{remark*}

\noindent Case 3. $a+b \geq n$, hence $a \nplus b = 2n-4-(a+b)$.
The calculation is the same as for the unevaluated case, except we have to change the upper limit of summation to $a \nplus (b-1) = 2n-4-(a+b-1)$, obtaining
\begin{align*}
p_a \otimes p_b
\ \ &= \ \ 
\sumtwo{k=a-b+1}{2n-4-(a+b-1)} \lambda(a,b-1,k)
\left( \frac{[k-n_k]^2}{[k][k+1]}
\begin{gathered}
\psscalebox{1.0 1.0} % Change this value to rescale the drawing.
{
\begin{pspicture}(0,-0.792027)(0.94241494,0.792027)
\psline[linecolor=black, linewidth=0.02](0.45697075,0.20267501)(0.45697075,-0.21881929)
\psline[linecolor=black, linewidth=0.02](0.4562587,0.20204714)(0.8494084,0.5487268)
\psline[linecolor=black, linewidth=0.02](0.067692176,0.5459411)(0.46084186,0.19926146)
\psline[linecolor=black, linewidth=0.02](0.06372307,-0.56065685)(0.45687276,-0.21397717)
\psline[linecolor=black, linewidth=0.02](0.45066956,-0.20976295)(0.84381926,-0.5564426)
\rput[bl](0.010811806,0.6225269){$\scriptstyle a$}
\rput[bl](0.832415,0.61202705){$\scriptstyle b$}
\rput[bl](0.52679443,-0.08153543){$\scriptstyle k-1$}
\rput[bl](0.0,-0.78152716){$\scriptstyle a$}
\rput[bl](0.8216032,-0.792027){$\scriptstyle b$}
\end{pspicture}
}
\end{gathered}
\quad + \ \ 
\begin{gathered}
\psscalebox{1.0 1.0} % Change this value to rescale the drawing.
{
\begin{pspicture}(0,-0.792027)(0.94241494,0.792027)
\psline[linecolor=black, linewidth=0.02](0.45697075,0.20267501)(0.45697075,-0.21881929)
\psline[linecolor=black, linewidth=0.02](0.4562587,0.20204714)(0.8494084,0.5487268)
\psline[linecolor=black, linewidth=0.02](0.067692176,0.5459411)(0.46084186,0.19926146)
\psline[linecolor=black, linewidth=0.02](0.06372307,-0.56065685)(0.45687276,-0.21397717)
\psline[linecolor=black, linewidth=0.02](0.45066956,-0.20976295)(0.84381926,-0.5564426)
\rput[bl](0.010811806,0.6225269){$\scriptstyle a$}
\rput[bl](0.832415,0.61202705){$\scriptstyle b$}
\rput[bl](0.52679443,-0.08153543){$\scriptstyle k+1$}
\rput[bl](0.0,-0.78152716){$\scriptstyle a$}
\rput[bl](0.8216032,-0.792027){$\scriptstyle b$}
\end{pspicture}
}
\end{gathered} \ 
\right)
\end{align*}
(compare equation \eqref{eq:JWP-tensorprod-identity_sumtwo_k} in Theorem \ref{thm:JWP-tensorprod-identity}).
Reindexing and simplifying as in the non-truncated proof, we get that
\begin{multline*}
p_a \otimes p_b
\ \ = \ \ 
\lambda(a,b-1,a-b+1) \frac{[a-b+1-n_{a-b+1}]^2}{[a-b+1][a-b+2]}
\begin{gathered}
\psscalebox{1.0 1.0} % Change this value to rescale the drawing.
{
\begin{pspicture}(0,-0.792027)(0.94241494,0.792027)
\psline[linecolor=black, linewidth=0.02](0.45697075,0.20267501)(0.45697075,-0.21881929)
\psline[linecolor=black, linewidth=0.02](0.4562587,0.20204714)(0.8494084,0.5487268)
\psline[linecolor=black, linewidth=0.02](0.067692176,0.5459411)(0.46084186,0.19926146)
\psline[linecolor=black, linewidth=0.02](0.06372307,-0.56065685)(0.45687276,-0.21397717)
\psline[linecolor=black, linewidth=0.02](0.45066956,-0.20976295)(0.84381926,-0.5564426)
\rput[bl](0.010811806,0.6225269){$\scriptstyle a$}
\rput[bl](0.832415,0.61202705){$\scriptstyle b$}
\rput[bl](0.52679443,-0.08153543){$\scriptstyle a-b$}
\rput[bl](0.0,-0.78152716){$\scriptstyle a$}
\rput[bl](0.8216032,-0.792027){$\scriptstyle b$}
\end{pspicture}
}
\end{gathered} \\
+ \ \ 
\sumtwo{k^\prime = a-b+2}{2n-4-(a+b)} \left[ \left( \lambda(a,b-1,k^\prime-1) + \lambda(a,b-1,k^\prime+1)\frac{[k^\prime+1-n_{k^\prime+1}]^2}{[k^\prime+1][k^\prime+2]} \right)
\begin{gathered}
\psscalebox{1.0 1.0} % Change this value to rescale the drawing.
{
\begin{pspicture}(0,-0.792027)(0.94241494,0.792027)
\psline[linecolor=black, linewidth=0.02](0.45697075,0.20267501)(0.45697075,-0.21881929)
\psline[linecolor=black, linewidth=0.02](0.4562587,0.20204714)(0.8494084,0.5487268)
\psline[linecolor=black, linewidth=0.02](0.067692176,0.5459411)(0.46084186,0.19926146)
\psline[linecolor=black, linewidth=0.02](0.06372307,-0.56065685)(0.45687276,-0.21397717)
\psline[linecolor=black, linewidth=0.02](0.45066956,-0.20976295)(0.84381926,-0.5564426)
\rput[bl](0.010811806,0.6225269){$\scriptstyle a$}
\rput[bl](0.832415,0.61202705){$\scriptstyle b$}
\rput[bl](0.52679443,-0.08153543){$\scriptstyle k^\prime$}
\rput[bl](0.0,-0.78152716){$\scriptstyle a$}
\rput[bl](0.8216032,-0.792027){$\scriptstyle b$}
\end{pspicture}
}
\end{gathered} \ \right] \\
+ \ \ \lambda(a,b-1,2n-4-(a+b-1))
\begin{gathered}
\psscalebox{1.0 1.0} % Change this value to rescale the drawing.
{
\begin{pspicture}(0,-0.792027)(2.1,0.792027)
\psline[linecolor=black, linewidth=0.02](0.45697075,0.20267501)(0.45697075,-0.21881929)
\psline[linecolor=black, linewidth=0.02](0.4562587,0.20204714)(0.8494084,0.5487268)
\psline[linecolor=black, linewidth=0.02](0.067692176,0.5459411)(0.46084186,0.19926146)
\psline[linecolor=black, linewidth=0.02](0.06372307,-0.56065685)(0.45687276,-0.21397717)
\psline[linecolor=black, linewidth=0.02](0.45066956,-0.20976295)(0.84381926,-0.5564426)
\rput[bl](0.010811806,0.6225269){$\scriptstyle a$}
\rput[bl](0.832415,0.61202705){$\scriptstyle b$}
\rput[bl](0.52679443,-0.15){$\scriptstyle 2n-2-(a+b)$}
\rput[bl](0.0,-0.78152716){$\scriptstyle a$}
\rput[bl](0.8216032,-0.792027){$\scriptstyle b$}
\end{pspicture}
}
\end{gathered} \ .
\end{multline*}
Now $k^\prime \leq 2n-4-(a+b) \leq 2(a+b)-4-(a+b) = a+b-4$, and by the non-truncated proof we know that for all $k^\prime \in \{a-b, a-b+2, \dotsc, a+b-4 \}$, the coefficients of the terms indexed by $k^\prime$ in the above expansion are the coefficients $\lambda(a,b,k^\prime) = \frac{[k^\prime+1]}{\theta(a,b,k^\prime)}$ in the tensor product identity.
Furthermore the last term again arises due to us writing $p_{n-2} \otimes p_1$ as a sum involving a negligible morphism, it is itself negligible since $a+b+(2n-2-(a+b))=2n-2$, and hence is equal to zero.
Thus
\begin{align*}
p_a \otimes p_b
\ \ &= \ \ 
\sumtwo{k=a-b}{2n-4-(a+b)} \frac{[k+1]}{\theta(a,b,k)}
\begin{gathered}
\psscalebox{1.0 1.0} % Change this value to rescale the drawing.
{
\begin{pspicture}(0,-0.792027)(0.94241494,0.792027)
\psline[linecolor=black, linewidth=0.02](0.45697075,0.20267501)(0.45697075,-0.21881929)
\psline[linecolor=black, linewidth=0.02](0.4562587,0.20204714)(0.8494084,0.5487268)
\psline[linecolor=black, linewidth=0.02](0.067692176,0.5459411)(0.46084186,0.19926146)
\psline[linecolor=black, linewidth=0.02](0.06372307,-0.56065685)(0.45687276,-0.21397717)
\psline[linecolor=black, linewidth=0.02](0.45066956,-0.20976295)(0.84381926,-0.5564426)
\rput[bl](0.010811806,0.6225269){$\scriptstyle a$}
\rput[bl](0.832415,0.61202705){$\scriptstyle b$}
\rput[bl](0.52679443,-0.08153543){$\scriptstyle k$}
\rput[bl](0.0,-0.78152716){$\scriptstyle a$}
\rput[bl](0.8216032,-0.792027){$\scriptstyle b$}
\end{pspicture}
}
\end{gathered} \\
&= \ \ 
\sumtwo{k=a-b}{a \nplus b} \frac{[k+1]}{\theta(a,b,k)}
\begin{gathered}
\psscalebox{1.0 1.0} % Change this value to rescale the drawing.
{
\begin{pspicture}(0,-0.792027)(0.94241494,0.792027)
\psline[linecolor=black, linewidth=0.02](0.45697075,0.20267501)(0.45697075,-0.21881929)
\psline[linecolor=black, linewidth=0.02](0.4562587,0.20204714)(0.8494084,0.5487268)
\psline[linecolor=black, linewidth=0.02](0.067692176,0.5459411)(0.46084186,0.19926146)
\psline[linecolor=black, linewidth=0.02](0.06372307,-0.56065685)(0.45687276,-0.21397717)
\psline[linecolor=black, linewidth=0.02](0.45066956,-0.20976295)(0.84381926,-0.5564426)
\rput[bl](0.010811806,0.6225269){$\scriptstyle a$}
\rput[bl](0.832415,0.61202705){$\scriptstyle b$}
\rput[bl](0.52679443,-0.08153543){$\scriptstyle k$}
\rput[bl](0.0,-0.78152716){$\scriptstyle a$}
\rput[bl](0.8216032,-0.792027){$\scriptstyle b$}
\end{pspicture}
}
\end{gathered} \ .
\end{align*}

Hence by induction the truncated identity holds for all $b \leq a$.
Finally, reflecting all diagrams about the vertical axis gives the result for all $a,b \leq n-2$.
\end{proof}

As in generic $\TLJ$, we have maps
\begin{gather*}
\varphi \colon p_a \otimes p_b \longrightarrow p_{\abs{a-b}} \oplus p_{\abs{a-b}+2} \oplus \dotsb \oplus p_{a\nplus b}, \\
\psi \colon p_{\abs{a-b}} \oplus p_{\abs{a-b}+2} \oplus \dotsb \oplus p_{a\nplus b} \longrightarrow p_a \otimes p_b,
\end{gather*}
given by
\begin{gather*}
\varphi = \begin{bmatrix}
\begin{gathered}
\psscalebox{1.0 1.0} % Change this value to rescale the drawing.
{
\begin{pspicture}(0,-0.5435213)(1.1242857,0.5435213)
\psline[linecolor=black, linewidth=0.02](0.5709957,0.29500684)(0.5709957,-0.15044771)
\psline[linecolor=black, linewidth=0.02](0.56473863,-0.14758837)(0.9296337,-0.4030906)
\psline[linecolor=black, linewidth=0.02](0.5772527,-0.14282647)(0.2123577,-0.3983287)
\rput[bl](0.27593985,0.2935213){$\scriptstyle \abs{a-b}$}
\rput[bl](0.0,-0.5339975){$\scriptstyle a$}
\rput[bl](1.0142857,-0.5435213){$\scriptstyle b$}
\end{pspicture}
}
\end{gathered} \rule{0pt}{5ex} \ \\
\begin{gathered}
\psscalebox{1.0 1.0} % Change this value to rescale the drawing.
{
\begin{pspicture}(0,-0.5575564)(1.1242857,0.5575564)
\psline[linecolor=black, linewidth=0.02](0.5709957,0.28097174)(0.5709957,-0.1644828)
\psline[linecolor=black, linewidth=0.02](0.56473863,-0.16162346)(0.9296337,-0.4171257)
\psline[linecolor=black, linewidth=0.02](0.5772527,-0.15686156)(0.2123577,-0.4123638)
\rput[bl](0.09348371,0.3075564){$\scriptstyle\abs{a-b}+2$}
\rput[bl](0.0,-0.5480326){$\scriptstyle a$}
\rput[bl](1.0142857,-0.5575564){$\scriptstyle b$}
\end{pspicture}
}
\end{gathered} \rule[-5ex]{0pt}{5ex} \rule{0pt}{6.5ex} \ \\
\vdots \\
\begin{gathered}
\psscalebox{1.0 1.0} % Change this value to rescale the drawing.
{
\begin{pspicture}(0,-0.5676441)(1.1242857,0.5676441)
\psline[linecolor=black, linewidth=0.02](0.5709957,0.27088404)(0.5709957,-0.17457052)
\psline[linecolor=black, linewidth=0.02](0.56473863,-0.17171118)(0.9296337,-0.4272134)
\psline[linecolor=black, linewidth=0.02](0.5772527,-0.16694927)(0.2123577,-0.4224515)
\rput[bl](0.33208022,0.3676441){$\scriptstyle a\nplus b$}
\rput[bl](0.0,-0.5581203){$\scriptstyle a$}
\rput[bl](1.0142857,-0.5676441){$\scriptstyle b$}
\end{pspicture}
}
\end{gathered} \ \rule{0pt}{6.5ex}
\end{bmatrix} \\
\intertext{and}
\psi = \begin{bmatrix}
\displaystyle \frac{\left[\abs{a-b}+1\right]}{\theta(a,b,\abs{a-b})}
\begin{gathered}
\psscalebox{1.0 1.0} % Change this value to rescale the drawing.
{
\begin{pspicture}(0,-0.4926521)(1.2255411,0.4926521)
\psline[linecolor=black, linewidth=0.02](0.5276145,-0.43496695)(0.5276145,0.010487583)
\psline[linecolor=black, linewidth=0.02](0.5213575,0.007628242)(0.88625246,0.2631305)
\psline[linecolor=black, linewidth=0.02](0.53387153,0.0028663373)(0.16897652,0.25836858)
\rput[bl](0.62554115,-0.4926521){$\scriptstyle \abs{a-b}$}
\rput[bl](0.0,0.32832763){$\scriptstyle a$}
\rput[bl](0.9618136,0.31265208){$\scriptstyle b$}
\end{pspicture}
}
\end{gathered} \quad ,
& \cdots \ \ ,
& \displaystyle \frac{\left[(a\nplus b)+1\right]}{\theta(a,b,a\nplus b)}
\begin{gathered}
\psscalebox{1.0 1.0} % Change this value to rescale the drawing.
{
\begin{pspicture}(0,-0.4926521)(1.4255411,0.4926521)
\psline[linecolor=black, linewidth=0.02](0.5276145,-0.43496695)(0.5276145,0.010487583)
\psline[linecolor=black, linewidth=0.02](0.5213575,0.007628242)(0.88625246,0.2631305)
\psline[linecolor=black, linewidth=0.02](0.53387153,0.0028663373)(0.16897652,0.25836858)
\rput[bl](0.62554115,-0.4926521){$\scriptstyle a\nplus b$}
\rput[bl](0.0,0.32832763){$\scriptstyle a$}
\rput[bl](0.9618136,0.31265208){$\scriptstyle b$}
\end{pspicture}
}
\end{gathered}
\end{bmatrix}
\end{gather*}
for all objects $p_a \otimes p_b$ in $\TLJ(q)$, and the proof that these give isomorphisms
\[ p_a \otimes p_b \cong p_{\abs{a-b}} \oplus p_{\abs{a-b}+2} \oplus \dotsb \oplus p_{a\nplus b} \]
is exactly the same as for Lemma \ref{lemma:JWP-tensor-isomorphism}.
Accordingly, the proof of the following theorem is the same as for the generic case.

\begin{theorem}[Semisimplicity of $\TLJ(q)$]
Temperley-Lieb-Jones at a root of unity $q=e^{\pi i/n}$ is semisimple, with simple objects the Jones-Wenzl idempotents $p_i$ for $0 \leq i \leq n-2$.
\end{theorem}

\begin{remark}
$\TLJ(q)$ at a root of unity $q$ is in fact equivalent as a braided spherical tensor category to the category $\mathcal{R}ep\,U_q(\mathfrak{sl}_2(\mathbbm{C}))$ of representations of the quantum algebra $U_q(\mathfrak{sl}_2(\mathbbm{C}))$.
However a problem arises when we consider $\TLJ(q)$ and $\mathcal{R}ep\,U_q(\mathfrak{sl}_2(\mathbbm{C}))$ as \emph{pivotal} categories: for reasons we do not go into here, their standard pivotal structures do not match, hence in order to make them equivalent as pivotal categories we either change the pivotal structure or else negate the parameter $q$ in one of the categories.
That is to say, $\TLJ(q)$ and $\mathcal{R}ep\,U_{-q}(\mathfrak{sl}_2(\mathbbm{C}))$ are equivalent categories.
(See \cite{ST2008} and \cite{Tingley2010} for details.)
\end{remark}

\begin{defn}
A \textbf{spherical fusion category} is a category that is
\begin{enumerate}
\item spherical,
\item linear with finite-dimensional hom-spaces,
\item \emph{rigid} monoidal, having tensor product with duals and unit and evaluation morphisms satisfying the \emph{zigzag identities}
\begin{equation*}
\begin{gathered}
\psscalebox{1.0 1.0} % Change this value to rescale the drawing.
{
\begin{pspicture}(0,-0.89029413)(1.36,0.89029413)
\psframe[linecolor=black, linewidth=0.02, dimen=outer](1.36,0.7043372)(0.0,-0.6556628)
\psline[linecolor=black, linewidth=0.02](0.3323407,0.6950146)(0.3323407,-0.6383761)
\psbezier[linecolor=black, linewidth=0.02](0.6300458,-0.64468706)(0.65626836,0.21215205)(1.0268565,0.20896283)(1.0637808,-0.6543256)
\rput[bl](1.0051699,-0.89029413){$\scriptstyle a$}
\rput[bl](0.23164041,-0.89029413){$\scriptstyle a$}
\rput[bl](0.23164041,0.7802941){$\scriptstyle a$}
\rput[bl](0.5698757,-0.89029413){$\scriptstyle a^\ast$}
\end{pspicture}
}
\end{gathered}
\,\circ\,
\begin{gathered}
\psscalebox{1.0 1.0} % Change this value to rescale the drawing.
{
\begin{pspicture}(0,-0.8979412)(1.36,0.8979412)
\psframe[linecolor=black, linewidth=0.02, dimen=outer](1.3600001,0.64374894)(0.0,-0.716251)
\psline[linecolor=black, linewidth=0.02](1.0276594,0.63442636)(1.0276594,-0.69896436)
\psbezier[linecolor=black, linewidth=0.02](0.28164056,0.6355468)(0.28164056,-0.1644532)(0.7169347,-0.1762179)(0.7169347,0.6237821)
\rput[bl](0.98605233,0.7079412){$\scriptstyle a$}
\rput[bl](0.21399349,0.7079412){$\scriptstyle a$}
\rput[bl](0.62869936,0.7079412){$\scriptstyle a^\ast$}
\rput[bl](0.98605233,-0.8979412){$\scriptstyle a$}
\end{pspicture}
}
\end{gathered}
\ = \ 
\begin{gathered}
\psscalebox{1.0 1.0} % Change this value to rescale the drawing.
{
\begin{pspicture}(0,-0.68)(1.36,0.68)
\psframe[linecolor=black, linewidth=0.02, dimen=outer](1.36,0.68)(0.0,-0.68)
\psbezier[linecolor=black, linewidth=0.02](0.29234043,0.65644825)(0.32135603,-0.53872484)(0.5778809,-0.25374556)(0.67531914,-0.032913435)(0.77275735,0.18791868)(1.0266957,0.4569061)(1.0582979,-0.67972195)
\rput[bl](0.82,-0.54760563){$\scriptstyle a$}
\end{pspicture}
}
\end{gathered}
\ = \ 
\begin{gathered}
\psscalebox{1.0 1.0} % Change this value to rescale the drawing.
{
\begin{pspicture}(0,-0.68)(1.36,0.68)
\psframe[linecolor=black, linewidth=0.02, dimen=outer](1.3599999,0.67999995)(0.0,-0.68000007)
\psline[linecolor=black, linewidth=0.02](0.6795743,0.6699999)(0.6795743,-0.68000007)
\rput[bl](0.4407011,-0.54760563){$\scriptstyle a$}
\end{pspicture}
}
\end{gathered}
\ \ = \ \ 
\begin{gathered}
\psscalebox{1.0 1.0} % Change this value to rescale the drawing.
{
\begin{pspicture}(0,-0.68)(1.3600003,0.68)
\psframe[linecolor=black, linewidth=0.02, dimen=outer](1.3600001,0.68000007)(0.0,-0.67999995)
\psbezier[linecolor=black, linewidth=0.02](0.27106392,-0.679595)(0.31853592,0.49174958)(0.591165,0.3453503)(0.71361715,-0.058318403)(0.8360693,-0.46198708)(1.0949676,-0.3277481)(1.1136171,0.66508585)
\rput[bl](0.36748585,-0.54760563){$\scriptstyle a$}
\end{pspicture}
}
\end{gathered}
\ = \ 
\begin{gathered}
\psscalebox{1.0 1.0} % Change this value to rescale the drawing.
{
\begin{pspicture}(0,-0.87978303)(1.36,0.87978303)
\psbezier[linecolor=black, linewidth=0.02](0.31953484,-0.64073336)(0.3457574,0.21610577)(0.7163456,0.21291654)(0.7532698,-0.6503719)
\psframe[linecolor=black, linewidth=0.02, dimen=outer](1.3600001,0.70559084)(0.0,-0.65440917)
\psline[linecolor=black, linewidth=0.02](1.0276594,0.6962682)(1.0276594,-0.6371225)
\rput[bl](0.261872,-0.87978303){$\scriptstyle a$}
\rput[bl](0.63313526,-0.87978303){$\scriptstyle a^\ast$}
\rput[bl](0.98605233,0.769783){$\scriptstyle a$}
\rput[bl](0.9794949,-0.87978303){$\scriptstyle a$}
\end{pspicture}
}
\end{gathered}
\,\circ\,
\begin{gathered}
\psscalebox{1.0 1.0} % Change this value to rescale the drawing.
{
\begin{pspicture}(0,-0.9296191)(1.36,0.9296191)
\psframe[linecolor=black, linewidth=0.02, dimen=outer](1.36,0.66501224)(0.0,-0.6949878)
\psline[linecolor=black, linewidth=0.02](0.3323407,0.6556896)(0.3323407,-0.67770106)
\psbezier[linecolor=black, linewidth=0.02](0.6380531,0.65410995)(0.6380531,-0.14589003)(1.0733472,-0.15765473)(1.0733472,0.64234525)
\rput[bl](0.23164041,-0.9296191){$\scriptstyle a$}
\rput[bl](0.23164041,0.7396191){$\scriptstyle a$}
\rput[bl](1.022865,0.7396191){$\scriptstyle a$}
\rput[bl](0.6178989,0.7396191){$\scriptstyle a^\ast$}
\end{pspicture}
}
\end{gathered} \ ,
\end{equation*}
and
\item semisimple with finitely many simple objects, and the tensor identity is simple.
\end{enumerate}
\end{defn}

We have now shown that Temperley-Lieb-Jones at a root of unity is an example of such a category, and are finally ready to turn our attention to the construction of skein modules.

\chapter{Turaev-Viro skein modules for $n$-holed disks}
At the end of the last chapter we saw that $\TLJ(q)$ is a spherical fusion category.
One of the reasons these categories are interesting is that they allow us to construct $(2+1)$-dimensional topological quantum field theories (TQFTs).

One of the first TQFTs to be discovered was the Turaev-Viro TQFT \cite{TV1992}, which takes as input a spherical fusion category in order to construct free modules for $2$-surfaces and linear maps for $3$-cobordisms.
In this final chapter we deal only with the $2$-dimensional aspect of the theory and present an alternative construction of the Turaev-Viro skein modules for $n$-holed disks.

In Turaev and Viro's original construction, the skein module for a given surface $\Sigma$ is defined abstractly as a finite quotient of an infinite-dimensional vector space formed by considering all possible ``labellings'' over all triangulations of $\Sigma$.
In order to do concrete calculations one then has to invoke the (folkloric) spine lemma \cite{Mat2013}, which allows one to pass from triangulations to a spine for $\Sigma$, that is, a graph embedded in $\Sigma$ onto which $\Sigma$ deformation retracts, and this then gives us an explicit basis for the skein module.
Our approach in this chapter will be to short-circuit all this, and instead \emph{define} modules via spines, proving all the results we need about existence and uniqueness directly.
%This has the advantage of allowing us to jump straight in to calculations using our skein modules; in particular we will use them to obtain representations of the braid groups $B_n$.
However our construction is much weaker than the original, being valid for $n$-holed disks (or more generally, for compact surfaces with trivial $n$th homology for $n\geq 2$); in contrast the original Turaev-Viro construction works for any compact $2$-manifold.

%\section{Skein modules for $n$-holed disks}
For the rest of this chapter we work at a root of unity $q=e^{\pi i/n}$.

\begin{defn}
Let $\Sigma$ be a $n$-holed disk, i.e. $\Sigma = \overline{D^2} \setminus \coprod_{i=1}^n D^2_i$ is the compact smooth $2$-manifold with boundary, formed by taking the closed $2$-disk and removing $n$ disjoint copies $D_1^2, \dotsc, D_n^2$ of the open disk from its interior.
Let $M$ be a finite collection of distinguished points marked on the boundary of $\Sigma$.
Then we call $(\Sigma, M)$ a \textbf{$n$-holed disk with marked boundary}.
\end{defn}

\begin{defn}
Let $(\Sigma, M)$ be a $n$-holed disk with marked boundary.
A \textbf{spine} for $(\Sigma,M)$ is a planar graph $s$ embedded in $\Sigma$ such that
\begin{enumerate}
\item the set of vertices of degree $1$ in $s$ is precisely the set $M$ of marked points on the boundary of $\Sigma$, and
\item $\Sigma$ deformation retracts onto $s$.
\end{enumerate}
A \textbf{trivalent spine} is a spine that is further uni-trivalent, i.e. one whose vertices are all of degree $3$, except for those coinciding with the marked boundary points.
\end{defn}

Note that the marked boundary points are part of the data of spines for $n$-holed disks, so for example while the two surfaces in Figure \ref{fig:different_spines} are the same, the set of marked boundary points is different and hence a spine for one is not a spine for the other.

\begin{figure}[ht]
\begin{equation*}
\psscalebox{1.0 1.0} % Change this value to rescale the drawing.
{
\begin{pspicture}(0,-0.83111113)(3.38,0.83111113)
\definecolor{colour0}{rgb}{0.2,0.2,1.0}
\psellipse[linecolor=colour0, linewidth=0.02, dimen=outer](1.6899999,-0.021111004)(1.35,0.5539394)
\psellipse[linecolor=black, linewidth=0.02, dimen=outer](1.69,-0.02111111)(1.69,0.81)
\pscircle[linecolor=black, linewidth=0.02, dimen=outer](0.8627273,-0.02111111){0.22}
\pscircle[linecolor=black, linewidth=0.02, dimen=outer](1.69,-0.02111111){0.22}
\pscircle[linecolor=black, linewidth=0.02, dimen=outer](2.5172727,-0.02111111){0.22}
\pscircle[linecolor=colour0, linewidth=0.02, fillstyle=solid,fillcolor=colour0, dimen=outer](0.8666667,0.18222222){0.04888889}
\pscircle[linecolor=colour0, linewidth=0.02, fillstyle=solid,fillcolor=colour0, dimen=outer](1.6933334,0.18666667){0.04888889}
\pscircle[linecolor=colour0, linewidth=0.02, fillstyle=solid,fillcolor=colour0, dimen=outer](2.5244443,0.17777778){0.04888889}
\pscircle[linecolor=colour0, linewidth=0.02, fillstyle=solid,fillcolor=colour0, dimen=outer](0.74666667,0.64){0.04888889}
\pscircle[linecolor=colour0, linewidth=0.02, fillstyle=solid,fillcolor=colour0, dimen=outer](1.5822222,0.7822222){0.04888889}
\pscircle[linecolor=colour0, linewidth=0.02, fillstyle=solid,fillcolor=colour0, dimen=outer](2.431111,0.6888889){0.04888889}
\psline[linecolor=colour0, linewidth=0.02](1.6888889,0.18666667)(1.6933334,0.51555556)
\psline[linecolor=colour0, linewidth=0.02](0.74222225,0.63555557)(0.74222225,0.36444443)
\psline[linecolor=colour0, linewidth=0.02](0.8666667,0.18222222)(0.8666667,0.40444446)
\psline[linecolor=colour0, linewidth=0.02](2.5244443,0.18666667)(2.5244443,0.4)
\psline[linecolor=colour0, linewidth=0.02](2.4266667,0.69777775)(2.4266667,0.43555555)
\psline[linecolor=colour0, linewidth=0.02](1.2755556,0.4977778)(1.2755556,-0.5288889)
\psline[linecolor=colour0, linewidth=0.02](2.1288888,0.4888889)(2.1288888,-0.5377778)
\psline[linecolor=colour0, linewidth=0.02](1.5777777,0.7866667)(1.5777777,0.52)
\end{pspicture}
}
\qquad
\psscalebox{1.0 1.0} % Change this value to rescale the drawing.
{
\begin{pspicture}(0,-0.8099998)(3.3800004,0.8099998)
\definecolor{colour0}{rgb}{0.2,0.2,1.0}
\psellipse[linecolor=black, linewidth=0.02, dimen=outer](1.69,0.0)(1.69,0.81)
\pscircle[linecolor=black, linewidth=0.02, dimen=outer](0.76284236,-0.002413793){0.22}
\pscircle[linecolor=colour0, linewidth=0.02, fillstyle=solid,fillcolor=colour0, dimen=outer](0.76678175,0.19632185){0.04888889}
\pscircle[linecolor=colour0, linewidth=0.02, dimen=outer](0.7635634,0.0){0.40229884}
\psline[linecolor=colour0, linewidth=0.02](0.7600001,0.38632172)(0.7645978,0.22540219)
\pscircle[linecolor=black, linewidth=0.02, dimen=outer](1.6892792,-0.002413793){0.22}
\pscircle[linecolor=colour0, linewidth=0.02, fillstyle=solid,fillcolor=colour0, dimen=outer](1.6932186,0.19632185){0.04888889}
\pscircle[linecolor=colour0, linewidth=0.02, dimen=outer](1.6900002,0.0){0.40229884}
\psline[linecolor=colour0, linewidth=0.02](1.6887357,0.21160908)(1.6887357,0.39551714)
\pscircle[linecolor=black, linewidth=0.02, dimen=outer](2.615716,-0.002413793){0.22}
\pscircle[linecolor=colour0, linewidth=0.02, fillstyle=solid,fillcolor=colour0, dimen=outer](2.6196554,0.19632185){0.04888889}
\pscircle[linecolor=colour0, linewidth=0.02, dimen=outer](2.616437,0.0){0.40229884}
\psline[linecolor=colour0, linewidth=0.02](2.6220691,0.20701139)(2.6266668,0.39551714)
\psline[linecolor=colour0, linewidth=0.02](1.1600001,0.0047125295)(1.2887357,0.0047125295)
\psline[linecolor=colour0, linewidth=0.02](2.0887358,0.0047125295)(2.2174714,0.0047125295)
\end{pspicture}
}
\end{equation*}
\caption{Trivalent spines for different $n$-holed disks. \label{fig:different_spines}}
\end{figure}
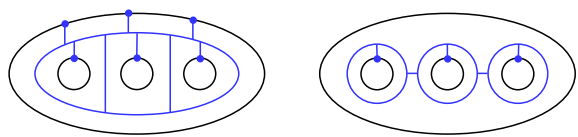

\begin{defn}
Let $s$ be a trivalent spine for a $n$-holed disk $(\Sigma,M)$.
A \textbf{coloring} of $s$ is a labelling of the edges of $s$ by simple objects $p_i$ in $\TLJ(q)$, such that at every trivalent vertex the labels of the edges adjacent to it form a $q$-admissible triple.
\end{defn}

Hence we see that a coloring of $s$ determines a particular configuration of TL diagrams drawn on the surface of $\Sigma$.
Such diagrams are in general called \emph{skein diagrams}.

\begin{defn}
The \textbf{skein module associated to the trivalent spine $s$} for the surface $(\Sigma,M)$ is the free $\mathbbm{C}$-module $C(s)$ having as a basis all colorings of $s$.
\end{defn}

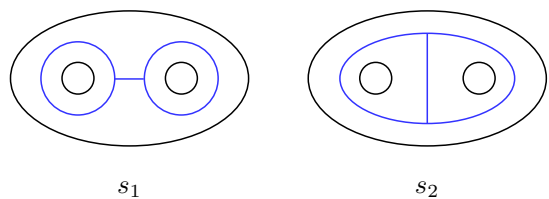
\begin{figure}
\begin{equation*}
\begin{gathered}
\psscalebox{1.0 1.0} % Change this value to rescale the drawing.
{
\begin{pspicture}(0,-0.9052379)(3.1514285,0.9052379)
\definecolor{colour0}{rgb}{0.2,0.2,1.0}
\psellipse[linecolor=black, linewidth=0.02, dimen=outer](1.5757143,0.0)(1.5757142,0.9052381)
\pscircle[linecolor=colour0, linewidth=0.02, dimen=outer](2.2557142,0.0){0.4952381}
\pscircle[linecolor=black, linewidth=0.02, dimen=outer](2.2576532,0.002461867){0.22}
\pscircle[linecolor=colour0, linewidth=0.02, dimen=outer](0.8957143,0.0){0.4952381}
\pscircle[linecolor=black, linewidth=0.02, dimen=outer](0.8976532,0.002461867){0.22}
\psline[linecolor=colour0, linewidth=0.02](1.3790946,-0.0072004185)(1.762768,-0.0072004185)
\end{pspicture}
} \\
s_1
\end{gathered}
\qquad
\begin{gathered}
\psscalebox{1.0 1.0} % Change this value to rescale the drawing.
{
\begin{pspicture}(0,-0.90523803)(3.151429,0.90523803)
\definecolor{colour0}{rgb}{0.2,0.2,1.0}
\pscircle[linecolor=black, linewidth=0.02, dimen=outer](2.2576535,0.0024617517){0.22}
\pscircle[linecolor=black, linewidth=0.02, dimen=outer](0.89765346,0.0024617517){0.22}
\psellipse[linecolor=black, linewidth=0.02, dimen=outer](1.5757146,0.0)(1.5757142,0.9052381)
\psellipse[linecolor=colour0, linewidth=0.02, dimen=outer](1.5757143,0.0)(1.1591836,0.60816324)
\psline[linecolor=colour0, linewidth=0.02](1.5757145,0.5911539)(1.5757145,-0.5925195)
\end{pspicture}
} \\
s_2
\end{gathered}
\end{equation*}
\caption{Trivalent spines for the doubly-holed disk. \label{fig:skein-module-eg-spines}}
\end{figure}

For example, let $(\overline{D^2}\setminus (D_1^2 \sqcup D_2^2),\emptyset)$ be the doubly-holed disk with empty boundary marking, and let $s_1$ and $s_2$ be the trivalent spines given in Figure \ref{fig:skein-module-eg-spines}.
Then at $q=e^{\pi i/3}$, the simple objects in $\TLJ(q)$ are $p_0$ and $p_1$, the only $q$-admissible triples are $(0,0,0)$ and $(0,1,1)$, and the skein modules associated to $s_1$ and $s_2$ are given by
\begin{equation*}
C(s_1) = \spn
\left\{
\begin{aligned}
& \begin{gathered}
\psscalebox{1.0 1.0} % Change this value to rescale the drawing.
{
\begin{pspicture}(0,-0.77607137)(2.6192858,0.77607137)
\definecolor{colour0}{rgb}{0.8,0.8,0.8}
\definecolor{colour1}{rgb}{0.2,0.2,1.0}
\psellipse[linecolor=colour0, linewidth=0.02, dimen=outer](1.3096428,0.0)(1.3096429,0.7760714)
\pscircle[linecolor=colour0, linewidth=0.02, dimen=outer](0.7482087,0.0){0.17416666}
\pscircle[linecolor=colour0, linewidth=0.02, dimen=outer](1.871077,0.0){0.17416666}
\pscircle[linecolor=colour0, linewidth=0.02, dimen=outer](0.7482087,0.0){0.17416666}
\pscircle[linecolor=colour0, linewidth=0.02, dimen=outer](1.871077,0.0){0.17416666}
\pscircle[linecolor=colour1, linewidth=0.02, dimen=outer](0.75115573,0.0031571442){0.37575758}
\pscircle[linecolor=colour1, linewidth=0.02, dimen=outer](1.868489,0.0031571442){0.37575758}
\psline[linecolor=colour1, linewidth=0.02](1.1150345,-0.0035903307)(1.5017011,-0.0035903307)
\rput[bl](1.250792,0.058434334){$\scriptstyle 0$}
\rput[bl](1.8176761,0.43892843){$\scriptstyle 0$}
\rput[bl](0.6843427,0.43892843){$\scriptstyle 0$}
\end{pspicture}
}
\end{gathered} \ ,
\quad
\begin{gathered}
\psscalebox{1.0 1.0} % Change this value to rescale the drawing.
{
\begin{pspicture}(0,-0.7760715)(2.6192856,0.7760715)
\definecolor{colour0}{rgb}{0.8,0.8,0.8}
\definecolor{colour1}{rgb}{0.2,0.2,1.0}
\psellipse[linecolor=colour0, linewidth=0.02, dimen=outer](1.3096429,0.0)(1.3096429,0.7760714)
\pscircle[linecolor=colour0, linewidth=0.02, dimen=outer](0.7482088,0.0){0.17416666}
\pscircle[linecolor=colour0, linewidth=0.02, dimen=outer](1.8710771,0.0){0.17416666}
\pscircle[linecolor=colour0, linewidth=0.02, dimen=outer](0.7482088,0.0){0.17416666}
\pscircle[linecolor=colour0, linewidth=0.02, dimen=outer](1.8710771,0.0){0.17416666}
\pscircle[linecolor=colour1, linewidth=0.02, dimen=outer](0.75115585,0.003157353){0.37575758}
\pscircle[linecolor=colour1, linewidth=0.02, dimen=outer](1.8684891,0.003157353){0.37575758}
\psline[linecolor=colour1, linewidth=0.02](1.1150346,-0.0035901219)(1.5017012,-0.0035901219)
\rput[bl](1.2507921,0.058434542){$\scriptstyle 0$}
\rput[bl](1.8287872,0.4167064){$\scriptstyle 1$}
\rput[bl](0.6954539,0.4167064){$\scriptstyle 1$}
\end{pspicture}
}
\end{gathered} \ , \\
& \begin{gathered}
\psscalebox{1.0 1.0} % Change this value to rescale the drawing.
{
\begin{pspicture}(0,-0.7760715)(2.6192858,0.7760715)
\definecolor{colour0}{rgb}{0.8,0.8,0.8}
\definecolor{colour1}{rgb}{0.2,0.2,1.0}
\psellipse[linecolor=colour0, linewidth=0.02, dimen=outer](1.3096428,0.0)(1.3096429,0.7760714)
\pscircle[linecolor=colour0, linewidth=0.02, dimen=outer](0.7482087,0.0){0.17416666}
\pscircle[linecolor=colour0, linewidth=0.02, dimen=outer](1.871077,0.0){0.17416666}
\pscircle[linecolor=colour0, linewidth=0.02, dimen=outer](0.7482087,0.0){0.17416666}
\pscircle[linecolor=colour0, linewidth=0.02, dimen=outer](1.871077,0.0){0.17416666}
\pscircle[linecolor=colour1, linewidth=0.02, dimen=outer](0.75115573,0.003157353){0.37575758}
\pscircle[linecolor=colour1, linewidth=0.02, dimen=outer](1.868489,0.003157353){0.37575758}
\psline[linecolor=colour1, linewidth=0.02](1.1150345,-0.0035901219)(1.5017012,-0.0035901219)
\rput[bl](1.2507921,0.058434542){$\scriptstyle 0$}
\rput[bl](1.8398982,0.4217064){$\scriptstyle 1$}
\rput[bl](0.7065649,0.4167064){$\scriptstyle 0$}
\end{pspicture}
}
\end{gathered} \ ,
\quad
\begin{gathered}
\psscalebox{1.0 1.0} % Change this value to rescale the drawing.
{
\begin{pspicture}(0,-0.7760715)(2.6192856,0.7760715)
\definecolor{colour0}{rgb}{0.8,0.8,0.8}
\definecolor{colour1}{rgb}{0.2,0.2,1.0}
\psellipse[linecolor=colour0, linewidth=0.02, dimen=outer](1.3096429,0.0)(1.3096429,0.7760714)
\pscircle[linecolor=colour0, linewidth=0.02, dimen=outer](0.7482089,0.0){0.17416666}
\pscircle[linecolor=colour0, linewidth=0.02, dimen=outer](1.8710771,0.0){0.17416666}
\pscircle[linecolor=colour0, linewidth=0.02, dimen=outer](0.7482089,0.0){0.17416666}
\pscircle[linecolor=colour0, linewidth=0.02, dimen=outer](1.8710771,0.0){0.17416666}
\pscircle[linecolor=colour1, linewidth=0.02, dimen=outer](0.7511559,0.0031570138){0.37575758}
\pscircle[linecolor=colour1, linewidth=0.02, dimen=outer](1.8684893,0.0031570138){0.37575758}
\psline[linecolor=colour1, linewidth=0.02](1.1150347,-0.0035904609)(1.5017014,-0.0035904609)
\rput[bl](1.2507923,0.058434203){$\scriptstyle 0$}
\rput[bl](1.8287873,0.41920608){$\scriptstyle 0$}
\rput[bl](0.695454,0.42420608){$\scriptstyle 1$}
\end{pspicture}
}
\end{gathered}
\end{aligned}
\right\}
\end{equation*}
and
\begin{equation*}
C(s_2) = \spn
\left\{
\begin{aligned}
& \begin{gathered}
\psscalebox{1.0 1.0} % Change this value to rescale the drawing.
{
\begin{pspicture}(0,-0.77607137)(2.6192858,0.77607137)
\definecolor{colour0}{rgb}{0.8,0.8,0.8}
\definecolor{colour1}{rgb}{0.2,0.2,1.0}
\psellipse[linecolor=colour0, linewidth=0.02, dimen=outer](1.3096429,0.0)(1.3096429,0.7760714)
\psline[linecolor=colour1, linewidth=0.02](1.311358,0.5323801)(1.311358,-0.53360957)
\psellipse[linecolor=colour1, linewidth=0.02, dimen=outer](1.309643,0.0)(1.0103464,0.54769814)
\pscircle[linecolor=colour0, linewidth=0.02, dimen=outer](0.7482088,0.0){0.17416666}
\pscircle[linecolor=colour0, linewidth=0.02, dimen=outer](1.8710771,0.0){0.17416666}
\rput[bl](2.3647316,-0.085){$\scriptstyle 0$}
\rput[bl](0.10654979,-0.085){$\scriptstyle 0$}
\rput[bl](1.099277,-0.085){$\scriptstyle 0$}
\end{pspicture}
}
\end{gathered} \ ,
\quad
\begin{gathered}
\psscalebox{1.0 1.0} % Change this value to rescale the drawing.
{
\begin{pspicture}(0,-0.7760713)(2.6192858,0.7760713)
\definecolor{colour0}{rgb}{0.8,0.8,0.8}
\definecolor{colour1}{rgb}{0.2,0.2,1.0}
\rput[bl](2.3792098,-0.081273645){$\scriptstyle 1$}
\rput[bl](0.14202273,-0.081273645){$\scriptstyle 1$}
\rput[bl](1.0992769,-0.08627365){$\scriptstyle 0$}
\pscircle[linecolor=colour0, linewidth=0.02, dimen=outer](0.74820864,0.0){0.17416666}
\pscircle[linecolor=colour0, linewidth=0.02, dimen=outer](1.871077,0.0){0.17416666}
\psline[linecolor=colour1, linewidth=0.02](1.3113579,0.53238004)(1.3113579,-0.5336096)
\psellipse[linecolor=colour1, linewidth=0.02, dimen=outer](1.3096429,0.0)(1.0103464,0.54769814)
\psellipse[linecolor=colour0, linewidth=0.02, dimen=outer](1.3096428,0.0)(1.3096429,0.7760714)
\end{pspicture}
}
\end{gathered} \ , \\
& \begin{gathered}
\psscalebox{1.0 1.0} % Change this value to rescale the drawing.
{
\begin{pspicture}(0,-0.7760715)(2.6192858,0.7760715)
\definecolor{colour0}{rgb}{0.8,0.8,0.8}
\definecolor{colour1}{rgb}{0.2,0.2,1.0}
\psellipse[linecolor=colour0, linewidth=0.02, dimen=outer](1.3096429,0.0)(1.3096429,0.7760714)
\psline[linecolor=colour1, linewidth=0.02](1.311358,0.53238016)(1.311358,-0.5336095)
\psellipse[linecolor=colour1, linewidth=0.02, dimen=outer](1.309643,0.0)(1.0103464,0.54769814)
\pscircle[linecolor=colour0, linewidth=0.02, dimen=outer](0.7482088,0.0){0.17416666}
\pscircle[linecolor=colour0, linewidth=0.02, dimen=outer](1.8710771,0.0){0.17416666}
\rput[bl](1.128368,-0.07900079){$\scriptstyle 1$}
\rput[bl](0.11293198,-0.084000796){$\scriptstyle 0$}
\rput[bl](2.379826,-0.07900079){$\scriptstyle 1$}
\end{pspicture}
}
\end{gathered} \ ,
\quad
\begin{gathered}
\psscalebox{1.0 1.0} % Change this value to rescale the drawing.
{
\begin{pspicture}(0,-0.7760715)(2.6192858,0.7760715)
\definecolor{colour0}{rgb}{0.8,0.8,0.8}
\definecolor{colour1}{rgb}{0.2,0.2,1.0}
\rput[bl](2.3720045,-0.08636352){$\scriptstyle 0$}
\rput[bl](0.11891365,-0.08136352){$\scriptstyle 1$}
\rput[bl](1.1283681,-0.08136352){$\scriptstyle 1$}
\pscircle[linecolor=colour0, linewidth=0.02, dimen=outer](0.74820906,0.0){0.17416666}
\pscircle[linecolor=colour0, linewidth=0.02, dimen=outer](1.8710773,0.0){0.17416666}
\psline[linecolor=colour1, linewidth=0.02](1.3113582,0.5323802)(1.3113582,-0.53360945)
\psellipse[linecolor=colour1, linewidth=0.02, dimen=outer](1.3096433,0.0)(1.0103464,0.54769814)
\psellipse[linecolor=colour0, linewidth=0.02, dimen=outer](1.3096431,0.0)(1.3096429,0.7760714)
\end{pspicture}
}
\end{gathered}
\end{aligned}
\right\}
\end{equation*}
where for simplicity we write $i$ instead of $p_i$ to indicate the values of the edge labels.

Observe that $C(s_1)$ and $C(s_2)$ are both $4$-dimensional and are thus isomorphic (although not naturally).
This is true in general: given any two trivalent spines $s_1$ and $s_2$ for a $n$-holed disk $(\Sigma,M)$ the skein modules $C(s_1), C(s_2)$ associated to the spines are always isomorphic, as we will soon show.
W will rely on the following two results, for proofs of which see Theorem 2 and Proposition 9 of \cite{KL1994}.

\begin{theorem}[Recoupling theorem, Kauffman]
\label{thm:recoupling}
Let $(a,b,j)$ and $(c,d,j)$ be $q$-admissible triples.
Then there exist unique $r_i \in \mathbbm{R}$ such that
\begin{equation*}
\begin{gathered}
\psscalebox{1.0 1.0} % Change this value to rescale the drawing.
{
\begin{pspicture}(0,-0.6988889)(2.058889,0.6988889)
\psline[linecolor=black, linewidth=0.02](0.5274074,-0.020864246)(1.5274074,-0.020864246)
\psline[linecolor=black, linewidth=0.02](0.5422222,-0.029012367)(0.17185186,0.48950616)
\psline[linecolor=black, linewidth=0.02](0.53481483,-0.021604959)(0.16444445,-0.54012346)
\psline[linecolor=black, linewidth=0.02](1.5150617,-0.029012324)(1.8854321,0.48950619)
\psline[linecolor=black, linewidth=0.02](1.5224692,-0.021604918)(1.8928396,-0.54012346)
\rput[bl](0.0,-0.6988889){$\scriptstyle a$}
\rput[bl](0.013333334,0.5188889){$\scriptstyle b$}
\rput[bl](1.9377778,0.5188889){$\scriptstyle c$}
\rput[bl](1.9288889,-0.6988889){$\scriptstyle d$}
\rput[bl](0.9866667,0.03888889){$\scriptstyle j$}
\end{pspicture}
}
\end{gathered}
\ \ = \ \ 
\sum_i r_i
\begin{gathered}
\psscalebox{1.0 1.0} % Change this value to rescale the drawing.
{
\begin{pspicture}(0,-0.8551537)(1.3858334,0.8551537)
\psline[linecolor=black, linewidth=0.02](0.68475306,0.37688375)(0.68475306,-0.44488034)
\psline[linecolor=black, linewidth=0.02](0.6766049,0.36470947)(1.1951234,0.66906655)
\psline[linecolor=black, linewidth=0.02](0.68401235,0.3707966)(0.1654938,0.6751537)
\psline[linecolor=black, linewidth=0.02](0.676605,-0.4347351)(1.1951234,-0.7390922)
\psline[linecolor=black, linewidth=0.02](0.68401235,-0.44082224)(0.16549385,-0.7451793)
\rput[bl](0.0,-0.8551537){$\scriptstyle a$}
\rput[bl](0.01,0.6751537){$\scriptstyle b$}
\rput[bl](1.2608334,0.688487){$\scriptstyle c$}
\rput[bl](1.2558334,-0.8551537){$\scriptstyle d$}
\rput[bl](0.7809338,-0.113877065){$\scriptstyle i$}
\end{pspicture}
}
\end{gathered}
\end{equation*}
where the sum on the right hand side runs over all $0 \leq i \leq n-2$ for which $(a,d,i)$ and $(b,c,i)$ are $q$-admissible.
\end{theorem}

The coefficients $r_i$ are called the \emph{$q\mhyphen 6j$ symbols}, and we write
\[ r_i = \qj{a,b,i}{c,d,j} \]
to emphasize the dependence of their values on the values of the edge labels.
They satisfy the following relation.

\begin{lemma}[Orthogonality identity]
\label{lemma:orthogonality_id}
Let $(a,b,j)$, $(c,d,j)$, $(a,b,k)$ and $(c,d,k)$ be $q$-admissible.
Then
\begin{equation*}
\sum_{i=0}^{n-2}\qj{a,b,i}{c,d,j}\qj{d,a,k}{b,c,i} = \delta_{jk}
\end{equation*}
where $\delta_{jk}$ is the Kronecker delta, and the sum is taken over all $i$ for which $(a,d,i)$ and $(b,c,i)$ are $q$-admissible.
\end{lemma}

\begin{remark}
An explicit formula for the $q\mhyphen 6j$ symbols is given in Proposition 11 of \cite{KL1994}, though we will not not need it here.
We remark however that our tensor product identities Theorem \ref{thm:truncated-tensorprod-identity} and Theorem \ref{thm:JWP-tensorprod-identity} (in the case of generic $q$) are special cases of the recoupling theorem with $j=0$, $a=b$ and $c=d$, which gives us that
\[ \qj{a,a,i}{c,c,0} = \frac{[i+1]}{\theta(a,c,i)}. \]
\end{remark}

One last thing we will need is to define the HI move on trivalent graphs.

\begin{defn}
Suppose a general (unlabelled) trivalent graph contains a subgraph
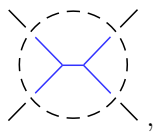
\begin{equation*}
\psscalebox{1.0 1.0} % Change this value to rescale the drawing.
{
\begin{pspicture}(0,-0.80035293)(1.7121252,0.80035293)
\definecolor{colour0}{rgb}{0.2,0.2,1.0}
\psline[linecolor=colour0, linewidth=0.02](0.71605885,0.0062687234)(0.28494775,0.45071316)
\psline[linecolor=colour0, linewidth=0.02](0.99828106,0.0062686554)(1.4160589,0.44626865)
\psline[linecolor=colour0, linewidth=0.02](0.71605885,-0.0026202332)(0.28494775,-0.44706467)
\psline[linecolor=colour0, linewidth=0.02](0.99828106,-0.0026201655)(1.4293922,-0.45150906)
\psline[linecolor=black, linewidth=0.02](1.7049477,0.73737985)(1.4160589,0.43960205)
\psline[linecolor=black, linewidth=0.02](1.7049477,-0.72928685)(1.4116144,-0.4270646)
\psline[linecolor=black, linewidth=0.02](0.007169967,-0.73373127)(0.31828108,-0.41373128)
\psline[linecolor=black, linewidth=0.02](0.011614411,0.73293537)(0.30939218,0.42626873)
\rput{-80.12028}(0.7100964,0.8444583){\pscircle[linecolor=white, linewidth=0.18, dimen=outer](0.85717,0.0){0.80035293}}
\rput{-90.0}(0.85272557,0.8684376){\pscircle[linecolor=black, linewidth=0.02, linestyle=dashed, dash=0.17638889cm 0.10583334cm, dimen=outer](0.8605816,0.007856025){0.72}}
\psline[linecolor=colour0, linewidth=0.02](0.7131964,0.0)(1.0011437,0.0)
\end{pspicture}
} \ ,
\end{equation*}
%(Here the numbering is just to keep track of the edges, and is not a labelling of the graph in the skein module sense.)
then we can replace this locally by
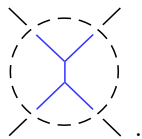
\begin{equation*}
\psscalebox{1.0 1.0} % Change this value to rescale the drawing.
{
\begin{pspicture}(0,-0.85606265)(1.6007059,0.85606265)
\definecolor{colour0}{rgb}{0.2,0.2,1.0}
\psline[linecolor=colour0, linewidth=0.02](0.7940842,-0.1400038)(0.34963974,-0.5711149)
\psline[linecolor=colour0, linewidth=0.02](0.79408425,0.14221843)(0.35408425,0.5599962)
\psline[linecolor=colour0, linewidth=0.02](0.80297315,-0.1400038)(1.2474176,-0.5711149)
\psline[linecolor=colour0, linewidth=0.02](0.8029731,0.14221843)(1.2518619,0.5733295)
\psline[linecolor=black, linewidth=0.02](0.06297309,0.8488851)(0.36075085,0.5599962)
\psline[linecolor=black, linewidth=0.02](1.5296397,0.8488851)(1.2274176,0.55555177)
\psline[linecolor=black, linewidth=0.02](1.5340842,-0.8488927)(1.2140841,-0.5377816)
\psline[linecolor=black, linewidth=0.02](0.06741753,-0.84444827)(0.3740842,-0.54667044)
\rput{9.87972}(0.0120591335,-0.13730845){\pscircle[linecolor=white, linewidth=0.18, dimen=outer](0.80035293,0.0011073303){0.80035293}}
\pscircle[linecolor=black, linewidth=0.02, linestyle=dashed, dash=0.17638889cm 0.10583334cm, dimen=outer](0.7924969,0.004518911){0.72}
\psline[linecolor=colour0, linewidth=0.02](0.80035305,-0.14286627)(0.80035305,0.145081)
\end{pspicture}
} \ .
\end{equation*}
One may think of as ``contracting'' the central edge down to a point and then extending again perpendicularly, keeping the edges outside of the circular region constant throughout.
We call performing such a local substitution a \textbf{HI move}.
Observe that HI moves are invertible; we simply apply the HI move (rotated $90^\circ$) to the same local region of the graph.
\end{defn}

Let $s$ be a trivalent spine for $(\Sigma,M)$, and assume $s^\prime$ is another spine for $(\Sigma,M)$ obtained by applying a single HI move to $s$.
Let $\gamma$ be the subgraph of $s$ to which we apply the HI move, and $\gamma^\prime$ the subgraph of $s^\prime$ which we hence obtain.
We define a linear map $\varphi \colon C(s) \rightarrow C(s^\prime)$ in the following manner:

Let $c$ be a coloring of $s$ for which the labelling of the subgraph $\gamma$ is given by
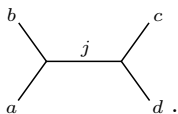
\begin{equation*}
\psscalebox{1.0 1.0} % Change this value to rescale the drawing.
{
\begin{pspicture}(0,-0.6988889)(2.058889,0.6988889)
\psline[linecolor=black, linewidth=0.02](0.5274074,-0.020864246)(1.5274074,-0.020864246)
\psline[linecolor=black, linewidth=0.02](0.5422222,-0.029012367)(0.17185186,0.48950616)
\psline[linecolor=black, linewidth=0.02](0.53481483,-0.021604959)(0.16444445,-0.54012346)
\psline[linecolor=black, linewidth=0.02](1.5150617,-0.029012324)(1.8854321,0.48950619)
\psline[linecolor=black, linewidth=0.02](1.5224692,-0.021604918)(1.8928396,-0.54012346)
\rput[bl](0.0,-0.6988889){$\scriptstyle a$}
\rput[bl](0.013333334,0.5188889){$\scriptstyle b$}
\rput[bl](1.9377778,0.5188889){$\scriptstyle c$}
\rput[bl](1.9288889,-0.6988889){$\scriptstyle d$}
\rput[bl](0.9866667,0.03888889){$\scriptstyle j$}
\end{pspicture}
} \ .
\end{equation*}
The image $\varphi(c)$ is defined to be the linear combination of colorings of $s^\prime$ whose labels agree with those of $s$ on the region $s^\prime\setminus\gamma^\prime = s\setminus\gamma$, and whose labels on $\gamma^\prime$ are given by the linear combination
\begin{equation*}
\sum_i \qj{a,b,i}{c,d,j}
\begin{gathered}
\psscalebox{1.0 1.0} % Change this value to rescale the drawing.
{
\begin{pspicture}(0,-0.8551537)(1.3858334,0.8551537)
\psline[linecolor=black, linewidth=0.02](0.68475306,0.37688375)(0.68475306,-0.44488034)
\psline[linecolor=black, linewidth=0.02](0.6766049,0.36470947)(1.1951234,0.66906655)
\psline[linecolor=black, linewidth=0.02](0.68401235,0.3707966)(0.1654938,0.6751537)
\psline[linecolor=black, linewidth=0.02](0.676605,-0.4347351)(1.1951234,-0.7390922)
\psline[linecolor=black, linewidth=0.02](0.68401235,-0.44082224)(0.16549385,-0.7451793)
\rput[bl](0.0,-0.8551537){$\scriptstyle a$}
\rput[bl](0.01,0.6751537){$\scriptstyle b$}
\rput[bl](1.2608334,0.688487){$\scriptstyle c$}
\rput[bl](1.2558334,-0.8551537){$\scriptstyle d$}
\rput[bl](0.7809338,-0.113877065){$\scriptstyle i$}
\end{pspicture}
}
\end{gathered}
\end{equation*}
as per the recoupling theorem.

Define $\psi \colon C(s^\prime) \rightarrow C(s)$ in a similar manner:
the image of the basis element $c^\prime \in C(s^\prime)$ is the linear combination of colorings of $s$ whose labels agree with those of $c^\prime$ on $s\setminus\gamma$, and whose labels on $\gamma$ are given by the linear combination 
\begin{equation*}
\sum_k \qj{d,a,k}{b,c,i}
\begin{gathered}
\psscalebox{1.0 1.0} % Change this value to rescale the drawing.
{
\begin{pspicture}(0,-0.6988889)(2.058889,0.6988889)
\psline[linecolor=black, linewidth=0.02](0.5274074,-0.020864246)(1.5274074,-0.020864246)
\psline[linecolor=black, linewidth=0.02](0.5422222,-0.029012367)(0.17185186,0.48950616)
\psline[linecolor=black, linewidth=0.02](0.53481483,-0.021604959)(0.16444445,-0.54012346)
\psline[linecolor=black, linewidth=0.02](1.5150617,-0.029012324)(1.8854321,0.48950619)
\psline[linecolor=black, linewidth=0.02](1.5224692,-0.021604918)(1.8928396,-0.54012346)
\rput[bl](0.0,-0.6988889){$\scriptstyle a$}
\rput[bl](0.013333334,0.5188889){$\scriptstyle b$}
\rput[bl](1.9377778,0.5188889){$\scriptstyle c$}
\rput[bl](1.9288889,-0.6988889){$\scriptstyle d$}
\rput[bl](0.9866667,0.03888889){$\scriptstyle k$}
\end{pspicture}
}\end{gathered}
\end{equation*}
we get by applying the recoupling theorem to
\begin{equation*}
\begin{gathered}
\psscalebox{1.0 1.0} % Change this value to rescale the drawing.
{
\begin{pspicture}(0,-0.8551537)(1.3858334,0.8551537)
\psline[linecolor=black, linewidth=0.02](0.68475306,0.37688375)(0.68475306,-0.44488034)
\psline[linecolor=black, linewidth=0.02](0.6766049,0.36470947)(1.1951234,0.66906655)
\psline[linecolor=black, linewidth=0.02](0.68401235,0.3707966)(0.1654938,0.6751537)
\psline[linecolor=black, linewidth=0.02](0.676605,-0.4347351)(1.1951234,-0.7390922)
\psline[linecolor=black, linewidth=0.02](0.68401235,-0.44082224)(0.16549385,-0.7451793)
\rput[bl](0.0,-0.8551537){$\scriptstyle a$}
\rput[bl](0.01,0.6751537){$\scriptstyle b$}
\rput[bl](1.2608334,0.688487){$\scriptstyle c$}
\rput[bl](1.2558334,-0.8551537){$\scriptstyle d$}
\rput[bl](0.7809338,-0.113877065){$\scriptstyle i$}
\end{pspicture}
}
\end{gathered} \ .
\end{equation*}

\begin{lemma}
\label{lemma:single-HI-iso}
The linear maps $\varphi$ and $\psi$ are inverse isomorphisms, and
\[ C(s) \cong C(s^\prime). \]
\end{lemma}

\begin{proof}
We abuse notation and represent colorings of the spines $s,s^\prime$ by the subgraphs $\gamma,\gamma^\prime$ on which we apply the HI move, for instance writing
\begin{equation*}
c = 
\begin{gathered}
\psscalebox{1.0 1.0} % Change this value to rescale the drawing.
{
\begin{pspicture}(0,-0.6988889)(2.058889,0.6988889)
\psline[linecolor=black, linewidth=0.02](0.5274074,-0.020864246)(1.5274074,-0.020864246)
\psline[linecolor=black, linewidth=0.02](0.5422222,-0.029012367)(0.17185186,0.48950616)
\psline[linecolor=black, linewidth=0.02](0.53481483,-0.021604959)(0.16444445,-0.54012346)
\psline[linecolor=black, linewidth=0.02](1.5150617,-0.029012324)(1.8854321,0.48950619)
\psline[linecolor=black, linewidth=0.02](1.5224692,-0.021604918)(1.8928396,-0.54012346)
\rput[bl](0.0,-0.6988889){$\scriptstyle a$}
\rput[bl](0.013333334,0.5188889){$\scriptstyle b$}
\rput[bl](1.9377778,0.5188889){$\scriptstyle c$}
\rput[bl](1.9288889,-0.6988889){$\scriptstyle d$}
\rput[bl](0.9866667,0.03888889){$\scriptstyle j$}
\end{pspicture}
}\end{gathered}
\end{equation*}
to mean the \emph{coloring} $c$ of $s$, and noting as we do so that the labellings of the edges of $c$ are constant outside of this local region.

Then
\begin{align*}
\psi\varphi(c) &= \psi\left(\sum_i \qj{a,b,i}{c,d,j}
\begin{gathered}
\psscalebox{1.0 1.0} % Change this value to rescale the drawing.
{
\begin{pspicture}(0,-0.8551537)(1.3858334,0.8551537)
\psline[linecolor=black, linewidth=0.02](0.68475306,0.37688375)(0.68475306,-0.44488034)
\psline[linecolor=black, linewidth=0.02](0.6766049,0.36470947)(1.1951234,0.66906655)
\psline[linecolor=black, linewidth=0.02](0.68401235,0.3707966)(0.1654938,0.6751537)
\psline[linecolor=black, linewidth=0.02](0.676605,-0.4347351)(1.1951234,-0.7390922)
\psline[linecolor=black, linewidth=0.02](0.68401235,-0.44082224)(0.16549385,-0.7451793)
\rput[bl](0.0,-0.8551537){$\scriptstyle a$}
\rput[bl](0.01,0.6751537){$\scriptstyle b$}
\rput[bl](1.2608334,0.688487){$\scriptstyle c$}
\rput[bl](1.2558334,-0.8551537){$\scriptstyle d$}
\rput[bl](0.7809338,-0.113877065){$\scriptstyle i$}
\end{pspicture}
}
\end{gathered} \ 
\right) \\
&= \sum_i \qj{a,b,i}{c,d,j} \left( \sum_k \qj{d,a,k}{b,c,i}
\begin{gathered}
\psscalebox{1.0 1.0} % Change this value to rescale the drawing.
{
\begin{pspicture}(0,-0.6988889)(2.058889,0.6988889)
\psline[linecolor=black, linewidth=0.02](0.5274074,-0.020864246)(1.5274074,-0.020864246)
\psline[linecolor=black, linewidth=0.02](0.5422222,-0.029012367)(0.17185186,0.48950616)
\psline[linecolor=black, linewidth=0.02](0.53481483,-0.021604959)(0.16444445,-0.54012346)
\psline[linecolor=black, linewidth=0.02](1.5150617,-0.029012324)(1.8854321,0.48950619)
\psline[linecolor=black, linewidth=0.02](1.5224692,-0.021604918)(1.8928396,-0.54012346)
\rput[bl](0.0,-0.6988889){$\scriptstyle a$}
\rput[bl](0.013333334,0.5188889){$\scriptstyle b$}
\rput[bl](1.9377778,0.5188889){$\scriptstyle c$}
\rput[bl](1.9288889,-0.6988889){$\scriptstyle d$}
\rput[bl](0.9866667,0.03888889){$\scriptstyle k$}
\end{pspicture}
}\end{gathered} \ 
\right) \\
&= \sum_k \left( \sum_i \qj{a,b,i}{c,d,j} \qj{d,a,k}{b,c,i}
\begin{gathered}
\psscalebox{1.0 1.0} % Change this value to rescale the drawing.
{
\begin{pspicture}(0,-0.6988889)(2.058889,0.6988889)
\psline[linecolor=black, linewidth=0.02](0.5274074,-0.020864246)(1.5274074,-0.020864246)
\psline[linecolor=black, linewidth=0.02](0.5422222,-0.029012367)(0.17185186,0.48950616)
\psline[linecolor=black, linewidth=0.02](0.53481483,-0.021604959)(0.16444445,-0.54012346)
\psline[linecolor=black, linewidth=0.02](1.5150617,-0.029012324)(1.8854321,0.48950619)
\psline[linecolor=black, linewidth=0.02](1.5224692,-0.021604918)(1.8928396,-0.54012346)
\rput[bl](0.0,-0.6988889){$\scriptstyle a$}
\rput[bl](0.013333334,0.5188889){$\scriptstyle b$}
\rput[bl](1.9377778,0.5188889){$\scriptstyle c$}
\rput[bl](1.9288889,-0.6988889){$\scriptstyle d$}
\rput[bl](0.9866667,0.03888889){$\scriptstyle k$}
\end{pspicture}
}\end{gathered} \ 
\right),
\end{align*}
but by the orthogonality identity (Lemma \ref{lemma:orthogonality_id}) all the coefficients are zero except those for which $k=j$, in which case
\begin{align*}
\psi\varphi(c) &= \sum_{i=0}^{n-2} \qj{a,b,i}{c,d,j} \qj{d,a,j}{b,c,i}
\begin{gathered}
\psscalebox{1.0 1.0} % Change this value to rescale the drawing.
{
\begin{pspicture}(0,-0.6988889)(2.058889,0.6988889)
\psline[linecolor=black, linewidth=0.02](0.5274074,-0.020864246)(1.5274074,-0.020864246)
\psline[linecolor=black, linewidth=0.02](0.5422222,-0.029012367)(0.17185186,0.48950616)
\psline[linecolor=black, linewidth=0.02](0.53481483,-0.021604959)(0.16444445,-0.54012346)
\psline[linecolor=black, linewidth=0.02](1.5150617,-0.029012324)(1.8854321,0.48950619)
\psline[linecolor=black, linewidth=0.02](1.5224692,-0.021604918)(1.8928396,-0.54012346)
\rput[bl](0.0,-0.6988889){$\scriptstyle a$}
\rput[bl](0.013333334,0.5188889){$\scriptstyle b$}
\rput[bl](1.9377778,0.5188889){$\scriptstyle c$}
\rput[bl](1.9288889,-0.6988889){$\scriptstyle d$}
\rput[bl](0.9866667,0.03888889){$\scriptstyle j$}
\end{pspicture}
}\end{gathered} \\
&=
\begin{gathered}
\psscalebox{1.0 1.0} % Change this value to rescale the drawing.
{
\begin{pspicture}(0,-0.6988889)(2.058889,0.6988889)
\psline[linecolor=black, linewidth=0.02](0.5274074,-0.020864246)(1.5274074,-0.020864246)
\psline[linecolor=black, linewidth=0.02](0.5422222,-0.029012367)(0.17185186,0.48950616)
\psline[linecolor=black, linewidth=0.02](0.53481483,-0.021604959)(0.16444445,-0.54012346)
\psline[linecolor=black, linewidth=0.02](1.5150617,-0.029012324)(1.8854321,0.48950619)
\psline[linecolor=black, linewidth=0.02](1.5224692,-0.021604918)(1.8928396,-0.54012346)
\rput[bl](0.0,-0.6988889){$\scriptstyle a$}
\rput[bl](0.013333334,0.5188889){$\scriptstyle b$}
\rput[bl](1.9377778,0.5188889){$\scriptstyle c$}
\rput[bl](1.9288889,-0.6988889){$\scriptstyle d$}
\rput[bl](0.9866667,0.03888889){$\scriptstyle j$}
\end{pspicture}
}\end{gathered} \ .
\end{align*}

Hence $\psi\varphi(c)=c$ for all basis elements $c$ of $C(s)$, and the same argument also gives that $\varphi\psi(c^\prime)=c^\prime$ for all basis elements $c^\prime$ of $C(s^\prime)$.
Thus $\psi\varphi$ and $\varphi\psi$ are the identities on the bases for $C(s),C(s^\prime)$, and so $\varphi,\psi$ are isomorphisms.
\end{proof}

Now we are ready to prove that different trivalent spines for the same surface yield isomorphic skein modules.

\begin{theorem}
\label{thm:skein-module-spine-iso}
\todo{I feel like it should have some kind of name.
``Well-definedness of skein modules for $n$-holed disks''?}
Let $(\Sigma,M)$ be a $n$-holed disk with marked boundary, and let $s, s^\prime$ be trivalent spines for $(\Sigma,M)$.
Then $C(s)\cong C(s^\prime)$.
\end{theorem}

The fundamental idea is to show that any two trivalent spines $s, s^\prime$ for $(\Sigma,M)$ are related via a sequence of HI moves
\[ s = s_1 \xrightarrow{\ z_1\ } s_2 \xrightarrow{\ z_2\ } s_3 \xrightarrow{\ z_3\ } \dotsc \xrightarrow{z_{k-2}} s_{k-1} \xrightarrow{\ z_{k-1}\ } s_k = s^\prime \]
where the $s_i$ are intermediate trivalent spines for $(\Sigma,M)$ and each $z_i$ is a single HI move applied to a local region of $s_i$.
Then by Lemma \ref{lemma:single-HI-iso} this yields a sequence of isomorphisms
\[ C(s) \stackrel{\widetilde{z_1}}{\cong} C(s_2) \stackrel{\widetilde{z_2}}{\cong} \dotsb \stackrel{\widetilde{z_{k-2}}}{\cong} C(s_{k-1}) \stackrel{\widetilde{z_{k-1}}}{\cong} C(s^\prime). \]
Our proof will use some basic concepts from Morse theory; the reader unfamiliar with some of the terms involved may consult Chapter 3 of \cite{Matsumoto2002}, Chapter 4 of \cite{Gompf1999} or any other introductory text.

\begin{proof}[Proof of Theorem \ref{thm:skein-module-spine-iso}]
Observe that any spine $s$ (not necessarily trivalent) for $(\Sigma,M)$ defines a handle decomposition $S$ for $\Sigma$ in terms of $2$-dimensional $0$ and $1$-handles in the following way:
for every vertex $v$ of $s$ we ``enlarge'' $v$ to a $0$-handle $H_v^0 = D^0\times D^2$ while keeping the edges adjacent to $v$ attached to the boundary of $H_v^0$, and for every edge $e$ connecting vertices $v, w$ we attach a $1$-handle $H_e^1 = D^1\times D^1$ to the corresponding $0$-handles $H_v^0,H_w^0$ in the obvious way, without twisting and by taking $e$ to be the core of $H_e^1$.
Conversely, suppose we have a handle decomposition for $\Sigma$ in terms of $0$ and $1$-handles, such that every marked boundary point on $\Sigma$ has an associated $0$-handle with exactly one $1$-handle attached.
Then it is clear that we can deformation retract each $1$-handle to its core and then each $0$-handle to a single point, and thus obtain a spine for $(\Sigma,M)$.
Therefore we see that giving a spine $s$ for $(\Sigma,M)$ is equivalent 
\todo{I'm not very happy about the word ``equivalent''... but I don't know how else to say this.}
to specifying a handle decomposition $S$ of $\Sigma$ ``with respect to the marked boundary $M$''.
It will be advantageous to first consider the case $M=\emptyset$, i.e. where $\Sigma$ has no marked boundary points.

Let $s,s^\prime$ be trivalent spines for $(\Sigma,\emptyset)$, and let $S,S^\prime$ be the associated handle decompositions for $\Sigma$.
By a fundamental result of Cerf theory, $S$ and $S^\prime$ are related (via homotopies of their corresponding Morse functions) by a finite sequence of handle pair creations, cancellations, and handle slides (cf. Theorem 4.2.12 of \cite{Gompf1999}); that is to say that $S^\prime$ can be obtained from $S$ by applying a sequence of such \emph{handle moves}, which are in fact diffeomorphisms.
Again, by considering homotopies of Morse functions it is not hard to show that since $S,S^\prime$ only involve handles of index $0$ or $1$, we may restrict ourselves to using only $(0,1)$-handle moves.

Observe that every $(0,1)$-handle move on handle decompositions has an interpretation as a ``spine move'', i.e. a transformation of the associated spine, as shown in Figure \ref{fig:handle-moves}.
The handle move theorem then asserts the existence of a finite sequence of spine moves that takes $s$ to $s^\prime$.
Our proof then proceeds in two steps:
\begin{enumerate}
\item First we show that given some sequence of spine moves taking $s$ to $s^\prime$, we can rewrite it into one that passes through spines $s_i$ whose largest vertex degree never exceeds $4$,
\item then we show that we can rearrange the moves in a sequence passing through such $4$-valent spines into one that reads as a sequence of HI moves that passes through trivalent spines.
\end{enumerate}

% ---------------- FIG: Handle moves ----------------
\begin{sidewaysfigure}
\begingroup
\addtolength{\tabcolsep}{0.7cm}
\begin{tabular}{c|c}
$
\begin{gathered}
\psscalebox{1.0 1.0} % Change this value to rescale the drawing.
{
\begin{pspicture}(0,-0.7830115)(1.7131,0.7830115)
\rput{5.904688}(0.0020926145,-0.08569061){\pscircle[linecolor=black, linewidth=0.02, dimen=outer](0.83180386,-0.022557722){0.36}}
\psbezier[linecolor=black, linewidth=0.02](0.64389503,0.27094656)(0.5224001,0.4018205)(0.47242284,0.4408185)(0.3660655,0.47840557)(0.2597082,0.51599264)(0.14945288,0.62318385)(0.107658476,0.78044295)
\psbezier[linecolor=black, linewidth=0.02](0.5400387,0.1635715)(0.4159967,0.32754168)(0.44860277,0.29068318)(0.31642264,0.3461138)(0.18424253,0.40154442)(0.058418732,0.5007801)(0.009501023,0.6497742)
\psbezier[linecolor=black, linewidth=0.02](0.5240142,-0.18811677)(0.4301928,-0.21851663)(0.27067184,-0.2807878)(0.23258346,-0.36356476)(0.19449507,-0.44634172)(0.29771188,-0.5015704)(0.015506186,-0.65171236)
\psbezier[linecolor=black, linewidth=0.02](0.6181795,-0.29706076)(0.34968305,-0.3850644)(0.35788035,-0.44280902)(0.3574318,-0.51543236)(0.35698324,-0.5880556)(0.23235133,-0.7076937)(0.093384214,-0.7739859)
\psbezier[linecolor=black, linewidth=0.02](1.1326588,-0.19549845)(1.2313857,-0.38046387)(1.2544206,-0.46247524)(1.4160249,-0.44775626)(1.5776293,-0.43303728)(1.6751446,-0.5407062)(1.7037829,-0.6141624)
\psbezier[linecolor=black, linewidth=0.02](1.0357862,-0.3130456)(1.0927285,-0.47623575)(1.1624985,-0.58233464)(1.2863052,-0.5947142)(1.4101119,-0.6070939)(1.539323,-0.5596656)(1.6121051,-0.7264929)
\rput{-90.0}(0.23769212,0.43160906){\rput[bl](0.33465058,0.09695847){$\scriptstyle \dots$}}
\rput{-90.0}(1.1728272,1.3667442){\rput[bl](1.2697858,0.09695847){$\scriptstyle \dots$}}
\psbezier[linecolor=black, linewidth=0.02](0.9909002,0.2865897)(1.0786161,0.38032737)(1.1913203,0.46654552)(1.3104743,0.5040191)(1.4296283,0.5414927)(1.5227652,0.5712527)(1.5486631,0.74682117)
\psbezier[linecolor=black, linewidth=0.02](1.0963905,0.20840776)(1.19631,0.29585934)(1.2645248,0.35904545)(1.42153,0.4090026)(1.5785353,0.45895976)(1.701317,0.5638022)(1.7029862,0.71731234)
\end{pspicture}
}
\end{gathered}
\quad \genfrac{}{}{0pt}{}{\xrightarrow{\qquad}}{\xleftarrow{\qquad}} \quad
\begin{gathered}
\psscalebox{1.0 1.0} % Change this value to rescale the drawing.
{
\begin{pspicture}(0,-0.8063061)(3.1449444,0.8063061)
\rput{5.904688}(0.0073156552,-0.23252036){\pscircle[linecolor=black, linewidth=0.02, dimen=outer](2.2579076,-0.045335997){0.36}}
\psbezier[linecolor=black, linewidth=0.02](2.5587626,-0.21827672)(2.6574895,-0.40324214)(2.6805243,-0.4852535)(2.8421288,-0.47053453)(3.0037332,-0.45581555)(3.1012483,-0.5634845)(3.1298866,-0.63694066)
\psbezier[linecolor=black, linewidth=0.02](2.46189,-0.33582386)(2.5188322,-0.49901402)(2.5886023,-0.6051129)(2.712409,-0.6174925)(2.8362157,-0.62987214)(2.9654267,-0.5824439)(3.038209,-0.74927115)
\rput{-90.0}(2.6217093,2.7700696){\rput[bl](2.6958895,0.07418019){$\scriptstyle \dots$}}
\psbezier[linecolor=black, linewidth=0.02](2.417004,0.26381144)(2.50472,0.3575491)(2.617424,0.44376725)(2.736578,0.48124084)(2.855732,0.5187144)(2.948869,0.54847443)(2.974767,0.7240429)
\psbezier[linecolor=black, linewidth=0.02](2.5224943,0.18562949)(2.6224139,0.27308106)(2.6906285,0.33626717)(2.8476338,0.38622433)(3.0046391,0.4361815)(3.1274207,0.5410239)(3.12909,0.69453406)
\rput{-90.0}(0.29968774,0.47251165){\rput[bl](0.3860997,0.08641195){$\scriptstyle \dots$}}
\psbezier[linecolor=black, linewidth=0.02](0.66962856,-0.3076073)(0.40113214,-0.39561093)(0.40932947,-0.45335552)(0.4088809,-0.52597886)(0.40843233,-0.5986022)(0.28380045,-0.7182402)(0.14483331,-0.7845325)
\psbezier[linecolor=black, linewidth=0.02](0.57546324,-0.1986633)(0.4816419,-0.22906315)(0.32212096,-0.2913343)(0.28403255,-0.37411126)(0.24594416,-0.45688826)(0.34916097,-0.5121169)(0.06695528,-0.66225886)
\psbezier[linecolor=black, linewidth=0.02](0.5914878,0.15302499)(0.4674458,0.31699517)(0.50005186,0.28013664)(0.36787173,0.3355673)(0.23569162,0.39099792)(0.109867826,0.4902336)(0.060950123,0.6392277)
\psbezier[linecolor=black, linewidth=0.02](0.69534415,0.26040006)(0.5738492,0.391274)(0.5238719,0.43027195)(0.41751462,0.46785903)(0.31115732,0.50544614)(0.20090197,0.61263734)(0.15910757,0.7698964)
\rput{5.904688}(0.0012806212,-0.09103934){\pscircle[linecolor=black, linewidth=0.02, dimen=outer](0.883253,-0.033104237){0.36}}
\psline[linecolor=black, linewidth=0.02](1.2257832,0.053254407)(1.9184662,0.053254407)
\psline[linecolor=black, linewidth=0.02](1.9038321,-0.117477305)(1.2209052,-0.117477305)
\end{pspicture}
}
\end{gathered}
$
&
$
\begin{gathered}
\psscalebox{1.0 1.0} % Change this value to rescale the drawing.
{
\begin{pspicture}(0,-0.7060292)(1.4067255,0.7060292)
\definecolor{colour0}{rgb}{1.0,0.0,0.0}
\pscircle[linecolor=black, linewidth=0.02, fillstyle=solid, fillcolor=black, dimen=outer](0.70989335,-0.0100997295){0.09795918}
\psline[linecolor=black, linewidth=0.02](0.0070710713,0.6989584)(0.7063211,0.0)
\psline[linecolor=black, linewidth=0.02](0.0097377375,-0.6962916)(0.7089878,0.002958397)
\psline[linecolor=black, linewidth=0.02](1.3996544,0.69629174)(0.7004044,-0.002958309)
\psline[linecolor=black, linewidth=0.02](1.3969878,-0.69895834)(0.69773775,0.0)
\rput{-90.0}(0.25476697,0.44648156){\rput[bl](0.35062426,0.0958573){$\scriptstyle \dots$}}
\rput{-90.0}(0.9095615,1.101276){\rput[bl](1.0054188,0.0958573){$\scriptstyle \dots$}}
\end{pspicture}
}
\end{gathered}
\quad \genfrac{}{}{0pt}{}{\xrightarrow{\qquad}}{\xleftarrow{\qquad}} \quad
\begin{gathered}
\psscalebox{1.0 1.0} % Change this value to rescale the drawing.
{
\begin{pspicture}(0,-0.70469606)(2.3667254,0.70469606)
\definecolor{colour0}{rgb}{1.0,0.0,0.0}
\psline[linecolor=colour0, linewidth=0.02](0.68291825,0.0)(1.660696,0.003777771)
\pscircle[linecolor=black, linewidth=0.02, fillstyle=solid,fillcolor=black, dimen=outer](0.6738933,0.0){0.09795918}
\pscircle[linecolor=black, linewidth=0.02, fillstyle=solid,fillcolor=black, dimen=outer](1.6671137,0.0){0.09795918}
\psline[linecolor=black, linewidth=0.02](2.3596544,0.69762504)(1.6604043,-0.001625006)
\psline[linecolor=black, linewidth=0.02](2.3569877,-0.697625)(1.6577376,0.0016250331)
\psline[linecolor=black, linewidth=0.02](0.0070709847,0.697625)(0.706321,-0.0016250467)
\psline[linecolor=black, linewidth=0.02](0.009737651,-0.69762504)(0.7089877,0.0016249925)
\rput{-90.0}(0.30744272,0.5093855){\rput[bl](0.40841413,0.1009714){$\scriptstyle \dots$}}
\rput{-90.0}(1.8389496,2.0408924){\rput[bl](1.939921,0.1009714){$\scriptstyle \dots$}}
\end{pspicture}
}
\end{gathered}
$
\\
Handle pair creation/cancellation & Edge-vertex creation/cancellation
\\\\\\
$
\begin{gathered}
\psscalebox{1.0 1.0} % Change this value to rescale the drawing.
{
\begin{pspicture}(0,-0.88987)(3.1449444,0.88987)
\psframe[linecolor=black, linewidth=0.02, dimen=outer](2.0441551,-0.052925844)(1.0831162,-0.19318558)
\rput{5.904688}(-0.0073159193,-0.09148269){\pscircle[linecolor=black, linewidth=0.02, fillstyle=solid, dimen=outer](0.883253,-0.11666808){0.36}}
\psbezier[linecolor=black, linewidth=0.02](0.69534415,0.17683621)(0.5738492,0.30771017)(0.5238719,0.34670812)(0.41751462,0.3842952)(0.31115732,0.42188227)(0.20090197,0.5290735)(0.15910757,0.6863326)
\psbezier[linecolor=black, linewidth=0.02](0.5914878,0.06946115)(0.4674458,0.23343134)(0.50005186,0.19657281)(0.36787173,0.25200343)(0.23569162,0.30743408)(0.109867826,0.40666977)(0.060950123,0.5556638)
\psbezier[linecolor=black, linewidth=0.02](0.57546324,-0.28222713)(0.4816419,-0.312627)(0.32212096,-0.37489814)(0.28403255,-0.45767513)(0.24594416,-0.54045206)(0.34916097,-0.5956807)(0.06695528,-0.74582267)
\psbezier[linecolor=black, linewidth=0.02](0.66962856,-0.39117113)(0.40113214,-0.47917476)(0.40932947,-0.53691936)(0.4088809,-0.60954267)(0.40843233,-0.682166)(0.28380045,-0.80180407)(0.14483331,-0.8680963)
\rput{-90.0}(0.38325158,0.38894778){\rput[bl](0.3860997,0.0028481106){$\scriptstyle \dots$}}
\psbezier[linecolor=black, linewidth=0.02](2.417004,0.1802476)(2.50472,0.27398527)(2.617424,0.36020342)(2.736578,0.397677)(2.855732,0.43515056)(2.948869,0.46491057)(2.974767,0.640479)
\psbezier[linecolor=black, linewidth=0.02](2.5224943,0.10206564)(2.6224139,0.18951723)(2.6906285,0.25270334)(2.8476338,0.3026605)(3.0046391,0.35261762)(3.1274207,0.45746005)(3.12909,0.6109702)
\rput{-90.0}(2.7052732,2.6865058){\rput[bl](2.6958895,-0.009383649){$\scriptstyle \dots$}}
\psbezier[linecolor=black, linewidth=0.02](2.46189,-0.4193877)(2.5188322,-0.5825779)(2.5886023,-0.6886768)(2.712409,-0.70105636)(2.8362157,-0.71343595)(2.9654267,-0.6660077)(3.038209,-0.832835)
\psbezier[linecolor=black, linewidth=0.02](2.5587626,-0.30184057)(2.6574895,-0.48680598)(2.6805243,-0.5688174)(2.8421288,-0.55409837)(3.0037332,-0.5393794)(3.1012483,-0.64704835)(3.1298866,-0.7205045)
\rput{5.904688}(-0.0012808851,-0.23296371){\pscircle[linecolor=black, linewidth=0.02, fillstyle=solid, dimen=outer](2.2579076,-0.12889984){0.36}}
\psbezier[linecolor=black, linewidth=0.02](0.8023159,0.22074248)(0.76304805,0.39015183)(0.75130284,0.34404728)(0.74826187,0.5274021)(0.74522084,0.7107569)(0.7581918,0.74706084)(0.710424,0.8866319)
\psbezier[linecolor=black, linewidth=0.02](0.9644781,0.21636163)(0.9180242,0.36659372)(0.9019366,0.45788458)(0.8942078,0.57559144)(0.88647896,0.6932983)(0.91600513,0.72330517)(0.8725862,0.8866319)
\psline[linecolor=black, linewidth=0.02](0.738412,0.7867156)(0.86730087,0.86671567)
\psline[linecolor=black, linewidth=0.02](0.74730086,0.6824277)(0.898412,0.775761)
\psline[linecolor=black, linewidth=0.02](0.74730086,0.5668721)(0.8895231,0.64687216)
\psline[linecolor=black, linewidth=0.02](0.75174534,0.45576102)(0.89396757,0.535761)
\psline[linecolor=black, linewidth=0.02](0.76507866,0.3535388)(0.9073009,0.42464992)
\psline[linecolor=black, linewidth=0.02](0.7961898,0.24687214)(0.9295231,0.3135388)
\rput[bl](1.0363629,0.46135986){$\rightsquigarrow$}
\end{pspicture}
}
\end{gathered}
\ \xrightarrow{\qquad} \ 
\begin{gathered}
\psscalebox{1.0 1.0} % Change this value to rescale the drawing.
{
\begin{pspicture}(0,-0.88987)(3.1449444,0.88987)
\psframe[linecolor=black, linewidth=0.02, dimen=outer](2.0441551,-0.052925844)(1.0831162,-0.19318558)
\rput{5.904688}(-0.0012808851,-0.23296371){\pscircle[linecolor=black, linewidth=0.02, fillstyle=solid, dimen=outer](2.2579076,-0.12889984){0.36}}
\psbezier[linecolor=black, linewidth=0.02](2.5587626,-0.30184057)(2.6574895,-0.48680598)(2.6805243,-0.5688174)(2.8421288,-0.55409837)(3.0037332,-0.5393794)(3.1012483,-0.64704835)(3.1298866,-0.7205045)
\psbezier[linecolor=black, linewidth=0.02](2.46189,-0.4193877)(2.5188322,-0.5825779)(2.5886023,-0.6886768)(2.712409,-0.70105636)(2.8362157,-0.71343595)(2.9654267,-0.6660077)(3.038209,-0.832835)
\rput{-90.0}(2.7052732,2.6865058){\rput[bl](2.6958895,-0.009383649){$\scriptstyle \dots$}}
\psbezier[linecolor=black, linewidth=0.02](2.417004,0.1802476)(2.50472,0.27398527)(2.617424,0.36020342)(2.736578,0.397677)(2.855732,0.43515056)(2.948869,0.46491057)(2.974767,0.640479)
\psbezier[linecolor=black, linewidth=0.02](2.5224943,0.10206564)(2.6224139,0.18951723)(2.6906285,0.25270334)(2.8476338,0.3026605)(3.0046391,0.35261762)(3.1274207,0.45746005)(3.12909,0.6109702)
\rput{-90.0}(0.38325158,0.38894778){\rput[bl](0.3860997,0.0028481106){$\scriptstyle \dots$}}
\psbezier[linecolor=black, linewidth=0.02](0.66962856,-0.39117113)(0.40113214,-0.47917476)(0.40932947,-0.53691936)(0.4088809,-0.60954267)(0.40843233,-0.682166)(0.28380045,-0.80180407)(0.14483331,-0.8680963)
\psbezier[linecolor=black, linewidth=0.02](0.57546324,-0.28222713)(0.4816419,-0.312627)(0.32212096,-0.37489814)(0.28403255,-0.45767513)(0.24594416,-0.54045206)(0.34916097,-0.5956807)(0.06695528,-0.74582267)
\psbezier[linecolor=black, linewidth=0.02](0.5914878,0.06946115)(0.4674458,0.23343134)(0.50005186,0.19657281)(0.36787173,0.25200343)(0.23569162,0.30743408)(0.109867826,0.40666977)(0.060950123,0.5556638)
\psbezier[linecolor=black, linewidth=0.02](0.69534415,0.17683621)(0.5738492,0.30771017)(0.5238719,0.34670812)(0.41751462,0.3842952)(0.31115732,0.42188227)(0.20090197,0.5290735)(0.15910757,0.6863326)
\rput{5.904688}(-0.0073159193,-0.09148269){\pscircle[linecolor=black, linewidth=0.02, fillstyle=solid, dimen=outer](0.883253,-0.11666808){0.36}}
\psbezier[linecolor=black, linewidth=0.02](2.1997185,0.2181451)(2.1604507,0.38755444)(2.1487055,0.3414499)(2.1456645,0.5248047)(2.1426234,0.7081595)(2.1555943,0.74446344)(2.1078267,0.8840345)
\psbezier[linecolor=black, linewidth=0.02](2.3618808,0.21376425)(2.3154268,0.36399633)(2.2993393,0.4552872)(2.2916105,0.57299405)(2.2838817,0.69070095)(2.3134077,0.72070783)(2.2699888,0.8840345)
\psline[linecolor=black, linewidth=0.02](2.1358147,0.78411824)(2.2647035,0.8641183)
\psline[linecolor=black, linewidth=0.02](2.1447034,0.6798303)(2.2958145,0.7731636)
\psline[linecolor=black, linewidth=0.02](2.1447034,0.5642747)(2.2869258,0.6442748)
\psline[linecolor=black, linewidth=0.02](2.149148,0.45316365)(2.2913702,0.53316367)
\psline[linecolor=black, linewidth=0.02](2.1624813,0.35094142)(2.3047035,0.42205253)
\psline[linecolor=black, linewidth=0.02](2.1935923,0.24427475)(2.3269258,0.31094143)
\end{pspicture}
}
\end{gathered}
$
&
$
\begin{gathered}
\psscalebox{1.0 1.0} % Change this value to rescale the drawing.
{
\begin{pspicture}(0,-0.7780391)(2.3667254,0.7780391)
\definecolor{colour0}{rgb}{1.0,0.0,0.0}
\psline[linecolor=colour0, linewidth=0.02](0.6792468,-0.06840599)(0.6734497,0.7779708)
\psline[linecolor=black, linewidth=0.02](0.6829183,-0.07400986)(1.660696,-0.069565415)
\pscircle[linecolor=black, linewidth=0.02, fillstyle=solid,fillcolor=black, dimen=outer](0.6738934,-0.07334319){0.09795918}
\pscircle[linecolor=black, linewidth=0.02, fillstyle=solid,fillcolor=black, dimen=outer](1.6671137,-0.07334319){0.09795918}
\psline[linecolor=black, linewidth=0.02](2.3596544,0.6242818)(1.6604044,-0.07496819)
\psline[linecolor=black, linewidth=0.02](2.3569877,-0.7709682)(1.6577377,-0.07171815)
\psline[linecolor=black, linewidth=0.02](0.0070710643,0.6242818)(0.7063211,-0.07496823)
\psline[linecolor=black, linewidth=0.02](0.009737731,-0.77096826)(0.7089878,-0.07171819)
\rput{-90.0}(0.380786,0.43604243){\rput[bl](0.4084142,0.027628217){$\scriptstyle \dots$}}
\rput{-90.0}(1.9122928,1.9675493){\rput[bl](1.939921,0.027628217){$\scriptstyle \dots$}}
\end{pspicture}
}
\end{gathered}
\ \xrightarrow{\qquad} \ 
\begin{gathered}
\psscalebox{1.0 1.0} % Change this value to rescale the drawing.
{
\begin{pspicture}(0,-0.7816756)(2.3667254,0.7816756)
\definecolor{colour0}{rgb}{1.0,0.0,0.4}
\psline[linecolor=colour0, linewidth=0.02](1.6756104,-0.064769745)(1.6698134,0.7816071)
\psline[linecolor=black, linewidth=0.02](0.6829183,-0.07764633)(1.660696,-0.07320189)
\pscircle[linecolor=black, linewidth=0.02, fillstyle=solid,fillcolor=black, dimen=outer](0.6738934,-0.07697966){0.09795918}
\pscircle[linecolor=black, linewidth=0.02, fillstyle=solid,fillcolor=black, dimen=outer](1.6671137,-0.07697966){0.09795918}
\psline[linecolor=black, linewidth=0.02](2.3596544,0.62064534)(1.6604044,-0.07860467)
\psline[linecolor=black, linewidth=0.02](2.3569877,-0.7746047)(1.6577377,-0.07535463)
\psline[linecolor=black, linewidth=0.02](0.0070710643,0.62064534)(0.7063211,-0.078604706)
\psline[linecolor=black, linewidth=0.02](0.009737731,-0.7746047)(0.7089878,-0.075354666)
\rput{-90.0}(0.38442248,0.43240595){\rput[bl](0.4084142,0.023991743){$\scriptstyle \dots$}}
\rput{-90.0}(1.9159293,1.9639128){\rput[bl](1.939921,0.023991743){$\scriptstyle \dots$}}
\end{pspicture}
}
\end{gathered}
$ \\
Handle slide & Edge slide
\end{tabular}
\endgroup
\caption{Handle moves for $(0,1)$-handle decompositions and their associated spine moves. \label{fig:handle-moves}}
\end{sidewaysfigure}
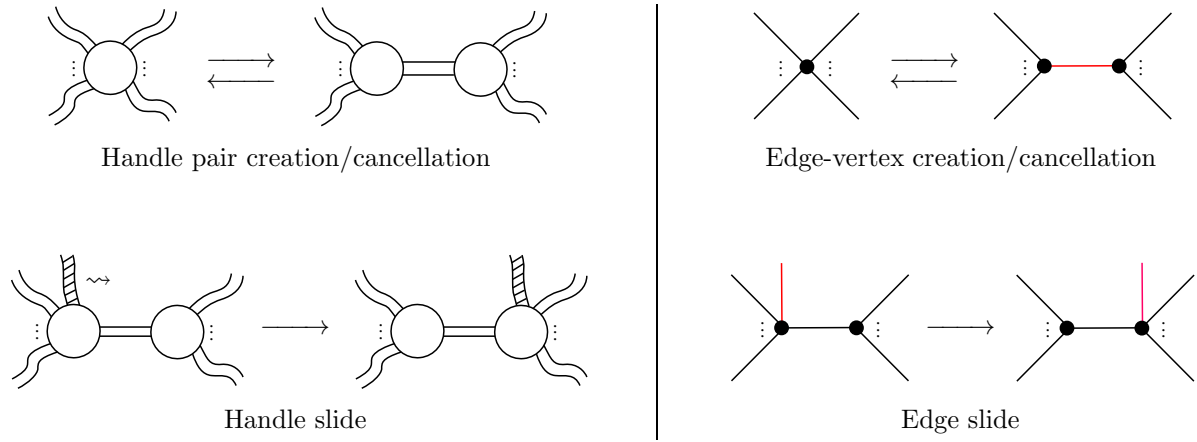
% ---------------- END FIG ----------------

Now we begin the actual work of the proof.
Let
\[ s = s_1 \xrightarrow{\ w_1\ } s_2 \xrightarrow{\ w_2\ } \cdots \xrightarrow{w_{k-2}} s_{k-1} \xrightarrow{w_{k-1}} s_k = s^\prime \]
be a sequence of spine moves from $s$ to $s^\prime$.
Without loss of generality we may assume that each move $w_i$ is one of the following:
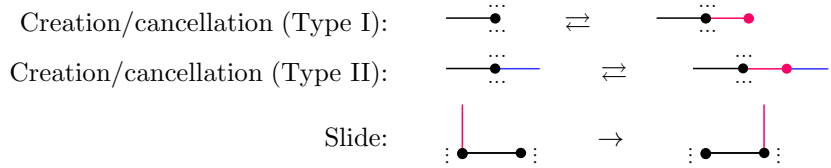
\begin{align*}
\text{Creation/cancellation (Type I):} & \qquad
\begin{gathered}
\psscalebox{1.0 1.0} % Change this value to rescale the drawing.
{
\begin{pspicture}(0,-0.185)(0.75574267,0.185)
\pscircle[linecolor=black, linewidth=0.02, fillstyle=solid,fillcolor=black, dimen=outer](0.6477786,-0.0018133641){0.068584576}
\rput[bl](0.5457427,0.155){$\scriptstyle \dots$}
\psline[linecolor=black, linewidth=0.02](0.65,0.002327803)(0.0,0.002327803)
\rput[bl](0.5457427,-0.185){$\scriptstyle \dots$}
\end{pspicture}
}
\end{gathered}
\qquad \rightleftarrows \qquad
\begin{gathered}
\psscalebox{1.0 1.0} % Change this value to rescale the drawing.
{
\begin{pspicture}(0,-0.185)(1.281236,0.185)
\definecolor{colour0}{rgb}{1.0,0.0,0.4}
\psline[linecolor=colour0, linewidth=0.02](1.1970313,0.0010212128)(0.66359574,0.0010212128)
\pscircle[linecolor=black, linewidth=0.02, fillstyle=solid,fillcolor=black, dimen=outer](0.6477786,-0.0018133641){0.068584576}
\pscircle[linecolor=colour0, linewidth=0.02, fillstyle=solid,fillcolor=colour0, dimen=outer](1.2126515,-0.002131798){0.068584576}
\rput[bl](0.5457427,0.155){$\scriptstyle \dots$}
\psline[linecolor=black, linewidth=0.02](0.65,0.002327803)(0.0,0.002327803)
\rput[bl](0.5457427,-0.185){$\scriptstyle \dots$}
\end{pspicture}
}
\end{gathered} \\
\text{Creation/cancellation (Type II):} & \qquad
\begin{gathered}
\psscalebox{1.0 1.0} % Change this value to rescale the drawing.
{
\begin{pspicture}(0,-0.185)(1.2323016,0.185)
\definecolor{colour1}{rgb}{0.2,0.2,1.0}
\psline[linecolor=colour1, linewidth=0.02](1.2323016,0.0)(0.7043613,0.0)
\pscircle[linecolor=black, linewidth=0.02, fillstyle=solid,fillcolor=black, dimen=outer](0.6451661,0.0015139843){0.06787803}
\psline[linecolor=black, linewidth=0.02](0.65000015,0.002327803)(0.0,0.002327803)
\rput[bl](0.5457428,0.155){$\scriptstyle \dots$}
\rput[bl](0.5457428,-0.185){$\scriptstyle \dots$}
\end{pspicture}
}
\end{gathered}
\qquad \rightleftarrows \qquad
\begin{gathered}
\psscalebox{1.0 1.0} % Change this value to rescale the drawing.
{
\begin{pspicture}(0,-0.185)(1.7814093,0.185)
\definecolor{colour1}{rgb}{0.2,0.2,1.0}
\definecolor{colour0}{rgb}{1.0,0.0,0.4}
\psline[linecolor=colour1, linewidth=0.02](1.7814091,-0.0017598893)(1.2666885,-0.0017598893)
\psline[linecolor=colour0, linewidth=0.02](1.2323738,-0.0017598893)(0.71765316,-0.0017598893)
\pscircle[linecolor=black, linewidth=0.02, fillstyle=solid,fillcolor=black, dimen=outer](0.65955615,0.004143681){0.06617837}
\pscircle[linecolor=colour0, linewidth=0.02, fillstyle=solid,fillcolor=colour0, dimen=outer](1.2334944,0.0){0.06617837}
\rput[bl](0.54574263,-0.185){$\scriptstyle \dots$}
\psline[linecolor=black, linewidth=0.02](0.6499999,0.002327803)(0.0,0.002327803)
\rput[bl](0.54574263,0.155){$\scriptstyle \dots$}
\end{pspicture}
}
\end{gathered} \\
\text{Slide:} & \qquad
\begin{gathered}
\psscalebox{1.0 1.0} % Change this value to rescale the drawing.
{
\begin{pspicture}(0,-0.38049808)(1.1869475,0.38049808)
\definecolor{colour0}{rgb}{1.0,0.0,0.4}
\psline[linecolor=colour0, linewidth=0.02](0.21446268,-0.2701044)(0.21446268,0.380498)
\psline[linecolor=black, linewidth=0.02](0.97322446,-0.26231462)(0.2309,-0.26231462)
\pscircle[linecolor=black, linewidth=0.02, fillstyle=solid,fillcolor=black, dimen=outer](0.21508288,-0.2651492){0.068584576}
\pscircle[linecolor=black, linewidth=0.02, fillstyle=solid,fillcolor=black, dimen=outer](0.98396206,-0.26216474){0.068584576}
\rput{-90.0}(1.3274455,0.98644936){\rput[bl](1.1569475,-0.17049807){$\scriptstyle \dots$}}
\pscircle[linecolor=black, linewidth=0.02, fillstyle=solid,fillcolor=black, dimen=outer](0.21508288,-0.27549806){0.068584576}
\psline[linecolor=black, linewidth=0.02](0.97322446,-0.26231462)(0.2309,-0.26231462)
\rput{-90.0}(0.17049807,-0.17049807){\rput[bl](0.0,-0.17049807){$\scriptstyle \dots$}}
\end{pspicture}
}
\end{gathered}
\qquad \rightarrow \qquad
\begin{gathered}
\psscalebox{1.0 1.0} % Change this value to rescale the drawing.
{
\begin{pspicture}(0,-0.38049808)(1.1869475,0.38049808)
\definecolor{colour0}{rgb}{1.0,0.0,0.4}
\psline[linecolor=colour0, linewidth=0.02](0.981586,-0.2701044)(0.981586,0.380498)
\psline[linecolor=black, linewidth=0.02](0.97322446,-0.26231462)(0.2309,-0.26231462)
\pscircle[linecolor=black, linewidth=0.02, fillstyle=solid,fillcolor=black, dimen=outer](0.21508288,-0.2651492){0.068584576}
\pscircle[linecolor=black, linewidth=0.02, fillstyle=solid,fillcolor=black, dimen=outer](0.98396206,-0.26216474){0.068584576}
\rput{-90.0}(1.3274455,0.98644936){\rput[bl](1.1569475,-0.17049807){$\scriptstyle \dots$}}
\pscircle[linecolor=black, linewidth=0.02, fillstyle=solid,fillcolor=black, dimen=outer](0.21508288,-0.27549806){0.068584576}
\psline[linecolor=black, linewidth=0.02](0.97322446,-0.26231462)(0.2309,-0.26231462)
\rput{-90.0}(0.17049807,-0.17049807){\rput[bl](0.0,-0.17049807){$\scriptstyle \dots$}}
\end{pspicture}
}
\end{gathered}
\end{align*}
(where the ellipses indicate zero or more edges attached to a vertex) as we can always write a general spine move as a composition consisting only of these moves.

For a sequence of spine moves $W$ taking $s$ to $s^\prime$, define a complexity function $\Phi$ on $W$ by the ordered pair
\begin{equation*}
\Phi(W) = (d,T)
\end{equation*}
where $d\geq 3$ is the maximum degree over all vertices over all spines in the sequence, and  $T\geq 1$ is the maximum contiguous length of time for which some vertex in the sequence achieves degree $d$.
That is, $T$ is the length of the longest subsequence
\[ s_{i_1} \xrightarrow{w_{i_1}} s_{i_2} \xrightarrow{w_{i_2}} \cdots \xrightarrow{w_{i_T-1}} s_{i_T} \]
in which the same vertex $v$ appears in every $s_{i_j}$ having degree $d$ throughout.

We claim that given a sequence $W$ from $s$ to $s^\prime$ having complexity function $\Phi(W)=(d,T)$ where $d>4$ and $T\geq 1$, we can always rewrite $W$ into a sequence $W^\prime$ for which
\[ \Phi(W^\prime) = (d^\prime,T^\prime), \]
where either $d^\prime < d$, or else $d^\prime=d$ and $T^\prime < T$.
That is, if the maximum degree $d$ over all vertices in a sequence is greater than $4$, we can always reduce the length of time for which vertices of degree $d$ appear in $W$, hence we can always find a sequence of maximum vertex degree at most $4$.
To prove this claim, let
\[ s_{i_1} \xrightarrow{w_{i_1}} \cdots \xrightarrow{w_{i_T-1}} s_{i_{T}} \]
be a subsequence of $W$ of length $T$ containing a vertex $v$ of degree $d$ throughout.
Then $v$ must have degree $d-1$ in the spine $s_{i_1-1}$ immediately preceding $s_{i_1}$, and also in the spine $s_{i_T+1}$ immediately following $s_{i_T}$.
\begin{equation}
\begin{array}{cccccccccc}
 & s_{i_1-1} & \xrightarrow{w_{i_1-1}} & s_{i_1} & \xrightarrow{w_{i_1}} & \cdots & \xrightarrow{w_{i_T-1}} & s_{i_T} & \xrightarrow{w_{i_T}} & s_{i_T+1} \\
\scriptstyle \deg v & \scriptstyle d-1 & & \scriptstyle d & & \scriptstyle \cdots & & \scriptstyle d & & \scriptstyle d-1
\end{array}
\label{eq:d-subseq}
\end{equation}
Thus $w_{i_1-1}$ and $w_{i_T}$ are moves that respectively increase and decrease the degree of $v$ by $1$.
Since there are a finite number of possibilities for such moves, we can go through each and show that we can always rewrite the subsequence \eqref{eq:d-subseq} in order to reduce the length of time for which $v$ has degree $d$.
The crucial fact that allows us to do this is that none of the moves $w_{i_1}, \dotsc, w_{i_T-1}$ change the degree of $v$, which means that each is either some move that does not involve $v$ or any of the edges adjacent to it, or else is a Type II creation/cancellation.
Even so the proof that we can always accomplish such a rewriting is a lengthy case-bash, which we relegate to Appendix A.
Hence by rewriting every subsequence of length $T$ in $W$ which contains a vertex of degree $d>4$ throughout, we reduce the total complexity of $W$.
Iterating this process we eventually obtain a sequence $Y$ taking $s$ to $s^\prime$, for which
\[ \Phi(Y) = (4,T). \]

It finally remains to show that we can rewrite $Y$ as a sequence $Z$ with complexity
\[ \Phi(Z) = (4,1), \]
the proof of which we omit.
Then such a sequence $Z$ is in fact a sequence of HI moves, since every subsequence
\[ s_{i_1} \rightarrow \cdots \rightarrow s_{i_T} \]
either does not contain a vertex of degree $4$, or else is of the form
\[ s_{i_1} \xrightarrow{w_{i_1}} s_{i_2} \xrightarrow{w_{i_2}} s_{i_3} \]
where $s_{i_1}, s_{i_3}$ contain no vertices of degree $4$ while $s_{i_2}$ contains exactly one vertex of degree $4$.
In the former case the spines $s_{i_j}$ are all isotopic, and in the latter case $w_{i_1}$ and $w_{i_2}$ are necessarily slides, and either $s_{i_1},s_{i_2},s_{i_3}$ are all isotopic, or else the subsequence is of the form
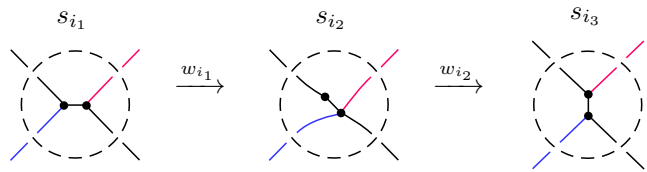
\begin{equation*}
\begin{gathered}
s_{i_1} \\
\begin{gathered}
\psscalebox{1.0 1.0} % Change this value to rescale the drawing.
{
\begin{pspicture}(0,-0.80035293)(1.7121252,0.80035293)
\definecolor{colour0}{rgb}{0.2,0.2,1.0}
\definecolor{colour1}{rgb}{1.0,0.0,0.4}
\psline[linecolor=black, linewidth=0.02](0.7131964,0.0)(1.0011437,0.0)
\psline[linecolor=black, linewidth=0.02](0.011614411,0.73293537)(0.30939218,0.42626873)
\psline[linecolor=colour0, linewidth=0.02](0.007169967,-0.73373127)(0.31828108,-0.41373128)
\psline[linecolor=black, linewidth=0.02](1.7049477,-0.72928685)(1.4116144,-0.4270646)
\psline[linecolor=colour1, linewidth=0.02](1.7049477,0.73737985)(1.4160589,0.43960205)
\psline[linecolor=black, linewidth=0.02](0.99828106,-0.0026201655)(1.4293922,-0.45150906)
\psline[linecolor=colour0, linewidth=0.02](0.71605885,-0.0026202332)(0.28494775,-0.44706467)
\psline[linecolor=colour1, linewidth=0.02](0.99828106,0.0062686554)(1.4160589,0.44626865)
\psline[linecolor=black, linewidth=0.02](0.71605885,0.0062687234)(0.28494775,0.45071316)
\pscircle[linecolor=black, linewidth=0.02, fillstyle=solid,fillcolor=black, dimen=outer](0.7111203,-0.00599374){0.056338027}
\pscircle[linecolor=black, linewidth=0.02, fillstyle=solid,fillcolor=black, dimen=outer](1.006895,-0.00599374){0.056338027}
\rput{-80.12028}(0.7100964,0.8444583){\pscircle[linecolor=white, linewidth=0.18, dimen=outer](0.85717,0.0){0.80035293}}
\rput{-90.0}(0.85272557,0.8684376){\pscircle[linecolor=black, linewidth=0.02, linestyle=dashed, dash=0.17638889cm 0.10583334cm, dimen=outer](0.8605816,0.007856025){0.72}}
\end{pspicture}
}
\end{gathered}
\end{gathered}
\quad \xrightarrow{w_{i_1}} \quad
\begin{gathered}
s_{i_2} \\
\begin{gathered}
\psscalebox{1.0 1.0} % Change this value to rescale the drawing.
{
\begin{pspicture}(0,-0.80035293)(1.7121252,0.80035293)
\definecolor{colour0}{rgb}{0.2,0.2,1.0}
\definecolor{colour1}{rgb}{1.0,0.0,0.4}
\psbezier[linecolor=colour0, linewidth=0.02](0.37062737,-0.36726135)(0.53562737,-0.20226134)(0.8356274,-0.13226135)(0.9556274,-0.11726134)
\psbezier[linecolor=colour1, linewidth=0.02](0.97062737,-0.07226135)(1.0706273,0.057738654)(1.1956273,0.23273866)(1.3456273,0.37273866)
\psline[linecolor=black, linewidth=0.02](0.73697454,0.11838881)(0.93736553,-0.08838905)
\psline[linecolor=black, linewidth=0.02](0.011614398,0.73293537)(0.30939218,0.42626873)
\psline[linecolor=colour0, linewidth=0.02](0.007169954,-0.73373127)(0.31828105,-0.41373128)
\psline[linecolor=black, linewidth=0.02](1.7049477,-0.72928685)(1.4116144,-0.4270646)
\psline[linecolor=colour1, linewidth=0.02](1.7049477,0.73737985)(1.4160589,0.43960205)
\pscircle[linecolor=black, linewidth=0.02, fillstyle=solid,fillcolor=black, dimen=outer](0.74112034,0.10900626){0.056338027}
\pscircle[linecolor=black, linewidth=0.02, fillstyle=solid,fillcolor=black, dimen=outer](0.951895,-0.10599374){0.056338027}
\rput{-80.12028}(0.7100964,0.8444582){\pscircle[linecolor=white, linewidth=0.18, dimen=outer](0.85717,0.0){0.80035293}}
\rput{-90.0}(0.8527255,0.8684376){\pscircle[linecolor=black, linewidth=0.02, linestyle=dashed, dash=0.17638889cm 0.10583334cm, dimen=outer](0.8605816,0.007856025){0.72}}
\psbezier[linecolor=black, linewidth=0.02](1.3556274,-0.38226134)(1.2106273,-0.24726135)(1.1006274,-0.19726135)(0.9656274,-0.107261345)
\psbezier[linecolor=black, linewidth=0.02](0.74562734,0.102738656)(0.53562737,0.23273866)(0.5106274,0.24773866)(0.36562738,0.38773865)
\end{pspicture}
}
\end{gathered}
\end{gathered}
\quad \xrightarrow{w_{i_2}} \quad
\begin{gathered}
s_{i_3} \\
\begin{gathered}
\psscalebox{1.0 1.0} % Change this value to rescale the drawing.
{
\begin{pspicture}(0,-0.85606265)(1.6007059,0.85606265)
\definecolor{colour1}{rgb}{1.0,0.0,0.4}
\definecolor{colour0}{rgb}{0.2,0.2,1.0}
\psline[linecolor=colour1, linewidth=0.02](0.8029731,0.14221843)(1.2518619,0.5733295)
\psline[linecolor=colour0, linewidth=0.02](0.7940842,-0.1400038)(0.34963974,-0.5711149)
\psline[linecolor=black, linewidth=0.02](0.79408425,0.14221843)(0.35408425,0.5599962)
\psline[linecolor=black, linewidth=0.02](0.80297315,-0.1400038)(1.2474176,-0.5711149)
\psline[linecolor=black, linewidth=0.02](0.06297309,0.8488851)(0.36075085,0.5599962)
\psline[linecolor=black, linewidth=0.02](1.5340842,-0.8488927)(1.2140841,-0.5377816)
\psline[linecolor=colour0, linewidth=0.02](0.06741753,-0.84444827)(0.3740842,-0.54667044)
\psline[linecolor=black, linewidth=0.02](0.80035305,-0.14286627)(0.80035305,0.145081)
\pscircle[linecolor=black, linewidth=0.02, fillstyle=solid,fillcolor=black, dimen=outer](0.80053353,0.14181936){0.056338027}
\pscircle[linecolor=black, linewidth=0.02, fillstyle=solid,fillcolor=black, dimen=outer](0.80053353,-0.14425659){0.056338027}
\psline[linecolor=colour1, linewidth=0.02](1.5296397,0.8488851)(1.2274176,0.55555177)
\rput{9.87972}(0.0120591335,-0.13730845){\pscircle[linecolor=white, linewidth=0.18, dimen=outer](0.80035293,0.0011073303){0.80035293}}
\pscircle[linecolor=black, linewidth=0.02, linestyle=dashed, dash=0.17638889cm 0.10583334cm, dimen=outer](0.7924969,0.004518911){0.72}
\end{pspicture}
}
\end{gathered}
\end{gathered}
\end{equation*}
and is a HI move.
Thus any trivalent spines $s$ and $s^\prime$ for $(\Sigma,\emptyset)$ are related by a sequence of HI moves.

Suppose now that the marked boundary $M$ is nonempty, and $s, s^\prime$ are spines for $(\Sigma,M)$.
By removing from $s$ and $s^\prime$ all vertices in $M$ together with every edge connected to some $v\in M$, we obtain a sequence $Z^\prime$ of HI moves from $s\setminus M$ to $s^\prime \setminus M$, and one can convert this to a sequence $Z$ of HI moves from $s$ to $s^\prime$ by inserting a HI move every time an edge in $Z^\prime$ is slid past the endpoint of some edge connected to a marked boundary point.
\end{proof}

Finally, we define our skein modules for $n$-holed disks.

\begin{defn}
Let $B$ be the set of all spines $s$ for the $n$-holed disk $(\Sigma,M)$, and let $E\xrightarrow{p} B$ be the bundle over $B$ with fiber $p^{-1}(s) = C(s)$.
We define the \textbf{skein module for $(\Sigma,M)$} to be the set $C(\Sigma)$ of all functions
\begin{align*}
f \colon & B \rightarrow E \\
& s \mapsto x \in p^{-1}(s)
\end{align*}
such that for all sequences $Z$ of HI moves from one spine $s$ to another spine $s^\prime$, we have that
\[ \widetilde{Z}(f(s)) = f(s^\prime) \]
where $\widetilde{Z}$ is the isomorphism on $C(s)$ induced by $Z$.
\end{defn}

Then $C(\Sigma) \cong C(s)$ for any fixed $s \in B$, by the map that sends
\begin{align*}
f & \longmapsto f(s), \\
(f \colon s \mapsto x) & \mathrel{\reflectbox{\ensuremath{\longmapsto}}} x.
\end{align*}

For this map to be well-defined we require that the following ``coherence condition'' holds: namely that for all spines $s,s^\prime$ and any two sequences $Z_1$ and $Z_2$ of HI moves sending $s$ to $s^\prime$,
\[ \widetilde{Z_1} = \widetilde{Z_2}. \]
That is, any sequence of HI moves between the same two spines yields the same isomorphism.
This is a consequence of the pentagon (Beidenharn-Elliot) identity for the $q$-$6j$ symbols (see Proposition 10 of \cite{KL1994}).

Defining the skein module $C(\Sigma)$ for $(\Sigma,M)$ as above guarantees its invariance under diffeomorphisms of the surface $\Sigma$.
However in order to do concrete computations one generally chooses a spine $s$ and calculates in $C(s)$.

\appendix
\chapter{Rewriting subsequences to reduce complexity}
In this appendix we give an indication of how one should go about rewriting subsequences as defined by \eqref{eq:d-subseq} in Theorem \ref{thm:skein-module-spine-iso} in order to reduce their complexity.
We prove three of the four cases that need to be considered, leaving the case where $T=1$ and $w_{i_1-1}$ is a slide to the reader.

The following notation for the spine moves will make the case-check a little less painful.

\begin{equation*}
\begin{array}{rccc}
\text{Creation/cancellation (Type I):}\quad
&
\begin{gathered}
\psscalebox{1.0 1.0} % Change this value to rescale the drawing.
{
\begin{pspicture}(0,-0.185)(0.9147945,0.185)
\pscircle[linecolor=black, linewidth=0.02, fillstyle=solid,fillcolor=black, dimen=outer](0.6477786,-0.0018133641){0.068584576}
\rput[bl](0.5457427,0.155){$\scriptstyle \dots$}
\psline[linecolor=black, linewidth=0.02](0.65,0.002327803)(0.0,0.002327803)
\rput[bl](0.5457427,-0.185){$\scriptstyle \dots$}
\rput[bl](0.7847945,-0.069863975){$\scriptstyle v$}
\end{pspicture}
}
\end{gathered}
& \substack{cr_v(u) \\ \longrightarrow \\ \longleftarrow \\ cn(u)} &
\begin{gathered}
\psscalebox{1.0 1.0} % Change this value to rescale the drawing.
{
\begin{pspicture}(0,-0.2562329)(1.2854794,0.2562329)
\psline[linecolor=black, linewidth=0.02](1.1970313,0.07225409)(0.66359574,0.07225409)
\pscircle[linecolor=black, linewidth=0.02, fillstyle=solid,fillcolor=black, dimen=outer](0.6477786,0.06941951){0.068584576}
\pscircle[linecolor=black, linewidth=0.02, fillstyle=solid,fillcolor=black, dimen=outer](1.2126515,0.06910108){0.068584576}
\rput[bl](0.5484824,0.22623287){$\scriptstyle \dots$}
\psline[linecolor=black, linewidth=0.02](0.65,0.07356068)(0.0,0.07356068)
\rput[bl](0.5484824,-0.2562329){$\scriptstyle \dots$}
\pscircle[linecolor=black, linewidth=0.02, fillstyle=solid,fillcolor=black, dimen=outer](1.2126515,0.06910108){0.068584576}
\psline[linecolor=black, linewidth=0.02](1.1970313,0.07225409)(0.66359574,0.07225409)
\pscircle[linecolor=black, linewidth=0.02, fillstyle=solid,fillcolor=black, dimen=outer](1.2126515,0.06910108){0.068584576}
\pscircle[linecolor=black, linewidth=0.02, fillstyle=solid,fillcolor=black, dimen=outer](1.2126515,0.06910108){0.068584576}
\rput[bl](0.5820548,-0.17397356){$\scriptstyle v$}
\rput[bl](1.1354795,-0.17397356){$\scriptstyle u$}
\end{pspicture}
}
\end{gathered} \\\\
\text{Creation/cancellation (Type II):}\quad
&
\begin{gathered}
\psscalebox{1.0 1.0} % Change this value to rescale the drawing.
{
\begin{pspicture}(0,-0.2562329)(1.2323016,0.2562329)
\psline[linecolor=black, linewidth=0.02](1.2323016,0.07428997)(0.7043613,0.07428997)
\pscircle[linecolor=black, linewidth=0.02, fillstyle=solid,fillcolor=black, dimen=outer](0.6451661,0.07656744){0.06787803}
\psline[linecolor=black, linewidth=0.02](0.65000015,0.07738126)(0.0,0.07738126)
\rput[bl](0.5433338,-0.2562329){$\scriptstyle \dots$}
\rput[bl](0.57690626,-0.17397356){$\scriptstyle v$}
\rput[bl](0.5433338,0.22623287){$\scriptstyle \dots$}
\rput[bl](0.9851254,-0.11643932){$\scriptstyle e$}
\end{pspicture}
}
\end{gathered}
& \substack{cr_e(e^\prime,u) \\ \longrightarrow \\ \longleftarrow \\ cn(e^\prime)} &
\begin{gathered}
\psscalebox{1.0 1.0} % Change this value to rescale the drawing.
{
\begin{pspicture}(0,-0.30503377)(1.7814093,0.30503377)
\psline[linecolor=black, linewidth=0.02](1.7814091,0.037589908)(1.2666885,0.037589908)
\psline[linecolor=black, linewidth=0.02](1.2323738,0.037589908)(0.71765316,0.037589908)
\pscircle[linecolor=black, linewidth=0.02, fillstyle=solid,fillcolor=black, dimen=outer](0.65955615,0.04349348){0.06617837}
\pscircle[linecolor=black, linewidth=0.02, fillstyle=solid,fillcolor=black, dimen=outer](1.2334944,0.038591377){0.06617837}
\psline[linecolor=black, linewidth=0.02](0.6499999,0.0416776)(0.0,0.0416776)
\rput[bl](0.5576337,-0.30503377){$\scriptstyle \dots$}
\rput[bl](0.5912061,-0.22277446){$\scriptstyle v$}
\rput[bl](0.5576337,0.17743199){$\scriptstyle \dots$}
\rput[bl](1.5309321,-0.1542813){$\scriptstyle e$}
\rput[bl](0.892576,0.095033765){$\scriptstyle e^\prime$}
\rput[bl](1.1583294,-0.21455528){$\scriptstyle u$}
\end{pspicture}
}
\end{gathered} \\\\
\text{Slide:}\quad
&
\begin{gathered}
\psscalebox{1.0 1.0} % Change this value to rescale the drawing.
{
\begin{pspicture}(0,-0.48519978)(1.1869475,0.48519978)
\psline[linecolor=black, linewidth=0.02](0.21446268,-0.16540268)(0.21446268,0.48519972)
\psline[linecolor=black, linewidth=0.02](0.97322446,-0.15761289)(0.2309,-0.15761289)
\pscircle[linecolor=black, linewidth=0.02, fillstyle=solid,fillcolor=black, dimen=outer](0.21508288,-0.16044746){0.068584576}
\pscircle[linecolor=black, linewidth=0.02, fillstyle=solid,fillcolor=black, dimen=outer](0.98396206,-0.15746301){0.068584576}
\rput{-90.0}(1.2227437,1.0911511){\rput[bl](1.1569475,-0.065796345){$\scriptstyle \dots$}}
\pscircle[linecolor=black, linewidth=0.02, fillstyle=solid,fillcolor=black, dimen=outer](0.21508288,-0.17079635){0.068584576}
\rput{-90.0}(0.065796345,-0.065796345){\rput[bl](0.0,-0.065796345){$\scriptstyle \dots$}}
\rput[bl](0.2752688,0.11480021){$\scriptstyle e$}
\rput[bl](0.5245839,-0.48519978){$\scriptstyle f$}
\end{pspicture}
}
\end{gathered}
& \xrightarrow{sl_f(e)} &
\begin{gathered}
\psscalebox{1.0 1.0} % Change this value to rescale the drawing.
{
\begin{pspicture}(0,-0.48246)(1.1869475,0.48246)
\psline[linecolor=black, linewidth=0.02](0.981586,-0.16814235)(0.981586,0.48246008)
\psline[linecolor=black, linewidth=0.02](0.97322446,-0.15487309)(0.2309,-0.15487309)
\pscircle[linecolor=black, linewidth=0.02, fillstyle=solid,fillcolor=black, dimen=outer](0.21508288,-0.15770768){0.068584576}
\pscircle[linecolor=black, linewidth=0.02, fillstyle=solid,fillcolor=black, dimen=outer](0.98396206,-0.15472323){0.068584576}
\rput{-90.0}(1.220004,1.0938909){\rput[bl](1.1569475,-0.06305655){$\scriptstyle \dots$}}
\pscircle[linecolor=black, linewidth=0.02, fillstyle=solid,fillcolor=black, dimen=outer](0.21508288,-0.16805655){0.068584576}
\rput{-90.0}(0.06305655,-0.06305655){\rput[bl](0.0,-0.06305655){$\scriptstyle \dots$}}
\rput[bl](1.0423921,0.11206055){$\scriptstyle e$}
\rput[bl](0.5245839,-0.48246){$\scriptstyle f$}
\end{pspicture}
}
\end{gathered}
\end{array}
\end{equation*}

\begin{sidewaystable}
\begin{tabular*}{1.1\textwidth}{m{0.55\textwidth}|m{0.55\textwidth}}
\multicolumn{2}{c}{\textbf{Case 1.} $w_{i_1-1}=$ Creation, $w_{i_T}=$ Cancellation.} \\ \hline
Old subsequence: & Replace with the modification: \\ \hline
% ----- Create & cancel the same edge -----
Create and cancel the same edge via Type I moves. & The moves $w_{i_1}, \dotsc, w_{i_T-1}$ do not involve the created edge-vertex pair in any way, hence we may simply remove $w_{i_1-1}$ and $w_{i_T}$ from the subsequence. \\
$
\begin{array}{cccccccccc}
% Label line
\text{step:} & \scriptstyle s_{i_1-1} & & \scriptstyle s_{i_1} & & & & \scriptstyle s_{i_T} & & \scriptstyle s_{i_T+1} \\\\
% Diagram line
& \begin{gathered}
\psscalebox{1.0 1.0} % Change this value to rescale the drawing.
{
\begin{pspicture}(0,-0.17426474)(0.33607042,0.17426474)
\pscircle[linecolor=black, linewidth=0.02, fillstyle=solid,fillcolor=black, dimen=outer](0.1790413,0.06845801){0.06787803}
\rput{-90.0}(-0.17426474,0.17426474){\rput[bl](0.0,0.17426474){$\scriptstyle \dots$}}
\rput[bl](0.20607042,-0.17426474){$\scriptstyle v$}
\end{pspicture}
}
\end{gathered}
& \xrightarrow{cr_v(u)} &
\begin{gathered}
\psscalebox{1.0 1.0} % Change this value to rescale the drawing.
{
\begin{pspicture}(0,-0.17545515)(0.8362277,0.17545515)
\psline[linecolor=black, linewidth=0.02](0.74777955,0.07077251)(0.21434395,0.07077251)
\pscircle[linecolor=black, linewidth=0.02, fillstyle=solid,fillcolor=black, dimen=outer](0.19852683,0.06793793){0.068584576}
\pscircle[linecolor=black, linewidth=0.02, fillstyle=solid,fillcolor=black, dimen=outer](0.7633997,0.067619495){0.068584576}
\rput[bl](0.13280304,-0.17545515){$\scriptstyle v$}
\rput[bl](0.6862277,-0.17545515){$\scriptstyle u$}
\rput{-90.0}(-0.17545515,0.17545515){\rput[bl](0.0,0.17545515){$\scriptstyle \dots$}}
\end{pspicture}
}
\end{gathered}
& \xrightarrow{w_{i_1}} &
\cdots
& \xrightarrow{w_{i_T-1}} &
\begin{gathered}
\psscalebox{1.0 1.0} % Change this value to rescale the drawing.
{
\begin{pspicture}(0,-0.17545515)(0.8362277,0.17545515)
\psline[linecolor=black, linewidth=0.02](0.74777955,0.07077251)(0.21434395,0.07077251)
\pscircle[linecolor=black, linewidth=0.02, fillstyle=solid,fillcolor=black, dimen=outer](0.19852683,0.06793793){0.068584576}
\pscircle[linecolor=black, linewidth=0.02, fillstyle=solid,fillcolor=black, dimen=outer](0.7633997,0.067619495){0.068584576}
\rput[bl](0.13280304,-0.17545515){$\scriptstyle v$}
\rput[bl](0.6862277,-0.17545515){$\scriptstyle u$}
\rput{-90.0}(-0.17545515,0.17545515){\rput[bl](0.0,0.17545515){$\scriptstyle \dots$}}
\end{pspicture}
}
\end{gathered}
& \xrightarrow{cn(u)} &
\begin{gathered}
\psscalebox{1.0 1.0} % Change this value to rescale the drawing.
{
\begin{pspicture}(0,-0.17426474)(0.33607042,0.17426474)
\pscircle[linecolor=black, linewidth=0.02, fillstyle=solid,fillcolor=black, dimen=outer](0.1790413,0.06845801){0.06787803}
\rput{-90.0}(-0.17426474,0.17426474){\rput[bl](0.0,0.17426474){$\scriptstyle \dots$}}
\rput[bl](0.20607042,-0.17426474){$\scriptstyle v$}
\end{pspicture}
}
\end{gathered} \\\\
% Degree line
\deg v: & \scriptstyle d-1 & & \scriptstyle d & & \scriptstyle \cdots & & \scriptstyle d & & \scriptstyle d-1
\end{array}
$
&
% ----- Create & cancel the same edge MODIFICATION -----
$
\begin{array}{ccccc}
% Label line
\scriptstyle s_{i_1-1} & & & & \scriptstyle s_{i_T+1} \\\\
% Diagram line
\begin{gathered}
\psscalebox{1.0 1.0} % Change this value to rescale the drawing.
{
\begin{pspicture}(0,-0.17426474)(0.33607042,0.17426474)
\pscircle[linecolor=black, linewidth=0.02, fillstyle=solid,fillcolor=black, dimen=outer](0.1790413,0.06845801){0.06787803}
\rput{-90.0}(-0.17426474,0.17426474){\rput[bl](0.0,0.17426474){$\scriptstyle \dots$}}
\rput[bl](0.20607042,-0.17426474){$\scriptstyle v$}
\end{pspicture}
}
\end{gathered}
& \xrightarrow{w_{i_1}} &
\cdots
& \xrightarrow{w_{i_T-1}} &
\begin{gathered}
\psscalebox{1.0 1.0} % Change this value to rescale the drawing.
{
\begin{pspicture}(0,-0.17426474)(0.33607042,0.17426474)
\pscircle[linecolor=black, linewidth=0.02, fillstyle=solid,fillcolor=black, dimen=outer](0.1790413,0.06845801){0.06787803}
\rput{-90.0}(-0.17426474,0.17426474){\rput[bl](0.0,0.17426474){$\scriptstyle \dots$}}
\rput[bl](0.20607042,-0.17426474){$\scriptstyle v$}
\end{pspicture}
}
\end{gathered} \\\\
% Degree line
\scriptstyle d-1 & & \scriptstyle \cdots & & \scriptstyle d-1
\end{array}
$
\\ \hline
% ----- Create, cancel different edges -----
Create an edge and cancel a different one via Type I moves.
&
Again, $w_{i_1}, \dotsc, w_{i_T-1}$ do not involve the created edge-vertex pair, hence we may make the following modification. \\
% ----- Pre-modification sequence -----
$
\begin{array}{ccccccccc}
% LABEL LINE
\scriptstyle s_{i_1-1} && \scriptstyle s_{i_1} &&&& \scriptstyle s_{i_T} && \scriptstyle s_{i_T+1} \\\\
% DIAGRAM LINE
\begin{gathered}
\psscalebox{1.0 1.0} % Change this value to rescale the drawing.
{
\begin{pspicture}(0,-0.17426474)(0.33607042,0.17426474)
\pscircle[linecolor=black, linewidth=0.02, fillstyle=solid,fillcolor=black, dimen=outer](0.1790413,0.06845801){0.06787803}
\rput{-90.0}(-0.17426474,0.17426474){\rput[bl](0.0,0.17426474){$\scriptstyle \dots$}}
\rput[bl](0.20607042,-0.17426474){$\scriptstyle v$}
\end{pspicture}
}
\end{gathered}
& \xrightarrow{cr_v(u)} &
\begin{gathered}
\psscalebox{1.0 1.0} % Change this value to rescale the drawing.
{
\begin{pspicture}(0,-0.17545515)(0.8362277,0.17545515)
\psline[linecolor=black, linewidth=0.02](0.74777955,0.07077251)(0.21434395,0.07077251)
\pscircle[linecolor=black, linewidth=0.02, fillstyle=solid,fillcolor=black, dimen=outer](0.19852683,0.06793793){0.068584576}
\pscircle[linecolor=black, linewidth=0.02, fillstyle=solid,fillcolor=black, dimen=outer](0.7633997,0.067619495){0.068584576}
\rput[bl](0.13280304,-0.17545515){$\scriptstyle v$}
\rput[bl](0.6862277,-0.17545515){$\scriptstyle u$}
\rput{-90.0}(-0.17545515,0.17545515){\rput[bl](0.0,0.17545515){$\scriptstyle \dots$}}
\end{pspicture}
}
\end{gathered}
& \xrightarrow{w_{i_1}} &
\cdots
& \xrightarrow{w_{i_T-1}} &
\begin{gathered}
\psscalebox{1.0 1.0} % Change this value to rescale the drawing.
{
\begin{pspicture}(0,-0.2562329)(1.2458904,0.2562329)
\psline[linecolor=black, linewidth=0.02](1.1574422,0.083212994)(0.6240067,0.083212994)
\pscircle[linecolor=black, linewidth=0.02, fillstyle=solid,fillcolor=black, dimen=outer](0.6081896,0.08037841){0.068584576}
\pscircle[linecolor=black, linewidth=0.02, fillstyle=solid,fillcolor=black, dimen=outer](1.1730624,0.08005998){0.068584576}
\rput[bl](0.54246575,-0.16301467){$\scriptstyle v$}
\rput[bl](1.0958904,-0.16301467){$\scriptstyle u$}
\rput[bl](0.5088934,-0.2562329){$\scriptstyle \dots$}
\rput[bl](0.5088934,0.22623287){$\scriptstyle \dots$}
\psline[linecolor=black, linewidth=0.02](0.6176003,0.07817162)(0.102879696,0.07817162)
\pscircle[linecolor=black, linewidth=0.02, fillstyle=solid,fillcolor=black, dimen=outer](0.0696856,0.07917309){0.06617837}
\rput[bl](0.0,-0.21232973){$\scriptstyle u^\prime$}
\end{pspicture}
}
\end{gathered}
& \xrightarrow{cn(u^\prime)} &
\begin{gathered}
\psscalebox{1.0 1.0} % Change this value to rescale the drawing.
{
\begin{pspicture}(0,-0.17545515)(0.8362277,0.17545515)
\psline[linecolor=black, linewidth=0.02](0.74777955,0.07077251)(0.21434395,0.07077251)
\pscircle[linecolor=black, linewidth=0.02, fillstyle=solid,fillcolor=black, dimen=outer](0.19852683,0.06793793){0.068584576}
\pscircle[linecolor=black, linewidth=0.02, fillstyle=solid,fillcolor=black, dimen=outer](0.7633997,0.067619495){0.068584576}
\rput[bl](0.13280304,-0.17545515){$\scriptstyle v$}
\rput[bl](0.6862277,-0.17545515){$\scriptstyle u$}
\rput{-90.0}(-0.17545515,0.17545515){\rput[bl](0.0,0.17545515){$\scriptstyle \dots$}}
\end{pspicture}
}
\end{gathered} \\\\
% DEGREE LINE
\scriptstyle d-1 & & \scriptstyle d & & \scriptstyle \cdots & & \scriptstyle d & & \scriptstyle d-1
\end{array}
$
&
% ----- Modified sequence -----
$\begin{array}{ccccccccc}
{\scriptstyle s_{i_{1}-1}} &  &  &  &  &  &  &  & {\scriptstyle s_{i_{T}+1}}\\
\\
\begin{gathered}
\psscalebox{1.0 1.0} % Change this value to rescale the drawing.
{
\begin{pspicture}(0,-0.17426474)(0.33607042,0.17426474)
\pscircle[linecolor=black, linewidth=0.02, fillstyle=solid,fillcolor=black, dimen=outer](0.1790413,0.06845801){0.06787803}
\rput{-90.0}(-0.17426474,0.17426474){\rput[bl](0.0,0.17426474){$\scriptstyle \dots$}}
\rput[bl](0.20607042,-0.17426474){$\scriptstyle v$}
\end{pspicture}
}
\end{gathered}
 & \xrightarrow{w_{i_{1}}} & \cdots & \xrightarrow{w_{i_{T}-1}} & 
\begin{gathered}
\psscalebox{1.0 1.0} % Change this value to rescale the drawing.
{
\begin{pspicture}(0,-0.20559214)(0.8834983,0.20559214)
\psline[linecolor=black, linewidth=0.02](0.6630855,0.10408729)(0.1483649,0.10408729)
\pscircle[linecolor=black, linewidth=0.02, fillstyle=solid,fillcolor=black, dimen=outer](0.09026789,0.109990865){0.06617837}
\pscircle[linecolor=black, linewidth=0.02, fillstyle=solid,fillcolor=black, dimen=outer](0.66420615,0.10508876){0.06617837}
\rput[bl](0.0,-0.18915378){$\scriptstyle u^\prime$}
\rput{-90.0}(0.6479062,1.0590905){\rput[bl](0.85349834,0.20559214){$\scriptstyle \dots$}}
\rput[bl](0.6082192,-0.20559214){$\scriptstyle v$}
\end{pspicture}
}
\end{gathered}
 & \xrightarrow{cn(u^{\prime})} & 
\begin{gathered}
\psscalebox{1.0 1.0} % Change this value to rescale the drawing.
{
\begin{pspicture}(0,-0.17426474)(0.33607042,0.17426474)
\pscircle[linecolor=black, linewidth=0.02, fillstyle=solid,fillcolor=black, dimen=outer](0.1790413,0.06845801){0.06787803}
\rput{-90.0}(-0.17426474,0.17426474){\rput[bl](0.0,0.17426474){$\scriptstyle \dots$}}
\rput[bl](0.20607042,-0.17426474){$\scriptstyle v$}
\end{pspicture}
}
\end{gathered} 
 & \xrightarrow{cr_{v}(u)} & 
\begin{gathered}
\psscalebox{1.0 1.0} % Change this value to rescale the drawing.
{
\begin{pspicture}(0,-0.17545515)(0.8362277,0.17545515)
\psline[linecolor=black, linewidth=0.02](0.74777955,0.07077251)(0.21434395,0.07077251)
\pscircle[linecolor=black, linewidth=0.02, fillstyle=solid,fillcolor=black, dimen=outer](0.19852683,0.06793793){0.068584576}
\pscircle[linecolor=black, linewidth=0.02, fillstyle=solid,fillcolor=black, dimen=outer](0.7633997,0.067619495){0.068584576}
\rput[bl](0.13280304,-0.17545515){$\scriptstyle v$}
\rput[bl](0.6862277,-0.17545515){$\scriptstyle u$}
\rput{-90.0}(-0.17545515,0.17545515){\rput[bl](0.0,0.17545515){$\scriptstyle \dots$}}
\end{pspicture}
}
\end{gathered}
 \\
\\
{\scriptstyle d-1} &  & {\scriptstyle \cdots} &  & {\scriptstyle d-1} &  & {\scriptstyle d-2} &  & {\scriptstyle d-1}
\end{array}$
\\ \hline \\
\multicolumn{2}{c}{Note that we do not consider Type II creation/cancellations for $w_{i_1-1},w_{i_T}$ as these do not change the degree of $v$.} \\
\end{tabular*}
\end{sidewaystable}

\begin{sidewaystable}
\begin{tabular*}{1.1\textwidth}{m{0.55\textwidth}|m{0.55\textwidth}}
\multicolumn{2}{c}{\textbf{Case 2.} $w_{i_1-1}=$ Creation, $w_{i_T}=$ Slide.} \\ \hline
Old subsequence: & Replace with the modification: \\ \hline
Create an edge with a Type I move, and slide the created edge off. & First perform $w_{i_1}, \dotsc, w_{i_T-1}$, then create an edge on $u^\prime$.\\
$\begin{array}{ccccccccc}
{\scriptstyle s_{i_{1}-1}} &  & {\scriptstyle s_{i_{1}}} &  &  &  & {\scriptstyle s_{i_{T}}} &  & {\scriptstyle s_{i_{T}+1}}\\
\\
\begin{gathered}
\psscalebox{1.0 1.0} % Change this value to rescale the drawing.
{
\begin{pspicture}(0,-0.17426474)(0.33607042,0.17426474)
\pscircle[linecolor=black, linewidth=0.02, fillstyle=solid,fillcolor=black, dimen=outer](0.1790413,0.06845801){0.06787803}
\rput{-90.0}(-0.17426474,0.17426474){\rput[bl](0.0,0.17426474){$\scriptstyle \dots$}}
\rput[bl](0.20607042,-0.17426474){$\scriptstyle v$}
\end{pspicture}
}
\end{gathered} 
 & \xrightarrow{cr_{v}(u)} & 
\begin{gathered}
\psscalebox{1.0 1.0} % Change this value to rescale the drawing.
{
\begin{pspicture}(0,-0.48924658)(0.718256,0.48924658)
\rput[bl](0.13554277,-0.48924658){$\scriptstyle v$}
\pscircle[linecolor=black, linewidth=0.02, fillstyle=solid,fillcolor=black, dimen=outer](0.20380263,-0.23870556){0.06787803}
\rput{-90.0}(0.11915819,-0.11915819){\rput[bl](0.0,-0.11915819){$\scriptstyle \dots$}}
\pscircle[linecolor=black, linewidth=0.02, fillstyle=solid,fillcolor=black, dimen=outer](0.650378,0.2626643){0.06787803}
\rput[bl](0.5656797,0.37924656){$\scriptstyle u$}
\rput[bl](0.22595371,0.050479453){$\scriptstyle e$}
\psline[linecolor=black, linewidth=0.02](0.20677564,-0.23719178)(0.623214,0.23952055)
\end{pspicture}
}
\end{gathered} 
  & \xrightarrow{w_{i_{1}}} & \cdots & \xrightarrow{w_{i_{T}-1}} & 
\begin{gathered}
\psscalebox{1.0 1.0} % Change this value to rescale the drawing.
{
\begin{pspicture}(0,-0.5148892)(1.2647779,0.5148892)
\rput[bl](0.14824118,-0.4890008){$\scriptstyle v$}
\pscircle[linecolor=black, linewidth=0.02, fillstyle=solid,fillcolor=black, dimen=outer](0.20380263,-0.21404803){0.06787803}
\rput{-90.0}(0.09351559,-0.09351559){\rput[bl](0.0,-0.09351559){$\scriptstyle \dots$}}
\pscircle[linecolor=black, linewidth=0.02, fillstyle=solid,fillcolor=black, dimen=outer](0.650378,0.2883069){0.06787803}
\rput[bl](0.5656797,0.40488917){$\scriptstyle u$}
\rput[bl](0.22595371,0.07612205){$\scriptstyle e$}
\psline[linecolor=black, linewidth=0.02](0.20677564,-0.21154918)(0.623214,0.26516315)
\pscircle[linecolor=black, linewidth=0.02, fillstyle=solid,fillcolor=black, dimen=outer](1.0408659,-0.21404803){0.06787803}
\psline[linecolor=black, linewidth=0.02](0.21982782,-0.20114437)(1.0422997,-0.20114437)
\rput[bl](0.9298317,-0.49651977){$\scriptstyle u^\prime$}
\rput{-90.0}(1.3413843,1.1281716){\rput[bl](1.2347779,-0.10660637){$\scriptstyle \dots$}}
\rput[bl](0.54280764,-0.5148892){$\scriptstyle f$}
\end{pspicture}
}
\end{gathered}  
   & \xrightarrow{sl_{f}(e)} & 
   \begin{gathered}
   \psscalebox{1.0 1.0} % Change this value to rescale the drawing.
{
\begin{pspicture}(0,-0.5148892)(1.2647779,0.5148892)
\rput[bl](0.14824118,-0.4890008){$\scriptstyle v$}
\pscircle[linecolor=black, linewidth=0.02, fillstyle=solid,fillcolor=black, dimen=outer](0.20380263,-0.21404803){0.06787803}
\rput{-90.0}(0.09351559,-0.09351559){\rput[bl](0.0,-0.09351559){$\scriptstyle \dots$}}
\pscircle[linecolor=black, linewidth=0.02, fillstyle=solid,fillcolor=black, dimen=outer](0.650378,0.2883069){0.06787803}
\rput[bl](0.5656797,0.40488917){$\scriptstyle u$}
\rput[bl](0.8886403,0.07612205){$\scriptstyle e$}
\pscircle[linecolor=black, linewidth=0.02, fillstyle=solid,fillcolor=black, dimen=outer](1.0408659,-0.21404803){0.06787803}
\psline[linecolor=black, linewidth=0.02](0.21982782,-0.20114437)(1.0422997,-0.20114437)
\rput[bl](0.9298317,-0.49651977){$\scriptstyle u^\prime$}
\rput{-90.0}(1.3413843,1.1281716){\rput[bl](1.2347779,-0.10660637){$\scriptstyle \dots$}}
\rput[bl](0.54280764,-0.5148892){$\scriptstyle f$}
\psline[linecolor=black, linewidth=0.02](0.63310975,0.29893947)(1.0271395,-0.19061276)
\end{pspicture}
}
\end{gathered}      
   \\
\\
{\scriptstyle d-1} &  & {\scriptstyle d} &  & {\scriptstyle \cdots} &  & {\scriptstyle d} &  & {\scriptstyle d-1}
\end{array}$
&
$\begin{array}{cccccccc}
{\scriptstyle s_{i_{1}-1}} &  &  &  &  &  & {\scriptstyle s_{i_{T}+1}}\\
\\
\begin{gathered}
 \psscalebox{1.0 1.0} % Change this value to rescale the drawing.
{
\begin{pspicture}(0,-0.17426474)(0.33607042,0.17426474)
\pscircle[linecolor=black, linewidth=0.02, fillstyle=solid,fillcolor=black, dimen=outer](0.1790413,0.06845801){0.06787803}
\rput{-90.0}(-0.17426474,0.17426474){\rput[bl](0.0,0.17426474){$\scriptstyle \dots$}}
\rput[bl](0.20607042,-0.17426474){$\scriptstyle v$}
\end{pspicture}
}
\end{gathered}
& \xrightarrow{w_{i_{1}}} & \cdots & \xrightarrow{w_{i_{T}-1}} & 
\begin{gathered}
\psscalebox{1.0 1.0} % Change this value to rescale the drawing.
{
\begin{pspicture}(0,-0.21068679)(1.2647779,0.21068679)
\rput[bl](0.14824118,-0.1847984){$\scriptstyle v$}
\pscircle[linecolor=black, linewidth=0.02, fillstyle=solid,fillcolor=black, dimen=outer](0.20380263,0.09015436){0.06787803}
\rput{-90.0}(-0.21068679,0.21068679){\rput[bl](0.0,0.21068679){$\scriptstyle \dots$}}
\pscircle[linecolor=black, linewidth=0.02, fillstyle=solid,fillcolor=black, dimen=outer](1.0408659,0.09015436){0.06787803}
\psline[linecolor=black, linewidth=0.02](0.21982782,0.10305801)(1.0422997,0.10305801)
\rput[bl](0.9298317,-0.19231738){$\scriptstyle u^\prime$}
\rput{-90.0}(1.037182,1.432374){\rput[bl](1.2347779,0.19759601){$\scriptstyle \dots$}}
\rput[bl](0.54280764,-0.21068679){$\scriptstyle f$}
\end{pspicture}
}
\end{gathered}
 & \xrightarrow{cr_{u^{\prime}}(u)} & 
\begin{gathered}
\psscalebox{1.0 1.0} % Change this value to rescale the drawing.
{
\begin{pspicture}(0,-0.5148892)(1.2647779,0.5148892)
\rput[bl](0.14824118,-0.4890008){$\scriptstyle v$}
\pscircle[linecolor=black, linewidth=0.02, fillstyle=solid,fillcolor=black, dimen=outer](0.20380263,-0.21404803){0.06787803}
\rput{-90.0}(0.09351559,-0.09351559){\rput[bl](0.0,-0.09351559){$\scriptstyle \dots$}}
\pscircle[linecolor=black, linewidth=0.02, fillstyle=solid,fillcolor=black, dimen=outer](0.650378,0.2883069){0.06787803}
\rput[bl](0.5656797,0.40488917){$\scriptstyle u$}
\rput[bl](0.8886403,0.07612205){$\scriptstyle e$}
\pscircle[linecolor=black, linewidth=0.02, fillstyle=solid,fillcolor=black, dimen=outer](1.0408659,-0.21404803){0.06787803}
\psline[linecolor=black, linewidth=0.02](0.21982782,-0.20114437)(1.0422997,-0.20114437)
\rput[bl](0.9298317,-0.49651977){$\scriptstyle u^\prime$}
\rput{-90.0}(1.3413843,1.1281716){\rput[bl](1.2347779,-0.10660637){$\scriptstyle \dots$}}
\rput[bl](0.54280764,-0.5148892){$\scriptstyle f$}
\psline[linecolor=black, linewidth=0.02](0.63310975,0.29893947)(1.0271395,-0.19061276)
\end{pspicture}
}
\end{gathered} 
 \\
\\
{\scriptstyle d-1} &  & {\scriptstyle {\scriptstyle \cdots}} &  & {\scriptstyle d-1} &  & {\scriptstyle d-1}
\end{array}$
\\ \hline
Create an edge via a Type I move and slide another edge off along the created edge. & First perform $w_{i_1},\dotsc,w_{i_T-1}$, and then perform a Type II creation to extend the edge $e$. \\
$\begin{array}{ccccccccc}
{\scriptstyle s_{i_{1}-1}} &  & {\scriptstyle s_{i_{1}}} &  &  &  & {\scriptstyle s_{i_{T}}} &  & {\scriptstyle s_{i_{T}+1}}\\
\\
\begin{gathered}
\psscalebox{1.0 1.0} % Change this value to rescale the drawing.
{
\begin{pspicture}(0,-0.17426474)(0.33607042,0.17426474)
\pscircle[linecolor=black, linewidth=0.02, fillstyle=solid,fillcolor=black, dimen=outer](0.1790413,0.06845801){0.06787803}
\rput{-90.0}(-0.17426474,0.17426474){\rput[bl](0.0,0.17426474){$\scriptstyle \dots$}}
\rput[bl](0.20607042,-0.17426474){$\scriptstyle v$}
\end{pspicture}
}
\end{gathered}
 & \xrightarrow{cr_{v}(u)} & 
\begin{gathered}
\psscalebox{1.0 1.0} % Change this value to rescale the drawing.
{
\begin{pspicture}(0,-0.20970172)(1.0525466,0.20970172)
\psline[linecolor=black, linewidth=0.02](0.97322446,0.11788519)(0.2309,0.11788519)
\pscircle[linecolor=black, linewidth=0.02, fillstyle=solid,fillcolor=black, dimen=outer](0.21508288,0.11505061){0.068584576}
\pscircle[linecolor=black, linewidth=0.02, fillstyle=solid,fillcolor=black, dimen=outer](0.98396206,0.118035056){0.068584576}
\rput{-90.0}(-0.20970172,0.20970172){\rput[bl](0.0,0.20970172){$\scriptstyle \dots$}}
\rput[bl](0.5245839,-0.20970172){$\scriptstyle f$}
\rput[bl](0.90126055,-0.14275509){$\scriptstyle u$}
\rput[bl](0.16873144,-0.13081479){$\scriptstyle v$}
\end{pspicture}
}
\end{gathered} 
  & \xrightarrow{w_{i_{1}}} & \cdots & \xrightarrow{w_{i_{T}-1}} & 
\begin{gathered}
\psscalebox{1.0 1.0} % Change this value to rescale the drawing.
{
\begin{pspicture}(0,-0.48843014)(1.0525466,0.48843014)
\psline[linecolor=black, linewidth=0.02](0.223377,-0.1621722)(0.223377,0.4884302)
\psline[linecolor=black, linewidth=0.02](0.97322446,-0.16084325)(0.2309,-0.16084325)
\pscircle[linecolor=black, linewidth=0.02, fillstyle=solid,fillcolor=black, dimen=outer](0.21508288,-0.16367783){0.068584576}
\pscircle[linecolor=black, linewidth=0.02, fillstyle=solid,fillcolor=black, dimen=outer](0.98396206,-0.16069338){0.068584576}
\rput{-90.0}(0.06902671,-0.06902671){\rput[bl](0.0,-0.06902671){$\scriptstyle \dots$}}
\rput[bl](0.2841831,0.11803069){$\scriptstyle e$}
\rput[bl](0.5245839,-0.48843014){$\scriptstyle f$}
\rput[bl](0.90126055,-0.42148352){$\scriptstyle u$}
\rput[bl](0.16873144,-0.40954322){$\scriptstyle v$}
\end{pspicture}
}
\end{gathered}  
   & \xrightarrow{sl_{f}(e)} & 
\begin{gathered}
\psscalebox{1.0 1.0} % Change this value to rescale the drawing.
{
\begin{pspicture}(0,-0.49141523)(1.1683623,0.49141523)
\psline[linecolor=black, linewidth=0.02](0.9875561,-0.15918712)(0.9875561,0.4914153)
\psline[linecolor=black, linewidth=0.02](0.97322446,-0.16382833)(0.2309,-0.16382833)
\pscircle[linecolor=black, linewidth=0.02, fillstyle=solid,fillcolor=black, dimen=outer](0.21508288,-0.1666629){0.068584576}
\pscircle[linecolor=black, linewidth=0.02, fillstyle=solid,fillcolor=black, dimen=outer](0.98396206,-0.16367845){0.068584576}
\rput{-90.0}(0.07201178,-0.07201178){\rput[bl](0.0,-0.07201178){$\scriptstyle \dots$}}
\rput[bl](1.0483623,0.121015765){$\scriptstyle e$}
\rput[bl](0.5245839,-0.49141523){$\scriptstyle f$}
\rput[bl](0.90126055,-0.4244686){$\scriptstyle u$}
\rput[bl](0.16873144,-0.4125283){$\scriptstyle v$}
\end{pspicture}
}
\end{gathered}   
   \\
\\
{\scriptstyle d-1} &  & {\scriptstyle d} &  & {\scriptstyle \cdots} &  & {\scriptstyle d} &  & {\scriptstyle d-1}
\end{array}$
&
$\begin{array}{ccccccc}
{\scriptstyle s_{i_{1}-1}} &  &  &  &  &  & {\scriptstyle s_{i_{T}+1}}\\
\\
\begin{gathered}
\psscalebox{1.0 1.0} % Change this value to rescale the drawing.
{
\begin{pspicture}(0,-0.17426474)(0.33607042,0.17426474)
\pscircle[linecolor=black, linewidth=0.02, fillstyle=solid,fillcolor=black, dimen=outer](0.1790413,0.06845801){0.06787803}
\rput{-90.0}(-0.17426474,0.17426474){\rput[bl](0.0,0.17426474){$\scriptstyle \dots$}}
\rput[bl](0.20607042,-0.17426474){$\scriptstyle v$}
\end{pspicture}
}
\end{gathered}
 & \xrightarrow{w_{i_{1}}} & \cdots & \xrightarrow{w_{i_{T}-1}} & 
\begin{gathered}
\psscalebox{1.0 1.0} % Change this value to rescale the drawing.
{
\begin{pspicture}(0,-0.44898668)(0.40418312,0.44898668)
\psline[linecolor=black, linewidth=0.02](0.223377,-0.20161566)(0.223377,0.44898674)
\pscircle[linecolor=black, linewidth=0.02, fillstyle=solid,fillcolor=black, dimen=outer](0.21508288,-0.20312129){0.068584576}
\rput{-90.0}(0.10847017,-0.10847017){\rput[bl](0.0,-0.10847017){$\scriptstyle \dots$}}
\rput[bl](0.2841831,0.07858723){$\scriptstyle e$}
\rput[bl](0.16873144,-0.44898668){$\scriptstyle v$}
\end{pspicture}
}
\end{gathered} 
  & \xrightarrow{cr_{e}(f,u)} & 
\begin{gathered}
\psscalebox{1.0 1.0} % Change this value to rescale the drawing.
{
\begin{pspicture}(0,-0.49141523)(1.1683623,0.49141523)
\psline[linecolor=black, linewidth=0.02](0.9875561,-0.15918712)(0.9875561,0.4914153)
\psline[linecolor=black, linewidth=0.02](0.97322446,-0.16382833)(0.2309,-0.16382833)
\pscircle[linecolor=black, linewidth=0.02, fillstyle=solid,fillcolor=black, dimen=outer](0.21508288,-0.1666629){0.068584576}
\pscircle[linecolor=black, linewidth=0.02, fillstyle=solid,fillcolor=black, dimen=outer](0.98396206,-0.16367845){0.068584576}
\rput{-90.0}(0.07201178,-0.07201178){\rput[bl](0.0,-0.07201178){$\scriptstyle \dots$}}
\rput[bl](1.0483623,0.121015765){$\scriptstyle e$}
\rput[bl](0.5245839,-0.49141523){$\scriptstyle f$}
\rput[bl](0.90126055,-0.4244686){$\scriptstyle u$}
\rput[bl](0.16873144,-0.4125283){$\scriptstyle v$}
\end{pspicture}
}
\end{gathered}  
  \\
\\
{\scriptstyle d-1} &  & {\scriptstyle {\scriptstyle \cdots}} &  & {\scriptstyle d-1} &  & {\scriptstyle d-1}
\end{array}$
\\ \hline
Create an edge via a Type I move and slide another edge off an edge different to the recently-created one. & Replace with the following: \\
$\begin{array}{ccccccccc}
{\scriptstyle s_{i_{1}-1}} &  & {\scriptstyle s_{i_{1}}} &  &  &  & {\scriptstyle s_{i_{T}}} &  & {\scriptstyle s_{i_{T}+1}}\\
\\
\begin{gathered}
\psscalebox{1.0 1.0} % Change this value to rescale the drawing.
{
\begin{pspicture}(0,-0.17426474)(0.33607042,0.17426474)
\pscircle[linecolor=black, linewidth=0.02, fillstyle=solid,fillcolor=black, dimen=outer](0.1790413,0.06845801){0.06787803}
\rput{-90.0}(-0.17426474,0.17426474){\rput[bl](0.0,0.17426474){$\scriptstyle \dots$}}
\rput[bl](0.20607042,-0.17426474){$\scriptstyle v$}
\end{pspicture}
}
\end{gathered}
 & \xrightarrow{cr_{v}(u)} & 
\begin{gathered}
\psscalebox{1.0 1.0} % Change this value to rescale the drawing.
{
\begin{pspicture}(0,-0.20970172)(1.0525466,0.20970172)
\psline[linecolor=black, linewidth=0.02](0.97322446,0.11788519)(0.2309,0.11788519)
\pscircle[linecolor=black, linewidth=0.02, fillstyle=solid,fillcolor=black, dimen=outer](0.21508288,0.11505061){0.068584576}
\pscircle[linecolor=black, linewidth=0.02, fillstyle=solid,fillcolor=black, dimen=outer](0.98396206,0.118035056){0.068584576}
\rput{-90.0}(-0.20970172,0.20970172){\rput[bl](0.0,0.20970172){$\scriptstyle \dots$}}
\rput[bl](0.90126055,-0.14275509){$\scriptstyle u$}
\rput[bl](0.16873144,-0.13081479){$\scriptstyle v$}
\end{pspicture}
}
\end{gathered} 
  & \xrightarrow{w_{i_{1}}} & \cdots & \xrightarrow{w_{i_{T}-1}} & 
\begin{gathered}
\psscalebox{1.0 1.0} % Change this value to rescale the drawing.
{
\begin{pspicture}(0,-0.48918197)(1.5255715,0.48918197)
\psline[linecolor=black, linewidth=0.02](0.6964019,-0.16142038)(0.6964019,0.48918203)
\psline[linecolor=black, linewidth=0.02](1.4462494,-0.16009143)(0.7039249,-0.16009143)
\pscircle[linecolor=black, linewidth=0.02, fillstyle=solid,fillcolor=black, dimen=outer](0.6881077,-0.162926){0.068584576}
\pscircle[linecolor=black, linewidth=0.02, fillstyle=solid,fillcolor=black, dimen=outer](1.4569869,-0.15994155){0.068584576}
\rput[bl](0.77511847,0.16057356){$\scriptstyle f$}
\rput[bl](1.3742855,-0.4207317){$\scriptstyle u$}
\rput[bl](0.6417563,-0.4087914){$\scriptstyle v$}
\psline[linecolor=black, linewidth=0.02](0.65000004,-0.16125716)(0.0,-0.16125716)
\rput[bl](0.80519885,-0.03843571){$\scriptstyle \dots$}
\rput[bl](0.5992287,-0.48918197){$\scriptstyle \dots$}
\rput[bl](0.27064082,-0.42748615){$\scriptstyle g$}
\end{pspicture}
}
\end{gathered}  
   & \xrightarrow{sl_{g}(f)} & 
\begin{gathered}
\psscalebox{1.0 1.0} % Change this value to rescale the drawing.
{
\begin{pspicture}(0,-0.24932836)(1.5255715,0.24932836)
\psline[linecolor=black, linewidth=0.02](1.4462494,0.0797622)(0.7039249,0.0797622)
\pscircle[linecolor=black, linewidth=0.02, fillstyle=solid,fillcolor=black, dimen=outer](0.6881077,0.07692762){0.068584576}
\pscircle[linecolor=black, linewidth=0.02, fillstyle=solid,fillcolor=black, dimen=outer](1.4569869,0.07991207){0.068584576}
\rput[bl](1.3742855,-0.18087809){$\scriptstyle u$}
\rput[bl](0.6417563,-0.16893779){$\scriptstyle v$}
\psline[linecolor=black, linewidth=0.02](0.65000004,0.07859646)(0.0,0.07859646)
\rput[bl](0.5962436,0.21932836){$\scriptstyle \dots$}
\rput[bl](0.5992287,-0.24932836){$\scriptstyle \dots$}
\rput[bl](0.27064082,-0.18763252){$\scriptstyle g$}
\end{pspicture}
}
\end{gathered}   
   \\
\\
{\scriptstyle d-1} &  & {\scriptstyle d} &  & {\scriptstyle \cdots} &  & {\scriptstyle d} &  & {\scriptstyle d-1}
\end{array}$
&
$\begin{array}{ccccccccc}
{\scriptstyle s_{i_{1}-1}} &  &  &  &  &  &  &  & {\scriptstyle s_{i_{T}+1}}\\
\\
\begin{gathered}
\psscalebox{1.0 1.0} % Change this value to rescale the drawing.
{
\begin{pspicture}(0,-0.17426474)(0.33607042,0.17426474)
\pscircle[linecolor=black, linewidth=0.02, fillstyle=solid,fillcolor=black, dimen=outer](0.1790413,0.06845801){0.06787803}
\rput{-90.0}(-0.17426474,0.17426474){\rput[bl](0.0,0.17426474){$\scriptstyle \dots$}}
\rput[bl](0.20607042,-0.17426474){$\scriptstyle v$}
\end{pspicture}
}
\end{gathered}
 & \xrightarrow{w_{i_{1}}} & \cdots & \xrightarrow{w_{i_{T}-1}} & 
\begin{gathered}
\psscalebox{1.0 1.0} % Change this value to rescale the drawing.
{
\begin{pspicture}(0,-0.45833406)(0.91511846,0.45833406)
\psline[linecolor=black, linewidth=0.02](0.6964019,-0.1922683)(0.6964019,0.45833412)
\pscircle[linecolor=black, linewidth=0.02, fillstyle=solid,fillcolor=black, dimen=outer](0.6881077,-0.19377393){0.068584576}
\rput[bl](0.77511847,0.12972564){$\scriptstyle f$}
\rput[bl](0.6417563,-0.43963933){$\scriptstyle v$}
\psline[linecolor=black, linewidth=0.02](0.65000004,-0.19210508)(0.0,-0.19210508)
\rput[bl](0.27064082,-0.45833406){$\scriptstyle g$}
\rput{-90.0}(0.9686568,0.7665563){\rput[bl](0.8676066,-0.10105022){$\scriptstyle \dots$}}
\end{pspicture}
}
\end{gathered} 
  & \xrightarrow{sl_{g}(f)} & 
\begin{gathered}
\psscalebox{1.0 1.0} % Change this value to rescale the drawing.
{
\begin{pspicture}(0,-0.17864192)(0.89760655,0.17864192)
\pscircle[linecolor=black, linewidth=0.02, fillstyle=solid,fillcolor=black, dimen=outer](0.6881077,0.08591822){0.068584576}
\rput[bl](0.6417563,-0.15994719){$\scriptstyle v$}
\psline[linecolor=black, linewidth=0.02](0.65000004,0.087587066)(0.0,0.087587066)
\rput[bl](0.27064082,-0.17864192){$\scriptstyle g$}
\rput{-90.0}(0.6889646,1.0462484){\rput[bl](0.8676066,0.17864192){$\scriptstyle \dots$}}
\end{pspicture}
}
\end{gathered}  
   & \xrightarrow{cr_{v}(u)} & 
\begin{gathered}
\psscalebox{1.0 1.0} % Change this value to rescale the drawing.
{
\begin{pspicture}(0,-0.24932836)(1.5255715,0.24932836)
\psline[linecolor=black, linewidth=0.02](1.4462494,0.0797622)(0.7039249,0.0797622)
\pscircle[linecolor=black, linewidth=0.02, fillstyle=solid,fillcolor=black, dimen=outer](0.6881077,0.07692762){0.068584576}
\pscircle[linecolor=black, linewidth=0.02, fillstyle=solid,fillcolor=black, dimen=outer](1.4569869,0.07991207){0.068584576}
\rput[bl](1.3742855,-0.18087809){$\scriptstyle u$}
\rput[bl](0.6417563,-0.16893779){$\scriptstyle v$}
\psline[linecolor=black, linewidth=0.02](0.65000004,0.07859646)(0.0,0.07859646)
\rput[bl](0.5962436,0.21932836){$\scriptstyle \dots$}
\rput[bl](0.5992287,-0.24932836){$\scriptstyle \dots$}
\rput[bl](0.27064082,-0.18763252){$\scriptstyle g$}
\end{pspicture}
}
\end{gathered}   
   \\
\\
{\scriptstyle d-1} &  & {\scriptstyle {\scriptstyle \cdots}} &  & {\scriptstyle d-1} &  & {\scriptstyle d-2} &  & {\scriptstyle d-1}
\end{array}$
\\ \hline
\end{tabular*}
\end{sidewaystable}

\begin{sidewaystable}
\begin{tabular*}{1.1\textwidth}{m{0.5\textwidth}|m{0.5\textwidth}}
\multicolumn{2}{c}{\textbf{Case 3.} $T>1$ and $w_{i_1-1}=$ Slide.} \\ \hline
\multicolumn{2}{l}{Then the subsequence has the following form:} \\
\multicolumn{2}{c}{
$\begin{array}{ccccc}
{\scriptstyle s_{i_{1}-1}} &  & {\scriptstyle s_{i_{1}}} &  & {\scriptstyle s_{i_{2}}}\\
\\
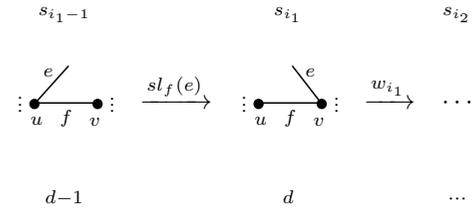
\begin{gathered}
\psscalebox{1.0 1.0} % Change this value to rescale the drawing.
{
\begin{pspicture}(0,-0.40430105)(1.2528377,0.40430105)
\rput[bl](0.14824118,-0.37841266){$\scriptstyle u$}
\pscircle[linecolor=black, linewidth=0.02, fillstyle=solid,fillcolor=black, dimen=outer](0.20380263,-0.10345992){0.06787803}
\rput{-90.0}(-0.017072527,0.017072527){\rput[bl](0.0,0.017072527){$\scriptstyle \dots$}}
\rput[bl](0.3214761,0.25835195){$\scriptstyle e$}
\pscircle[linecolor=black, linewidth=0.02, fillstyle=solid,fillcolor=black, dimen=outer](1.0408659,-0.10345992){0.06787803}
\psline[linecolor=black, linewidth=0.02](0.21982782,-0.09055626)(1.0422997,-0.09055626)
\rput[bl](0.9298317,-0.38593164){$\scriptstyle v$}
\rput{-90.0}(1.2128857,1.2327895){\rput[bl](1.2228377,0.009951896){$\scriptstyle \dots$}}
\rput[bl](0.54280764,-0.40430105){$\scriptstyle f$}
\psline[linecolor=black, linewidth=0.02](0.19728883,-0.10987539)(0.6569903,0.3975873)
\end{pspicture}
}
\end{gathered}
 & \xrightarrow{sl_{f}(e)} & 
\begin{gathered}
\psscalebox{1.0 1.0} % Change this value to rescale the drawing.
{
\begin{pspicture}(0,-0.40402052)(1.2528377,0.40402052)
\rput[bl](0.14824118,-0.37813213){$\scriptstyle u$}
\pscircle[linecolor=black, linewidth=0.02, fillstyle=solid,fillcolor=black, dimen=outer](0.20380263,-0.10317938){0.06787803}
\rput{-90.0}(-0.01735306,0.01735306){\rput[bl](0.0,0.01735306){$\scriptstyle \dots$}}
\rput[bl](0.82296866,0.25266236){$\scriptstyle e$}
\pscircle[linecolor=black, linewidth=0.02, fillstyle=solid,fillcolor=black, dimen=outer](1.0408659,-0.10317938){0.06787803}
\psline[linecolor=black, linewidth=0.02](0.21982782,-0.09027572)(1.0422997,-0.09027572)
\rput[bl](0.9298317,-0.3856511){$\scriptstyle v$}
\rput{-90.0}(1.2126052,1.2330701){\rput[bl](1.2228377,0.010232429){$\scriptstyle \dots$}}
\rput[bl](0.54280764,-0.40402052){$\scriptstyle f$}
\psline[linecolor=black, linewidth=0.02](0.65102017,0.39786783)(1.0331097,-0.09168441)
\end{pspicture}
}
\end{gathered} 
  & \xrightarrow{w_{i_{1}}} & \cdots\\
\\
{\scriptstyle d-1} &  & {\scriptstyle d} &  & {\scriptstyle \cdots}
\end{array}$
} \\
\multicolumn{2}{l}{and there are 3 subcases:} \\ \hline
Old subsequence: & Replace with the modification: \\ \hline
1. $w_{i_1}$ does not involve the edges $e$ or $f$.
&
Then we first perform $w_{i_1}$, then slide $e$ along $f$ to $v$, and continue with $w_{i_2},\dotsc$ as before.
This reduces the length of time for which $v$ has degree $d$ by $1$. \\
$\begin{array}{ccccccc}
{\scriptstyle s_{i_{1}-1}} &  & {\scriptstyle s_{i_{1}}} &  & {\scriptstyle s_{i_{2}}}\\
\\
\begin{gathered}
\psscalebox{1.0 1.0} % Change this value to rescale the drawing.
{
\begin{pspicture}(0,-0.40430105)(1.2528377,0.40430105)
\rput[bl](0.14824118,-0.37841266){$\scriptstyle u$}
\pscircle[linecolor=black, linewidth=0.02, fillstyle=solid,fillcolor=black, dimen=outer](0.20380263,-0.10345992){0.06787803}
\rput{-90.0}(-0.017072527,0.017072527){\rput[bl](0.0,0.017072527){$\scriptstyle \dots$}}
\rput[bl](0.3214761,0.25835195){$\scriptstyle e$}
\pscircle[linecolor=black, linewidth=0.02, fillstyle=solid,fillcolor=black, dimen=outer](1.0408659,-0.10345992){0.06787803}
\psline[linecolor=black, linewidth=0.02](0.21982782,-0.09055626)(1.0422997,-0.09055626)
\rput[bl](0.9298317,-0.38593164){$\scriptstyle v$}
\rput{-90.0}(1.2128857,1.2327895){\rput[bl](1.2228377,0.009951896){$\scriptstyle \dots$}}
\rput[bl](0.54280764,-0.40430105){$\scriptstyle f$}
\psline[linecolor=black, linewidth=0.02](0.19728883,-0.10987539)(0.6569903,0.3975873)
\end{pspicture}
}
\end{gathered}
 & \xrightarrow{sl_{f}(e)} & 
\begin{gathered}
\psscalebox{1.0 1.0} % Change this value to rescale the drawing.
{
\begin{pspicture}(0,-0.40402052)(1.2528377,0.40402052)
\rput[bl](0.14824118,-0.37813213){$\scriptstyle u$}
\pscircle[linecolor=black, linewidth=0.02, fillstyle=solid,fillcolor=black, dimen=outer](0.20380263,-0.10317938){0.06787803}
\rput{-90.0}(-0.01735306,0.01735306){\rput[bl](0.0,0.01735306){$\scriptstyle \dots$}}
\rput[bl](0.82296866,0.25266236){$\scriptstyle e$}
\pscircle[linecolor=black, linewidth=0.02, fillstyle=solid,fillcolor=black, dimen=outer](1.0408659,-0.10317938){0.06787803}
\psline[linecolor=black, linewidth=0.02](0.21982782,-0.09027572)(1.0422997,-0.09027572)
\rput[bl](0.9298317,-0.3856511){$\scriptstyle v$}
\rput{-90.0}(1.2126052,1.2330701){\rput[bl](1.2228377,0.010232429){$\scriptstyle \dots$}}
\rput[bl](0.54280764,-0.40402052){$\scriptstyle f$}
\psline[linecolor=black, linewidth=0.02](0.65102017,0.39786783)(1.0331097,-0.09168441)
\end{pspicture}
}
\end{gathered} 
  & \xrightarrow{w_{i_{1}}} & 
\begin{gathered}
\psscalebox{1.0 1.0} % Change this value to rescale the drawing.
{
\begin{pspicture}(0,-0.40402052)(1.2528377,0.40402052)
\rput[bl](0.14824118,-0.37813213){$\scriptstyle u$}
\pscircle[linecolor=black, linewidth=0.02, fillstyle=solid,fillcolor=black, dimen=outer](0.20380263,-0.10317938){0.06787803}
\rput{-90.0}(-0.01735306,0.01735306){\rput[bl](0.0,0.01735306){$\scriptstyle \dots$}}
\rput[bl](0.82296866,0.25266236){$\scriptstyle e$}
\pscircle[linecolor=black, linewidth=0.02, fillstyle=solid,fillcolor=black, dimen=outer](1.0408659,-0.10317938){0.06787803}
\psline[linecolor=black, linewidth=0.02](0.21982782,-0.09027572)(1.0422997,-0.09027572)
\rput[bl](0.9298317,-0.3856511){$\scriptstyle v$}
\rput{-90.0}(1.2126052,1.2330701){\rput[bl](1.2228377,0.010232429){$\scriptstyle \dots$}}
\rput[bl](0.54280764,-0.40402052){$\scriptstyle f$}
\psline[linecolor=black, linewidth=0.02](0.65102017,0.39786783)(1.0331097,-0.09168441)
\end{pspicture}
}
\end{gathered}  
   & \xrightarrow{w_{i_{2}}} & \cdots\\
\\
{\scriptstyle d-1} &  & {\scriptstyle d} &  & {\scriptstyle d} &  & {\scriptstyle \cdots}
\end{array}$
&
$\begin{array}{ccccccc}
{\scriptstyle s_{i_{1}-1}} &  &  &  & {\scriptstyle s_{i_{2}}}\\
\\
\begin{gathered}
\psscalebox{1.0 1.0} % Change this value to rescale the drawing.
{
\begin{pspicture}(0,-0.40430105)(1.2528377,0.40430105)
\rput[bl](0.14824118,-0.37841266){$\scriptstyle u$}
\pscircle[linecolor=black, linewidth=0.02, fillstyle=solid,fillcolor=black, dimen=outer](0.20380263,-0.10345992){0.06787803}
\rput{-90.0}(-0.017072527,0.017072527){\rput[bl](0.0,0.017072527){$\scriptstyle \dots$}}
\rput[bl](0.3214761,0.25835195){$\scriptstyle e$}
\pscircle[linecolor=black, linewidth=0.02, fillstyle=solid,fillcolor=black, dimen=outer](1.0408659,-0.10345992){0.06787803}
\psline[linecolor=black, linewidth=0.02](0.21982782,-0.09055626)(1.0422997,-0.09055626)
\rput[bl](0.9298317,-0.38593164){$\scriptstyle v$}
\rput{-90.0}(1.2128857,1.2327895){\rput[bl](1.2228377,0.009951896){$\scriptstyle \dots$}}
\rput[bl](0.54280764,-0.40430105){$\scriptstyle f$}
\psline[linecolor=black, linewidth=0.02](0.19728883,-0.10987539)(0.6569903,0.3975873)
\end{pspicture}
}
\end{gathered}
 & \xrightarrow{w_{i_{1}}} & 
 \begin{gathered}
\psscalebox{1.0 1.0} % Change this value to rescale the drawing.
{
\begin{pspicture}(0,-0.40430105)(1.2528377,0.40430105)
\rput[bl](0.14824118,-0.37841266){$\scriptstyle u$}
\pscircle[linecolor=black, linewidth=0.02, fillstyle=solid,fillcolor=black, dimen=outer](0.20380263,-0.10345992){0.06787803}
\rput{-90.0}(-0.017072527,0.017072527){\rput[bl](0.0,0.017072527){$\scriptstyle \dots$}}
\rput[bl](0.3214761,0.25835195){$\scriptstyle e$}
\pscircle[linecolor=black, linewidth=0.02, fillstyle=solid,fillcolor=black, dimen=outer](1.0408659,-0.10345992){0.06787803}
\psline[linecolor=black, linewidth=0.02](0.21982782,-0.09055626)(1.0422997,-0.09055626)
\rput[bl](0.9298317,-0.38593164){$\scriptstyle v$}
\rput{-90.0}(1.2128857,1.2327895){\rput[bl](1.2228377,0.009951896){$\scriptstyle \dots$}}
\rput[bl](0.54280764,-0.40430105){$\scriptstyle f$}
\psline[linecolor=black, linewidth=0.02](0.19728883,-0.10987539)(0.6569903,0.3975873)
\end{pspicture}
}
\end{gathered}
 & \xrightarrow{sl_{f}(e)} & 
\begin{gathered}
\psscalebox{1.0 1.0} % Change this value to rescale the drawing.
{
\begin{pspicture}(0,-0.40402052)(1.2528377,0.40402052)
\rput[bl](0.14824118,-0.37813213){$\scriptstyle u$}
\pscircle[linecolor=black, linewidth=0.02, fillstyle=solid,fillcolor=black, dimen=outer](0.20380263,-0.10317938){0.06787803}
\rput{-90.0}(-0.01735306,0.01735306){\rput[bl](0.0,0.01735306){$\scriptstyle \dots$}}
\rput[bl](0.82296866,0.25266236){$\scriptstyle e$}
\pscircle[linecolor=black, linewidth=0.02, fillstyle=solid,fillcolor=black, dimen=outer](1.0408659,-0.10317938){0.06787803}
\psline[linecolor=black, linewidth=0.02](0.21982782,-0.09027572)(1.0422997,-0.09027572)
\rput[bl](0.9298317,-0.3856511){$\scriptstyle v$}
\rput{-90.0}(1.2126052,1.2330701){\rput[bl](1.2228377,0.010232429){$\scriptstyle \dots$}}
\rput[bl](0.54280764,-0.40402052){$\scriptstyle f$}
\psline[linecolor=black, linewidth=0.02](0.65102017,0.39786783)(1.0331097,-0.09168441)
\end{pspicture}
}
\end{gathered} 
  & \xrightarrow{w_{i_{2}}} & \cdots\\
\\
{\scriptstyle d-1} &  & {\scriptstyle d-1} &  & {\scriptstyle d} &  & {\scriptstyle \cdots}
\end{array}$ \\ \hline
2. $w_{i_1}$ involves $e$. Then $w_{i_1}$ is a Type II creation/cancellation ($\ast$). & Then we first perform ($\ast$) at step $s_{i_1-1}$, then slide the resulting edge $e^\prime$ over $f$ to $v$, then continue with $w_{i_2},\dotsc$. \\
$\begin{array}{ccccccc}
{\scriptstyle s_{i_{1}-1}} &  & {\scriptstyle s_{i_{1}}} &  & {\scriptstyle s_{i_{2}}}\\
\\
\begin{gathered}
\psscalebox{1.0 1.0} % Change this value to rescale the drawing.
{
\begin{pspicture}(0,-0.40430105)(1.2528377,0.40430105)
\rput[bl](0.14824118,-0.37841266){$\scriptstyle u$}
\pscircle[linecolor=black, linewidth=0.02, fillstyle=solid,fillcolor=black, dimen=outer](0.20380263,-0.10345992){0.06787803}
\rput{-90.0}(-0.017072527,0.017072527){\rput[bl](0.0,0.017072527){$\scriptstyle \dots$}}
\rput[bl](0.3214761,0.25835195){$\scriptstyle e$}
\pscircle[linecolor=black, linewidth=0.02, fillstyle=solid,fillcolor=black, dimen=outer](1.0408659,-0.10345992){0.06787803}
\psline[linecolor=black, linewidth=0.02](0.21982782,-0.09055626)(1.0422997,-0.09055626)
\rput[bl](0.9298317,-0.38593164){$\scriptstyle v$}
\rput{-90.0}(1.2128857,1.2327895){\rput[bl](1.2228377,0.009951896){$\scriptstyle \dots$}}
\rput[bl](0.54280764,-0.40430105){$\scriptstyle f$}
\psline[linecolor=black, linewidth=0.02](0.19728883,-0.10987539)(0.6569903,0.3975873)
\end{pspicture}
}
\end{gathered}
 & \xrightarrow{sl_{f}(e)} & 
\begin{gathered}
\psscalebox{1.0 1.0} % Change this value to rescale the drawing.
{
\begin{pspicture}(0,-0.40402052)(1.2528377,0.40402052)
\rput[bl](0.14824118,-0.37813213){$\scriptstyle u$}
\pscircle[linecolor=black, linewidth=0.02, fillstyle=solid,fillcolor=black, dimen=outer](0.20380263,-0.10317938){0.06787803}
\rput{-90.0}(-0.01735306,0.01735306){\rput[bl](0.0,0.01735306){$\scriptstyle \dots$}}
\rput[bl](0.82296866,0.25266236){$\scriptstyle e$}
\pscircle[linecolor=black, linewidth=0.02, fillstyle=solid,fillcolor=black, dimen=outer](1.0408659,-0.10317938){0.06787803}
\psline[linecolor=black, linewidth=0.02](0.21982782,-0.09027572)(1.0422997,-0.09027572)
\rput[bl](0.9298317,-0.3856511){$\scriptstyle v$}
\rput{-90.0}(1.2126052,1.2330701){\rput[bl](1.2228377,0.010232429){$\scriptstyle \dots$}}
\rput[bl](0.54280764,-0.40402052){$\scriptstyle f$}
\psline[linecolor=black, linewidth=0.02](0.65102017,0.39786783)(1.0331097,-0.09168441)
\end{pspicture}
}
\end{gathered} 
  & \xrightarrow{\displaystyle\ast} & 
\begin{gathered}
\psscalebox{1.0 1.0} % Change this value to rescale the drawing.
{
\begin{pspicture}(0,-0.40402052)(1.2528377,0.40402052)
\rput[bl](0.14824118,-0.37813213){$\scriptstyle u$}
\pscircle[linecolor=black, linewidth=0.02, fillstyle=solid,fillcolor=black, dimen=outer](0.20380263,-0.10317938){0.06787803}
\rput{-90.0}(-0.01735306,0.01735306){\rput[bl](0.0,0.01735306){$\scriptstyle \dots$}}
\rput[bl](0.82296866,0.25266236){$\scriptstyle e^\prime$}
\pscircle[linecolor=black, linewidth=0.02, fillstyle=solid,fillcolor=black, dimen=outer](1.0408659,-0.10317938){0.06787803}
\psline[linecolor=black, linewidth=0.02](0.21982782,-0.09027572)(1.0422997,-0.09027572)
\rput[bl](0.9298317,-0.3856511){$\scriptstyle v$}
\rput{-90.0}(1.2126052,1.2330701){\rput[bl](1.2228377,0.010232429){$\scriptstyle \dots$}}
\rput[bl](0.54280764,-0.40402052){$\scriptstyle f$}
\psline[linecolor=black, linewidth=0.02](0.65102017,0.39786783)(1.0331097,-0.09168441)
\end{pspicture}
}
\end{gathered}  
   & \xrightarrow{w_{i_{2}}} & \cdots\\
\\
{\scriptstyle d-1} &  & {\scriptstyle d} &  & {\scriptstyle d} &  & {\scriptstyle \cdots}
\end{array}$
&
$\begin{array}{ccccccc}
{\scriptstyle s_{i_{1}-1}} &  &  &  & {\scriptstyle s_{i_{2}}}\\
\\
\begin{gathered}
\psscalebox{1.0 1.0} % Change this value to rescale the drawing.
{
\begin{pspicture}(0,-0.40430105)(1.2528377,0.40430105)
\rput[bl](0.14824118,-0.37841266){$\scriptstyle u$}
\pscircle[linecolor=black, linewidth=0.02, fillstyle=solid,fillcolor=black, dimen=outer](0.20380263,-0.10345992){0.06787803}
\rput{-90.0}(-0.017072527,0.017072527){\rput[bl](0.0,0.017072527){$\scriptstyle \dots$}}
\rput[bl](0.3214761,0.25835195){$\scriptstyle e$}
\pscircle[linecolor=black, linewidth=0.02, fillstyle=solid,fillcolor=black, dimen=outer](1.0408659,-0.10345992){0.06787803}
\psline[linecolor=black, linewidth=0.02](0.21982782,-0.09055626)(1.0422997,-0.09055626)
\rput[bl](0.9298317,-0.38593164){$\scriptstyle v$}
\rput{-90.0}(1.2128857,1.2327895){\rput[bl](1.2228377,0.009951896){$\scriptstyle \dots$}}
\rput[bl](0.54280764,-0.40430105){$\scriptstyle f$}
\psline[linecolor=black, linewidth=0.02](0.19728883,-0.10987539)(0.6569903,0.3975873)
\end{pspicture}
}
\end{gathered}
 & \xrightarrow{{\displaystyle \ast}} & 
\begin{gathered}
\psscalebox{1.0 1.0} % Change this value to rescale the drawing.
{
\begin{pspicture}(0,-0.40430105)(1.2528377,0.40430105)
\rput[bl](0.14824118,-0.37841266){$\scriptstyle u$}
\pscircle[linecolor=black, linewidth=0.02, fillstyle=solid,fillcolor=black, dimen=outer](0.20380263,-0.10345992){0.06787803}
\rput{-90.0}(-0.017072527,0.017072527){\rput[bl](0.0,0.017072527){$\scriptstyle \dots$}}
\rput[bl](0.3214761,0.25835195){$\scriptstyle e^\prime$}
\pscircle[linecolor=black, linewidth=0.02, fillstyle=solid,fillcolor=black, dimen=outer](1.0408659,-0.10345992){0.06787803}
\psline[linecolor=black, linewidth=0.02](0.21982782,-0.09055626)(1.0422997,-0.09055626)
\rput[bl](0.9298317,-0.38593164){$\scriptstyle v$}
\rput{-90.0}(1.2128857,1.2327895){\rput[bl](1.2228377,0.009951896){$\scriptstyle \dots$}}
\rput[bl](0.54280764,-0.40430105){$\scriptstyle f$}
\psline[linecolor=black, linewidth=0.02](0.19728883,-0.10987539)(0.6569903,0.3975873)
\end{pspicture}
}
\end{gathered} 
  & \xrightarrow{sl_{f}(e^{\prime})} & 
\begin{gathered}
\psscalebox{1.0 1.0} % Change this value to rescale the drawing.
{
\begin{pspicture}(0,-0.40402052)(1.2528377,0.40402052)
\rput[bl](0.14824118,-0.37813213){$\scriptstyle u$}
\pscircle[linecolor=black, linewidth=0.02, fillstyle=solid,fillcolor=black, dimen=outer](0.20380263,-0.10317938){0.06787803}
\rput{-90.0}(-0.01735306,0.01735306){\rput[bl](0.0,0.01735306){$\scriptstyle \dots$}}
\rput[bl](0.82296866,0.25266236){$\scriptstyle e^\prime$}
\pscircle[linecolor=black, linewidth=0.02, fillstyle=solid,fillcolor=black, dimen=outer](1.0408659,-0.10317938){0.06787803}
\psline[linecolor=black, linewidth=0.02](0.21982782,-0.09027572)(1.0422997,-0.09027572)
\rput[bl](0.9298317,-0.3856511){$\scriptstyle v$}
\rput{-90.0}(1.2126052,1.2330701){\rput[bl](1.2228377,0.010232429){$\scriptstyle \dots$}}
\rput[bl](0.54280764,-0.40402052){$\scriptstyle f$}
\psline[linecolor=black, linewidth=0.02](0.65102017,0.39786783)(1.0331097,-0.09168441)
\end{pspicture}
}
\end{gathered}  
 & \xrightarrow{w_{i_{2}}} & \cdots\\
\\
{\scriptstyle d-1} &  & {\scriptstyle d-1} &  & {\scriptstyle d} &  & {\scriptstyle \cdots}
\end{array}$ \\ \hline
\end{tabular*}
\end{sidewaystable}

\begin{sidewaystable}
\begin{tabular*}{1.1\textwidth}{m{0.55\textwidth}|m{0.55\textwidth}}
\hline
3. $w_{i_1}$ involves $f$. Then again $w_{i_1}$ must be a Type II creation/cancellation.
If it is a creation:
&
Then we first perform the creation, and then two successive slides to move $e$ back on to $v$.
Then we continue with $w_{i_2},\dotsc$.
\\
$\begin{array}{ccccccc}
{\scriptstyle s_{i_{1}-1}} &  & {\scriptstyle s_{i_{1}}} &  & {\scriptstyle s_{i_{2}}}\\
\\
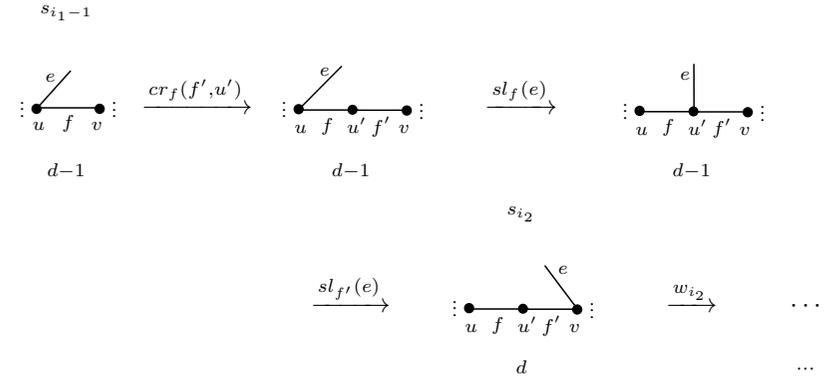
\begin{gathered}
\psscalebox{1.0 1.0} % Change this value to rescale the drawing.
{
\begin{pspicture}(0,-0.40430105)(1.2528377,0.40430105)
\rput[bl](0.14824118,-0.37841266){$\scriptstyle u$}
\pscircle[linecolor=black, linewidth=0.02, fillstyle=solid,fillcolor=black, dimen=outer](0.20380263,-0.10345992){0.06787803}
\rput{-90.0}(-0.017072527,0.017072527){\rput[bl](0.0,0.017072527){$\scriptstyle \dots$}}
\rput[bl](0.3214761,0.25835195){$\scriptstyle e$}
\pscircle[linecolor=black, linewidth=0.02, fillstyle=solid,fillcolor=black, dimen=outer](1.0408659,-0.10345992){0.06787803}
\psline[linecolor=black, linewidth=0.02](0.21982782,-0.09055626)(1.0422997,-0.09055626)
\rput[bl](0.9298317,-0.38593164){$\scriptstyle v$}
\rput{-90.0}(1.2128857,1.2327895){\rput[bl](1.2228377,0.009951896){$\scriptstyle \dots$}}
\rput[bl](0.54280764,-0.40430105){$\scriptstyle f$}
\psline[linecolor=black, linewidth=0.02](0.19728883,-0.10987539)(0.6569903,0.3975873)
\end{pspicture}
}
\end{gathered}
 & \xrightarrow{sl_{f}(e)} & 
\begin{gathered}
\psscalebox{1.0 1.0} % Change this value to rescale the drawing.
{
\begin{pspicture}(0,-0.40402052)(1.2528377,0.40402052)
\rput[bl](0.14824118,-0.37813213){$\scriptstyle u$}
\pscircle[linecolor=black, linewidth=0.02, fillstyle=solid,fillcolor=black, dimen=outer](0.20380263,-0.10317938){0.06787803}
\rput{-90.0}(-0.01735306,0.01735306){\rput[bl](0.0,0.01735306){$\scriptstyle \dots$}}
\rput[bl](0.82296866,0.25266236){$\scriptstyle e$}
\pscircle[linecolor=black, linewidth=0.02, fillstyle=solid,fillcolor=black, dimen=outer](1.0408659,-0.10317938){0.06787803}
\psline[linecolor=black, linewidth=0.02](0.21982782,-0.09027572)(1.0422997,-0.09027572)
\rput[bl](0.9298317,-0.3856511){$\scriptstyle v$}
\rput{-90.0}(1.2126052,1.2330701){\rput[bl](1.2228377,0.010232429){$\scriptstyle \dots$}}
\rput[bl](0.54280764,-0.40402052){$\scriptstyle f$}
\psline[linecolor=black, linewidth=0.02](0.65102017,0.39786783)(1.0331097,-0.09168441)
\end{pspicture}
}
\end{gathered}  
  & \xrightarrow{cr_{f}(f^{\prime},u^{\prime})} & 
\begin{gathered}
\psscalebox{1.0 1.0} % Change this value to rescale the drawing.
{
\begin{pspicture}(0,-0.4621732)(1.8498526,0.4621732)
\rput[bl](0.14824118,-0.41742328){$\scriptstyle u$}
\pscircle[linecolor=black, linewidth=0.02, fillstyle=solid,fillcolor=black, dimen=outer](0.20380263,-0.10461564){0.06787803}
\rput{-90.0}(-0.015916798,0.015916798){\rput[bl](0.0,0.015916798){$\scriptstyle \dots$}}
\rput[bl](1.3841627,0.34077832){$\scriptstyle e$}
\pscircle[linecolor=black, linewidth=0.02, fillstyle=solid,fillcolor=black, dimen=outer](1.6378808,-0.12252609){0.06787803}
\psline[linecolor=black, linewidth=0.02](0.17206661,-0.115592584)(0.99453855,-0.115592584)
\rput[bl](1.5268466,-0.41742328){$\scriptstyle v$}
\rput{-90.0}(1.8289669,1.8107383){\rput[bl](1.8198526,-0.00911428){$\scriptstyle \dots$}}
\rput[bl](0.5010166,-0.43530753){$\scriptstyle f$}
\psline[linecolor=black, linewidth=0.02](1.6360948,-0.11103112)(1.206244,0.45613307)
\psline[linecolor=black, linewidth=0.02](1.574569,-0.11821205)(0.97626626,-0.11821205)
\pscircle[linecolor=black, linewidth=0.02, fillstyle=solid,fillcolor=black, dimen=outer](0.91816926,-0.11230848){0.06617837}
\rput[bl](0.8458118,-0.41742328){$\scriptstyle u^\prime$}
\rput[bl](1.1666882,-0.4621732){$\scriptstyle f^\prime$}
\end{pspicture}
}
\end{gathered}  
   & \xrightarrow{w_{i_{2}}} & \cdots\\
\\
{\scriptstyle d-1} &  & {\scriptstyle d} &  & {\scriptstyle d} &  & {\scriptstyle \cdots}
\end{array}$
&
$\begin{array}{ccccccccc}
{\scriptstyle s_{i_{1}-1}} &  &  &  &  &  & \\
\\
\begin{gathered}
\psscalebox{1.0 1.0} % Change this value to rescale the drawing.
{
\begin{pspicture}(0,-0.40430105)(1.2528377,0.40430105)
\rput[bl](0.14824118,-0.37841266){$\scriptstyle u$}
\pscircle[linecolor=black, linewidth=0.02, fillstyle=solid,fillcolor=black, dimen=outer](0.20380263,-0.10345992){0.06787803}
\rput{-90.0}(-0.017072527,0.017072527){\rput[bl](0.0,0.017072527){$\scriptstyle \dots$}}
\rput[bl](0.3214761,0.25835195){$\scriptstyle e$}
\pscircle[linecolor=black, linewidth=0.02, fillstyle=solid,fillcolor=black, dimen=outer](1.0408659,-0.10345992){0.06787803}
\psline[linecolor=black, linewidth=0.02](0.21982782,-0.09055626)(1.0422997,-0.09055626)
\rput[bl](0.9298317,-0.38593164){$\scriptstyle v$}
\rput{-90.0}(1.2128857,1.2327895){\rput[bl](1.2228377,0.009951896){$\scriptstyle \dots$}}
\rput[bl](0.54280764,-0.40430105){$\scriptstyle f$}
\psline[linecolor=black, linewidth=0.02](0.19728883,-0.10987539)(0.6569903,0.3975873)
\end{pspicture}
}
\end{gathered}
 & \xrightarrow{cr_{f}(f^{\prime},u^{\prime})} & 
\begin{gathered}
\psscalebox{1.0 1.0} % Change this value to rescale the drawing.
{
\begin{pspicture}(0,-0.4656737)(1.8498526,0.4656737)
\rput[bl](0.14824118,-0.4209238){$\scriptstyle u$}
\pscircle[linecolor=black, linewidth=0.02, fillstyle=solid,fillcolor=black, dimen=outer](0.20380263,-0.108116165){0.06787803}
\rput{-90.0}(-0.012416279,0.012416279){\rput[bl](0.0,0.012416279){$\scriptstyle \dots$}}
\rput[bl](0.48267013,0.3372778){$\scriptstyle e$}
\pscircle[linecolor=black, linewidth=0.02, fillstyle=solid,fillcolor=black, dimen=outer](1.6378808,-0.12602662){0.06787803}
\psline[linecolor=black, linewidth=0.02](0.17206661,-0.1190931)(0.99453855,-0.1190931)
\rput[bl](1.5268466,-0.4209238){$\scriptstyle v$}
\rput{-90.0}(1.8324674,1.8072377){\rput[bl](1.8198526,-0.0126148){$\scriptstyle \dots$}}
\rput[bl](0.5010166,-0.43880805){$\scriptstyle f$}
\psline[linecolor=black, linewidth=0.02](0.23310973,-0.08468089)(0.7763933,0.4586027)
\psline[linecolor=black, linewidth=0.02](1.574569,-0.12171257)(0.97626626,-0.12171257)
\pscircle[linecolor=black, linewidth=0.02, fillstyle=solid,fillcolor=black, dimen=outer](0.91816926,-0.115809){0.06617837}
\rput[bl](0.8458118,-0.4209238){$\scriptstyle u^\prime$}
\rput[bl](1.1666882,-0.4656737){$\scriptstyle f^\prime$}
\end{pspicture}
}
\end{gathered} 
  & \xrightarrow{sl_{f}(e)} & 
\begin{gathered}
\psscalebox{1.0 1.0} % Change this value to rescale the drawing.
{
\begin{pspicture}(0,-0.48742482)(1.8498526,0.48742482)
\rput[bl](0.14824118,-0.4426749){$\scriptstyle u$}
\pscircle[linecolor=black, linewidth=0.02, fillstyle=solid,fillcolor=black, dimen=outer](0.20380263,-0.12986726){0.06787803}
\rput{-90.0}(0.00933482,-0.00933482){\rput[bl](0.0,-0.00933482){$\scriptstyle \dots$}}
\rput[bl](0.73938656,0.28567597){$\scriptstyle e$}
\pscircle[linecolor=black, linewidth=0.02, fillstyle=solid,fillcolor=black, dimen=outer](1.6378808,-0.1477777){0.06787803}
\psline[linecolor=black, linewidth=0.02](0.17206661,-0.1408442)(0.99453855,-0.1408442)
\rput[bl](1.5268466,-0.4426749){$\scriptstyle v$}
\rput{-90.0}(1.8542185,1.7854867){\rput[bl](1.8198526,-0.034365896){$\scriptstyle \dots$}}
\rput[bl](0.5010166,-0.46055916){$\scriptstyle f$}
\psline[linecolor=black, linewidth=0.02](0.925647,-0.10046184)(0.9241213,0.48739886)
\psline[linecolor=black, linewidth=0.02](1.574569,-0.14346367)(0.97626626,-0.14346367)
\pscircle[linecolor=black, linewidth=0.02, fillstyle=solid,fillcolor=black, dimen=outer](0.91816926,-0.1375601){0.06617837}
\rput[bl](0.8458118,-0.4426749){$\scriptstyle u^\prime$}
\rput[bl](1.1666882,-0.48742482){$\scriptstyle f^\prime$}
\end{pspicture}
}
\end{gathered}
\\
{\scriptstyle d-1} &  & {\scriptstyle d-1} &  & {\scriptstyle d-1} &  & 
\\
& & & {\scriptstyle s_{i_{2}}} \\\\
& & \xrightarrow{sl_{f^{\prime}}(e)} & 
\begin{gathered}
\psscalebox{1.0 1.0} % Change this value to rescale the drawing.
{
\begin{pspicture}(0,-0.4621732)(1.8498526,0.4621732)
\rput[bl](0.14824118,-0.41742328){$\scriptstyle u$}
\pscircle[linecolor=black, linewidth=0.02, fillstyle=solid,fillcolor=black, dimen=outer](0.20380263,-0.10461564){0.06787803}
\rput{-90.0}(-0.015916798,0.015916798){\rput[bl](0.0,0.015916798){$\scriptstyle \dots$}}
\rput[bl](1.3841627,0.34077832){$\scriptstyle e$}
\pscircle[linecolor=black, linewidth=0.02, fillstyle=solid,fillcolor=black, dimen=outer](1.6378808,-0.12252609){0.06787803}
\psline[linecolor=black, linewidth=0.02](0.17206661,-0.115592584)(0.99453855,-0.115592584)
\rput[bl](1.5268466,-0.41742328){$\scriptstyle v$}
\rput{-90.0}(1.8289669,1.8107383){\rput[bl](1.8198526,-0.00911428){$\scriptstyle \dots$}}
\rput[bl](0.5010166,-0.43530753){$\scriptstyle f$}
\psline[linecolor=black, linewidth=0.02](1.6360948,-0.11103112)(1.206244,0.45613307)
\psline[linecolor=black, linewidth=0.02](1.574569,-0.11821205)(0.97626626,-0.11821205)
\pscircle[linecolor=black, linewidth=0.02, fillstyle=solid,fillcolor=black, dimen=outer](0.91816926,-0.11230848){0.06617837}
\rput[bl](0.8458118,-0.41742328){$\scriptstyle u^\prime$}
\rput[bl](1.1666882,-0.4621732){$\scriptstyle f^\prime$}
\end{pspicture}
}
\end{gathered}  
 & \xrightarrow{w_{i_{2}}} & \cdots
\\
& & & {\scriptstyle d} &  & {\scriptstyle \cdots}
\end{array}$ \\ \hline
If it is a cancellation: & Then we first slide $e$ across $g$, then cancel $f$, then slide $e$ back across $g$ on to $v$, and continue with $w_{i_2},\dotsc$.
\\
$\begin{array}{ccccccc}
{\scriptstyle s_{i_{1}-1}} &  & {\scriptstyle s_{i_{1}}} &  & {\scriptstyle s_{i_{2}}}\\
\\
\begin{gathered}
\psscalebox{1.0 1.0} % Change this value to rescale the drawing.
{
\begin{pspicture}(0,-0.4549811)(1.8498526,0.4549811)
\rput[bl](0.15713006,-0.41023114){$\scriptstyle a$}
\pscircle[linecolor=black, linewidth=0.02, fillstyle=solid,fillcolor=black, dimen=outer](0.20380263,-0.12409018){0.06787803}
\rput{-90.0}(0.0035577407,-0.0035577407){\rput[bl](0.0,-0.0035577407){$\scriptstyle \dots$}}
\rput[bl](0.86489236,0.29908156){$\scriptstyle e$}
\pscircle[linecolor=black, linewidth=0.02, fillstyle=solid,fillcolor=black, dimen=outer](1.6378808,-0.14200063){0.06787803}
\psline[linecolor=black, linewidth=0.02](0.17206661,-0.13506712)(0.99453855,-0.13506712)
\rput[bl](1.5357355,-0.41023114){$\scriptstyle v$}
\rput{-90.0}(1.8484414,1.7912638){\rput[bl](1.8198526,-0.028588818){$\scriptstyle \dots$}}
\rput[bl](0.5099055,-0.4281154){$\scriptstyle g$}
\psline[linecolor=black, linewidth=0.02](0.9308875,-0.10509936)(1.1363933,0.45151755)
\psline[linecolor=black, linewidth=0.02](1.574569,-0.1376866)(0.97626626,-0.1376866)
\pscircle[linecolor=black, linewidth=0.02, fillstyle=solid,fillcolor=black, dimen=outer](0.91816926,-0.13178302){0.06617837}
\rput[bl](0.8547007,-0.41023114){$\scriptstyle u$}
\rput[bl](1.1755772,-0.4549811){$\scriptstyle f$}
\end{pspicture}
}
\end{gathered}
 & \xrightarrow{sl_{f}(e)} & 
\begin{gathered}
\psscalebox{1.0 1.0} % Change this value to rescale the drawing.
{
\begin{pspicture}(0,-0.46493173)(1.8498526,0.46493173)
\rput[bl](0.15713006,-0.42018178){$\scriptstyle a$}
\pscircle[linecolor=black, linewidth=0.02, fillstyle=solid,fillcolor=black, dimen=outer](0.20380263,-0.13404082){0.06787803}
\rput{-90.0}(0.0135083785,-0.0135083785){\rput[bl](0.0,-0.0135083785){$\scriptstyle \dots$}}
\rput[bl](1.4320565,0.32495183){$\scriptstyle e$}
\pscircle[linecolor=black, linewidth=0.02, fillstyle=solid,fillcolor=black, dimen=outer](1.6378808,-0.15195127){0.06787803}
\psline[linecolor=black, linewidth=0.02](0.17206661,-0.14501776)(0.99453855,-0.14501776)
\rput[bl](1.5357355,-0.42018178){$\scriptstyle v$}
\rput{-90.0}(1.858392,1.7813132){\rput[bl](1.8198526,-0.038539458){$\scriptstyle \dots$}}
\rput[bl](0.5099055,-0.43806604){$\scriptstyle g$}
\psline[linecolor=black, linewidth=0.02](1.6473054,-0.13296044)(1.2617664,0.45947737)
\psline[linecolor=black, linewidth=0.02](1.574569,-0.14763723)(0.97626626,-0.14763723)
\pscircle[linecolor=black, linewidth=0.02, fillstyle=solid,fillcolor=black, dimen=outer](0.91816926,-0.14173366){0.06617837}
\rput[bl](0.8547007,-0.42018178){$\scriptstyle u$}
\rput[bl](1.1755772,-0.46493173){$\scriptstyle f$}
\end{pspicture}
}
\end{gathered} 
  & \xrightarrow{cn(f)} & 
\begin{gathered}
\psscalebox{1.0 1.0} % Change this value to rescale the drawing.
{
\begin{pspicture}(0,-0.45746902)(1.1991063,0.45746902)
\rput[bl](0.15713006,-0.43958476){$\scriptstyle a$}
\pscircle[linecolor=black, linewidth=0.02, fillstyle=solid,fillcolor=black, dimen=outer](0.20380263,-0.1534438){0.06787803}
\rput{-90.0}(0.032911364,-0.032911364){\rput[bl](0.0,-0.032911364){$\scriptstyle \dots$}}
\rput[bl](0.78131026,0.31748915){$\scriptstyle e$}
\pscircle[linecolor=black, linewidth=0.02, fillstyle=solid,fillcolor=black, dimen=outer](0.9871346,-0.15941395){0.06787803}
\psline[linecolor=black, linewidth=0.02](0.17206661,-0.16442074)(0.99453855,-0.16442074)
\rput[bl](0.88498926,-0.42764446){$\scriptstyle v$}
\rput{-90.0}(1.2151085,1.1231042){\rput[bl](1.1691064,-0.046002142){$\scriptstyle \dots$}}
\rput[bl](0.5099055,-0.45746902){$\scriptstyle g$}
\psline[linecolor=black, linewidth=0.02](0.99655914,-0.14042313)(0.61102015,0.45201468)
\end{pspicture}
}
\end{gathered}  
   & \xrightarrow{w_{i_{2}}} & \cdots\\
\\
{\scriptstyle d-1} &  & {\scriptstyle d} &  & {\scriptstyle d} &  & {\scriptstyle \cdots}
\end{array}$
&
$\begin{array}{ccccccccc}
{\scriptstyle s_{i_{1}-1}} &  &  &  &  &  & \\
\\
\begin{gathered}
\psscalebox{1.0 1.0} % Change this value to rescale the drawing.
{
\begin{pspicture}(0,-0.4549811)(1.8498526,0.4549811)
\rput[bl](0.15713006,-0.41023114){$\scriptstyle a$}
\pscircle[linecolor=black, linewidth=0.02, fillstyle=solid,fillcolor=black, dimen=outer](0.20380263,-0.12409018){0.06787803}
\rput{-90.0}(0.0035577407,-0.0035577407){\rput[bl](0.0,-0.0035577407){$\scriptstyle \dots$}}
\rput[bl](0.86489236,0.29908156){$\scriptstyle e$}
\pscircle[linecolor=black, linewidth=0.02, fillstyle=solid,fillcolor=black, dimen=outer](1.6378808,-0.14200063){0.06787803}
\psline[linecolor=black, linewidth=0.02](0.17206661,-0.13506712)(0.99453855,-0.13506712)
\rput[bl](1.5357355,-0.41023114){$\scriptstyle v$}
\rput{-90.0}(1.8484414,1.7912638){\rput[bl](1.8198526,-0.028588818){$\scriptstyle \dots$}}
\rput[bl](0.5099055,-0.4281154){$\scriptstyle g$}
\psline[linecolor=black, linewidth=0.02](0.9308875,-0.10509936)(1.1363933,0.45151755)
\psline[linecolor=black, linewidth=0.02](1.574569,-0.1376866)(0.97626626,-0.1376866)
\pscircle[linecolor=black, linewidth=0.02, fillstyle=solid,fillcolor=black, dimen=outer](0.91816926,-0.13178302){0.06617837}
\rput[bl](0.8547007,-0.41023114){$\scriptstyle u$}
\rput[bl](1.1755772,-0.4549811){$\scriptstyle f$}
\end{pspicture}
}
\end{gathered}
 & \xrightarrow{sl_{g}(e)} & 
\begin{gathered}
\psscalebox{1.0 1.0} % Change this value to rescale the drawing.
{
\begin{pspicture}(0,-0.47115532)(1.8498526,0.47115532)
\rput[bl](0.15713006,-0.4264054){$\scriptstyle a$}
\pscircle[linecolor=black, linewidth=0.02, fillstyle=solid,fillcolor=black, dimen=outer](0.20380263,-0.14026442){0.06787803}
\rput{-90.0}(0.019731982,-0.019731982){\rput[bl](0.0,-0.019731982){$\scriptstyle \dots$}}
\rput[bl](0.38131028,0.3187282){$\scriptstyle e$}
\pscircle[linecolor=black, linewidth=0.02, fillstyle=solid,fillcolor=black, dimen=outer](1.6378808,-0.15817487){0.06787803}
\psline[linecolor=black, linewidth=0.02](0.17206661,-0.15124136)(0.99453855,-0.15124136)
\rput[bl](1.5357355,-0.4264054){$\scriptstyle v$}
\rput{-90.0}(1.8646157,1.7750895){\rput[bl](1.8198526,-0.04476306){$\scriptstyle \dots$}}
\rput[bl](0.5099055,-0.44428965){$\scriptstyle g$}
\psline[linecolor=black, linewidth=0.02](0.21446958,-0.1332139)(0.65878135,0.46519408)
\psline[linecolor=black, linewidth=0.02](1.574569,-0.15386084)(0.97626626,-0.15386084)
\pscircle[linecolor=black, linewidth=0.02, fillstyle=solid,fillcolor=black, dimen=outer](0.91816926,-0.14795727){0.06617837}
\rput[bl](0.8547007,-0.4264054){$\scriptstyle u$}
\rput[bl](1.1755772,-0.47115532){$\scriptstyle f$}
\end{pspicture}
}
\end{gathered} 
  & \xrightarrow{cn(f)} & 
\begin{gathered}
\psscalebox{1.0 1.0} % Change this value to rescale the drawing.
{
\begin{pspicture}(0,-0.45753303)(1.1991063,0.45753303)
\rput[bl](0.15713006,-0.43964878){$\scriptstyle a$}
\pscircle[linecolor=black, linewidth=0.02, fillstyle=solid,fillcolor=black, dimen=outer](0.20380263,-0.15350781){0.06787803}
\rput{-90.0}(0.032975376,-0.032975376){\rput[bl](0.0,-0.032975376){$\scriptstyle \dots$}}
\rput[bl](0.35145953,0.31742513){$\scriptstyle e$}
\pscircle[linecolor=black, linewidth=0.02, fillstyle=solid,fillcolor=black, dimen=outer](0.9871346,-0.15947796){0.06787803}
\psline[linecolor=black, linewidth=0.02](0.17206661,-0.16448475)(0.99453855,-0.16448475)
\rput[bl](0.88498926,-0.42770848){$\scriptstyle v$}
\rput{-90.0}(1.2151724,1.1230402){\rput[bl](1.1691064,-0.046066154){$\scriptstyle \dots$}}
\rput[bl](0.5099055,-0.45753303){$\scriptstyle g$}
\psline[linecolor=black, linewidth=0.02](0.22043973,-0.12854683)(0.61102015,0.45195067)
\end{pspicture}
}
\end{gathered}
\\
{\scriptstyle d-1} &  & {\scriptstyle d-1} &  & {\scriptstyle d-1} &  & 
\\
& & & {\scriptstyle s_{i_{2}}} \\\\
& & \xrightarrow{sl_{g}(e)} & 
\begin{gathered}
\psscalebox{1.0 1.0} % Change this value to rescale the drawing.
{
\begin{pspicture}(0,-0.45746902)(1.1991063,0.45746902)
\rput[bl](0.15713006,-0.43958476){$\scriptstyle a$}
\pscircle[linecolor=black, linewidth=0.02, fillstyle=solid,fillcolor=black, dimen=outer](0.20380263,-0.1534438){0.06787803}
\rput{-90.0}(0.032911364,-0.032911364){\rput[bl](0.0,-0.032911364){$\scriptstyle \dots$}}
\rput[bl](0.78131026,0.31748915){$\scriptstyle e$}
\pscircle[linecolor=black, linewidth=0.02, fillstyle=solid,fillcolor=black, dimen=outer](0.9871346,-0.15941395){0.06787803}
\psline[linecolor=black, linewidth=0.02](0.17206661,-0.16442074)(0.99453855,-0.16442074)
\rput[bl](0.88498926,-0.42764446){$\scriptstyle v$}
\rput{-90.0}(1.2151085,1.1231042){\rput[bl](1.1691064,-0.046002142){$\scriptstyle \dots$}}
\rput[bl](0.5099055,-0.45746902){$\scriptstyle g$}
\psline[linecolor=black, linewidth=0.02](0.99655914,-0.14042313)(0.61102015,0.45201468)
\end{pspicture}
}
\end{gathered} 
  & \xrightarrow{w_{i_{2}}} & \cdots
\\
& & & {\scriptstyle d} &  & {\scriptstyle \cdots}
\end{array}$
\\ \hline
\end{tabular*}
\end{sidewaystable}

\cleardoublepage
\phantomsection
\addcontentsline{toc}{chapter}{Bibliography}
\bibliographystyle{alpha}
\bibliography{/home/jayce/Bibtex/thesis}

\end{document}